\documentclass[12pt,oneside,reqno]{amsart}
\usepackage[margin=1in]{geometry}
\usepackage{color}
\usepackage{esint,amssymb}
\usepackage{graphicx}
\usepackage{MnSymbol}
\usepackage{mathtools}
\usepackage[colorlinks=true, pdfstartview=FitV, linkcolor=blue, citecolor=blue, urlcolor=blue,pagebackref=false]{hyperref}
\usepackage{microtype}
\usepackage{amsmath}
\usepackage{slashed}
\usepackage[normalem]{ulem}

\usepackage{relsize}
\usepackage{cancel}
\usepackage{xifthen}
\usepackage{verbatim}
\usepackage{amsxtra}
\definecolor{darkgreen}{rgb}{0,0.5,0}
\definecolor{darkblue}{rgb}{0,0,0.7}
\definecolor{darkred}{rgb}{0.9,0.1,0.1}

\newcommand{\pfstep}[1]{\smallskip \noindent {\it #1.}}

\newtheorem{theorem}{Theorem}
\newtheorem{proposition}[theorem]{Proposition}
\newtheorem{lemma}[theorem]{Lemma}
\newtheorem{corollary}[theorem]{Corollary}

\theoremstyle{definition}
\newtheorem{remark}[theorem]{Remark}

\newtheorem{definition}[theorem]{Definition}

\newtheorem{question}[theorem]{Question}

\newcommand{\cref}[1]{Corollary~\ref{c.#1}}

\numberwithin{equation}{section}
\numberwithin{theorem}{section}

\newcommand{\Z}{\mathbb{Z}}
\newcommand{\N}{\mathbb{N}}
\newcommand{\R}{\mathbb{R}}

\renewcommand{\subset}{\subseteq}

\newcommand{\ep}{\varepsilon}

\newcommand{\test}[1][]{%
\ifthenelse{\equal{#1}{}}{omitted}{given}%
}

\newcommand{\derv}[4]{Y_H^{#1}\partial_{Q_T}^{#2}\partial_H^{#3}\partial_M^{#4}}

\renewcommand{\d}[1]{\ensuremath{\operatorname{d}\!{#1}}}
\DeclarePairedDelimiter{\norm}{\lVert}{\rVert}

\DeclareMathOperator{\supp}{supp}

\DeclareMathOperator{\ini}{in}
\newcommand{\jap}[1]{\left\langle {#1} \right\rangle}
\newcommand{\brk}[1]{\left\langle {#1} \right\rangle}

\newcommand{\D}{\mathfrak{D}}

\renewcommand{\bar}{\overline}

\renewcommand{\part}{\partial}
\DeclareMathOperator{\Kep}{Kep}

\usepackage{dsfont}

\newcommand{\calD}{\mathcal D}
\newcommand{\calE}{\mathcal E}

\newcommand{\calL}{\mathcal L}

\newcommand{\calN}{\mathcal N}

\newcommand{\calR}{\mathcal R}
\newcommand{\calS}{\mathcal S}
\newcommand{\calT}{\mathcal T}

\def\f {\frac}
\def\rd {\partial}
\def\ls {\lesssim}
\def\de {\delta}
\def\i {\infty}
\def\alp {\alpha}
\def\bt {\beta}

\def\ep {\epsilon}
\def\om {\omega}

\def\Omg {\Omega}

\newcommand{\ud}{\mathrm{d}}

\def\wtgmm {\widetilde{\gamma}}

\begin{document}

\title[Phase mixing under an external Kepler potential]{Linear and nonlinear phase mixing \\for the gravitational Vlasov--Poisson system\\ under an external Kepler potential}
\begin{abstract}
In Newtonian gravity, a self-gravitating collisionless gas around a massive object such as a star or a planet is modeled via the Vlasov--Poisson system with an external Kepler potential. The presence of this attractive potential allows for bounded trajectories along which the gas neither falls in towards the object nor escape to infinity. 

We study this system focusing on the regime with bounded trajectories. First, we prove quantitative linear phase mixing estimates in three dimensions outside symmetry. Second, our main result is a long-time nonlinear phase mixing theorem for spherically symmetric data with finite regularity. The mechanism is phenomenologically similar to Landau damping on a torus and our result applies to the same time scale (modulo logarithms) as the known results on Landau damping with finite regularity. However, in contrast with Landau damping, we need to contend with weaker linear estimates as well as use a system of dynamically defined action angle variables.
\end{abstract}

\author[Sanchit Chaturvedi]{Sanchit Chaturvedi}
\address[Sanchit Chaturvedi]{Courant Institute of Mathematical Sciences, New York University, New York, NY 10002, USA}
\email{chaturvedisanchit@nyu.edu}
\author[Jonathan Luk]{Jonathan Luk}
\address[Jonathan Luk]{Department of Mathematics, Stanford University, 450 Jane Stanford Way, Stanford, CA 94305, USA}
\email{jluk@stanford.edu}

\maketitle

\section{Introduction}
Consider the Vlasov--Poisson system in $3$ spatial dimensions:
\begin{equation}\label{eq:transport}
\begin{split}
\mathfrak D f := \rd_s f +  \sum_{i=1}^3 v_i \rd_{x_i} f - \sum_{i=1}^3  \rd_{x_i} \Phi  \rd_{v_i} f - \sum_{i=1}^3 \rd_{x_i} \varphi \partial_{v_i}f = 0&,\\
\varphi = \Delta^{-1}\left(\int_{\R^3}f(s,x,v)\d v\right)&,
\end{split}
\end{equation}
for an unknown function $f: [0,\infty)_s \times (\mathbb R^3_x \setminus \{0\}) \times \mathbb R^3_v \to \mathbb R_{\geq 0}$.
Here, $\Phi:\mathbb R^3_x \setminus \{0\} \to \mathbb R$ is the (time-independent) \emph{Kepler potential}, which is a smooth and spherically symmetric external confining potential given by
\begin{equation}\label{eq:Kepler.potential}
\Phi(x) = -\f 1{|x|}.
\end{equation}

The model in question can be thought of as governing the motion of a self-gravitating gas under the influence of the gravitational field of an external point mass at $x = 0$ (modeling a massive object such as a star). One could also change the second equation in \eqref{eq:transport} to $\varphi = - \Delta^{-1}(\int_{\R^3}f(t,x,v)\d v)$ and all the results we consider will still apply (though for the sake of exposition, we will not explicitly write down that case). With this modification, the equations can be viewed as modeling a charge-plasma system where the motion of the plasma is determined both by the mean field force that it generates and an external attractive charge. In this paper, we will mostly be interested in the nonlinear model \eqref{eq:transport}--\eqref{eq:Kepler.potential} in spherical symmetry (though our linear result applies more generally, see Theorem~\ref{thm:lin-stability}). 

In general, global existence\footnote{We remark that the global existence of \emph{weak} solutions in spherical symmetry has been established in \cite{jCxwZjbW2015}. The work \cite{jCxwZjbW2015} considered a slightly different setting where the point charge can move according to the electric field of the plasma and applies more generally without symmetry assumptions, but it reduces to our setup in spherical symmetry.}, regularity and asymptotics of solutions to \eqref{eq:transport}--\eqref{eq:Kepler.potential} remains unknown.  We are particularly interested in the small-data regime and are motivated by the following fundamental question:
\begin{question}
Given sufficiently small and localized (and possibly spherically symmetric) initial data, are the solutions to \eqref{eq:transport} global? If so, what are the long-time asymptotics?
\end{question}

For the corresponding question when $\Phi \equiv 0$, it is known that small-data solutions are always global \cite{BaDe85}. The long-time asymptotics are characterized by dispersion, similar to the linear transport equation on $\mathbb R^3_x \times \mathbb R^3_v$, except that scattering is modified by nonlinear effects. (Various results in this direction will be discussed in Section~\ref{sec:related.vacuum}.) The case where the external potential is repulsive, i.e., $\Phi(x) = +\f 1{|x|}$ (note the opposite sign as \eqref{eq:Kepler.potential}) was studied in \cite{bPkW2021, bPkWjqY2022}, which proved that small-data solutions are still global. In this case, the linear flow is very different from the free transport, but is nonetheless governed by dispersion.

However, the situation is very different in our case. Even just from the point of view of the background characteristics, there are trapped trajectories which stay in a bounded set of phase space (and away from $x = 0$), which in particular means that the solution does not disperse. As a result, there exist arbitrarily small stationary solutions \cite[Section~6.2]{Straub.thesis}. In particular, the $0$ solution is not asymptotically stable. Instead, one can at best hope that small solutions settle down to stationary ones, via a phase mixing mechanism similar to that seen in Landau damping of confined plasmas \cite{Landau1946}.

In order to focus on phase mixing, in both our linear and nonlinear results, we will only consider solutions with data supported on trapped trajectories. Our first result concerns the linearized problem, for which we only consider $f_{\mathrm{in}}$ which is compactly supported in phase space in a way such that for some $\mathfrak c_0,\, \mathfrak h_0,\, \mathfrak l_1,\, \mathfrak l_2,\, \mathfrak n_0,\, \mathfrak n_1 >0$,
\begin{equation}\label{eq:f.support}
\begin{split}
&\: \mathrm{supp}(f_{\mathrm{in}}) \subset  \mathcal S^{\mathrm{lin}}_0 \\
 := &\: \Big\{ (x,v):  -\f 1{2L(x,v)} + \mathfrak c_0 \leq H_{\Kep}(x,v) \leq -\mathfrak h_0 <0,\, \mathfrak l_1 \leq L(x,v) \leq \mathfrak l_2,\\
 &\: \qquad \qquad \qquad \qquad \qquad \qquad |\mathbf{n}(x,v)| > \mathfrak n_0 >0, -1 + \mathfrak n_1 \leq  \f{\mathbf{n} \cdot \mathbf{e}}{|\mathbf{n}||\mathbf{e}|} \leq  1-\mathfrak n_1 \Big\},
\end{split}
\end{equation}
where 
\begin{align*}
H_{\Kep}(x,v) = \f{|v|^2}2 -\f 1{|x|},\quad L(x,v) = |x|^2 |v|^2 - (x\cdot v)^2,\\
\mathbf{n}(x,v) = (0,0,1)\times (x \times v), \quad \mathbf{e} = \Big(|v|^2 - \f{1}{|x|} \Big)x - (x\cdot v) v.
\end{align*}
We remark that \eqref{eq:f.support} also guarantees that $f_{\mathrm{in}}$ is supported away from $x = 0$ (see \eqref{eq:bounded.away.from.0}), where the Kepler potential is singular.

For the linearized problem, we prove estimates for the particle density $f$ and for the macroscopic density $\rho$ and macroscopic potential $\varphi$ (see \eqref{eq:intro.rho.def}). While these do not decay in time, we prove that the \emph{time derivatives} of $\rho$ and $\varphi$ decay, and thus $\rho$ and $\varphi$ themselves converge to stationary configurations with a quantitative rate depending on the regularity of the initial data as $s \to \infty$. This decay is a manifestation of a phase mixing mechanism.
\begin{theorem}\label{thm:lin-stability}
Let $f:[0,\infty)_s\times \mathbb R^3_x\times \mathbb R^3_v$ be a solution to linear equation
\begin{equation}\label{eq:linear}
\begin{split}
\rd_s f +  \sum_{i=1}^3 v_i \rd_{x_i} f - \sum_{i=1}^3  \rd_{x_i} \Phi  \rd_{v_i} f = 0,\quad f \restriction_{t=0} = f_{\mathrm{in}},
\end{split}
\end{equation}
where $\Phi(x) = -\f 1{|x|}$ as in \eqref{eq:Kepler.potential} and the initial function $f_{\mathrm{in}}$ satisfies \eqref{eq:f.support}.

Define the macroscopic density and potential by
\begin{equation}\label{eq:intro.rho.def}
\rho(s,x):= \int_{\R^3}f(s,x,v)\d v,\quad \varphi(s,x):=\Delta^{-1}\rho(s,x).
\end{equation}
For $N \in \mathbb N$, $A >0$, assume that the initial data satisfy 
\begin{equation}\label{eq:linear.assumption}
\sum_{|\alpha|+|\beta| \leq N}\sup_{x,v}|\rd_x^{\alpha}\rd_v^\beta f_{\ini}|(x,v) \leq A.
\end{equation}
Then the solution exists for all $s\geq 0$ with $\mathrm{supp}(f(s)) \subseteq \mathcal S^{\mathrm{lin}}_0 $, and, moreover, the following estimates hold for all $s \geq 0$ (or $t\geq 0$ in \eqref{eq:thm.bdd-lin.2}), where the implicit constants in $\ls$ depend on $\mathfrak c_0$, $\mathfrak h_0$, $\mathfrak l_1$, $\mathfrak l_2$, $\mathfrak n_0$, $\mathfrak n_1$ and $N$: 
\begin{enumerate}
\item (Boundedness and controlled growth of higher derivatives) 
\begin{equation}\label{eq:thm.bdd-lin.1}
 \sup_{x,v}|\rd_x^\alp \rd_v^\bt f|(s,x,v) \ls A \jap{s}^{|\alp|+|\bt|},\quad |\alp| + |\bt| \leq N.
\end{equation}
\item (Improved higher order estimates in Delaunay variables) There exists an action angle change of coordinates $(s,x,v)\mapsto (t=s,J,\mathfrak{L}_z,\mathfrak{L},Q,\theta_{\mathfrak{L}_z}, \theta_{\mathfrak{L}})$ such that 
\begin{equation}\label{eq:thm.bdd-lin.2}
\sup_{J,\mathfrak{L}_z,\mathfrak{L},Q,\theta_{\mathfrak{L}_z}, \theta_{\mathfrak{L}}}|\rd_{J}^{i_1}\rd_{\mathfrak{L}_{z}}^{i_2}\rd_{\mathfrak{L}}^{i_3}\rd_Q^{i_4}\rd_{\theta_{\mathfrak{L}_z}}^{i_5} \rd_{\theta_{\mathfrak{L}}}^{i_6}f|(t,J, \cdots, \theta_{\mathfrak{L}}) \lesssim  A \jap{t}^{i_1},\quad i_1+\cdots+i_6\leq N.
\end{equation}
\item (Inverse polynomial decay of $\rd_s \rho$ and $\rd_s \varphi$) The following decay estimates hold for all $s$
\begin{equation}\label{eq:thm.decay-rho-lin}
\sup_{x \in \mathbb R^3} |\rd_x^{\alp} \rd_s \rho(s,x)| \ls A \brk{s}^{-(N-|\alp|)},\quad |\alp|\leq N-2,
\end{equation}
and for any $p \in [1,\infty)$ and $R>0$, the following holds:
\begin{equation}\label{eq:thm.decay-lin}
 \|\rd_x^{\alp} \rd_s \varphi(s,\cdot)\|_{L^p(|x|\leq R)} \ls_{p,R} A \brk{s}^{-\min\{N-|\alp|+2,N\}},\quad |\alp|\leq N,
\end{equation}
where the implicit constant may additionally depend on $p$ and $R$.
\end{enumerate}

\end{theorem}

For our main result, we now turn to the nonlinear equation \eqref{eq:transport}. We will only consider the case where the initial data are spherically symmetric. Since $\Phi$ is spherically symmetric, it is easy to check that spherically symmetric initial data to \eqref{eq:transport} induce spherically symmetric solutions. 
\begin{definition}\label{def:r.w.L}
We say that the initial data to \eqref{eq:transport} are \textbf{spherically symmetric} if $f_{\mathrm{in}}$ depends on $(x,v)$ only through
$$r(x,v) = |x|,\quad w(x,v) = \f{x\cdot v}{|x|},\quad L(x,v)= |x|^2|v|^2 - (x\cdot v)^2,$$
i.e., there exists a function $F$ such that $f_{\mathrm{in}}(x,v) = F(r(x,v),w(x,v),L(x,v))$.
\end{definition}
When dealing with the nonlinear problem, we use spherical symmetry in a very crucial way. As we will explain (see Section~\ref{sec:intro.dyn.var}), we need to construct a \emph{dynamical} system of action angle variables associated to the nonlinear potential at a future time. These new variables are important for proving decay and closing the nonlinear estimates. The construction of such a change of variables relies on spherical symmetry; this is related to the fact that the system of ODEs corresponding to the characteristics is completely integrable for any spherically symmetric potentials.

Our main nonlinear result is a small data (in $C^{N+1}$) long-time stability result in spherical symmetry for initial data supported on trapped trajectories (in the sense of \eqref{eq:f.support.SS}). We use phase mixing to obtain a time scale for stability beyond that obtained using standard local existence results (see Remark~\ref{rmk:timescale}).
\begin{theorem}[Long-time nonlinear stability and phase mixing in spherical symmetry]\label{thm:main}
Given $\mathfrak c_0,\, \mathfrak h_0,\, \mathfrak l_1,\, \mathfrak l_2 >0$ and $N \in \mathbb N$ with $N \geq 8$, there exist $\de>0$ and $\ep_0>0$ (depending only on $\mathfrak c_0$, $\mathfrak h_0$, $\mathfrak l_1$, $\mathfrak l_2$ and $N$) such that the following holds for all $\ep \in (0, \ep_0]$.

Consider spherically symmetric initial data to \eqref{eq:transport} satisfying 
\begin{equation}\label{eq:f.support.SS}
\mathrm{supp}(f_{\mathrm{in}}) \subset \mathcal S_0 := \{(x,v):  -\f 1{2L(x,v)} + \mathfrak c_0 \leq H(x,v) \leq -\mathfrak h_0 <0,\, \mathfrak l_1 \leq L(x,v) \leq \mathfrak l_2 \},
\end{equation}
where 
\begin{equation}\label{eq:H.SS.defined}
H=w^2/2+L/2r^2+\Phi(r)+\varphi(t,r)
\end{equation}
and $L$ is as in \eqref{def:r.w.L}. Assume also that the following size condition holds:
\begin{equation}\label{eq:main.assumption}
\sup_{x,v} \sum_{|\alp|+|\bt| \leq N+1} |\rd_x^{\alp} \rd_v^{\bt} f_{\mathrm{in}}|(x,v) \leq \de \ep.
\end{equation}

Then the following holds:
\begin{enumerate}
\item (Long-time existence) The solution $f$ arising from the given initial data exists and remains regular on the time interval $0\leq s \leq T_{\mathrm{final}} := \ep^{-1}(\log \f 1{\ep})^{-2}$.
\item (Controlled growth of higher derivatives of $f$) The following estimates hold for all $s \in [0,T_{\mathrm{final}}]$:
\begin{equation}\label{eq:thm.very.weak}
\sup_{x,v}|\rd_x^\alp \rd_v^\bt f|(s,x,v) \ls \de \ep \brk{s}^{|\alp|+|\bt|},\quad |\alp|+|\bt| \leq N.
\end{equation}
\item (Improved estimates for $f$ in dynamical action angle variables) There exists a future-defined system of coordinates $(t=s, Q_{T_{\mathrm{final}}}, H, M=L)$ such that the following estimates hold for all $t \in [0,T_{\mathrm{final}}]$:
\begin{equation}\label{eq:thm.boundedness}
\sup_{Q_{T_{\mathrm{final}}}, H,M}| \rd_{Q_{T_{\mathrm{final}}}}^{i_1} \rd_H^{i_2} \rd_M^{i_3} f|(t, Q_{T_{\mathrm{final}}}, H, M) \ls \de \ep \brk{t}^{i_2},\quad i_1+i_2+i_3 \leq N.
\end{equation}
\item (Decay for $\rd_s \varphi$) The following decay estimates hold for all $s \in [0,T_{\mathrm{final}}]$:
\begin{equation}\label{eq:thm.decay}
\sup_{|x| \in [\f{\mathfrak l_1}{2}, \f{2}{\mathfrak h}]} |\rd_r^{j} \rd_s \varphi(s,x)| \ls \de \ep \brk{s}^{-\min\{N-|\alp|+2,N\}},\quad j \leq N.
\end{equation}
\end{enumerate}
\end{theorem}

A few remarks are in order.
\begin{remark}[The time scale $\ep^{-1}(\log \f 1{\ep})^{-2}$]\label{rmk:timescale}
We note that the time scale $\ep^{-1}(\log \f 1{\ep})^{-2}$ should be compared with the ``trivial'' time scale $\ep^{-\f 12}$ that can be achieved using a suitable local existence result that does not capture the decay estimate in \eqref{eq:thm.decay}.

Notice that even for the easier problem of Vlasov--Poisson on a torus, $\ep^{-1}$ is the best known time of existence for Sobolev data of size $\de \ep$ (possibly up to logarithms). At that time scale, the nonlinear effect of plasma echoes begin to appear; see \cite{jB2021}.
\end{remark}

\begin{remark}[Gevrey data]
In the case of Vlasov--Poisson on a torus, \emph{global} existence is known for analytic or suitable Gevrey data \cite{jBnMcM2016, eGtNiR2020a, cMcV2011}. It would be of interest to study the analogous problem in this setting, though one must contend with weaker estimates. 
\end{remark}

In the broader context, our result can be viewed as a stability result for kinetic models with bounded trajectories. See Section~\ref{sec:related} for discussion on related works. In particular, the problem we study here can be viewed as a simplified non-relativistic model for understanding the nonlinear stability of the Schwarzschild black hole solution in the context of the Einstein--massive Vlasov system in spherical symmetry. In that case, even though global existence of solutions in the black hole exterior (in the sense of the completeness of future null infinity) can be established using soft methods \cite{mD2005, mDaR2005}, the long-time asymptotics remain unknown. It is known, however, that the Schwarzschild solution itself is not asymptotically stable, as there are arbitrarily small spherically symmetric perturbations of Schwarzschild which lead to a stationary black hole \cite{feJ2021, gR1994}. It would be of interest to at least obtain an analogue of Theorem~\ref{thm:main} in that setting.

\subsection{Ideas of proof}

\subsubsection{Action angle variables and linear decay}\label{sec:intro.ideas.linear} To prove even the linear decay, it is important to pass to action angle variables instead of using the original variables. In these variables, the linearized equation simplifies, and one can capture the fact that different derivatives have different decay/growth property. More precisely, comparing the estimates \eqref{eq:thm.bdd-lin.1} with \eqref{eq:thm.bdd-lin.2}, one sees that in the $x,v$ coordinate system, each $\rd_x$ and $\rd_v$ lose a power of $t$, while one gets more refined higher order estimates with respect to the action angles variables.

After passing to action angle variables $(s,x,v)\mapsto (t,J,\mathfrak{L},\mathfrak{L}_z,Q,\theta_{\mathfrak{L}},\theta_{\mathfrak{L}_z})$ (where $t = s$ and the other coordinate functions are defined in Definition~\ref{def:Delaunay-action}), the linear equation takes the form
\begin{equation}\label{eq:intro.aa.eqn}
\rd_t f + \f{1}{J^3}\rd_{Q} f = 0,
\end{equation}
The linear estimates \eqref{eq:thm.bdd-lin.2} for $f$ follow immediately from \eqref{eq:intro.aa.eqn}. Indeed, since $\rd_{\mathfrak{L}_{z}}^{i_2}\rd_{\mathfrak{L}}^{i_3}\rd_Q^{i_4}\rd_{\theta_{\mathfrak{L}_z}}^{i_5} \rd_{\theta_{\mathfrak{L}}}^{i_6}$ commutes with $\rd_t + \f 1{J^3} \rd_Q$, it follows that $\rd_{\mathfrak{L}_{z}}^{i_2}\rd_{\mathfrak{L}}^{i_3}\rd_Q^{i_4}\rd_{\theta_{\mathfrak{L}_z}}^{i_5} \rd_{\theta_{\mathfrak{L}}}^{i_6}f$ is bounded by initial data. On the other hand, since $[\rd_J, \rd_t + \f 1{J^3} \rd_Q] \neq 0$, every $\rd_J$ commutation gives a power of $t$ growth.

Moreover, \eqref{eq:intro.aa.eqn} can be solved explicitly using Fourier series in $Q$, where the solution $f$ is given by 
\begin{equation}
f = \sum_{k \in \mathbb Z} \widehat{f}_k(J,\mathfrak{L},\mathfrak{L}_z,\theta_{\mathfrak{L}},\theta_{\mathfrak{L}_z}) e^{ikQ} e^{-i J^{-3} k t}.
\end{equation} 
To show the $s$ decay of\footnote{In the linear setting, $\rd_s$ in the $(s,x,v)$ coordinate system coincides with $\rd_t$ in the $(t,J,M,M_z,Q,\theta_M,\theta_{M_z})$ coordinate system. Notice that this will not be the case when using the dynamical action angle variables in the nonlinear problem, resulting in some extra terms to be controlled.} $\rd_s \rho = \rd_t \rho$, we write
\begin{equation}
\rd_s \rho(s,x) = \int_{\mathbb R^3} \rd_t  \Big(\sum_{k \in \mathbb Z}  \widehat{f}(J,\mathfrak{L},\mathfrak{L}_z,\theta_{\mathfrak{L}},\theta_{\mathfrak{L}_z}) e^{ikQ} e^{-i J^{-3} k t} \Big) \, \ud v 
\end{equation}
where the integral is taken over $v \in \mathbb R^3$ for a fixed $x \in \mathbb R^3$. The needed decay \eqref{eq:thm.decay-lin} is now an easy consequence of nonstationary phase, after noting that in the $(x^1,x^2,x^3,w,L,\vartheta)$ coordinate system, $\ud v = \ud L \ud w \ud \vartheta$, and that $\rd_{L}J^{-3}$ is bounded away from $0$.


\subsubsection{Dynamical action angle variables}\label{sec:intro.dyn.var} In the nonlinear setting of Theorem~\ref{thm:main}, we introduce dynamically defined (i.e., depending on the nonlinear solution) action angle variables\footnote{Technically, the coordinates we use are a slight modification of the action angle variables, which share important features, but the change of variable map is not a symplectomorphism. This slightly more flexible choice is already sufficient for our purposes.}. Before we discuss the dynamical action angle variables, it is useful to first rewrite \eqref{eq:intro.aa.eqn} in $(t,H_{\Kep},\mathfrak{L},\mathfrak{L}_z,Q,\theta_{\mathfrak{L}},\theta_{\mathfrak{L}_z})$ coordinates, where $H_{\Kep}=-\f{1}{2J^2}$. Then \eqref{eq:intro.aa.eqn} takes the form
\begin{equation}\label{eq:intro.aa'.eqn}
\rd_t f + (-2H_{\Kep})^{3/2} \rd_{Q} f = 0.
\end{equation}
When reduced to the spherically symmetric setting, the functions $(t,Q,H_{\Kep},M)$ form a coordinate system and the \emph{linear} transport equation takes the form \eqref{eq:intro.aa'.eqn}.

For the nonlinear system \eqref{eq:transport} in the setting of theorem~\ref{thm:main}, however, if we continue to use the coordinate system $(t,Q,H_{\Kep},M)$, the nonlinear equation takes the form
\begin{equation}\label{eq:intro.bad.equation}
\rd_t f + ((-2H_{\Kep})^{3/2} - \rd_r\varphi \rd_w Q) \rd_{Q} f  - \rd_r\varphi w \rd_{H_{\Kep}} f = 0.
\end{equation}
One expects $\rd_r\varphi w \rd_{H_{\Kep}} f$ to be the most problematic term because even linearly $\rd_r\varphi$ does not decay in $t$ and $\rd_{H_{\mathrm{Kep}}} f$ grows linearly in $t$. This results in a nonlinear term growing as $O(\ep^2 t)$ which can only be bounded by $\ep$ on the shorter time-scale $t\sim \ep^{-1/2}$. (In fact, this would be the time-scale one could even achieve with standard energy estimates, without using action angle variables at all.)

Instead, we introduce the dynamically defined coordinates $(t,Q_T,H,M)$, where $t=s$, $M=L$, but $H = \f{w^2}2+\f{L}{2r^2}+\Phi(r)+\varphi(t,r)$ is the Hamiltonian that \emph{depends on the solution} $\varphi(t,r)$, and $Q_T$ is an angle variable that is normalized \emph{depending on the solution $\varphi(T,r)$} at some \emph{future} time $T$. See Definition~\ref{def:coordinates} for more details.

The upshot of using this coordinate system is that a nonlinear term of the form $\rd_r\varphi w \rd_{H_{\Kep}} f$ in \eqref{eq:intro.bad.equation} is no longer present. Instead, the nonlinear term in $\rd_H f$ now takes the form $\rd_s \varphi \rd_H f$ (see \eqref{eq:vlasov-action angle} for the equation), where $\rd_s$ is the time derivative in the original coordinate system. Importantly, while $\varphi$ itself and its spatial derivatives do not decay in $t$, its time derivative decays, and this can compensate the growth of $\rd_H f$.

We should note that while this change of variables is helpful in handling decay, because its definition depends on the solution, we will need to control the change of variable map. It is in particular important both (1) to show that the change of variables map is close to the background values and (2) to control the regularity of the change of variables map. This will be carried out in Section~\ref{sec:dyn.coord}.

\subsubsection{Comparison with Vlasov--Poisson on $\mathbb T$}\label{sec:compare.with.flat}
After passing to dynamical action angle variables, our problem shares some similarity with the Vlasov--Poisson system on the torus $\mathbb T^d$ (we write it for $d=1$ as the essence of the problem is similar for all dimensions $d$):
\begin{equation}\label{eq:intro.VP.flat}
\rd_t f + v\rd_x f - \rd_x \varphi \rd_v f = 0,\quad \varphi(t,x) = \Delta^{-1} (\int_{\mathbb R} f(t,x,v) \, \ud v - \f 1{2\pi}\int_{\mathbb T\times \mathbb R} f(t,x,v) \, \ud v\ud x).
\end{equation}
The problem \eqref{eq:intro.VP.flat} has been by now very well-studied, in fact many results have been obtained not only for small data solutions, but also for small perturbations for a large class of spatially homogeneous solutions. The problem \eqref{eq:intro.VP.flat} shares the following features with our problem:
\begin{enumerate} 
\item The linear decay is generated by phase mixing. 
\item The nonlinear structure is similar: for our problem the nonlinear term is  $\rd_s \varphi \rd_H f$ (see Section~\ref{sec:intro.dyn.var}), where $\rd_s \varphi$ decays due to phase mixing (in the regularity dependent manner) and $\rd_{H} f$ grows linearly in $t$. In \eqref{eq:intro.VP.flat}, the main nonlinear term is $\rd_x \varphi \rd_v f$, where now $\rd_x\varphi$ decays by phase mixing and $\rd_v f$ grows linearly in $t$.
\end{enumerate}

However, there are important differences. Perhaps the most important difference is that in the case \eqref{eq:intro.VP.flat}, $\rd_x$ derivatives of $\varphi$ have the same $t$-decay, while in our case, higher $\rd_r$ derivatives of $\varphi$ lose decay in powers of $t$. Ultimately, this is related to the fact that the $\rd_x$ derivative for $\varphi$ in the case of \eqref{eq:intro.VP.flat} coincides with a ``good'' derivative for the transport equation, while in our case the $\rd_r$ derivative for $\varphi$ does not coincide with the ``good'' $\rd_{Q_T}$ derivative. 


\subsubsection{Nonlinear estimates}

To carry out the nonlinear estimates, we introduce a scheme (similar to that for \eqref{eq:intro.VP.flat} in \cite{jB2017, jBnMcM2016, VPL, eGtNiR2020a, cMcV2011}) to have independent estimates for $f$ and for $\varphi$ (and $\rd_s \varphi$). In particular, we use the structure of the equations to obtain decay estimates for $\rd_s \varphi$ which are \emph{better} than those directly derived from using the vector field bounds on $f$.

Recall that the data are of size $\de\ep$. For $f$, we will prove the following estimates: 
\begin{equation}\label{eq:est.f.intro}
\| \rd_{Q_T}^{i_1} \rd_{H}^{i_2} \rd_{M}^{i_3} f\|_{L^\i}(t) \ls \de \ep \brk{t}^{i_2}.
\end{equation}

As for $\varphi$ and its derivatives, we prove the following estimates:
\begin{align}
\sup_{ r\in [\f{\mathfrak l_1}{2}, \f{2}{\mathfrak h}]} |\partial_r^{j}\rd_s\varphi(s,r)| \ls &\: \delta\epsilon \jap{s}^{-\min\{N-j+2,N\}}, \quad 0 \leq j \leq N, \label{eq:intro.varphi_t} \\
\sup_{r\in [\f{\mathfrak l_1}{2}, \f{2}{\mathfrak h}]} |\partial_r^{\ell}\varphi(s,r)| \ls &\: \epsilon,\quad 0 \leq \ell \leq N+2. \label{eq:intro.varphi_r} 
\end{align}
We are only concerned with estimates on a compact set away from $r = 0$ because as this is where the nonlinear interaction terms in the Vlasov equation will be supported. Notice that $\rd_s \varphi$ decays, but the decay rate becomes worse with higher $\rd_r$ derivatives. In addition to the decay estimates \eqref{eq:intro.varphi_t}, we also prove decay estimates for $\varphi(s,r) - \varphi(T,r)$, where $T$ is the future time at which the coordinates are chosen:
\begin{equation}
\sup_{ r\in [\f{\mathfrak l_1}{2}, \f{2}{\mathfrak h}]} |\rd_r^k (\varphi(s,r) - \varphi(T,r))| \ls \de \ep,\quad 0 \leq k \leq N+1.\label{eq:intro.varphi_diff} 
\end{equation}

Assuming the estimates \eqref{eq:intro.varphi_t}, \eqref{eq:intro.varphi_r} and \eqref{eq:intro.varphi_diff}, the bound \eqref{eq:est.f.intro} for $f$ is straightforward. The boundedness estimates \eqref{eq:intro.varphi_r} then essentially follows from the bounds for $f$. The most difficult part of the proof is thus the \emph{decay} bounds \eqref{eq:intro.varphi_t} and \eqref{eq:intro.varphi_diff}, where we need to extend the linear estimates in Section~\ref{sec:intro.ideas.linear} to the nonlinear setting. We will explain this in the next two subsubsections, focussing on \eqref{eq:intro.varphi_t}, as \eqref{eq:intro.varphi_diff} can be achieved along similar lines.

\subsubsection{Decay estimate and the commuting vector field method}

In order to close the decay estimate for $\rd_s \varphi(s,r)$, we argue as in the linear estimates, but in addition need to control nonlinear inhomogeneous terms such as 
\begin{equation}\label{eq:intro.main.error}
\begin{split}
&\:\int_0^r\int_{r_2}^\infty \int_{-\infty}^\infty\int_{0}^\infty\frac{1}{r r_1}\int_0^t S(t-\tau,Q_T,H,M)[ \rd_s \varphi(\tau,r_1)\\
&\hspace{13em}\times \partial_H f(\tau,Q_T,H,M)]\ud\tau\d L\d w\d r_1\d r_2,
\end{split}
\end{equation}
where $S(\cdot,Q_T,H,M)$ is the semigroup associated to the linear flow of $\rd_t + \Omega(H,M) \rd_{Q_T}$. (Here, $\rd_t + \Omega(H,M) \rd_{Q_T}$ is a dynamically defined linear transport operator.)

Recall from \eqref{eq:est.f.intro} and \eqref{eq:intro.varphi_t} that $|\rd_s \varphi|(\tau,\cdot)\ls \de \ep \jap{\tau}^{-N}$ and $|\rd_H f|(\tau, \cdot)\ls \de \ep \jap{\tau}$. These estimates by themselves are too weak to conclude that \eqref{eq:intro.main.error} has $t$-decay. Here, we are inspired by the strategy used for the Vlasov--Poisson problem on $\mathbb T$ (see Section~\ref{sec:compare.with.flat}) developed in \cite{jBnMcM2016, VPL, eGtNiR2020a, cMcV2011} to obtain decay via understanding resonance between different modes. We decompose both 
\begin{align*}
\rd_s \varphi = \sum_{k \in \mathbb Z} e^{ikQ_T}(\widehat{\rd_s \varphi})_k(\tau,H,M), \quad \rd_H f= \sum_{k \in \mathbb Z} e^{ikQ_T}(\widehat{\rd_H f})_k(\tau,H,M)
\end{align*}
in Fourier series in $Q_T$. For the product in \eqref{eq:intro.main.error}, we need to control
$$(\widehat{\rd_s \varphi})_\ell(\tau) S_k(t-\tau) (\widehat{\rd_H f})_{k-\ell}(\tau),$$
where $S_k(\cdot)$ is the semigroup associate to $\rd_t + ik \Omega(H,M) \rd_{Q_T}$. Here, one can obtain decay estimates for $S_k(t-\tau) (\widehat{\rd_H f})_{k-\ell}(\tau)$ in $kt - \ell \tau$ so that when combined with the $\tau$-decay of $\rd_s \varphi$, there is hope of obtaining $t$-decay. More precisely, we achieve this using commuting vector fields as in \cite{VPL}. This means that in addition to using $\rd_{Q_T}$, $\rd_{H}$ and $\rd_{M}$ as commuting vector fields as in \eqref{eq:est.f.intro}, we also introduce a $t$-weighted vector field $Y_{H} = t\partial_{H}\{\Omega(H,M)\}\partial_{Q_T}+\partial_H$ and control $\| Y_H^{i_1} \rd_{Q_T}^{i_2} \rd_H^{i_3} \rd_M^{i_4} f\|_{L^\i}(t)$. Roughly, we would like to use the control of $Y_{H} f$ get decay in $kt - \ell \tau$ (see Lemma~\ref{lem:main-terms-den}).

The use of commuting vector fields to understand resonances allows us to bypass explicit Fourier computation in $H$ (cf.~\cite{jBnMcM2016, eGtNiR2020a, cMcV2011}), and is particularly suited for our variable-coefficient setting. We remark, however, that when compared to the case of \eqref{eq:intro.VP.flat} there are some new technical challenges in the implementation of the strategy. We explain this for the particular term \eqref{eq:intro.main.error}: 
\begin{enumerate}
\item When using $Y_H$ derivatives of $f$ to generate decay, we also need to differentiate $\rd_s\varphi$. (See Lemma~\ref{lem:main-terms-den} and Remark~\ref{rmk:also.diff.psi}.) 
\item In our case, higher $Y_H$ derivatives of $f$ necessarily grows in $t$ (see \eqref{eq:boot-f} when $i_1$ is large and Remark~\ref{rmk:large.i_1}). This limits the number of $Y_H$ derivatives that can be used. In particular, this will not give sufficient decay in $t$. In order to obtain sufficient $t$-decay, we will also crucially use the $\rd_H$ vector field: while it does not generate extra factors of time decay, it allows us to exchange $\tau$-decay for $t$-decay.
\end{enumerate}
We will defer to the beginning of the proof of Proposition~\ref{prop:T12} for a more detailed discussion of the strategy for handling this error term.

\subsubsection{Final remarks on higher derivatives of $\rd_s\varphi$}
Finally, in the estimates \eqref{eq:intro.varphi_t} and \eqref{eq:intro.varphi_diff}, we need to take higher $\rd_r$ derivatives. In particular, we need to take $\rd_r$ derivatives of \eqref{eq:intro.main.error}. The following remarks are important in these higher order derivative estimates:
\begin{enumerate}
\item Note that the first two derivatives either act on the $\frac{1}{r r_1}$ factor or remove the (up to two) outer integrals in \eqref{eq:intro.main.error}. In other words, the first two $\rd_r$ derivatives do not cost decay in $t$ (which is simply a reflection of the fact that the Poisson equation gains two spatial derivatives). This turns out to be important in our bootstrap scheme in order to have sufficient decay to close.
\item When differentiating \eqref{eq:intro.main.error} by $\rd_{r}$, we need to control terms when it hits on $\rd_{H} f$. The vector field $\rd_{r}$ has a non-trivial $\rd_{H}$ component in the $(\rd_{Q_T}, \rd_{H}, \rd_{M})$ basis corresponding to the $(t,Q_T,H,M)$ coordinate system, and the $\rd_{H}$ derivatives on $f$ are very costly in terms of decay. Instead, we write $\rd_{r}$ as a linear combination of $(\rd_{Q_T}, \rd_{L}, \rd_{M})$ (see \eqref{eq:r-in-good-things}), where $\rd_{L}$ is to be understood as the coordinate vector field in the $(s,r,w,L)$ coordinates. The upshot is that while $\rd_{L}$ acting on $f$ still incurs growth as before, it can be integrated by parts away so that the derivative falls on $\rd_{s} \varphi$ (or even better, on the coefficients).
\end{enumerate}

\subsection{Related works}\label{sec:related}

\subsubsection{Linear estimates}\label{sec:related.linear}

For a class of linear models related to that in Theorem~\ref{thm:lin-stability}, decay without a rate was established by Rioseco--Sarbach in \cite{pRoS2020}. For a related but somewhat different $1$-dimensional model, quantitative phase mixing estimates were proven in our earlier paper \cite{sCjL2022}. These results were generalized in \cite{mMpRhVDB2022} to a larger class of models. 

Very recently, Had\u zi\'c--Rein--Schrecker--Straub \cite{mHgRmScS2024} proved linear decay estimates for a large class of transport equations that are motivated by stability of steady state solutions to the Vlasov--Poisson system (see Section~\ref{sec:related.steady.states}). Their result in particular applies to the equation in Theorem~\ref{thm:lin-stability}, at least when restricted to low-order derivatives and spherical symmetry. Moreover, their estimates allow for initial data that are supported near elliptic points.

\subsubsection{Landau damping}\label{sec:related.Landau} Our result can be viewed as a generalization of Landau damping. Landau damping is a damping mechanism through phase mixing in plasma; linear decay was first observed in Landau's seminal work \cite{Landau1946}. A mathematical  breakthrough was achieved by Mouhot--Villani \cite{cMcV2011}, justifying Landau damping in a nonlinear, but analytic, setting. This has been simplified and extended to Gevrey settings in \cite{jBnMcM2016, eGtNiR2020a}. Notice that the works \cite{jBnMcM2016, eGtNiR2020a, cMcV2011} do not explicitly treat the case of finite regularity data, but their methods allow one to control the solutions up to time $O(\ep^{-1})$ for size $O(\ep)$ data (see Remark~\ref{rmk:timescale} and \cite{jB2021}).

There are many other related results, we refer the reader to \cite{jB2021, jBnMcM2018, jBnMcM2020, eCcM1998, VPL, rGjS1994, rGjS1995, eGtNiR2020b, hjHjjlW2009, dHKttNfR2021a, dHKttNfR2024, dHKttNfR2021b, adIbPxcWkW2023, adIbPxcWkW2024b, adIbPxcWkW2024, bY2016, cZ2021} and the references therein for a sample of results. See also the expository notes \cite{jB2022} for further references.

\subsubsection{Stability results for the Vlasov--Poisson system}\label{sec:related.vacuum}

In addition to the results in Section~\ref{sec:related.Landau}, there are other regimes for which stability results are known for the Vlasov--Poisson system. In the near vacuum regime, this was proven in \cite{BaDe85}. For more recent refinements, extensions and discussions, including modified scattering results, see \cite{lBrVR2024, shCssK2016, xlDuan2022, pFzmOYbPkW2023, hjHaRjjlV2011, aIbPxcWkW2022, pS2022, vSmT2024, Sm16, Wa18.2}.
 
In the presence of a repulsive potential $\Phi(x) = +\f 1{|x|}$, stability was proven in \cite{bPkW2021, bPkWjqY2022}.  See also stability results for a different repulsive potential in \cite{Velozo.Ruiz.brothers+, Velozo.Ruiz.brothers}.

\subsubsection{Stability of steady state solutions to the Vlasov--Poisson system}\label{sec:related.steady.states}

Closely related to our setup, but much more difficult analytically, is the problem of stability of steady state solutions to the Vlasov--Poisson system. 

A linear stability condition for these steady states was identified in the classical work of Antonov \cite{vaA1960}. Remarkably, when the stability condition is satisfied, orbital stability for the \emph{nonlinear} system has been proven \cite{yGgR2007, mLfMpR2011, mLfMpR2012}; see also \cite{jDoSjS2004, yG1999, yG2000, yGzwL2008, yGgR1999, yGgR2001, mH2007, mLfMpR2008, oSjS2006, gW1999} and the survey \cite{cM2013}. 

As for asymptotic stability, it has been suggested in the pioneering work of Lynden-Bell that phase mixing may act as a damping mechanism \cite{dLB1962, dLB1967}, though it is possible also for linearized perturbations to oscillate \cite{sdM1990}. For some recent mathematical results concerning damping versus oscillation in the linear setting, we refer the reader to \cite{mHgRcS2022, mHgRmScS2023, mK2021, mMpRhVDB2023}. It should be noted, however, that quantitative decay estimates remain elusive even for the linearized Vlasov--Poisson system around steady states\footnote{This is despite the fact that decay estimates can be proven for an equation involving the purely transport part, see Section~\ref{sec:related.linear} and the discussions in \cite{mHgRmScS2024, cS2024}.}.
 
\subsubsection{Black hole stability for the Einstein--Vlasov system}

The problem in this paper can also be viewed as a simplified non-relativistic model for understanding the nonlinear stability of the Schwarzschild black hole solution in the context of the Einstein--massive Vlasov system in spherical symmetry.

The problem shares a similar feature as the problem in the present paper in that even for the linearized problem, there are bound trajectories for the massive Vlasov equation, making decay estimates subtle. If one considers the linear Vlasov equation on the exterior of a fixed Schwarzschild black hole, and \emph{assumes} that the initial data are supported away from trapped trajectories, then quantitative decay estimates have been proven in \cite{velozo_ruiz_2022}. Turning to the nonlinear problem of the Einstein--massive Vlasov system in spherical symmetry, it is known that solutions arising from small perturbations of Schwarzschild possess a complete future null infinity in the black hole exterior and have singularity structure similar to that of Schwarzschild \cite{mD2005, mDaR2005, mDaR2016}. It is likely that orbital stability can be proven along the lines of \cite{mDgH2006}, but in order to understand asymptotic behavior of solutions, one would need to contend with issues similar to those in the present paper.

We remark that much more is known concerning black hole stability if the Vlasov field is assumed to be \emph{massless}. This is particularly due to the fact that trapped characteristics are unstable in this setting. Decay estimates for the linear massless Vlasov system on fixed Kerr spacetimes were established in \cite{lApBjJ2018} (see also \cite{lB2023}). When restricted to spherical symmetry, nonlinear stability of Schwarzschild for the Einstein--massless Vlasov system has been proven in \cite{velozo_ruiz_2022}.

\subsection{Outline of the paper} 

In \textbf{Section~\ref{sec:linear}}, we prove the linear estimates stated in Theorem~\ref{thm:lin-stability}. In Sections~\ref{sec:bootstrap}--\ref{sec:continuity}, we prove our main nonlinear theorem (Theorem~\ref{thm:lin-stability}). We begin with introducing the notations in \textbf{Section~\ref{sec:notation}} and the main bootstrap assumptions in \textbf{Section~\ref{sec:bootstrap}}. Then in \textbf{Section~\ref{sec:dyn.coord}}, we introduce the dynamical action angle variables and control the change of variables map. The next three sections are devoted to the main estimates: in \textbf{Section~\ref{sec:f}}, we prove the estimates for $f$; in \textbf{Section~\ref{sec:top.boundedness}}, we prove the top order estimate for $f$ and the boundedness estimates for $\varphi$; in \textbf{Section~\ref{sec:density}}, we prove the decay estimates for $\rd_s \varphi$, $\varphi(t,\cdot) - \varphi(T,\cdot)$ (and their derivatives). Finally, we put together the estimates and conclude the proof of Theorem~\ref{thm:lin-stability} in \textbf{Section~\ref{sec:continuity}}.

\subsection*{Acknowledgements} We gratefully acknowledge the support of the National Science Foundation under the grants DMS-2005435 and DMS-2304445. The first author is also supported by the Simons foundation award 1141490 and the second author is also supported by a Terman fellowship. 

We thank Saehoon Eo for helpful comments on a previous version of the manuscript.

Most of this work was carried out when the first author was a PhD student at Stanford University.

\section{Proof of Theorem~\ref{thm:lin-stability}: Linear estimates without symmetry assumptions}\label{sec:linear}

\subsection{Action angle variables in the linear setting}\label{sub:action angle-lin}
In this subsection, we use the Delaunay action angle variables to transform the linear Vlasov--Poisson equation \eqref{eq:linear} with Kepler potential, i.e.~with $\Phi(x)=-\f{1}{|x|}$. We refer the reader to Chapter 3 and Appendix E of \cite{BinTre2011} for more details.

For the linear problem it is useful to define the Kepler Hamiltonian, $H_{\text{Kep}}$ as follows
\begin{equation}\label{eq:Kepler-Ham}
H_{\text{Kep}}=\f{|v|^2}2-\f 1{|x|}.
\end{equation}

We work with the following action angle variables.
\begin{definition}\label{def:Delaunay-action}[Delaunay action angle variables] 
\begin{enumerate}
\item Let $\mathbf{L}$ be the angular momentum vector, i.e.,
\begin{equation}\label{eq:L.def}
\mathbf{L}:= x\times v, 
\end{equation}
and 
define
\begin{equation}\label{eq:n.def}
\mathbf{n}:=(0,0,1) \times \mathbf{L}
\end{equation}
\item Define
\begin{equation}\label{eq:e.def}
\mathbf{e} := \Big(|v|^2 - \f{1}{|x|} \Big)x - (x\cdot v) v
\end{equation}
so that
\begin{equation}\label{eq:emod.def}
|\mathbf{e}| =\sqrt{(| v|^2 |x|-1)^2 + (\f{2}{|x|} - | v|^2)(x\cdot  v)^2}, 
\end{equation}
and define $E \in \mathbb R/(2\pi \mathbb Z)$ by
\begin{equation}\label{eq:E.def}
(\cos E, \sin E) = \Big(\f{1}{|\mathbf{e}|}(| v|^2 |x|-1), \f{1}{|\mathbf{e}|}(x\cdot  v)\sqrt{\f{2}{|x|} - | v|^2}\Big).
\end{equation}
\item (Action variables) Recalling \eqref{eq:Kepler-Ham} and \eqref{eq:L.def}, we define the \textbf{action variables} $(\mathfrak{L}_z, \mathfrak{L}, J)$ by 
\begin{equation}\label{eq:Del-action}
\mathfrak{L}_z:= (0,0,1) \cdot \mathbf{L}, \hspace{1em}\mathfrak{L}:=|\mathbf{L}| \hspace{1em} \text{and } J:=\f{1}{\sqrt{-2H_{\text{Kep}}}}.
\end{equation}
\item (Angle variables) Recalling \eqref{eq:n.def}, \eqref{eq:e.def}, \eqref{eq:emod.def} and \eqref{eq:E.def}, define the angle variable $\theta_{\mathfrak L} \in \mathbb R/(2\pi\mathbb Z)$ by
\begin{equation}\label{eq:theta.frkL.def}
(\cos \theta_{\mathfrak{L}},\sin \theta_{\mathfrak{L}})=(\f{n_1}{|\mathbf{n}|}, \f{n_2}{|\mathbf{n}|}),
\end{equation}
the angle variable $\theta_{\mathfrak{L}_z} \in [0,\pi]$ by 
\begin{equation}
\cos \theta_{\mathfrak{L}_z} = \f{\mathbf{n}\cdot \mathbf{e}}{|\mathbf{n}| |\mathbf{e}|},
\end{equation}
and the angle variable $Q \in \mathbb R/(2\pi\mathbb Z)$ by 
\begin{equation}\label{eq:Q.def}
 Q = E - |\mathbf e|\sin E,
\end{equation}
where $n_1$, $n_2$ denotes the first and second component of the vector field $\mathbf{n}$, respectively.
\end{enumerate}
\end{definition}

\begin{remark}[Well-definedness of the Delaunay variables I]\label{rmk:Delaunay.1}
We claim that $(\mathfrak{L}_z, \mathfrak{L}, J, \theta_{\mathfrak{L}_z}, \theta_{\mathfrak{L}}, E)$ are well-defined (and analytic) on the set $\calS_0^{\mathrm{lin}}$ defined in \eqref{eq:f.support}. Given the expression of their definition, it suffices to check that $|\mathbf{e}| >0$, $|\mathbf{n}| >0$ and $-1 < \f{\mathbf{n} \cdot \mathbf{e}}{|\mathbf{n}||\mathbf{e}|} <1$ on $\calS_0^{\mathrm{lin}}$. That $|\mathbf{n}| >0$ and $-1 < \f{\mathbf{n} \cdot \mathbf{e}}{|\mathbf{n}||\mathbf{e}|} <1$ hold are immediate from the definition in \eqref{eq:f.support}.

It remains to show that $|\mathbf{e}| >0$ on $\calS_0^{\mathrm{lin}}$. Using \eqref{eq:emod.def},
\begin{equation}
\begin{split}
|\mathbf{e}|^2 = &\: (| v|^2 |x|-1)^2 + (\f{2}{|x|} - | v|^2)(x\cdot  v)^2 \\
= &\: (|v|^2 - \f 2{|x|})(|v|^2 |x|^2 - (x\cdot  v)^2) + 1 = 2 H_{\Kep}(x,v) L(x,v) + 1.
\end{split}
\end{equation}
Using the conditions in \eqref{eq:f.support}, it follows that $0< |\mathbf{e}| <1$.

\end{remark}

\begin{remark}[Well-definedness of the Delaunay variables II]\label{rmk:Delaunay.2}
To see that $Q$ in \eqref{eq:Q.def} is well-defined, note that if we think of $Q$ as a function of $E$ (for fixed $\mathbf{e}$) 
\begin{equation*}
\begin{split}
Q(2n\pi) = 2n\pi \hbox{ (for $n \in \mathbb Z$)},\quad \f{\ud Q}{\ud E} = 1 - |\mathbf{e}|\cos E >0
\end{split}
\end{equation*}
with the latter holding since $0 < |\mathbf{e}| < 1$ (see Remark~\ref{rmk:Delaunay.1}). Hence, $Q:\mathbb R/(2\pi\mathbb Z) \to \mathbb R/(2\pi\mathbb Z)$ is a diffeomorphism.
\end{remark}

\begin{remark}[The Delaunay variables as coordinates]\label{rmk:Delaunay.3}
The variables $(J,\mathfrak{L},\mathfrak{L}_z,Q,\theta_{\mathfrak{L}},\theta_{\mathfrak{L}_z})$ form a coordinate system on $\calS_0^{\mathrm{lin}}$. In fact, the change of variables is canonical and the following holds:
\begin{equation}\label{eq:canonical}
\sum_{i}\ud x_i\wedge \ud v_i=\ud J\wedge \ud Q+\ud \mathfrak{L}\wedge \ud \theta_{\mathfrak{L}}+\ud \mathfrak{L}_z\wedge \ud \theta_{\mathfrak{L}_z}.
\end{equation}
See, for instance, \cite[p.283]{Arnold2013}.
\end{remark}

\begin{remark}[Invariance of the set $\calS_0^{\mathrm{lin}}$]\label{rmk:invariance}
We know from Remark~\ref{rmk:Delaunay.1}--Remark~\ref{rmk:Delaunay.3} that the Delaunay variables are well-defined and form a coordinate system in $\calS_0^{\mathrm{lin}}$. We claim that in fact if $\mathrm{supp}(f_{\mathrm{in}}) \subset \mathcal S^{\mathrm{lin}}_0$, then for the solution to the linear problem, $\mathrm{supp}(f(t,\cdot)) \subset \mathcal S^{\mathrm{lin}}_0$ for all $t \geq 0$. As a result, we can use the Delaunay variables as a coordinate system for all $t \geq 0$.

To check our claim above, it is straightforward to check that $\D H_{\Kep} = \D \mathbf{L} = \D \mathbf{e} = 0$. From this, it also follows that $\D L = \D \mathbf{n} = \D \f{\mathbf{n} \cdot \mathbf{e}}{|\mathbf{n}||\mathbf{e}|} = 0$. Hence, all the functions used in the definition of $\calS_0^{\mathrm{lin}}$ in \eqref{eq:f.support} are invariant under $\D$ and our claim on the invariance of $\calS_0^{\mathrm{lin}}$ holds.
%
\end{remark}

\begin{proposition}\label{prop:lin-VP-actang}
In the coordinate system $(t,J,\mathfrak{L},\mathfrak{L}_z,Q,\theta_{\mathfrak{L}},\theta_{\mathfrak{L}_z})$ defined as in Definition~\ref{def:Delaunay-action}, the equation \eqref{eq:linear} reduces to
\begin{equation}\label{eq:lin-VP-actang}
\D f = \: \rd_t f +\f{1}{J^3}\rd_Q f = 0
\end{equation}
on the set $\mathcal S^{\mathrm{lin}}_0$ in \eqref{eq:f.support}.
\end{proposition}
\begin{proof}
It suffices to check that 
\begin{equation}\label{eq:DQ.goal.easy.part}
\D \mathfrak{L}_z = \D \mathfrak{L} = \D J = \D \theta_{\mathfrak{L}_z} = \D \theta_{\mathfrak{L}} = 0
\end{equation} 
and
\begin{equation}\label{eq:DQ.goal}
\D Q = \f 1{J^3}.
\end{equation}

Since \eqref{eq:DQ.goal.easy.part} follows from considerations in Remark~\ref{rmk:invariance}, it remains to prove \eqref{eq:DQ.goal}. We first compute
\begin{equation}\label{eq:Dv2x-1}
\begin{split}
\D (| v|^2 |x|-1) = \f{|v|^2(x\cdot v)}{|x|} - \f{2(x\cdot v)}{|x|^2}
\end{split}
\end{equation}
Using the fact that $\D |\mathbf{e}| = 0$ and \eqref{eq:Dv2x-1}, we then deduce that 
\begin{equation}
\begin{split}
\D E = &\: \f{-1}{|\mathbf{e}|\sqrt{1 - \f{1}{|\mathbf{e}|^2} (| v|^2 |x|-1)^2}} \D (| v|^2 |x|-1)  \\
= &\: \f{-1}{(x\cdot  v)\sqrt{\f{2}{|x|} - | v|^2}} \Big( \f{|v|^2(x\cdot v)}{|x|} - \f{2(x\cdot v)}{|x|^2} \Big) \\
= &\: \f{1}{|x|}\sqrt{\f{2}{|x|} - | v|^2}.
\end{split}
\end{equation}
Using the definition of \eqref{eq:Q.def} of $Q$, we thus obtain
\begin{equation}
\begin{split}
\D Q = &\: (1 - |\mathbf e|\cos E )\D E  \\
= &\: (1 - (| v|^2 |x|-1))\f{1}{|x|}\sqrt{\f{2}{|x|} - | v|^2}
=  \Big(\f{2}{|x|} - | v|^2 \Big)^{\f 32} = J^{3},
\end{split}
\end{equation}
as desired. \qedhere
\end{proof}

\begin{remark}
The equation \eqref{eq:lin-VP-actang} can be proven by noting that $(x,v) \mapsto (\mathfrak{L}_z, \mathfrak{L}, J, \theta_{\mathfrak{L}_z}, \theta_{\mathfrak{L}}, Q)$ is a canonical change of coordinates (see Remark~\ref{rmk:Delaunay.3}) so that
$$\D f = \partial_t f +\{H,f\},$$ 
where $\{\cdot,\cdot\}$ is the Poisson bracket. Then, the desired conclusion follows from the fact that $H_{\Kep} = -\f 1{2J^2}$.
\end{remark}

\subsection{Proof of Theorem~\ref{thm:lin-stability}}

We now turn to the proof of Theorem~\ref{thm:lin-stability} and begin with parts (1) and (2). 

Consider the linear equation \eqref{eq:linear} and rewrite it as in \eqref{eq:lin-VP-actang} in the $(t,J,\mathfrak{L},\mathfrak{L}_z,Q,\theta_{\mathfrak{L}},\theta_{\mathfrak{L}_z})$ coordinates. Using the method of characteristics, the solution is given by
\begin{equation}\label{eq:duhamel-no-sym}
\begin{split}
f(t,J,\mathfrak{L},\mathfrak{L}_z,Q,\theta_{\mathfrak{L}},\theta_{\mathfrak{L}_z})=f_{\ini}(J,\mathfrak{L},\mathfrak{L}_z,Q-\f{t}{J^3},\theta_{\mathfrak{L}},\theta_{\mathfrak{L}_z}).
\end{split}
\end{equation}

We are now ready to proof parts (1) and (2) of Theorem~\ref{thm:lin-stability}.
\begin{proof}[Proof of parts (1) and (2) of Theorem~\ref{thm:lin-stability}]
We begin with part (2), which is an immediate consequence of the solution formula \eqref{eq:duhamel-no-sym}. For part (1), in order to control $\rd_x^\alp \rd_v^\bt f$, we transform $\rd_x$ and $\rd_v$ into to the basis $(\rd_{J},\rd_{\mathfrak{L}},\rd_{\mathfrak{L}_z},\rd_{Q},\rd_{\theta_{\mathfrak{L}}},\rd_{\theta_{\mathfrak{L}_z}})$. Then the desired conclusion follows from part (2). (Notice that each of $\rd_{x^i}$ or $\rd_{v^i}$ may give a non-trivial $\rd_{J}$ component, which explains the growth in $\jap{t}$.) \qedhere
\end{proof}

To prove decay for $\rd_s\rho$ and $\rd_s \varphi$, we continue to rely on the representation formula \eqref{eq:duhamel-no-sym}. However, in order to deal with the $v$-integral, it is useful to introduce the $(x^1,x^2,x^3,w,L,\vartheta)$ coordinates, where 
$$w(x,v) = \f{x\cdot v}{|x|},\quad L(x,v)= |x|^2|v|^2 - (x\cdot v)^2,$$
and $\vartheta  \in \mathbb R/(2\pi \mathbb Z)$ is the angle defined so that for every fixed $(x^1,x^2,x^3)$, $(w,\sqrt{L/|x|^2},\vartheta)$ forms a system of cylindrical coordinates of $v$-space with $L = 0$ being the axis\footnote{Note that we have not fully specify $\vartheta$, but any smooth choice of $\vartheta$ so that $(w,\sqrt{L/|x|^2},\vartheta)$ forms a system of cylindrical coordinates will work for the argument below.}. 

We are now read to prove part (3) of Theorem~\ref{thm:lin-stability}.

\begin{proof}[Proof of part (3) of Theorem~\ref{thm:lin-stability}]

Since $\ud v^1 \ud v^2 \ud v^3 = \sqrt{L/|x|^2} \ud \sqrt{L/|x|^2} \ud w \ud \vartheta = \f 1{2|x|^2} \ud L \ud w \ud \vartheta$ in $v$-space, in these coordinates, we have
\begin{equation}\label{eq:Poisson-cyl}
\rho(s,x)=\f{1}{2|x|^2}\int_0^{2\pi}\int_{-\infty}^\infty \int_{0}^\infty f(s,x,w,L,\vartheta)\ud L\ud w\ud \vartheta.
\end{equation}
Note that $f$ (and thus $\rho$) is supported away from $|x| = 0$. (To see this, note that by \eqref{eq:f.support},  
\begin{equation}\label{eq:bounded.away.from.0}
- \f{\mathfrak l_1}{2|x|^2}\geq H_{\Kep} - \f{L}{2|x|^2} = \f{(x\cdot v)^2}{2|x|^2} - \f{1}{|x|} \geq -\f{1}{|x|}.)
\end{equation} 
It thus suffices to prove decay estimates for the following auxilliary quantity
\begin{equation}\label{eq:aux}
\Gamma(s,x):=2|x|^2\rd_s \rho =\int_0^{2\pi}\int_{-\infty}^\infty \int_{0}^\infty \rd_s f(s,x,w,L,\vartheta)\ud L\ud w\ud \vartheta.
\end{equation}

Denote
$$f(s,x,w,L,\vartheta) = \bar{f}(t,J,\mathfrak{L},\mathfrak{L}_z,Q,\theta_{\mathfrak{L}},\theta_{\mathfrak{L}_z}).$$
Taking Fourier series in $Q$ we write
$$\bar{f}_{\ini}(t,J,\mathfrak{L},\mathfrak{L}_z,Q,\theta_{\mathfrak{L}},\theta_{\mathfrak{L}_z}) = \sum_{k\in \mathbb Z} \widehat{\bar{f}}_{\ini,k}(J,\mathfrak{L},\mathfrak{L}_z,\theta_{\mathfrak{L}},\theta_{\mathfrak{L}_z}) e^{ikQ}.$$ 
Thus, according to \eqref{eq:duhamel-no-sym}, we have
\begin{equation}
\begin{split}
\rd_s f(s,x,w,L,\vartheta) =&\: -\sum_{k\in \mathbb Z} \f{ik}{J^3}\widehat{\bar{f}}_{\ini,k}(J,\mathfrak{L},\mathfrak{L}_z,\theta_{\mathfrak{L}},\theta_{\mathfrak{L}_z}) e^{ik(Q-\f{t}{J^3})}.
\end{split}
\end{equation}
In particular, after differentiating $\Gamma$ in \eqref{eq:aux} with $|\alp| \leq N$, we have
\begin{equation}
\rd_x^\alp \Gamma(s,x) = -\sum_{k\in \Z\backslash \{0\}} \sum_{\alp' + \alp'' = \alp} \prod_{\ell=1}^{\ell=3} \f{\alp_\ell!}{\alp_\ell'!\alp_\ell''!} \Gamma_{k,\alp',\alp''}(s,x),
\end{equation}
where
\begin{equation}\label{eq:Gamma.k.alp'.alp''}
\Gamma_{k,\alp',\alp''} = \int_0^{2\pi}\int_{-\infty}^\infty \int_{0}^\infty \rd_x^{\alp'} (\f{ik}{J^3}\widehat{\bar{f}}_{\ini,k} e^{ikQ}) (\rd_x^{\alp''} e^{-\f{ikt}{J^3}}) \ud L\ud w\ud \vartheta. 
\end{equation}

We now use a stationary phase argument. Notice that \eqref{eq:Kepler-Ham} can be rewritten as 
$$H_{\Kep} = \f{w^2}{2} + \f{L}{2|x|^2} - \f 1{|x|}$$
so that in the $(x^1,x^2,x^3,w,L,\vartheta)$ coordinate system, $\rd_L e^{-\f{ikt}{J^3}} = -ikt e^{-\f{ikt}{J^3}} \rd_L (-2H_{\Kep})^{3/2} = \f 32 ikt e^{-\f{ikt}{J^3}}   (-2H_{\Kep})^{1/2} \f 1{|x|^2}$. Since $\f{(-2H_{\Kep})^{1/2}}{|x|^2}$ is bounded away from $0$ on $\calS_0^{\mathrm{lin}}$, this implies that for $|kt|\geq 1$,
\begin{equation}\label{eq:linear.stat.phase}
e^{-\f{ikt}{J^3}}  =  \Big( \f{2|x|^2}{3 ikt   (-2H_{\Kep})^{1/2}} \rd_L \Big)^{N-|\alp'|} e^{-\f{ikt}{J^3}}.
\end{equation} 

Plugging \eqref{eq:linear.stat.phase} into \eqref{eq:Gamma.k.alp'.alp''}, we can write 
\begin{equation*}
\begin{split}
\Gamma_{k,\alp',\alp''}
= &\: (kt)^{-N} \int_0^{2\pi}\int_{-\infty}^\infty \int_{0}^\infty \rd_x^{\alp'} (\f{ik}{J^3}\widehat{\bar{f}}_{\ini,k}e^{ikQ}) \rd_x^{\alp''}\Big(  \Big( \f{2|x|^2}{3 i   (-2H_{\Kep})^{1/2}} \rd_L \Big)^{N-|\alp'|}   e^{-\f{ikt}{J^3}} \Big) \ud L\ud w\ud \vartheta.
\end{split}
\end{equation*}
We now integrate by parts $\rd_L$ for $N-|\alp'|$ times, and observe the following:
\begin{itemize} 
\item Each of the $(N-|\alp'|)$ $L$ derivatives that is integrated by parts will either fall on $\widehat{\bar{f}}_{\ini,k}$, or else will fall on the other quantities and at worst give a factor of $k$ (from differentiating $e^{ikQ}$). 
\item The $\rd_x^{\alp'}$ derivatives will distribute onto $\widehat{\bar{f}}_{\ini,k}$ and $\f{e^{ikQ}}{J^3}$. When they fall on $\f{e^{ikQ}}{J^3}$, each derivative gives a factor $k$.
\item The $\rd_x^{\alp''}$ derivatives will either fall on $\f{2|x|^2}{3 i   (-2H_{\Kep})^{1/2}}$ (which gives a bounded contribution), or else will fall on $e^{-\f{ikt}{J^3}}$, which at worst give $|kt|^{|\alp''|}$.
\end{itemize}
With these observations, and using also that $\f{(-2H_{\Kep})^{1/2}}{|x|^2} \gtrsim 1$ on $\calS_0^{\mathrm{lin}}$, altogether we obtain
\begin{equation}
\begin{split}
|\Gamma_{k,\alp',\alp''}| \ls |kt|^{-N} |kt|^{|\alp''|} \sum_{N' + |\gamma| \leq N} |k|^{1+N'} |\rd^\gamma \widehat{\bar{f}}_{\ini,k}(J,\mathfrak{L},\mathfrak{L}_z,\theta_{\mathfrak{L}},\theta_{\mathfrak{L}_z})|
\end{split}
\end{equation}
where the multi-index $\gamma$ is used for the variables $(J,\mathfrak{L},\mathfrak{L}_z,\theta_{\mathfrak{L}},\theta_{\mathfrak{L}_z})$. 
Hence, recalling $\alp = \alp' + \alp''$, we deduce that when $t \geq 1$,
\begin{equation}\label{eq:linear.dx.Gamma}
|\rd_x^\alp \Gamma| \ls \sum_{N'+|\gamma| \leq N}\sum_{k\in \Z\backslash \{0\}} |k|^{-N+|\alp|+2}|t|^{-N+|\alp|} |k|^{N'-1}|\rd^\gamma  \widehat{\bar{f}}_{\ini,k}(J,\mathfrak{L},\mathfrak{L}_z,\theta_{\mathfrak{L}},\theta_{\mathfrak{L}_z})|.
\end{equation}
By the Cauchy--Schwarz inequality and Plancherel's theorem for each fixed $(J,\mathfrak{L},\mathfrak{L}_z,\theta_{\mathfrak{L}},\theta_{\mathfrak{L}_z})$, we obtain
\begin{equation}\label{eq:linear.Plancherel}
\begin{split}
&\: \sum_{N'+|\gamma| \leq N} \sum_{k \in \mathbb Z \setminus \{0\}}  |k|^{N'-1}|\rd^\gamma  \widehat{\bar{f}}_{\ini,k}|(J,\mathfrak{L},\mathfrak{L}_z,\theta_{\mathfrak{L}},\theta_{\mathfrak{L}_z}) \\
\ls &\: \sum_{N'+|\gamma| \leq N} \Big(\sum_{k \in \mathbb Z \setminus \{0\}} |k|^{2N'}|\rd^\gamma  \widehat{\bar{f}}_{\ini,k}(J,\mathfrak{L},\mathfrak{L}_z,\theta_{\mathfrak{L}},\theta_{\mathfrak{L}_z})|^2 \Big)^{1/2} \Big(\sum_{k \in \mathbb Z \setminus \{0\}} |k|^{-2} \Big)^{1/2} \\
\ls &\: \sum_{N'+|\gamma| \leq N} \Big( \int_0^{2\pi}  |\rd_Q^{N'}\rd^\gamma  \widehat{\bar{f}}_{\ini}(J,\mathfrak{L},\mathfrak{L}_z,\theta_{\mathfrak{L}},\theta_{\mathfrak{L}_z})|^2 \, \ud Q \Big)^{1/2} \\
\ls &\: \sum_{|\alp'| + |\bt'| \leq N}  \sup_{x,v} |\rd_x^{\alp'} \rd_v^{\bt'}  \widehat{\bar{f}}_{\ini}| \ls A, 
\end{split}
\end{equation}
where we used \eqref{eq:linear.assumption} in the last line.

Plugging \eqref{eq:linear.Plancherel} into \eqref{eq:linear.dx.Gamma}, and recalling \eqref{eq:aux}, we thus obtain the desired estimate \eqref{eq:thm.decay-rho-lin} for $\rd_x^\alp \rd_s \rho$ when $|\alp| \leq N-2$.

Finally, we can prove the desire estimates \eqref{eq:thm.decay-lin} for $\rd_x^\alp \rd_s \varphi = \Delta^{-1} \rd_x^\alp \rd_s \rho$ using the bound \eqref{eq:thm.decay-rho-lin} and standard elliptic estimates. We leave out the details. \qedhere
\end{proof}

\section{Proof of Theorem~\ref{thm:main}: notations and setup in spherical symmetry}\label{sec:notation}

For the remainder of the paper, we turn to the proof of Theorem~\ref{thm:main}. In particular, we will restrict ourselves to spherical symmetric solutions. 

Recall that in spherical symmetry, the function $f$ depends on $r$, $w$, $L$ given as in Definition~\ref{def:r.w.L}. It will be convenient to work in the coordinates $(s,r,w,L)$. We will abuse notation to write $f(s,r(x),w(x,v),L(x,v)) = f(s,x,v)$ and $\varphi(s,r(x))= \varphi(s,x)$.
\begin{proposition}\label{prop:transport-spherical}
Under the spherical symmetry assumption, the equation \eqref{eq:transport} in $(s,r,w,L)$ coordinates reduces to
\begin{align}
\D f = &\: \rd_s f + w \rd_{r} f +\left(L/r^3-\rd_{r} \Phi-\rd_{r} \varphi\right)  \rd_{w} f = 0, \label{eq:transport_spherical}\\
\varphi= &\: - \pi \int_0^r\int_{r_2}^\infty \int_{-\infty}^\infty\int_{0}^\infty\frac{1}{rr_1}f(s,r_1,w,L)\d L\d w\d r_1\d r_2. \label{eq:Poisson_spherical}
\end{align}
\end{proposition}
\begin{proof}
\pfstep{Step~1: Transport equation in spherical symmetry}  Using the transport equation \eqref{eq:transport} in the $(s,x,v)$ coordinates, we have
\begin{align}
\mathfrak D s =&\:  1, \\
\mathfrak D r =&\: \f{v_i x_i}r = w, \\
\mathfrak D w =&\: v_i  v_j (\f{\de_{ij}}{r} - \f{x_i x_j}{r^3}) - \f{x_i}{r} \rd_{x_i} (\Phi + \varphi) = L/r^3 - \partial_r (\Phi + \varphi), \label{eq:spherical.transport.Dw}\\
\mathfrak D L =&\: v_i (2 x_i |v|^2 - 2 v_i (x\cdot v)) - \rd_{x_i} (\Phi + \varphi) (2 r^2 v_i - 2 x_i (x\cdot v)) = 0, \label{eq:spherical.transport.DL}
\end{align}
where in \eqref{eq:spherical.transport.DL}, we have used spherical symmetry of $\Phi$ and $\varphi$ to deduce that $\rd_{x_i} (\Phi + \varphi) = \f{x_i}r \rd_r(\Phi + \varphi)$.
The equation \eqref{eq:transport_spherical} then follows immediately.

\pfstep{Step~2: Poisson's equation in spherical symmetry} Define
$$\rho(s,x):= \int_{\mathbb R^3} f(s, r\omega ,v)\, \ud v.$$ 
Slightly abusing notation as above, we write $\rho(s,r) = \rho(s,|x|)$ in spherical symmetry. We begin with writing $\rho$ in spherical coordinates.
To do this, note that for every fixed $x$, $(w, L)$ can be viewed as cylindrical coordinates, where $w$ is the height in the direction parallel to $\f xr$ and $\sqrt{L/r^2}$ is the radius in the plane orthogonal to $\f xr$. Thus, we have
\begin{equation}\label{eq.rho.in.spherical}
\rho(s,r) =  \f{\pi}{r^2} \int_{-\infty}^{\infty} \int_0^\infty f(s,r,w,L) \,\d L\,\d w.
\end{equation}

We now solve the Poisson equation. Since $\varphi$ is spherically symmetric, $\Delta \varphi = \f 1 r \rd_r^2 (r\varphi)$. For a sufficiently regular $f$, the operator $\Delta^{-1}$ imposes the boundary condition $\lim_{r \to 0} r\varphi(r) =0$ and $\lim_{r\to +\infty} \rd_r (r\varphi)(r) = 0$. Thus, using $\Delta \varphi = \rho$ and \eqref{eq.rho.in.spherical}, we obtain
$$\varphi(s,r)=- \pi \int_0^r\int_{r_2}^\infty \int_{-\infty}^\infty\int_{0}^\infty \f{f}{r_1 r}(s,r_1,w,L)\d L\d w\d r_1\d r_2.$$


\end{proof}

\section{Proof of Theorem~\ref{thm:main}: bootstrap assumptions and the main a priori estimates}\label{sec:bootstrap}
We introduce a bootstrap argument on a time interval $[0,T_{\text{B}})$. We make the following bootstrap assumptions on $\rd_s \varphi(s,r)$ and $\varphi(s,r)$ for all $0\leq s <T_{\text{B}}$:
\begin{align}
\sup_{ r\in [\f{\mathfrak l_1}{2}, \f{2}{\mathfrak h}]} |\partial_r^{j}\rd_s\varphi(s,r)| \leq &\: \delta^{3/4}\epsilon \jap{s}^{-\min\{N-j+2,N\}}, \quad 0 \leq j \leq N, \label{eq:varphi_t} \\
\sup_{r\in [\f{\mathfrak l_1}{2}, \f{2}{\mathfrak h}]} |\partial_r^{\ell}\varphi(s,r)| \leq &\: \de^{3/4} \epsilon,\quad 0 \leq \ell \leq N+2, \label{eq:varphi_r} 
\end{align}
as well as the following bootstrap assumption for $\varphi(s_1,r)-\varphi(s_2,r)$ for all $0 \leq s_1 \leq s_2 < T_{\text{B}}$:
\begin{align}
\sup_{r\in [\f{\mathfrak l_1}{2}, \f{2}{\mathfrak h}]} |\partial_r^{k}(\varphi(s_1,r)-\varphi(s_2,r))| \leq &\: \delta^{3/4}\epsilon \jap{s_1}^{-\min\{N-k+2,N\}},\quad 0 \leq k \leq N+1. \label{eq:varphi_diff}
\end{align}


The following is the main bootstrap theorem for $\varphi$ and its derivatives. The end of the proof of the theorem can be found in Section~\ref{sec:density}.
\begin{theorem}\label{thm:boot}
Given $\mathfrak c_0,\, \mathfrak h_0,\, \mathfrak l_1,\, \mathfrak l_2 >0$ and $N \in \mathbb N$ with $N \geq 8$, there exist $\de>0$ and $\ep_0>0$ (depending only on $\mathfrak c_0$, $\mathfrak h_0$, $\mathfrak l_1$, $\mathfrak l_2$ and $N$) such that the following holds for all $\ep \in (0, \ep_0]$.

Consider data as in Theorem~\ref{thm:main}. Suppose there exists $T_{\text{B}}\in (0,  \ep^{-1}(\log 1/\ep)^{-1}]$ such that the solution to \eqref{eq:transport_spherical} remains smooth in the time interval $[0,T_{\text{B}})$ 
and moreover that $\varphi$ (as defined in \eqref{eq:transport_spherical}) satisfies \eqref{eq:varphi_t}, \eqref{eq:varphi_r} and \eqref{eq:varphi_diff} for any $0 \leq s < T_{\text{B}}$ and $0 \leq s_1 \leq s_2 < T_{\text{B}}$.

Then in fact the bounds \eqref{eq:varphi_t}, \eqref{eq:varphi_r} and \eqref{eq:varphi_diff} hold with $\delta^{3/4}$ replaced by $C\delta$. Here, $C>0$ is a constant depending only on $\mathfrak c_0$, $\mathfrak h_0$, $\mathfrak l_1$, $\mathfrak l_2$ and $N$, and is independent of $\de$ and $\ep$.
\end{theorem}

A standard continuity argument using the a priori estimates in Theorem~\ref{thm:boot} implies the existence and regularity statements, as well as the estimates in Theorem~\ref{thm:main}; see Section~\ref{sec:continuity} for details. From now on until Section~\ref{sec:density}, we focus on the proof of Theorem~\ref{thm:boot}. \textbf{We fix $T_{\text{B}} \in (0,  \ep^{-1}(\log 1/\ep)^{-1}]$ as in Theorem~\ref{thm:boot} and work under the assumptions of Theorem~\ref{thm:boot}.}

\textbf{In the following, we allow all implicit constants in $\ls$ or in the big-$O$ notation to depend on $\mathfrak c_0$, $\mathfrak h_0$, $\mathfrak l_1$, $\mathfrak l_2$ and $N$. However, importantly, the implicit constants are to be \underline{in}dependent of $\de$ and $\ep$.}

\section{Proof of Theorem~\ref{thm:main}: dynamical action angle variables}\label{sec:dyn.coord}

\subsection{Weak bootstrap assumptions}
We introduce the weak bootstrap assumptions that are needed in the rest of the section. Notice that these bootstrap assumptions are needed because our action angle variables (see Definition~\ref{def:coordinates} below) depend on the nonlinear solution. On the other hand, the bootstrap assumptions we need here are much weaker than those we introduce later in the core of the argument. 

Consider initial data satisfying the assumptions of Theorem~\ref{thm:main}. Introduce the following \textbf{weak bootstrap assumptions}:
\begin{align}
\sup_{s \in [0,T_{\text{B}}),\, r\in [\f{\mathfrak l_1}{2}, \f{2}{\mathfrak h}]}  \sum_{k=0}^{N+2} |\rd_r^k \varphi(s,r)| \leq &\: \ep,  \label{eq:BA.weak.1}\\
\sup_{s \in [0,T_{\text{B}}),\, r\in [\f{\mathfrak l_1}{2}, \f{2}{\mathfrak h}]} |\rd_s\varphi(s,r)| \leq &\:  \ep \langle s \rangle^{-2}, \label{eq:BA.weak.2} \\
\sup_{s \in [0,T_{\text{B}}),\, r\in [\f{\mathfrak l_1}{2}, \f{2}{\mathfrak h}]} \sum_{k=1}^{N} |\rd_r^k \rd_s \varphi(s,r)| \leq &\: 1, \label{eq:BA.weak.4}
\end{align}
where $N \geq 8$ is as in Theorem~\ref{thm:main}. We note that the weak bootstrap assumptions \eqref{eq:BA.weak.1}--\eqref{eq:BA.weak.4} follow from the bootstrap assumptions \eqref{eq:varphi_t} and \eqref{eq:varphi_r}, but we separate them out to emphasize that the assumptions needed to control the dynamical coordinates are weaker. In particular, there is no need of decay estimates for $\varphi(s,r) - \varphi(T,r)$ as in \eqref{eq:varphi_diff}.

\subsection{Analyzing the support of $f$}\label{sec:support}

\begin{lemma}\label{lem:support}
Let \begin{equation}\label{eq:calS.def}
\mathcal S := \{(H,L): -\f 1{2L} + \mathfrak c \leq H \leq -\mathfrak h <0,\, \mathfrak l_1 \leq L \leq \mathfrak l_2 \},
\end{equation} 
where $\mathfrak c := \f{\mathfrak c_0}2$, $\mathfrak h := \f{\mathfrak h_0}2$ and $\mathfrak c_0$, $\mathfrak h_0$ are as in \eqref{eq:f.support.SS}. 

For every $s\in [0,T_{\text{B}})$, $f(s,r,w,L)\neq 0 \implies (H(s,r,w,L),L) \in \mathcal S$.

\end{lemma}
\begin{proof}
Recall that the initial $f|_{t=0}$ is supported in the set $\mathcal S_0$ (see \eqref{eq:f.support.SS}).
We compute (using \eqref{eq:transport_spherical}) that 
$$\mathfrak D H(s,r,w,L) = \rd_s \varphi(s,r),\quad \mathfrak D L = 0.$$
In particular, using the weak bootstrap assumption \eqref{eq:BA.weak.2}, we see that if $\ep$ is sufficiently small, then $H$ can at most change by $O(\ep)$ along a characteristic and $L$ is conserved along a characteristic. The conclusion of the lemma thus follows. \qedhere
\end{proof}

From now on, we take $\mathfrak c = \f{\mathfrak c_0}2$, $\mathfrak h = \f{\mathfrak h_0}2$ as in Lemma~\ref{lem:support}.

\begin{proposition}\label{prop:support}
For $\ep>0$ sufficiently small, the following holds at all time $s\in [0,T_{\text{B}})$:
\begin{enumerate}
\item If $(H(s,r,w,L),L) \in \calS$, then the following upper and lower bounds hold for $r$:
\begin{equation}\label{eq:r.supp}
\f{\mathfrak l_1} 2 \leq r \leq \f{1}{\mathfrak h} + O(\ep).
\end{equation}
\item If $(H(s,r,w,L),L) \in \calS$, then the following bound holds for $|w|$:
\begin{equation}\label{eq:w.supp}
|w| \leq \f 2{\sqrt{\mathfrak l_1}} + O(\ep).
\end{equation}
\item For every $(H,L) \in \mathcal S$ and every fixed $s \in [0, T_{\text{B}})$, there exist exactly two distinct $r$ values $r_-$ and $r_+$ (defined so that $r_- < r_+$) such that
$$H = \f{L}{2r_\pm^2} - \f 1{r_\pm} + \varphi(s,r_\pm).$$
Moreover,
\begin{equation}\label{eq:rpm.est}
r_\pm = \f{1\pm \sqrt{1+2HL}}{(-2H)} + O(\sqrt{\ep}).
\end{equation}
\item Let $(s,H,L)$ and $r_\pm$ be as in (3). There exists
$\mathfrak b> 0$ (depending only on $\mathfrak c$, $\mathfrak h$ and $\mathfrak l_1$) such that we have
$$|r_\pm(H,L) - L |\geq \mathfrak b >0.$$
\item Let $(s,H,L)$ and $r_\pm$ be as in (3), and $\mathfrak b$ be as in (4). Define $U(s,r,L) := \f{L}{2r^2} -\f 1r + \varphi(s,r)$. There exists $\mathfrak d>0$ (depending only on $\mathfrak c$, $\mathfrak h$, $\mathfrak l_1$ and $\mathfrak l_2$) such that the following holds for all $r \in [\f{\mathfrak l_1}{2}, \f 2{\mathfrak h}]$:
$$|r-L|\geq \f {\mathfrak b} 2 \implies |\rd_r U(s,r,L)| \geq \mathfrak d >0$$
and
$$|r- L|\leq \f {\mathfrak b}2 \implies |H - U(s,r,L)|\geq \mathfrak d >0.$$
\end{enumerate}
\end{proposition}
\begin{proof}
Using that 
\begin{equation}\label{eq:H.neg.basic}
\f{w^2}2 +\f{L}{2r^2} - \f 1r + \varphi(s,r) = H \leq - \mathfrak h,
\end{equation}
and the assumed bound for $\varphi$ in \eqref{eq:BA.weak.1}, we obtain
$$ \f{L}{2r^2} \leq \f 1r - \mathfrak h + O(\ep) \leq \f 1r.$$
From this we deduce the lower bound
$$r\geq \f L2 \geq \f{\mathfrak l_1}2.$$
To prove the upper bound, we use again \eqref{eq:H.neg.basic} together with $w^2 \geq 0$ to deduce
$$(2\mathfrak h + O(\ep)) r^2 -2r +L \leq 0.$$
Notice that the roots of the quadratic polynomial above are given by $\f{2\pm \sqrt{4-8\mathfrak hL + O(\ep)}}{(4\mathfrak h + O(\ep))}$.
Thus, if $(H(r,w,L),L) \in \calS$, we must have
$$r \leq \f{2+ \sqrt{4-8\mathfrak hL + O(\ep)}}{(4\mathfrak h + O(\ep))} \leq \f{1}{\mathfrak h} + O(\ep),$$
giving the desired upper bound on $r$. This finishes the proof of (1).

For (2), we start with \eqref{eq:H.neg.basic} and use $\f{L}{r^2} \geq 0$, $-\mathfrak h <0$ to get
$$\f{w^2}2 < \f 1 r + \varphi(s,r).$$
Using the lower bound for $r$ in \eqref{eq:r.supp} and \eqref{eq:BA.weak.1}, we thus obtain
$$\f{w^2}2 \leq \f{2}{\mathfrak l_1} + O(\ep).$$
Rearranging then gives \eqref{eq:w.supp}.

For (3), first observe that $\mathfrak r_{\pm} = \f{1\pm \sqrt{1+2HL}}{(-2H)}$ are the only two roots to $H = \f{L}{2\mathfrak r_\pm^2} - \f 1{\mathfrak r_{\pm}}$. Note that $\mathfrak r_{\pm}$ are distinct since $1 + 2HL \geq \mathfrak c$ on $\mathcal S$. 

Consider now $\bt(r) = H - \f{L}{2r^2} + \f 1{r} -\varphi(s,r)$. Our goal is to show that $\bt$ has exactly two zeros and that they are $O(\sqrt{\ep})$-close to $\mathfrak r_{\pm}$. Note that for $\ep >0$ sufficiently small, 
\begin{equation}\label{eq:for.only.2.zeros}
H - \f{L}{2r^2} + \f 1{r} 
\begin{cases}
\ls -\sqrt{\ep} & \hbox{if $r \leq \mathfrak r_- - \sqrt{\ep}$ or $r \geq \mathfrak r_+ + \sqrt{\ep}$} \\
\gtrsim \sqrt{\ep} & \hbox{if $\mathfrak r_- + \sqrt{\ep} \leq r \leq \mathfrak r_+ - \sqrt{\ep} $}
\end{cases}
\end{equation}
Taking $\ep$ smaller if necessary, the bound $|\varphi|(s,r) \leq \ep$ from \eqref{eq:BA.weak.1} implies that 
\begin{equation}
\bt(r) \begin{cases}
\ls -\sqrt{\ep} & \hbox{if $r \leq \mathfrak r_- - \sqrt{\ep}$ or $f \geq \mathfrak r_+ + \sqrt{\ep}$} \\
\gtrsim \sqrt{\ep} & \hbox{if $\mathfrak r_- + \sqrt{\ep} \leq r \leq \mathfrak r_+ - \sqrt{\ep} $}
\end{cases}
\end{equation}
In particular, $\bt$ has no zeros outside $(\mathfrak r_- - \sqrt{\ep}, \mathfrak r_- + \sqrt{\ep}) \cup (\mathfrak r_+ - \sqrt{\ep}, \mathfrak r_+ + \sqrt{\ep})$. Moreover, the intermediate value theorem implies that $\bt$ has at least one zero in each of $(\mathfrak r_- - \sqrt{\ep}, \mathfrak r_- + \sqrt{\ep})$ and $(\mathfrak r_+ - \sqrt{\ep}, \mathfrak r_+ + \sqrt{\ep})$. 

It thus remains to show that $(\mathfrak r_- - \sqrt{\ep}, \mathfrak r_- + \sqrt{\ep})$ and $(\mathfrak r_+ - \sqrt{\ep}, \mathfrak r_+ + \sqrt{\ep})$ each has at most one (hence exactly one) zero. For this, first note that $\f{\ud}{\ud r} (H - \tfrac{L}{2r^2} + \f 1{r} )(\mathfrak r_{\pm}) \neq 0$. Thus, using the bound \eqref{eq:BA.weak.1} for $|\rd_r \varphi| \leq \ep$, we see that $\bt$ is strictly monotonic in $(\mathfrak r_- - \sqrt{\ep}, \mathfrak r_- + \sqrt{\ep})$ and $(\mathfrak r_+ - \sqrt{\ep}, \mathfrak r_+ + \sqrt{\ep})$. Putting all these together gives the desired conclusion.

For (4), we first compute that
$$\mathfrak r_{\pm} - L = \f{1\pm \sqrt{1+2HL}}{(-2H)} - L = \f 1{(-2H)} \Big( 1+2HL \pm \sqrt{1+2HL} \Big).$$
Let $h(y) = y - \sqrt{y}$. It is easy to check that there exists $\mathfrak a> 0$ (depending on $\mathfrak c$, $\mathfrak h$ and $\mathfrak l_1$) such that 
$$\inf_{y\in [2 \mathfrak l_1 \mathfrak c, 1-2 \mathfrak h \mathfrak l_1]} |h(y)|\geq \mathfrak a >0.$$
Now notice that $(H,L) \in \mathcal S \implies 1 +2HL \in [2 \mathfrak l_1 \mathfrak c, 1-2 \mathfrak h \mathfrak l_1]$. It therefore follows that if $(H,L) \in \mathcal S$, we have
$$|\mathfrak r_{\pm} - L| \geq \f{1}{2\mathfrak h}\min\{ 2\mathfrak l_1 \mathfrak c, \mathfrak a\}.$$
The conclusion of (4) thus follows from \eqref{eq:rpm.est}.

Finally, we turn to (5). For the first implication, observe that if $|r-L|\geq \f{\mathfrak b}{2}$, then using also (1), we obtain $|-\f{L}{r^3} + \f 1 {r^2}| \geq \f{\mathfrak b}{2}(\f{2}{\mathfrak l_1})^3$. Choosing $\ep>0$ sufficiently small, we thus obtain the lower bound $|\rd_r U(s,r,L)| = |-\f{L}{r^3} + \f 1 {r^2} + \rd_r\varphi(s,r)|\geq \f{\mathfrak b}{2}(\f{4}{\mathfrak l_1})^3$. For the second implication, notice that if $|r - L|\leq \f{\mathfrak b}{2}$, then $|r - r_{\pm}(H,L)|\geq \f{\mathfrak b} 2$ by (4). It follows that $|H - \f{L}{2r^2} + \f 1{r}|\gtrsim \mathfrak b$, and thus, upon choosing $\ep$ smaller, $|H - U(s,r,L)|\gtrsim \mathfrak b$. \qedhere
\end{proof}

\subsection{Defining the action angle variables}

The action angle type variables that we will use are defined in Definition~\ref{def:coordinates} below. For any $T \in (0,T_{\text{B}})$, we will introduce a (different) set of coordinates on the time interval $[0,T]$. (In particular, the subscript $Q_T$ emphasizes its dependence on $T$.) The reason that we need to introduce a different set of coordinates for each $T$ is that when we improve the bound \eqref{eq:varphi_diff}, which involves two times, we will use a system of coordinates that is initiated at the later time so that we can obtain decay. From now on, $T \in (0,T_{\text{B}})$.

First, we start with some conventions.
\begin{definition}\label{def:r-pm-U}[The notations $r_\pm$, $U_T$]
\begin{enumerate}
\item Suppose $(H_T,L) \in \mathcal S$. Define $r_\pm(H_T, L)$ by the conditions 
\begin{enumerate}
\item $r_-(H_T, L) <  r_+(H_T,L)$,
\item $H_T = \f{L}{2r_\pm^2} - \f 1{r_\pm} + \varphi(T,r_\pm)$.
\end{enumerate}
(The well-definedness is guaranteed by part (3) of Proposition~\ref{prop:support}).
\item Denote
\begin{equation}\label{eq:U.def}
U_T(r,L) := \f{L}{2r^2} - \f 1r + \varphi(T,r).
\end{equation}
\end{enumerate}
\end{definition}

Next, we define $Q_T$, which will be one of our coordinate functions.
\begin{definition}\label{def:QT}(The variable $Q_T$)
\begin{enumerate}
\item Define $$\widetilde{Q}_T(r,H_T,L) := \int_{r_-(H_T,L)}^r \f 1{\sqrt{H_T -U_T(\rho,L)}}\, \ud \rho$$ and 
$$\bar{Q}_T(r,w,L) :=
\mathrm{sgn}(w) \widetilde{Q}_T\Big(r, \f {w^2}2 + \f{L}{2r^2} - \f 1r + \varphi(T,r), L\Big),$$
where (from now on) $\mathrm{sgn}(w) = 
\begin{cases}
1 & \hbox{if $w\geq 0$} \\ -1 & \hbox{if $w <0$}
\end{cases}$.
\item Define
\begin{equation}\label{eq:period-actang-nonlinear}
\widetilde{\mathfrak T}(H_T,L) := 2\int_{r_-(H_T,L)}^{r_+(H_T,L)} \f{1}{\sqrt{H_T - U_T(\rho,L)}}\, \ud \rho
\end{equation}
and
\begin{equation}\label{eq:period}
\mathfrak T(r,w,L) := \widetilde{\mathfrak T}\Big(\f{w^2}2+ \f{L}{2r^2} - \f 1r + \varphi(T,r) ,L\Big).
\end{equation}
\item Define the variable $Q_T$ by 
\begin{equation}\label{eq:QT}
Q_T(r,w,L):= \f{2\pi \bar{Q}_T(r,w,L)}{\mathfrak T(r,w,L)} .
\end{equation}
\end{enumerate}
\end{definition}

\begin{remark}[$Q_T$ is not globally smooth]\label{rmk:QT.global}
Notice that $Q_T$ as given in Definition~\ref{def:QT} is discontinuous. (Indeed, it is discontinuous at $(r =r_+(H_T,L),w=0, L)$ for any $(H_T,L)$, with a jump from $-\pi$ to $\pi$.) Nonetheless, this is just similar to the usual discontinuity associated to polar coordinates: indeed, $\sin Q_T$ and $\cos Q_T$ are smooth functions and, moreover, the coordinate vector fields $(\rd_t, \rd_{Q_T}, \rd_{H}, \rd_{M})$ in the $(t,Q_T,H,M)$ coordinate system defined in Definition~\ref{def:coordinates} below are regular vector fields.
\end{remark}

Finally, we define the action angle coordinate system.
\begin{definition}[The $(t,Q_T,H,M)$ coordinate system]\label{def:coordinates}
For every $T \in (0,T_{\text{B}})$, introduce the $(t,Q_T,H,M)$ coordinate system on $\{ (s,r,w,L): s \in [0,T], (H(s,r,w,L),L) \in \calS\}$, where $Q_T$ is as in \eqref{eq:QT} in Definition~\ref{def:QT} and 
$$t := s,\quad H := \f{w^2}2 + \f{L}{2r^2} - \f 1r + \varphi(s,r),\quad M := L.$$
\end{definition}

\begin{remark}
While we could have denoted the new coordinate system as $(s,Q_T, H, L)$, we prefer to introduce the new notations $(t,Q_T,H,M)$ so that in the later parts of the paper, vector fields such as $\rd_t$, $\rd_{Q_T}$, etc.~will be unambiguously defined. See Definition~\ref{def:coordinate.vector.fields}.
\end{remark}

\begin{remark}
Even though we are already using the word ``coordinate system,'' at this point, we have not proven that $(t,Q_T,H,M)$ indeed form a coordinate system. Nevertheless, this will be proven below in Lemma~\ref{lem:Jac}.
\end{remark}

\subsection{Estimating the action angle variables and the change of variable maps}

In this subsection, we control the functions defined in Definition~\ref{def:QT} (thought of as functions of $(r,w,L)$) and also give bounds on the map $(s,r,w,L) \mapsto (t,Q_T,H,M)$. 

Because of the way that $\bar{Q}_T$ is defined as an integral, we need to be careful when deriving higher order derivative estimates when near $r_\pm$, when $H_T - U_T(\rho,L)$ is close to $0$. We will mostly focus on the estimates near $r_-$, with the understanding that the estimates near $r_+$ are similar; see the proof of Proposition~\ref{prop:barQ}.

%
%

\begin{lemma}\label{lem:dr-}
The following identities hold:
\begin{align}
\rd_{H_T} (r_-(H_T,L)) = &\: \f{1}{\rd_r U_T(r_-(H,L),L) }, \label{eq:dr-.1}\\
\rd_L (r_-(H_T,L))  = &\: -\f{1}{2r_-^2(H_T,L)\rd_rU_T(r_-(H_T,L),L) }. \label{eq:dr-.2}
\end{align}
\end{lemma}
\begin{proof}
By definition $r_-(H_T,L)$ satisfies
\begin{equation}\label{eq:r-.def}
H_T - U_T(r_-(H_T,L),L) = 0.
\end{equation}
Differentiating \eqref{eq:r-.def} by $\rd_{H_T}$, we obtain $1 - \rd_r U_T(r_-(H_T,L),L)\rd_{H_T} r_-(H_T,L) = 0$, which yields \eqref{eq:dr-.1}.

Differentiating \eqref{eq:r-.def} by $\rd_{L}$, and recalling that $\rd_L U_T(r,L) = \f{1}{2r^2}$, we obtain 
\begin{equation*}
\begin{split}
0 = &\: \rd_rU_T(r_-(H_T,L),L) \rd_L r_-(H_T,L)  + \rd_L U_T (r_-(H_T,L),L) \\
= &\: \rd_rU_T (r_-(H_T,L),L) \rd_L r_-(H_T,L)  + \f{1}{2r_-^2(H_T,L)}.
\end{split}
\end{equation*}
Rearranging yields \eqref{eq:dr-.2}. \qedhere
\end{proof}

In the next few lemmas, we compute integrals of the form \eqref{eq:int.sqrt.downstairs}. We will only carry out the computations for the $\int_{r_-(H_T,L)}^{r}$ integral. A completely analogous computation can be performed for the $\int_{r}^{r_+(H_T,L)}$ integral but will be omitted. They will then be used for obtaining estimates for the derivatives of $Q_T$.

\begin{lemma}\label{lem:int.sqrt.downstairs}
Let $\gamma(r,H_T,L)$ be a $C^1$ function. Then 
\begin{equation}\label{eq:int.sqrt.downstairs}
\begin{split}
&\: \int_{r_-(H_T,L)}^{r}  \f{\gamma(\rho,H_T,L)\, \ud \rho}{ [H_T - U_T(\rho,L)]^{1/2}} \\
= &\: -\f{2\gamma(r_-(H_T,L),H_T,L)}{\rd_r U_T(r_-(H_T,L),L)}\sqrt{H_T - U_T(r,L)}  + \int_{r_-(H_T,L)}^{r}  \f{\wtgmm(r,H_T,L)}{[H_T - U_T(\rho,L)]^{1/2}} \, \ud \rho,
\end{split}
\end{equation}
where
\begin{equation}
\begin{split}
\wtgmm(r,H_T,L) = &\:\gamma(r,H_T,L) - \f{\gamma(r_-(H_T,L),H_T,L)\rd_r U_T(r,L)}{\rd_rU_T(r_-(H_T,L),L)}
\end{split}
\end{equation}
which satisfies the property that
\begin{equation}\label{eq:state.wtgmm.vanishes}
\lim_{r\to r_-} \wtgmm(r,H_T,L) = 0.
\end{equation}
\end{lemma}
\begin{proof}
For brevity, we write $r_- = r_-(H_T,L)$ when there is no danger of confusion. We first rewrite the integral in \eqref{eq:int.sqrt.downstairs} as follows:
\begin{equation}
\begin{split}
&\: \int_{r_-}^{r}  \f{\gamma(\rho,H_T,L)}{ [H_T - U_T(\rho,L)]^{1/2}} \, \ud \rho  \\
= &\: \f{\gamma(r_-,H_T,L)}{\rd_rU_T(r_-,L)}\int_{r_-}^{r}  \f{\rd_r U_T(\rho,L)}{ [H_T - U_T(\rho,L)]^{1/2}} \, \ud \rho \\
&\: + \int_{r_-}^{r}  \Big[\gamma(\rho,H_T,L) - \f{\gamma(r_-,H_T,L)\rd_r U_T(\rho,L)}{\rd_rU_T(r_-,L)} \Big]\f{\ud \rho}{ [H_T - U_T(\rho,L)]^{1/2}}  =: \mathfrak I_1 + \mathfrak I_2.
\end{split}
\end{equation}

The term $\mathfrak I_2$ is already in the form as needed in \eqref{eq:int.sqrt.downstairs}. We further compute the integral in $\mathfrak I_1$. Changing variables to $V = U_T(\rho,L)$, we have
\begin{equation}
\int_{r_-}^{r}  \f{\rd_r U_T(\rho,L)}{ [H_T - U_T(\rho,L)]^{1/2}} \, \ud \rho = \int_{H_T}^{U(t,r)}  \f {\ud V}{\sqrt{H - V}} = -2\sqrt{H_T - U_T(r,L)}.
\end{equation}
Thus
\begin{equation}
\mathfrak I_1 = -\f{2\gamma(r_-,H_T,L)}{\rd_rU_T(r_-,L)}\sqrt{H_T - U_T(r,L)}.
\end{equation}
This proves \eqref{eq:int.sqrt.downstairs}.

It remains to prove that \eqref{eq:state.wtgmm.vanishes} holds. First, note that $\rd_rU_T(r_-,L)$ is bounded away from $0$ by Proposition~\ref{prop:support}.(5). We then define
\begin{equation}\label{eq:alp.for.general.int}
\alp(r,H_T,L) = \gamma(r,H_T,L) \rd_r U_T(r_-(H_T,L),L) - \gamma(r_-(H_T,L),H_T,L)\rd_r U_T(r,L).
\end{equation}
Notice that when $r \leq L- \f{\mathfrak b}2$, it follows from the definition of $r_-$ that $r_-(U_T(r,L),L) = r$. It then easily follows that
$$\alp(r,U_T(r,L),L) = 0,$$
which in turn implies \eqref{eq:state.wtgmm.vanishes}. \qedhere

\end{proof}


We notice that there are two types of terms on the right-hand side of \eqref{eq:int.sqrt.downstairs}. One type is of the form of $[H_T - U_T(r,L)]^{1/2}$ multiplied by a regular function in $\gamma(r,H_T,L)$, while the other type is an integral of the form $\int_{r_-(H_T,L)}^{r} \f{\wtgmm(\rho,H_T,L)}{ [H_T - U_T(\rho,L)]^{1/2}} \, \ud \rho$, but now with $\wtgmm$ vanishing at $r_-$. We consider the derivatives of these functions in the $(r,w,L)$ in the next two lemmas.

\begin{lemma}\label{lem:diff.trivial}
Suppose $\gamma(r,H_T,L)$ is a $C^1$ function. Define 
\begin{equation}
\bar\eta(r,H_T,L) =  \gamma(r,H_T,L) [H_T - U_T(r,L)]^{1/2} 
\end{equation}
and
\begin{equation}
\eta(r,w,L) = 
\mathrm{sgn}(w)
\bar{\eta}(r,\f {w^2}2 + \f{L}{2r^2} - \f 1r + \varphi(T,r),L).
\end{equation}
Then 
\begin{align}
\rd_r|_{(r,w,L)} \eta(r,w,L) = &\: \f{w}{\sqrt{2}} (\rd_r\gamma +\rd_r U\rd_{H_T} \gamma)(r,\f {w^2}2 + \f{L}{2r^2} - \f 1r + \varphi(T,r),L), \\
\rd_w|_{(r,w,L)} \eta(r,w,L) = &\: \f{1}{\sqrt{2}} (\gamma + w^2  \rd_{H_T}\gamma)(r,\f {w^2}2 + \f{L}{2r^2} - \f 1r + \varphi(T,r),L), \\
\rd_L|_{(r,w,L)} \eta(r,w,L) = &\: \f{w}{\sqrt{2}} (\f{1}{2r^2} \rd_{H_T}\gamma + \rd_L \gamma)(r,\f {w^2}2 + \f{L}{2r^2} - \f 1r + \varphi(T,r),L).
\end{align}
\end{lemma}
\begin{proof}
Notice that after setting $H_T = \f {w^2}2 + \f{L}{2r^2} - \f 1r + \varphi(T,r)$, we have $ [H_T - U_T(r,L)]^{1/2}  = \f{|w|}{\sqrt{2}}$. The lemma then follows easily from the chain rule. \qedhere
\end{proof}

We now compute the derivatives of general integrals of the form $\int_{r_-(H_T,L)}^{r} \f{\wtgmm(\rho,H_T,L)}{[H_T - U_T(\rho,L)]^{1/2}} \, \ud \rho$, where $\wtgmm$ vanishes at $r_-$.
\begin{lemma}\label{lem:diff.elliptic.int}
Suppose $\wtgmm(r,H_T,L)$ is defined for $(H_T,L) \in \calS$, $r \in [r_-(H_T,L), r_+(H_T,L)]$ and is a $C^1$ function (up to boundary) satisfying
\begin{equation}\label{eq:wtgmm.good.at.boundary}
\lim_{r\to r_-} \wtgmm(r,H_T,L) = 0.
\end{equation} 
Define 
\begin{equation}
\bar\eta(r,H_T,L) = \int_{r_-(H_T,L)}^{r} \f{\wtgmm(\rho,H_T,L)}{[H_T - U_T(\rho,L)]^{1/2}} \, \ud \rho
\end{equation}
and
\begin{equation}
\eta(r,w,L) = 
\mathrm{sgn}(w) \bar{\eta}(r,\f {w^2}2 + \f{L}{2r^2} - \f 1r + \varphi(T,r),L).
\end{equation}

Then the following identities hold for all $(r,w,L)$ satisfying $(H(r,w,L),L) \in \calS$ and $r \in [r_-(H_T,L), L-\f{\mathfrak b}2]$, where the right-hand side is understood as a function of $(r,w,L)$ after setting $H_T = \f {w^2}2 + \f{L}{2r^2} - \f 1r + \varphi(T,r)$:
\begin{equation}\label{eq:rdr.elliptic.int}
\begin{split}
 \rd_r|_{(r,w,L)} \eta(r,w,L) 
= &\: \mathrm{sgn}(w)
\rd_r U_T(r,L) \int_{r_-(H_T,L)}^{r} \f{[ \rd_{\rho} (\tfrac{\wtgmm}{\rd_{\rho} U_T}) + \rd_{H_T} \wtgmm](\rho,H_T,L)}{ [H_T - U_T(\rho,L)]^{1/2}} \, \ud \rho,
\end{split}
\end{equation}
\begin{equation}\label{eq:rdw.elliptic.int}
\begin{split}
&\: \rd_w|_{(r,w,L)} \eta(r,w,L) 
= -\f{-\sqrt{2} \wtgmm(r,H_T,L)}{\rd_r U_T(r,L)} + |w|\int_{r_-(H_T,L)}^{r}\frac{[\rd_{\rho} (\tfrac{\wtgmm}{\rd_{\rho} U_T}) + \rd_{H_T} \wtgmm](\rho,H_T,L)}{ [H_T - U_T(\rho,L)]^{1/2}} \, \ud \rho,
\end{split}
\end{equation}
and
\begin{equation}\label{eq:rdL.elliptic.int}
\begin{split}
&\: \rd_L|_{(r,w,L)} \eta(r,w,L) \\
= &\: \mathrm{sgn}(w) \int_{r_-(H_T,L)}^{r}\f{[ \rd_L \wtgmm - \f 12 \rd_{\rho} \Big( \tfrac{\wtgmm}{\rho^2 \rd_{\rho} U_T} \Big) + \f 1{2r^2}\rd_{\rho} (\f{\wtgmm}{\rd_{\rho} U_T}) + \f 1{2r^2}\rd_{H_T} \wtgmm ](\rho,H_T,L)}{ [H_T - U_T(\rho,L)]^{1/2}} \, \ud \rho.
\end{split}
\end{equation}
Notice that thanks to Proposition~\ref{prop:support}.(5), $\rd_r U_T$ is bounded away from zero and can thus be divided.
\end{lemma}
\begin{proof}
By \eqref{eq:wtgmm.good.at.boundary} and the fact that $\wtgmm$ is $C^1$, we have
\begin{equation}\label{eq:wtgmm.boundary.term.1}
\limsup_{r\to r_-} \f{|\wtgmm|(r,H_T,L)}{(r-r_-)} < \infty.
\end{equation}
Notice also that for $r \in [r_-(H_T,L), L-\f{\mathfrak b}2]$,
\begin{equation}\label{eq:wtgmm.boundary.term.2}
(r-r_-) \ls |H_T - U_T(r,L)| \ls (r-r_-)
\end{equation}
This follows from the fact that $H_T - U_T(r_-,L) = 0$ (by definition of $r_-$) and $|\rd_r |_{(r,H_T,L)} (H_T - U_T)(r_-,L)| = |(\rd_r U_T)(r_-,L)| \geq \mathfrak d$ (by part (5) of Proposition~\ref{prop:support}).

It is useful to begin with auxiliary computations in the $(r,H_T,L)$ coordinate system. In the following computation, we use $[\rd_{H_T} + \f 1{\rd_{\rho} U_{T}(\rho,L)} \rd_\rho] [H_T - U_T(\rho,L)]^{-1/2} = 0$ and integrate by parts. (Notice that due to \eqref{eq:wtgmm.boundary.term.1} and \eqref{eq:wtgmm.boundary.term.2}, the integration by parts gives no boundary terms at $r_-$. Moreover, for the same reasons, there are no contributions from $\rd_{H_T}$ hitting $r_-(H_T,L)$ on the left end-point.)
\begin{equation}\label{eq:rdHT.elliptic.int}
\begin{split}
&\: \rd_{H_T} |_{(r,H_T,L)} \int_{r_-(H_T,L)}^{r} \f{\wtgmm(\rho,H_T,L)}{[H_T - U_T(\rho,L)]^{1/2}} \, \ud \rho \\
= &\: -\int_{r_-(H_T,L)}^{r} \f{\wtgmm(\rho,H_T,L)}{\rd_{\rho} U_T(\rho,L)} \rd_\rho [H_T - U_T(\rho,L)]^{-1/2} \, \ud \rho  + \int_{r_-(H_T,L)}^{r} \f{\rd_{H_T} \wtgmm(\rho,H_T,L)}{ [H_T - U_T(\rho,L)]^{1/2}} \, \ud \rho\\
= &\: -\f{\wtgmm(r,H_T,L)}{\rd_{r} U_T(r,L)[H_T - U_T(r,L)]^{1/2} } \\
&\: + \int_{r_-(H_T,L)}^{r}\Big[ \rd_{\rho} (\f{\wtgmm}{\rd_{\rho} U_T}) + \rd_{H_T} \wtgmm\Big](\rho,H_T,L) [H_T - U_T(\rho,L)]^{-1/2} \, \ud \rho.
\end{split}
\end{equation}

We also compute the $\rd_{L} |_{(r,H_T,L)}$ derivative. For this, we instead use
$[\rd_{L} - \f{1}{2\rho^2(\rd_{\rho} U_T)(\rho,L)} \rd_\rho][H_T - U_T(\rho,L)]^{-1/2} = 0$ and integrate by parts. As before, there are no contributions at the $r = r_-$ boundary thanks to \eqref{eq:wtgmm.boundary.term.1} and \eqref{eq:wtgmm.boundary.term.2}.
\begin{equation}\label{eq:rdL.inRHL.elliptic.int}
\begin{split}
&\: \rd_{L} |_{(r,H_T,L)} \int_{r_-(H_T,L)}^{r} \f{\wtgmm(\rho,H_T,L)}{[H_T - U_T(\rho,L)]^{1/2}} \, \ud \rho \\
= &\: \int_{r_-(H_T,L)}^{r} \f{\wtgmm(\rho,H_T,L)}{2\rho^2\rd_{\rho} U_T(\rho,L)} \rd_\rho [H_T - U_T(\rho,L)]^{-1/2} \, \ud \rho  + \int_{r_-(H_T,L)}^{r} \f{\rd_{L} \wtgmm(\rho,H_T,L)}{ [H_T - U_T(\rho,L)]^{1/2}} \, \ud \rho\\
= &\: \f{\wtgmm(r,H_T,L)}{2r^2\rd_{r} U_T(r,L)[H_T - U_T(r,L)]^{1/2} } \\
&\: + \int_{r_-(H_T,L)}^{r}\Big[ - \f 12 \rd_{\rho} \Big( \f{\wtgmm}{\rho^2 \rd_{\rho} U_T} \Big) + \rd_{L} \wtgmm\Big](\rho,H_T,L) [H_T - U_T(\rho,L)]^{-1/2} \, \ud \rho.
\end{split}
\end{equation}

We now return to the derivatives in the $(r,w,L)$ coordinate system and prove the main identities. We first consider the $\rd_r|_{(r,w,L)}$ derivative and prove \eqref{eq:rdr.elliptic.int}. We use the chain rule and \eqref{eq:rdHT.elliptic.int}. Notice that the terms $-\wtgmm(r,H_T,L) [H_T - U_T(\rho,L)]^{-1/2}$ cancel.
\begin{equation*}
\begin{split}
&\: \rd_r|_{(r,w,L)} \int_{r_-(H_T,L)}^{r} \f{\wtgmm(\rho,H_T,L)}{[H_T - U_T(\rho,L)]^{1/2}} \, \ud \rho \\
= &\: \rd_r|_{(r,H_T,L)} \int_{r_-(H_T,L)}^{r} \f{\wtgmm(\rho,H_T,L)}{ [H_T - U_T(\rho,L)]^{1/2}} \, \ud \rho \\
&\: + \rd_r U_T(r,L) \rd_{H_T} |_{(r,H_T,L)} \int_{r_-(H_T,L)}^{r} \f{\wtgmm(\rho,H_T,L)}{ [H_T - U_T(\rho,L)]^{1/2}} \, \ud \rho \\
= &\: \rd_r U_T(r,L) \int_{r_-(H_T,L)}^{r}\Big[ \rd_{\rho} (\f{\wtgmm}{\rd_{\rho} U_T}) + \rd_{H_T} \wtgmm\Big](\rho,H_T,L) [H_T - U_T(\rho,L)]^{-1/2} \, \ud \rho.
\end{split}
\end{equation*}
This proves \eqref{eq:rdr.elliptic.int}.

Next, we turn to the $\rd_w|_{(r,w,L)}$ derivative, which is similar.
\begin{equation*}
\begin{split}
&\: \rd_w|_{(r,w,L)} \int_{r_-(H_T,L)}^{r} \f{\wtgmm(\rho,H_T,L)}{ [H_T - U_T(\rho,L)]^{1/2}} \, \ud \rho \\
= &\: w \rd_{H_T} |_{(r,H_T,L)} \int_{r_-(H_T,L)}^{r} \f{\wtgmm(\rho,H_T,L)}{ [H_T - U_T(\rho,L)]^{1/2}} \, \ud \rho \\
= &\: -\f{\sqrt{2} w\wtgmm(r,H_T,L)}{|w|\rd_r U_T(r,L)} + w\int_{r_-(H_T,L)}^{r}\Big[ \rd_{\rho} (\f{\wtgmm}{\rd_{\rho} U_T}) + \rd_{H_T} \wtgmm\Big](\rho,H_T,L) [H_T - U_T(\rho,L)]^{-1/2} \, \ud \rho,
\end{split}
\end{equation*}
where we used that $[H_T - U_T(\rho,L)]^{1/2} = \f 1{\sqrt{2}} |w|$. This proves \eqref{eq:rdw.elliptic.int}.

Finally, we compute the $\rd_L|_{(r,w,L)}$ derivative using both \eqref{eq:rdHT.elliptic.int} and \eqref{eq:rdL.inRHL.elliptic.int}. Notice that each of \eqref{eq:rdHT.elliptic.int} and \eqref{eq:rdL.inRHL.elliptic.int} gives a boundary term, but the total contributions cancel.
\begin{equation*}
\begin{split}
&\: \rd_L|_{(r,w,L)} \int_{r_-(H_T,L)}^{r} \f{\wtgmm(\rho,H_T,L)}{[H_T - U_T(\rho,L)]^{1/2}} \, \ud \rho \\
= &\: \rd_L|_{(r,H_T,L)} \int_{r_-(H_T,L)}^{r} \f{\wtgmm(\rho,H_T,L)}{[H_T - U_T(\rho,L)]^{1/2}} \, \ud \rho + \f 1{2r^2} \rd_{H_T}|_{(r,H_T,L)} \int_{r_-(H_T,L)}^{r} \f{\wtgmm(\rho,H_T,L)}{[H_T - U_T(\rho,L)]^{1/2}} \, \ud \rho \\
= &\: \int_{r_-(H_T,L)}^{r}\Big[ \rd_L \wtgmm - \f 12 \rd_{\rho} \Big( \f{\wtgmm}{\rho^2 \rd_{\rho} U_T} \Big) + \f 1{2r^2}\rd_{\rho} (\f{\wtgmm}{\rd_{\rho} U_T}) + \f 1{2r^2}\rd_{H_T} \wtgmm\Big](\rho,H_T,L) [H_T - U_T(\rho,L)]^{-1/2} \, \ud \rho.
\end{split}
\end{equation*}
This proves \eqref{eq:rdL.elliptic.int} and concludes the proof of the lemma. \qedhere
\end{proof}

Using the above computations, we can now control the derivatives of $\bar{Q}_T$. Notice that while $\bar{Q}_T$ is not globally defined (see~Remark~\ref{rmk:QT.global}), the derivatives of $\bar{Q}_T$ are well-defined everywhere when $(H_T,L) \in \calS$.
\begin{proposition}\label{prop:barQ}
Define 
\begin{equation}\label{eq:tQ.def}
\widetilde{Q}_T(r,H_T,L) = \int_{r_-(H_T,L)}^r \f 1{\sqrt{H_T -U_T(\rho,L)}}\, \ud \rho
\end{equation}
and 
\begin{equation}\label{eq:barQ.def}
\bar{Q}_T(r,w,L) = 
\mathrm{sgn}(w) \widetilde{Q}_T\Big(r, \f {w^2}2 + \f{L}{2r^2} - \f 1r + \varphi(T,r), L\Big).
\end{equation}
Then, for $I \leq N+2$,
\begin{equation}\label{eq:COV.est}
\sup_{(r,w,L): (H_T,L)\in \mathcal S} \sum_{i_1+ i_2+i_3 \leq I} |\rd_r^{i_1} \rd_w^{i_2} \rd_L^{i_3} \bar Q_T| \ls \sup_{r \in [\f{\mathfrak l_1}2, \f{2}{\mathfrak h}]} \sum_{i \leq I} (1+ |\rd_r^{i+1} \varphi |(T,r)).
\end{equation}
\end{proposition}
\begin{proof}
\pfstep{Step~1: $r\in [r_-, L - \f{\mathfrak b}2]$} In this region, we repeatedly apply Lemmas~\ref{lem:int.sqrt.downstairs}, \ref{lem:diff.trivial} and \ref{lem:diff.elliptic.int} to obtain the expressions for the derivatives of $\bar{Q_T}$. The estimate then easily follows after using \eqref{eq:wtgmm.boundary.term.2} and the bootstrap assumption \eqref{eq:BA.weak.1}.

\pfstep{Step~2: $r \in [L - \f{\mathfrak b}2, L + \f{\mathfrak b}2]$} In the case, we write
$$\widetilde{Q}_T(r,H_T,L) =  \int_{r_-}^{L - \f {\mathfrak b}2} \f 1{\sqrt{H_T -U_T(\rho,L)}}\, \ud \rho + \int_{L - \f {\mathfrak b}2}^r \f 1{\sqrt{H_T -U_T(\rho,L)}}\, \ud \rho = (\widetilde{Q}_{T,1} + \widetilde{Q}_{T,2})(r,H_T,L).$$
When viewed as a function in $r,w,L$, the derivatives of the first term obey the desired bounds using the bounds in Step~1. For the second term, notice that for $\rho \in [L - \f{\mathfrak b}2, L + \f{\mathfrak b}2]$, we have $\sqrt{H_T - U_T(\rho,L)} = \f{|w|}{\sqrt{2}} \geq \mathfrak d >0$ (by part (5) of Proposition~\ref{prop:support}). Thus, for the second term, we have the stronger estimate that
$$\sup_{\substack{ (H_T,L) \in \mathcal S \\ r \in [L - \f{\mathfrak b}2, L + \f{\mathfrak b}2]}} \sum_{i_1+ i_2+i_3 \leq I} |\rd_r^{i_1} \rd_{H_T}^{i_2} \rd_L^{i_3} \widetilde{Q}_{T,2}| \ls \sup_{r \in [L - \f{\mathfrak b}2, L + \f{\mathfrak b}2]} \sum_{i \leq I} (1+ |\rd_r^{i} \varphi |(T,r)),$$
which implies \eqref{eq:COV.est} a fortiori.

\pfstep{Step~3: $r \in [L + \f{\mathfrak b}2, r_+(H_T,L)]$} By Steps~1 and 2, it suffices to control
$$\int_{L + \f {\mathfrak b}2}^r \f 1{\sqrt{H_T -U_T(\rho,L)}}\, \ud \rho,$$ 
which can be written as 
\begin{equation}\label{eq:integral.near.r+}
\int_{L + \f {\mathfrak b}2}^{r_+(H_T,L)} \f 1{\sqrt{H_T -U_T(\rho,L)}}\, \ud \rho - \int_{r}^{r_+(H_T,L)} \f 1{\sqrt{H_T -U_T(\rho,L)}}\, \ud \rho.
\end{equation}
We notice that the same argument as in Step~1 holds when $\int_{r_-(H_T,L)}^{r}$ is replaced by $\int_{r}^{r_+(H_T,L)}$ and so both integrals in \eqref{eq:integral.near.r+} can be bounded as in Step~1. \qedhere
\end{proof}

\begin{proposition}\label{prop:der-period}
Let $\mathfrak T$ be defined by \eqref{eq:period}. Then,
\begin{equation}\sup_{(r,w,L): (H_T,L)\in \mathcal S} \sum_{i_1+ i_2+i_3 \leq N+1} |\rd_r^{i_1} \rd_w^{i_2} \rd_L^{i_3} \mathfrak T| \ls 1.
\end{equation}
 Furthermore,
\begin{equation}
\sup_{(r,w,L): (H_T,L)\in \mathcal S} \sum_{i_1+ i_2+i_3 \leq N+2} |\rd_r^{i_1} \rd_w^{i_2} \rd_L^{i_3} \mathfrak T| \ls  1+ |\rd_r^{N+3} \varphi |(T,r).
\end{equation}
\end{proposition}
\begin{proof}
This is a consequence of Proposition~\ref{prop:barQ}, after noting that 
$$\mathfrak T = 2\bar{Q}_T(r_+(H_T(r,w,L),L),w,L).$$ \qedhere
\end{proof}

\begin{corollary}\label{cor:der-Q}
Let $Q_T$ be defined as in \eqref{eq:QT}. Then,
\begin{equation}
\sup_{(r,w,L): (H_T,L)\in \mathcal S}  \sum_{i_1+ i_2+i_3 \leq N+1} |\rd_r^{i_1} \rd_w^{i_2} \rd_L^{i_3} Q_T| \ls  1.
\end{equation}
Furthermore,
\begin{equation}
\sup_{(r,w,L): (H_T,L)\in \mathcal S}  \sum_{i_1+ i_2+i_3 \leq N+2} |\rd_r^{i_1} \rd_w^{i_2} \rd_L^{i_3} Q_T| \ls   1+ |\rd_r^{N+3} \varphi |(T,r).
\end{equation}
\end{corollary}
\begin{proof}
This follows from combining the estimates in Proposition~\ref{prop:barQ} and Proposition~\ref{prop:der-period}. \qedhere
\end{proof}

\begin{lemma}\label{lem:Jac}
Let $Q_T$ be defined as in \eqref{eq:QT}. Then the following hold:
\begin{enumerate}
\item The change of variables
$$(s,r,w,L) \mapsto (t,Q_T,H,M)$$
is a $C^{I}$ map for $I\leq N+2$ and all the partial derivatives up to order $I$ are bounded $\ls \sum_{i \leq I} (1+ |\rd_r^{i+1} \varphi |(T,r))$. 
\item The Jacobian determinant of the change of variable map is $\f{2\sqrt{2}\pi}{\widetilde{\mathfrak T}}+ O(\ep)$. In particular, $(t,Q_T,H,M)$ forms a coordinate system.
\end{enumerate}
\end{lemma}
\begin{proof}
\pfstep{Step~1: Derivatives of the change of variables map} We write $H(s,r,w,L) = \f{w^2}2 + U(s,r,L)$, where $U(s,r,L) := \f{L}{2r^2} - \f 1r + \varphi(s,r)$. It is then easy to compute that
\begin{equation}\label{eq:change.of.variables.matrix}
\begin{split}
&\: \begin{bmatrix}
\rd_s t & \rd_s Q_T & \rd_s H & \rd_s M \\
\rd_r t & \rd_r Q_T & \rd_r H & \rd_r M \\
\rd_w t & \rd_w Q_T & \rd_w H & \rd_w M \\
\rd_L t & \rd_L Q_T & \rd_L H & \rd_L M
\end{bmatrix} 
= 
 \begin{bmatrix}
1 & 0 & \rd_s \varphi & 0 \\
0 & \rd_r|_{(s,r,w,L)} Q_T & \partial_r U & 0 \\
0 & \rd_w|_{(s,r,w,L)} Q_T & w & 0 \\
0 & \rd_L|_{(s,r,w,L)} Q_T & \f 1{2r^2} & 1
\end{bmatrix}.
\end{split}
\end{equation}

The fact that all the partial derivatives up to $N+1$ order of $(t,Q_T,H,M)$ are bounded $\ls 1$ thus follows from \eqref{eq:BA.weak.1} and Corollary~\ref{cor:der-Q}.


\pfstep{Step~2: Computation of the Jacobian determinant} Using \eqref{eq:change.of.variables.matrix} and the definition of $Q_T$ in Definition~\ref{def:QT}, we compute that Jacobian determinant of the change of variable map $(s,r,w,L) \mapsto (t,Q_T,H,M)$ equals to
\begin{equation}\label{eq:calJ.compute}
\begin{split}
\mathcal J := &\: w (\rd_r |_{(s,r,w,L)} Q_T) - \rd_r U(\rd_w |_{(s,r,w,L)} Q_T) \\
= &\: |w| \Big( \rd_r |_{(s,r,H_T,L)} \f{2\pi \widetilde{Q}_T}{\widetilde{\mathfrak T}} + (\rd_r |_{(s,r,w,L)} H_T)  (\rd_{H_T} |_{(s,r,H_T,L)} \f{2\pi \widetilde{Q}_T}{\widetilde{\mathfrak T}}) \Big) \\
&\: \qquad -  |w| \rd_r U(s,r,L) \rd_{H_T} |_{(s,r,H_T,L)} \f{2\pi \widetilde{Q}_T}{\widetilde{\mathfrak T}}.
\end{split}
\end{equation}
Notice that $\rd_r |_{(s,r,H_T,L)} \widetilde{Q}_T(r,H_T,L) = \f{1}{\sqrt{H_T - U(r,L)}}$. On the other hand, $\rd_r |_{(s,r,H_T,L)} \widetilde{\mathfrak T}(H_T,L) = 0$. Therefore, 
\begin{equation}\label{eq:drQ.in.rHL}
\rd_r |_{(s,r,H_T,L)} \f{2\pi \widetilde{Q}_T}{\widetilde{\mathfrak T}} = \f{2\pi}{\widetilde{\mathfrak T} \sqrt{H_T - U(r,L)}} = \f{2\sqrt{2}\pi}{ \widetilde{\mathfrak T} |w|},
\end{equation}
where in the last step we set $H_T = \f{w^2}2+ U_T(r,L)$ as in Definition~\ref{def:QT}.

For the remaining terms, we use \eqref{eq:BA.weak.1} to deduce that
\begin{equation}\label{eq:Jacobian.junk.1}
\rd_r |_{(s,r,w,L)} H_T - \rd_r U(s,r,L) = \rd_r \varphi(T,r) - \rd_r \varphi(s,r) = O(\ep).
\end{equation}
On the other hand, 
\begin{equation}\label{eq:Jacobian.junk.2}
|\rd_{H_T} |_{(s,r,H_T,L)} \widetilde{Q}_T| \ls \f 1{|w|},\quad |\rd_{H_T} |_{(s,r,H_T,L)} \widetilde{\mathfrak T}| \ls \f 1{|w|}.
\end{equation}
(This can be deduced  using \eqref{eq:int.sqrt.downstairs} and \eqref{eq:rdHT.elliptic.int} for $r \in [r_-,L-\f{\mathfrak b}2]$. For $r \geq L-\frac{\mathfrak{b}}2$, we split up the integral as in the proof of Proposition~\ref{prop:barQ}; we omit the details.) Combining \eqref{eq:Jacobian.junk.1} and \eqref{eq:Jacobian.junk.2}, we thus obtain
\begin{equation}\label{eq:Jacobian.junk.3}
|w| (\rd_r |_{(s,r,w,L)} H_T)  (\rd_{H_T} |_{(s,r,H_T,L)} \f{2\pi \widetilde{Q}_T}{\widetilde{\mathfrak T}}) - |w| \rd_r U(s,r,L) \rd_{H_T} |_{(s,r,H_T,L)} \f{2\pi \widetilde{Q}_T}{\widetilde{\mathfrak T}} = O(\ep).
\end{equation}

Finally, combining \eqref{eq:calJ.compute}, \eqref{eq:drQ.in.rHL} and \eqref{eq:Jacobian.junk.3}, we deduce that the Jacobian determinant $= \f{2\sqrt{2}\pi}{\widetilde{\mathfrak T}}+ O(\ep)$. \qedhere

\end{proof}

As a corollary, we can also control the higher derivatives of the inverse change of variable maps $(t,Q_T,H,M) \mapsto (s,r,w,L)$. 
\begin{corollary}\label{cor:der-inverse-map}
The following estimates hold:
\begin{equation}
\sup_{(H_T,L)\in \mathcal S} \sum_{i_1+i_2+i_3 \leq N+1} |\rd_{Q_T}^{i_1} \rd_H^{i_2} \rd_M^{i_3} r|,\, |\rd_{Q_T}^{i_1} \rd_H^{i_2} \rd_M^{i_3} w|,\, |\rd_{Q_T}^{i_1} \rd_H^{i_2} \rd_M^{i_3} L| \ls  1.
\end{equation}
Furthermore,
\begin{equation*}
\sup_{(H_T,L)\in \mathcal S} \sum_{i_1+i_2+i_3 \leq N+2} |\rd_{Q_T}^{i_1} \rd_H^{i_2} \rd_M^{i_3} r|,\, |\rd_{Q_T}^{i_1} \rd_H^{i_2} \rd_M^{i_3} w|,\, |\rd_{Q_T}^{i_1} \rd_H^{i_2} \rd_M^{i_3} L| \ls  \sum_{i \leq N+2} (1+ |\rd_r^{i+1} \varphi |(T,r)).
\end{equation*}
\end{corollary}
\begin{proof}
This is an immediate consequence of Lemma~\ref{lem:Jac}. \qedhere

%
\end{proof}

Now that we have introduced two different coordinates, we introduce the following convention for coordinate vector fields which we will use in the remainder of the paper.
\begin{definition}\label{def:coordinate.vector.fields}
We denote by $(\rd_s, \rd_r, \rd_w, \rd_L)$ the coordinate vector fields in the $(s,r,w,L)$ coordinate system, and by $(\rd_t, \rd_{Q_T}, \rd_H, \rd_M)$ the coordinate vector fields in the $(t,Q_T,H,M)$ coordinate system.
\end{definition}

We also need the following result for density estimates and estimates on $f$. Note that the function $w\rd_{H_T} |_{(s,r,H_T,L)} Q_T$ appears in the equation \eqref{eq:vlasov-action angle} below.
\begin{corollary}\label{cor:der-H-Q}
Let $Q_T$ be defined as in \eqref{eq:QT}. 
Then in $(s,r,w,L)$ coordinates
\begin{equation}
\sup_{(r,w,L): (H_T,L)\in \mathcal S}  \sum_{i_1+ i_2+i_3 \leq N} |\rd_r^{i_1} \rd_w^{i_2} \rd_L^{i_3} (w\rd_{H_T} |_{(s,r,H_T,L)} Q_T)| \ls  1.
\end{equation}
and in $(t,Q_T,H,M)$ coordinates
\begin{equation}
\sup_{(H_T,L)\in \mathcal S}  \sum_{i_1+ i_2+i_3 \leq N} |\rd_{Q_T}^{i_1} \rd_{H}^{i_2} \rd_{M}^{i_3} (w\rd_{H_T} |_{(s,r,H_T,L)} Q_T)| \ls 1.
\end{equation}
Furthermore,
\begin{equation}
\sup_{(r,w,L): (H_T,L)\in \mathcal S}  \sum_{i_1+ i_2+i_3 \leq N+1} |\rd_r^{i_1} \rd_w^{i_2} \rd_L^{i_3} (w\rd_{H_T} |_{(s,r,H_T,L)} Q_T)| \ls  1+ |\rd_r^{N+3} \varphi |(T,r).
\end{equation}

\begin{equation}
\sup_{(H_T,L)\in \mathcal S}  \sum_{i_1+ i_2+i_3 \leq N+1} |\rd_{Q_T}^{i_1} \rd_{H}^{i_2} \rd_{M}^{i_3} (w\rd_{H_T} |_{(s,r,H_T,L)} Q_T)| \ls  1+ |\rd_r^{N+3} \varphi |(T,r).
\end{equation}
\end{corollary}
\begin{proof}
This is similar to Proposition~\ref{prop:barQ} except that we apply Lemma~\ref{lem:int.sqrt.downstairs} and \eqref{eq:rdHT.elliptic.int}. We omit the details. \qedhere
\end{proof}

\subsection{Refined estimates on the nonlinear period function}

We next give more refined estimates (as compared to those obtained in Proposition~\ref{prop:der-period}) for the nonlinear period function defined in Definition~\ref{def:QT}. In addition to an upper bound, we show that the nonlinear period is close to the exact Kepler period (see Proposition~\ref{prop:time-period-compare}). 

First, we compute the Kepler period.
\begin{lemma}
Define $\widetilde{\mathfrak{T}}_{\Kep}(H_T,L)$ by 
\begin{equation}\label{eq:time-period-kepler}
\widetilde{\mathfrak{T}}_{\Kep}(H_T,L) := 2\int_{r_{-,\Kep}}^{r_{+,\Kep}} \f{\ud r}{\sqrt{H_T - U_{\Kep}(r,L)}},
\end{equation}
where
$$U_{\Kep}(r,L) := \f{L}{2r^2} - \f 1r,\quad r_{\pm,\Kep} = \f{1\pm \sqrt{2+2H_T L}}{(-2H_T)}.$$
Then $\widetilde{\mathfrak{T}}_{\Kep}(H_T,L)$ is independent of $L$, and 
$$\widetilde{\mathfrak{T}}_{\Kep}(H_T,L) = \widetilde{\mathfrak{T}}_{\Kep}(H_T) = \f{\pi}{(-H_T)^{3/2}}.$$ 
\end{lemma}
\begin{proof}
To keep the notation lean, we write $r_{\pm} = r_{\pm,\Kep}$. Note that $\f{r_+ + r_-}2 = -\f 1{2H_T}$.

Observe that $H_T - U_{\Kep}(r,L) = \f{-H_T}{r^2} (r_{+,\Kep}-r)(r-r_{-,\Kep})$. Then, first changing variables to $\rho = r - \f{r_{+} + r_{-}}{2} = r + \f 1{2H_T}$, and then to $s = \rho/(\tfrac{r_+-r_-}2)$, we obtain
\begin{equation}
\begin{split}
\widetilde{\mathfrak{T}}_{\Kep}(H_T,L) = &\:  \f{2}{\sqrt{-H_T}} \int_{r_{-}}^{r_{+}} \f{ r \ud r}{\sqrt{(r_+-r)(r-r_-)}} \\
= &\:  \f{2}{\sqrt{-H_T}} \int_{-\f{r_+-r_-}2}^{\f{r_+-r_-}2} \f{(\rho - \f{1}{2H_T}) \ud \rho}{\sqrt{(\tfrac{r_+-r_-}2)^2 - \rho^2}} \\
= &\: \f{1}{(-H_T)^{3/2}} \int_{-1}^{1} \f{\ud s}{\sqrt{1- s^2}} = \f{\pi}{(-H_T)^{3/2}},
\end{split}
\end{equation}
where we have used $\int_{-\f{r_+-r_-}2}^{\f{r_+-r_-}2} \f{\rho\, \ud \rho}{\sqrt{(\tfrac{r_+-r_-}2)^2 - \rho^2}} = 0$ by symmetry.
\end{proof}

In the next proposition, we show that $\widetilde{\mathfrak{T}}(H_T,L)$ is close to $\widetilde{\mathfrak{T}}_{\Kep}(H_T)$, thus giving a more accurate bound than Proposition~\ref{prop:der-period}. Notice, however, that we control one few derivative in Proposition~\ref{prop:time-period-compare} as compared to Proposition~\ref{prop:der-period}.

\begin{proposition}\label{prop:time-period-compare}
Let $\widetilde{\mathfrak{T}}(H_T,L)$ be as in \eqref{eq:period-actang-nonlinear} and $\widetilde{\mathfrak{T}}_{\Kep}(H_T)$ be as in \eqref{eq:time-period-kepler}.

Then, for $I \leq N$, 
\begin{equation*}
\begin{split}
\sup_{(r,w,L): (H_T,L)\in \calS} \sum_{i_1+i_2  \leq I}\Big|\rd_{H_T}^{i_1}  \rd_{L}^{i_2} (\widetilde{\mathfrak{T}}(H_T,L) - \widetilde{\mathfrak{T}}_{\Kep}(H_T))\Big|
\ls &\: \sup_{r\in [\f{\mathfrak l_1}2, \f 2{\mathfrak h}]} \sum_{i\leq I+2} |\rd_r^{i} \varphi(T,r)|.
\end{split}
\end{equation*}
\end{proposition}
\begin{proof}
\pfstep{Step~0: Preliminaries} For every $\zeta \in [0,1]$, define
\begin{equation}\label{eq:Tzeta.def}
\widetilde{\mathfrak{T}}_\zeta = \int_{r_-(\zeta)}^{r_+(\zeta)} \f{\ud r}{\sqrt{H_T - U_{\Kep}(r,L)-\zeta \varphi(T,r)}},
\end{equation}
where $r_-(\zeta) < r_+(\zeta)$ are the roots of $H_T - U_{\Kep}(r,L)-\zeta \varphi(T,r)$. (The existence of $r_\pm(\zeta)$ follows exactly as in Proposition~\ref{prop:support}.(3). Note also that $r_\pm(\zeta)$ depends on $H_T$, $L$ and $T$ as well, but we will only emphasize the $\zeta$ dependence.) Define also the notation $U_{\zeta} = U_{\Kep}(r,L)+\zeta \varphi(T,r)$.

We further split $\widetilde{\mathfrak{T}}_\zeta$ into three pieces: $\widetilde{\mathfrak{T}}_\zeta = \widetilde{\mathfrak{T}}_\zeta^{(1)} + \widetilde{\mathfrak{T}}_\zeta^{(2)} + \widetilde{\mathfrak{T}}_\zeta^{(3)}$, where
$$\widetilde{\mathfrak{T}}_\zeta^{(1)} = \int_{r_-(\zeta)}^{L-\mathfrak b/2} \cdots, \quad \widetilde{\mathfrak{T}}_\zeta^{(2)} = \int_{L-\mathfrak b/2}^{L+\mathfrak b/2} \cdots,\quad \widetilde{\mathfrak{T}}_\zeta^{(3)} = \int_{L+\mathfrak b/2}^{r_+(\zeta)} \cdots.$$
Using 
\begin{equation*}
\begin{split}
\widetilde{\mathfrak{T}} - \widetilde{\mathfrak{T}}_{\Kep} = \widetilde{\mathfrak{T}}_1  - \widetilde{\mathfrak{T}}_0 = \sum_{i=1}^3 (\widetilde{\mathfrak{T}}_1^{(i)}  - \widetilde{\mathfrak{T}}_0^{(i)}) =  \sum_{i=1}^3 \int_{0}^{1} \f{\ud}{\ud\zeta} \widetilde{\mathfrak{T}}_\zeta^{(i)} \, \ud \zeta,
\end{split}
\end{equation*}
it suffices to obtain the estimate 
\begin{equation}\label{eq:T.dzeta.goal}
\sup_{(r,w,L): (H_T,L)\in \calS} \sum_{i_1+i_2 \leq I}\Big|\rd_{H_T}^{i_1} \rd_{L}^{i_2} \Big( \f{\ud}{\ud\zeta} \widetilde{\mathfrak{T}}^{(i)}_\zeta\Big)\Big| \ls \sup_{r\in [\f{\mathfrak l_1}2, \f 2{\mathfrak h}]} |\rd_r^{I+1} \varphi(T,r)|.
\end{equation}
We will prove \eqref{eq:T.dzeta.goal} for $i = 1,2,3$ in Steps~1, 2, 3 respectively.

\pfstep{Step~1: Bounding $\widetilde{\mathfrak{T}}_1^{(1)} - \widetilde{\mathfrak{T}}_0^{(1)}$} Our goal is to compute $\f{\ud}{\ud\zeta} \widetilde{\mathfrak{T}}^{(1)}_\zeta$. We first rewrite
\begin{equation}\label{eq:T.zeta}
\begin{split}
 \widetilde{\mathfrak{T}}^{(1)}_\zeta 
= &\: \f{1}{\rd_r U_\zeta(r_-)}\int_{r_-}^{L-\mathfrak b/2} \f{\rd_r U_\zeta(\rho)}{\sqrt{H_T - U_\zeta(\rho)}}\, \ud \rho+ \int_{r_-}^{L-\mathfrak b/2} \Big(1 - \f{\rd_r U_\zeta(\rho)}{\rd_r U_\zeta(r_-)} \Big) \f {\ud \rho}{\sqrt{H_T - U_\zeta(r)}} \\
= &\: -\f{2 \sqrt{H_T - U_\zeta(L-\mathfrak b/2)}}{\rd_r U_\zeta(r_-)} + \int_{r_-}^{L-\mathfrak b/2} \Big(1 - \f{\rd_r U_\zeta(\rho)}{\rd_r U_\zeta(r_-)} \Big) \f {\ud \rho}{\sqrt{H_T - U_\zeta(\rho)}}.
\end{split}
\end{equation}
We now differentiate in $\zeta$. First note that 
$$\f{\ud}{\ud\zeta} (H_T - U_\zeta(r))^{-1/2} = \f 12 \f{\varphi(r)}{(H_T - U_\zeta(r))^{3/2}}= \f{\varphi(r)}{(\rd_r U_\zeta)(r)} \rd_r|_{(r,H_T,L)} (H_T - U_\zeta(r))^{-1/2}.$$
(Here, since $r \in r_-(\zeta) L - \f{\mathfrak b}2]$, the same argument as in Proposition~\ref{prop:support}.(5) guarantees that $\rd_{r}U_\zeta(r)$ is non-vanishing and can be divided.)
This implies, after integration by parts, that
\begin{equation}\label{eq:d.T1.intermediate}
\begin{split}
&\: \int_{r_-}^{L-\mathfrak b/2} \Big(1 - \f{\rd_r U_\zeta(\rho)}{\rd_r U_\zeta(r_-)} \Big)\Big( \f{\ud}{\ud\zeta} (H_T - U_\zeta(\rho))^{-1/2}\Big) \ud \rho \\
= &\: \f{\tfrac{\varphi(L-\mathfrak b/2)}{(\rd_r U_\zeta)(L-\mathfrak b/2)} - \tfrac{\varphi(L-\mathfrak b/2)}{\rd_r U_\zeta(r_-)}}{  \sqrt{H_T - U_\zeta(L-\mathfrak b/2)}} - \int_{r_-}^{L-\mathfrak b/2} \rd_\rho \Big( \tfrac{\varphi(\rho)}{(\rd_r U_\zeta)(\rho)} - \tfrac{\varphi(\rho)}{\rd_r U_\zeta(r_-)} \Big)\f{1}{\sqrt{H_T - U_\zeta(\rho)}} \ud \rho.
\end{split}
\end{equation}
We now differentiate the expression in \eqref{eq:T.zeta} and use the expression \eqref{eq:d.T1.intermediate} to obtain
\begin{equation}
\begin{split}
\f{\ud}{\ud\zeta} \widetilde{\mathfrak{T}}^{(1)}_\zeta 
= &\: \f{\varphi(L-\mathfrak b/2)}{\sqrt{H_T - U_\zeta(L-\mathfrak b/2)} \rd_r U_\zeta(r_-)} + \f{(\rd_r \varphi)(r_-)\sqrt{H_T - U_\zeta(L-\mathfrak b/2)}}{(\rd_r U_\zeta(r_-))^2} \\
&\: + \int_{r_-}^{L-\mathfrak b/2} \Big(- \f{ \rd_r \varphi(\rho)}{\rd_r U_\zeta(r_-)} + \f{ \rd_r \varphi(r_-) \rd_r U_\zeta(\rho)}{(\rd_r U_\zeta(r_-))^2} \Big) \f {\ud \rho}{\sqrt{H_T - U_\zeta(\rho)}} \\
&\: + \f{\tfrac{\varphi(L-\mathfrak b/2)}{(\rd_r U_\zeta)(L-\mathfrak b/2)} - \tfrac{\varphi(L-\mathfrak b/2)}{\rd_r U_\zeta(r_-)}}{  \sqrt{H_T - U_\zeta(L-\mathfrak b/2)}} -\int_{r_-}^{L-\mathfrak b/2} \rd_\rho \Big( \tfrac{\varphi(\rho)}{(\rd_r U_\zeta)(\rho)} - \tfrac{\varphi(\rho)}{\rd_r U_\zeta(r_-)} \Big)\f{1}{\sqrt{H_T - U_\zeta(\rho)}} \ud \rho.
\end{split}
\end{equation}
The terms without integration obviously satisfy the bounds as in \eqref{eq:T.dzeta.goal}, since $\sqrt{H_T - U_\zeta(L-\mathfrak b/2)}$ is bounded away from $0$. For the terms with integration, we repeatedly apply Lemma~\ref{lem:int.sqrt.downstairs} with $r = L-\mathfrak b/2$ before differentiating by $\rd_{H_T}$ and $\rd_L$. This then gives \eqref{eq:T.dzeta.goal} for $i=1$.

\pfstep{Step~2: Bounding $\widetilde{\mathfrak{T}}_1^{(2)} - \widetilde{\mathfrak{T}}_0^{(2)}$} This term is the simplest as we have 
\begin{equation}
\f{\ud}{\ud \zeta}  \widetilde{\mathfrak{T}}_\zeta^{(2)} = \int_{L-\mathfrak b/2}^{L+\mathfrak b/2} \f{\varphi(\rho)}{(H_T - U_\zeta(\rho))^{3/2}} \, \ud \rho.
\end{equation}

\pfstep{Step~3: Bounding $\widetilde{\mathfrak{T}}_1^{(3)} - \widetilde{\mathfrak{T}}_0^{(3)}$} This is analogous to Step~1; we omit the details. \qedhere

\end{proof}

\subsection{The Vlasov equation in the dynamical action angle variables}
In the following proposition we represent the equation \eqref{eq:transport_spherical} in $(t,Q_T,H,M)$ coordinate system as in Definition~\ref{def:coordinates}.
\begin{proposition}\label{prop:vlasov-action angle}
\begin{equation}\label{eq:vlasov-action angle}
\mathfrak{D} f=\partial_t f+\frac{2\sqrt{2}\pi}{\mathfrak{T}}\partial_{Q_T} f+(w\rd_{H_T} |_{(s,r,H_T,L)} Q_T)\left(\partial_r \varphi(T,r)-\partial_r\varphi(t,r)\right)\partial_{Q_T}f + \rd_s \varphi \partial_{H}f=0.
\end{equation}
\end{proposition}
\begin{proof}
Using \eqref{eq:transport_spherical}, we compute in $(s,r,w,L)$ coordinates that 
\begin{equation}\label{eq:vlasov.action.angle.equation.deduction.1}
\mathfrak{D} t = 1,\quad \mathfrak{D} H=\rd_s \varphi,\quad \mathfrak{D} M = 0.
\end{equation}
Using \eqref{eq:transport_spherical}, we also obtain
\begin{equation*}
\begin{split}
\mathfrak{D} Q_T = w \rd_{r}|_{(s,r,w,L)} Q_T - \rd_r U \rd_{w} |_{(s,r,w,L)} Q_T. 
\end{split}
\end{equation*}
Notice that the right-hand side is exactly the determinant in \eqref{eq:calJ.compute}. Thus, following the computation in Lemma~\ref{lem:Jac}, we obtain 
\begin{equation}\label{eq:vlasov.action.angle.equation.deduction.2}
\mathfrak{D} Q_T = \frac{2\sqrt{2}\pi}{\mathfrak{T}} + (w\rd_{H_T} |_{(s,r,H_T,L)} Q_T)\left(\partial_r \varphi(T,r)-\partial_r\varphi(t,r)\right).
\end{equation}
Combining \eqref{eq:vlasov.action.angle.equation.deduction.1} and \eqref{eq:vlasov.action.angle.equation.deduction.2} yields the desired conclusion. \qedhere
\end{proof}

\section{Proof of Theorem~\ref{thm:main}: estimates on $f$ and its derivatives}\label{sec:f}
We continue to work under the assumptions of Theorem~\ref{thm:boot}.

In this section, we prove pointwise estimates on $f$ and its derivatives in the (dynamically-defined) $(t,Q_T,H,M)$ coordinates (for any fixed $T \in (0,T_{\text{B}})$) introduced in the last section.

The following is the main result of this section:
\begin{theorem}\label{the:boot-f}
Suppose that the assumptions of Theorem~\ref{thm:boot} hold. 
Then, for $i_1+i_2+i_3+i_4 = I \leq N$, the estimate 
\begin{equation*}
\begin{split}
\|\derv{i_1}{i_2}{i_3}{i_4} f\|_{L^\i}(t)\leq C\delta\epsilon\jap{t}^{i_3}[1+\epsilon z(\jap{t},I-N+i_1)],
\end{split}
\end{equation*}
holds for some $C>0$ depending only on $\mathfrak c_0$, $\mathfrak h_0$, $\mathfrak l_1$, $\mathfrak l_2$ and $N$, and independent of $\de$ and $\ep$, where
\begin{equation}\label{eq:z.def}
z(\jap{t},I-N+i_1)=
\begin{cases}
\log \jap{t} & \hbox{when } I-N+i_1=0\\
\jap{t}^{\max\{0,I-N+i_1\}} & \hbox{when } I-N+i_1\neq 0,
\end{cases}
\end{equation}
\end{theorem}

In the statement of the main theorem, and below, $\log$ denotes the natural logarithm. Moreover, $\rd_{Q_T}$, $\rd_{H}$ and $\rd_{M}$ are coordinate vector fields with respect to the $(t,Q_T,H,M)$ coordinates in Definition~\ref{def:coordinates}. The vector field $Y_H$ will be defined in Definition~\ref{def:omega} below.

\begin{remark}\label{rmk:large.i_1}
We note that when $i_1$ becomes large, the estimates in \eqref{eq:boot-f} are not much better than if we just expand $Y_H$ out using \eqref{eq:Y}. In particular, at the top order, the estimate \eqref{eq:boot-f} is no better than $\|Y_H^N f\|_{L^\i}(t) \ls \de^{3/4}\ep^2 \jap{t}^N$. This is in stark contrast to the Vlasov--Poisson system on the torus (recall \eqref{eq:intro.VP.flat}), where one could prove boundedness of the solution after repeated $Y_H$ differentiation. This is related to the fact, already mentioned in Section~\ref{sec:compare.with.flat}, that the angle variable $Q_T$ do not align with the independent variable $r$. As a result, when applying the estimates with $Y_H$ differentiation, we will only be using at most one $Y_H$ derivative.
\end{remark}

In order to establish Theorem~\ref{the:boot-f}, we assume the following bootstrap assumptions on $\derv{i_1}{i_2}{i_3}{i_4}f$ for any $i_1+i_2+i_3+i_4=I\leq N$ and any $t\in[0,T]$ for any $T\in (0,T_{\text{B}})$:
\begin{equation}\label{eq:boot-f}
\begin{split}
\|\derv{i_1}{i_2}{i_3}{i_4} f\|_{L^\i}(t)\leq \delta^{3/4}\epsilon\jap{t}^{i_3}[1+\epsilon z(\jap{t},I-N+i_1)],
\end{split}
\end{equation}
where $z(\jap{t},I-N+i_1)$ is again given by \eqref{eq:z.def}. Until Section~\ref{sec:pointwise.proof}, we will work under assumption \eqref{eq:boot-f}, in addition to the assumptions of Theorem~\ref{thm:boot}.

The remainder of the section will be organized as follows. In \textbf{Section~\ref{sec:pointwise.vf}}, we define the set of vector fields commutators and derive their commutation properties. In \textbf{Section~\ref{sec:pointwise.setup}}, we derive the commuted equations and the main error terms. The next three subsections are devoted to controlling the error terms: in \textbf{Section~\ref{sec:pointwise.data}}, we bound the initial data term, in \textbf{Section~\ref{sec:pointwise.main.error}}, we estimate the main nonlinear error terms, in \textbf{Section~\ref{sec:pointwise.proof}}, we control the remaining error terms arising from commutation and conclude the proof of Theorem~\ref{the:boot-f}.

\subsection{Vector fields and commutation lemmas}\label{sec:pointwise.vf} We will use the following set of commuting vector fields:
\begin{equation}\label{eq:Gamma-set}
\{\partial_{Q_T},\partial_H,Y_H,{\partial_M}\}.
\end{equation}
This set of vector fields is similar to that in Subsection~\ref{sub:action angle-lin}, but importantly they are dynamically defined and take into account $\varphi(T,r)$.

In \eqref{eq:Gamma-set}, $\partial_{Q_T}$, $\partial_H$ and $\partial_M$ are coordinate vector fields with respect to the $(t,Q_T,H,M)$ coordinates in Definition~\ref{def:coordinates}.

The vector field $Y_H$ that we will used is defined as follows:
\begin{definition}\label{def:omega}
\begin{enumerate}
\item We define
\begin{equation}\label{eq:Omg.def}
\Omega(X,Z):=\f{2\sqrt{2}\pi}{\widetilde{\mathfrak{T}}(X,Z)},
\end{equation}
where $\widetilde{\mathfrak{T}}(X,Z)$ is the dynamical period as defined in \eqref{eq:period-actang-nonlinear}.
\item Define $Y_H$ by
\begin{equation}\label{eq:Y}
Y_H:=t\partial_{H}\{\Omega(H,M)\}\partial_{Q_T}+\partial_H,
\end{equation}
where $\Omega(H,M)$ is as in \eqref{eq:Omg.def} and $\partial_{Q_T}$, $\partial_H$ are again the coordinate vector fields in the $(t,Q_T,H,M)$ coordinates.
\end{enumerate}
\end{definition}
Some remarks are in order. 
\begin{remark} 
From now on, $X$ and $Z$ will be reserved the as dummy variables for the first and second slots of $\Omg$, respectively. This will remove possible confusion as we will consider both $\Omega(H_T,L)$ (as in the linear operator in \eqref{eq:lin-transport-act-ang}) and $\Omega(H,L)$ (as in the vector field $Y_H$). 
\end{remark}

\begin{remark}
We note explicitly that when we write $\Omega(H,L)$, $\Omega$ is still defined by \eqref{eq:Omg.def} and \eqref{eq:period-actang-nonlinear}, i.e., involving the function $\varphi(T,r)$ at time $T$, even though $H$ is the Hamiltonian at time $t$.
\end{remark}

\begin{remark}
$Y_H$ is defined with $\Omega(H,M)$ instead of $\Omega(H_T,M)$ as in \eqref{eq:lin-transport-act-ang} because this will be what is relevant for the density estimates in Section~\ref{sec:density}. This is in turn because to derive the density estimates, we use the flow of $\rd_t + \Omg(H,M)\rd_{Q_T}$ instead of $\rd_t + \Omg(H_T,M)\rd_{Q_T}$; see \eqref{eq:den-lin}.
\end{remark}

The function $\Omega$ defined in \eqref{eq:Omg.def} is also important because it is part of the linear operator, for which we denote as follows:
\begin{definition}
We denote the linear part of \eqref{eq:vlasov-action angle} by $\mathfrak{D}^{(\mathrm{lin})}$, i.e., we denote
\begin{equation}\label{eq:lin-transport-act-ang}
\mathfrak{D}^{(\mathrm{lin})} := \partial_t+\f{2\sqrt{2}\pi}{\widetilde{\mathfrak{T}}(H_T,M)}\partial_{Q_T} = \rd_t + \Omega(H_T,M) \rd_{Q_T}.
\end{equation}
\end{definition}

In Proposition~\ref{prop:precise-der-period}, Lemma~\ref{lem:trivial.computations} and Lemma~\ref{lem:der-Omega}, we will give some computations and estimates for $\Omega$ defined in \eqref{eq:Omg.def}. They will be useful for computing and estimating the commutators of the commuting vector fields \eqref{eq:Gamma-set} with the linear operator $\mathfrak{D}^{(\mathrm{lin})}$; see Lemma~\ref{lem:vec-field-comm} and Proposition~\ref{prop:comm-Vlasov} below.
\begin{proposition}\label{prop:precise-der-period}
There exist $\eta_1,\, \eta_2,\,\kappa>0$ such that the following estimates hold:
\begin{equation}\label{eq:time-period-est}
0<\f{\eta_1}2\leq|\Omega(H_T,M)|\leq2\eta_2
\end{equation}
and
\begin{equation}\label{eq:H-der-time-period}
\left|(\partial_X \Omega)(H_T,M)\right|>\kappa \quad \hbox{whenever $(H,M)\in \mathcal S$}.
\end{equation}
\end{proposition}
\begin{proof}
Notice that $H_T = H + \varphi(T,r) - \varphi(t,r)$ and thus for $\ep>0$ sufficiently small, $(-H_T)$ is both bounded and bounded away from $0$ when $(H,M)\in \mathcal S$ (see \eqref{eq:calS.def}). Using this, the bounds \eqref{eq:time-period-est} and \eqref{eq:H-der-time-period} follow from the definition \eqref{eq:Omg.def}, the comparison of $\widetilde{\mathfrak{T}}$ and $\widetilde{\mathfrak{T}}_{\Kep} = \f{\pi}{(-H_T)^{3/2}}$ in Proposition~\ref{prop:time-period-compare}, and the smallness of the derivatives of $\varphi$ in \eqref{eq:BA.weak.1}. \qedhere

%
%
\end{proof}

\begin{lemma}\label{lem:trivial.computations}
\begin{align}
\rd_{Q_T} \{ \Omega(H_T, M) \} = &\: (\rd_{Q_T} r) (\rd_r\varphi(T,r) - \rd_r\varphi(t,r)) (\rd_X \Omg)(H_T,M), \\
\rd_{H} \{ \Omega(H_T, M) \} = &\: [\rd_{H} r (\rd_r\varphi(T,r) - \rd_r\varphi(t,r)) + 1](\rd_X \Omg)(H_T,M), \\
\rd_{M} \{ \Omega(H_T, M) \} = &\: \rd_{M} r (\rd_r\varphi(T,r) - \rd_r\varphi(t,r)) (\rd_X \Omg)(H_T,M) + (\rd_Z \Omg)(H_T,M).
\end{align}
\end{lemma}
\begin{proof}
Since $H_T - H = \varphi(T,r) - \varphi(t,r)$ and $\rd_{Q_T} H = 0$  we obtain that
$$\rd_{Q_T} H_T = \rd_{Q_T}(H_T - H) = (\rd_{Q_T} r)(\rd_r\varphi(T,r) - \rd_r\varphi(t,r)).$$
A similar argument using $\rd_H H  =1$ and $\rd_{M} H = 0$ gives
$$\rd_{H} H_T = (\rd_{H} r)(\rd_r\varphi(T,r) - \rd_r\varphi(t,r)) + 1, \quad \rd_{M} H_T = (\rd_{M} r)(\rd_r\varphi(T,r) - \rd_r\varphi(t,r)).$$
The lemma then follows from the above computations and the chain rule. \qedhere
\end{proof}

In what follows we will need to bound the derivatives of $\rd_X \Omg$ and $\rd_Z \Omg$. For this purpose, we start with a general lemma:
\begin{lemma}\label{lem:der-omega}
Let $\om(X,Z)$ be a $C^1$ function. Then 
\begin{align}
\rd_{Q_T} \{\om(H_T,M) \} = &\: \mathfrak s_{(\rd_{Q_T})} (\rd_X\om)(H_T,M), \label{eq:der.general.HTM.1}\\
\rd_{H} \{\om(H_T,M) \} = &\: (\rd_X\om)(H_T,M) + \mathfrak s_{(\rd_H)} (\rd_X\om)(H_T,M), \label{eq:der.general.HTM.2}\\
\rd_{M} \{\om(H_T,M) \} = &\: (\rd_Z\om)(H_T,M) +\mathfrak s_{(\rd_{Q_T})} (\rd_X\om)(H_T,M), \label{eq:der.general.HTM.3}\\
Y_H \{\om(H_T,M)\} = &\: (\rd_X\om)(H_T,M) + \mathfrak s_{(Y_H)} (\rd_X\om)(H_T,M), \label{eq:der.general.HTM.4}
\end{align}
where $\mathfrak s_{(\rd_{Q_T})}$, $\mathfrak s_{(\rd_H)}$, $\mathfrak s_{(\rd_{M})}$, $\mathfrak s_{(Y_H)}$ are functions of $(t,T,Q_T,H,M)$ such that for $i_1+i_2+i_3+i_4 \leq N$, the following estimates hold when $(H,M)\in \mathcal S$:
\begin{align}
|Y_H^{i_1}\partial_{Q_T}^{i_2}\partial_{H}^{i_3}\partial_M^{i_4}\mathfrak s_{(\cdot)}| \ls &\:  \sum_{I\leq i_1+i_2+i_3+i_4}  \brk{t}^{i_1} \Big|\rd_r^{I+1} (\varphi(t,r) - \varphi(T,r))\Big| , \quad \cdot = \rd_{Q_T}, \rd_H,\rd_M,\label{eq:bound-s-not-Y}\\
 |Y_H^{i_1}\partial_{Q_T}^{i_2}\partial_{H}^{i_3}\partial_M^{i_4}\mathfrak s_{(Y_H)}| \ls &\:  \sum_{I\leq i_1+i_2+i_3+i_4}  \brk{t}^{i_1+1} \Big|\rd_r^{I+1} (\varphi(t,r) - \varphi(T,r)) \Big|\label{eq:bound-s-Y}.
\end{align}
\end{lemma}
\begin{proof}
For \eqref{eq:der.general.HTM.1}--\eqref{eq:der.general.HTM.3}, the exact formulas are as in Lemma~\ref{lem:trivial.computations} with $\Omg$ replaced by $\om$. Thus the desired expressions and formulas follow from using Lemma~\ref{lem:Jac} and Corollary~\ref{cor:der-inverse-map} to control quantities coming from the change of variables map. (Here, we convert all derivatives on $\varphi(t,r) - \varphi(T,r)$ to $\rd_r$ derivatives using Lemma~\ref{lem:Jac} since it is a function of $(t,r)$ alone.)

For \eqref{eq:der.general.HTM.4}, recall the formula \eqref{eq:Y} to see that it can be expressed in terms of \eqref{eq:der.general.HTM.1} and \eqref{eq:der.general.HTM.2}; the extra $\brk{t}$ weight comes from the $t$ weight in the vector field itself. \qedhere
\end{proof}
We then turn to the bounds for the derivatives of $\rd_X\Omg$ and $\rd_Z \Omg$:
\begin{lemma}\label{lem:der-Omega}
For $\Omega$ defined as in \eqref{eq:Omg.def}, we have the following high order estimates for $i_1+i_2+i_3+i_4\leq N-1$ and when $(H,M)\in \mathcal S$:
\begin{equation}\label{eq:precise-der-period.1}
\begin{split}
& \left|Y_H^{i_1}\partial_{Q_T}^{i_2}\partial_{H}^{i_3}\partial_M^{i_4}\{(\partial_Z\Omega)(H_T,M)\}\right|\\
&\hspace{5em}\lesssim \ep + \sum_{I\leq i_1+i_2+i_3+i_4-1} \jap{t}^{i_1}\Big|\rd_r^{I+1} (\varphi(t,r) - \varphi(T,r)) \Big|, 
\end{split}
\end{equation}
Moreover, for $i_1+i_2+i_3+i_4\leq N$ and when $(H,M)\in \mathcal S$, the following estimates hold:
\begin{equation}\label{eq:precise-der-period.2}
\begin{split}
& \left|Y_H^{i_1}\partial_{Q_T}^{i_2}\partial_{H}^{i_3}\partial_M^{i_4}\{(\partial_Z\Omega)(H_T,M)\}\right|\\
&\hspace{5em}\lesssim 1 + \sum_{I\leq i_1+i_2+i_3+i_4-1} \jap{t}^{i_1}\Big|\rd_r^{I+1} (\varphi(t,r) - \varphi(T,r)) \Big|,  
\end{split}
\end{equation}
\begin{equation}
\begin{split}
& \left|Y_H^{i_1}\partial_{Q_T}^{i_2}\partial_{H}^{i_3}\partial_M^{i_4}\{(\partial_X\Omega)(H_T,M)\}\right|\\
&\hspace{5em}\lesssim 1 + \sum_{I\leq i_1+i_2+i_3+i_4-1} \jap{t}^{i_1}\Big|\rd_r^{I+1} (\varphi(t,r) - \varphi(T,r)) \Big|. \label{eq:precise-der-period.3}
\end{split}
\end{equation}
\end{lemma}
\begin{proof}
We first consider the case \eqref{eq:precise-der-period.1}. We apply Lemma~\ref{lem:der-omega} repeatedly to get terms of the following three types
$$Y_H^{i_1}\partial_{Q_T}^{i_2}\partial_{H}^{i_3}\partial_M^{i_4}\{(\partial_Z\Omega)(H_T,M)\}=\mathcal{E}_1+\mathcal{E}_2+\mathcal{E}_3.$$
The first term, $\mathcal{E}_1$ is such that all the derivatives hit $\partial_Z\Omega(H_T,M)$ and we have the bound
$$|\mathcal{E}_1|\lesssim |(\partial_X^{i_1+i_3}\partial_Z^{i_4+1}\Omega)(H_T,M)|.$$
Further, we recall that $\widetilde{\mathfrak{T}}_{\Kep}(X,Z)$ is independent of $Z$. Thus, the bound
$$|\mathcal{E}_1|\lesssim |(\partial_X^{i_1+i_3}\partial_Z^{i_4+1}\Omega)(H_T,M)|\lesssim \epsilon,$$ follows from \eqref{eq:Omg.def}, Proposition~\ref{prop:time-period-compare}, Lemma~\ref{lem:Jac}, and the bound for $\varphi$ in \eqref{eq:BA.weak.1}. (We note that we need $i_1+i_3+i_4\leq N-1$, so that the total number of derivatives on $\Omg$ is $\leq N$ and Proposition~\ref{prop:time-period-compare} can be applied.)

The second term, $\mathcal{E}_2$ satisfies the bound
\begin{align*}
|\mathcal{E}_2|&\lesssim \sum_{j_1+j_2\leq N+1}\sum_{\substack{{i_1'\leq i_1}\\{i_2'\leq i_2}\\{i_3'\leq i_3}\\{i_4'\leq i_4}\\{i_1'+i_2'+i_3'+i_4'\leq i_1+i_2+i_3+i_4-1}}}| (\rd_X^{j_1}\rd_Z^{j_2}\Omega)(H_T,M) \derv{i_1'}{i_2'}{i_3'}{i_4'}\mathfrak{s}_{(\cdot)}|, \hspace{2em} \cdot=\rd_{Q_T}, \rd_H,\rd_M.
\end{align*}
(Notice that $\mathcal{E}_2$ includes terms which are e.g., quadratic in $\mathfrak{s}_{(\cdot)}$, for which the derivatives of one copy of $\mathfrak{s}_{(\cdot)}$ can be controlled using Lemma~\ref{lem:der-omega} and the bootstrap assumption \eqref{eq:varphi_diff}.)

To bound this term, we use \eqref{eq:bound-s-not-Y} the fact that the maximum number of terms hitting $\mathfrak{s}_{(\cdot)}$ is $i_1+i_2+i_3+i_4-1$, and Proposition~\ref{prop:der-period} and Lemma~\ref{lem:Jac}. We get that 
$$|\mathcal{E}_2|\lesssim \sum_{I\leq i_1+i_2+i_3+i_4-1} \jap{t}^{i_1}\Big|\rd_r^{I+1} (\varphi(t,r) - \varphi(T,r)) \Big|.$$

Finally, $\mathcal{E}_3$ satisfies the bound,
\begin{align*}
|\mathcal{E}_3|&\lesssim \sum_{j_1+j_2\leq N+1}\sum_{\substack{{i_1'\leq i_1-1}\\{i_2'\leq i_2}\\{i_3'\leq i_3}\\{i_4'\leq i_4}\\{i_1'+i_2'+i_3'+i_4'\leq i_1+i_2+i_3+i_4-1}}}| (\rd_X^{j_1}\rd_Z^{j_2}\Omega)(H_T,M) \derv{i_1'}{i_2'}{i_3'}{i_4'}\mathfrak{s}_{(Y_H)}|.
\end{align*}
We proceed as in the case of $\mathcal{E}_2$ but use the bound \eqref{eq:bound-s-Y} instead and note that $i_1'\leq i_1-1$ to get 
$$|\mathcal{E}_3|\lesssim \sum_{I\leq i_1+i_2+i_3+i_4-1} \jap{t}^{i_1}\Big|\rd_r^{I+1} (\varphi(t,r) - \varphi(T,r)) \Big|.$$

This proves \eqref{eq:precise-der-period.1}. We can prove \eqref{eq:precise-der-period.2} and \eqref{eq:precise-der-period.3} by proceeding in the same way, except for using Proposition~\ref{prop:der-period} and Lemma~\ref{lem:Jac} for the term $\calE_1$. 
\end{proof}

To conclude the subsection, we compute the commutators between $\mathfrak{D}^{(\mathrm{lin})}$ and the commuting vector fields.
\begin{lemma}\label{lem:vec-field-comm} For $\mathfrak{D}^{(\mathrm{lin})}$ defined as in \eqref{eq:lin-transport-act-ang}, the following commutation properties hold for the vector fields in \eqref{eq:Gamma-set}:
\begin{align}
\label{eq:com-Q-lin}[\mathfrak{D}^{(\mathrm{lin})},\partial_{Q_T}]&=(\partial_{Q_T} r)\{(\partial_X\Omega)(H_T,L)\}(\rd_r\varphi(t,r)-\rd_r\varphi(T,r))\partial_{Q_T},\\
\label{eq:com-H-lin}
[\mathfrak{D}^{(\mathrm{lin})},\partial_{H}]&=\left[\partial_H r(\rd_r\varphi(t,r)-\rd_r\varphi(T,r))-1\right]\{(\partial_X\Omega)(H_T,M)\}\partial_{Q_T}, \\
[\mathfrak{D}^{(\mathrm{lin})},Y_H]&=\partial_H\left[\int_0^1 (\partial_{X}\Omega)(H+\sigma(H_T-H),M)\d\sigma\right](\varphi(t,r)-\varphi(T,r))\partial_{Q_T} \notag\\
&\quad+(\partial_H r)\left[\int_0^1  (\partial_{X}\Omega)(H+\sigma(H_T-H),M)\d\sigma\right](\rd_r\varphi(t,r)-\rd_r\varphi(T,r))\partial_{Q_T} \notag\\
&\quad-t\left[\partial_H r(\rd_r\varphi(t,r)-\rd_r\varphi(T,r))-1\right](\partial_{Q_T} r)\{(\partial_X\Omega)(H_T,M)\}^2\notag\\
&\qquad\hspace{16em}\times(\rd_r\varphi(t,r)-\rd_r\varphi(T,r))\partial_{Q_T}, \label{eq:com-Y-lin}\\
\label{eq:com-M-lin}
[\mathfrak{D}^{(\mathrm{lin})},{\partial_M}]&=\left[\partial_M r(\rd_r\varphi(t,r)-\rd_r\varphi(T,r))\cdot\{(\partial_X\Omega)(H_T,M)\}-(\partial_Z\Omega)(H_T,M)\right]\partial_{Q_T}.
\end{align}
\end{lemma}

\begin{proof}

The identities \eqref{eq:com-Q-lin}, \eqref{eq:com-H-lin} and \eqref{eq:com-M-lin} are immediate consequences of Lemma~\ref{lem:trivial.computations}.

%

Next we turn our attention to \eqref{eq:com-Y-lin}. 
\begin{equation*}
\begin{split}
[\mathfrak{D}^{(\mathrm{lin})},Y_H]=&\left(\partial_t+\Omg(H_T,M)\partial_{Q_T}\right)\left(t\partial_{H}\{\Omega(H,M)\}\partial_{Q_T}+\partial_H\right)\\
&-\left(t\partial_{H}\{\Omega(H,M)\}\partial_{Q_T}+\partial_H\right)\left(\partial_t+\Omg(H_T,M)\partial_{Q_T}\right)\\
=&\partial_H(\Omega(H,M)-\Omega(H_T,M))\partial_{Q_T}-t\partial_{Q_T}\left(\Omega(H_T,M)\right)\partial_{H}\left(\Omega(H_T,M)\right)\partial_{Q_T}.
\end{split}
\end{equation*}
For the second term, we substitute in Lemma~\ref{lem:trivial.computations}. For the first term, we use fundamental theorem of calculus to get
$$\Omega(H,M)-\Omega(H_T,M)=\int_0^1(\rd_X \Omg)(H+\sigma(H_T-H),M)\d \sigma(H-H_T).$$
Combining this with $H-H_T = \varphi(t,r) - \varphi(T,r)$, we obtain the desired result. \qedhere
\end{proof}

\subsection{Setting up pointwise estimates}\label{sec:pointwise.setup} We now commute \eqref{eq:vlasov-action angle} with $\derv{i_1}{i_2}{i_3}{i_4}$ for $i_1+i_2+i_3+i_4=I\leq N+1$ and use Lemma~\ref{lem:vec-field-comm} to derive the equation satisfied by $\derv{i_1}{i_2}{i_3}{i_4} f$.
\begin{proposition}\label{prop:comm-Vlasov}
Let $f$ solve \eqref{eq:vlasov-action angle}. Then $\derv{i_1}{i_2}{i_3}{i_4}f$ solves the following equation:
\begin{equation}\label{eq:comm-Vlasov}
\begin{split}
&\partial_t (\derv{i_1}{i_2}{i_3}{i_4}f)+\Omg(H_T,M)\partial_{Q_T}(\derv{i_1}{i_2}{i_3}{i_4}f)=\\
&\quad\qquad-\derv{i_1}{i_2}{i_3}{i_4}\left[(w\rd_{H_T} |_{(s,r,H_T,L)} Q_T)\left(\rd_r\varphi(T,r)-\rd_r\varphi(t,r)\right)\partial_{Q_T}f\right]\\
&\quad\qquad-\derv{i_1}{i_2}{i_3}{i_4}\left[\partial_s \varphi \partial_H f\right]+\sum_{i_4'+i_4''=i_4-1}Y_H^{i_1}\partial_{Q_T}^{i_2}\partial_H^{i_3}\partial_{M}^{i_4'}\left[\mathfrak{A}_1(t,H,Q_T,M)\partial_{Q_T}\partial_M^{i_4''}f\right]\\
&\quad\qquad+\sum_{i_3'+i_3''=i_3-1}Y_H^{i_1}\partial_{Q_T}^{i_2}\partial_H^{i_3'}\left[\mathfrak{A}_2(t,H,Q_T,M)\partial_{Q_T}\partial_{H}^{i_3''}\partial_M^{i_4} f\right]\\
&\quad\qquad+\sum_{i_2'+i_2''=i_2-1}Y_H^{i_1}\partial_{Q_T}^{i_2'}\left[\mathfrak{A}_3(t,H,Q_T,M)\partial_{Q_T}\partial_{Q_T}^{i_2''}\partial_{H}^{i_3}\partial_M^{i_4} f\right]\\
&\quad\qquad\sum_{i_1'+i_1''=i_1-1}Y_H^{i_1'}\left[\mathfrak{A}_4(t,H,Q_T,M)\partial_{Q_T}\derv{i_1''}{i_2}{i_3}{i_4} f\right] \\
& \qquad =: T_1 + \cdots + T_6,
\end{split}
\end{equation}
where
$$\mathfrak{A}_1(t,H,Q_T,M)=\partial_M r(\rd_r\varphi(t,r)-\rd_r\varphi(T,r))\cdot\{(\partial_X\Omega)(H_T,M)\}-(\partial_Z\Omega)(H_T,M),$$
$$\mathfrak{A}_2(t,H,Q_T,M)=\left[\partial_H r(\rd_r\varphi(t,r)-\rd_r\varphi(T,r))-1\right]\{(\partial_X\Omega)(H_T,M)\},$$
$$\mathfrak{A}_3(t,H,Q_T,M)=(\partial_{Q_T} r)\{(\partial_X\Omega)(H_T,M)\}(\rd_r\varphi(t,r)-\rd_r\varphi(T,r))$$
and 
\begin{align*}
\mathfrak{A}_4(t,H,Q_T,M)=&\partial_H\left[\int_0^1 (\partial_{X}\Omega)(H+\sigma(H_T-H),M)\d\sigma\right](\varphi(t,r)-\varphi(T,r))\\
&\quad+(\partial_H r)\left[\int_0^1  (\partial_{X}\Omega)(H+\sigma(H_T-H),M)\d\sigma\right](\rd_r\varphi(t,r)-\rd_r\varphi(T,r))\\
&\quad-t\left[\partial_H r(\rd_r\varphi(t,r)-\rd_r\varphi(T,r))-1\right](\partial_{Q_T} r)\{(\partial_X\Omega)(H_T,M)\}^2\\
&\qquad\hspace{16em}\times(\rd_r\varphi(t,r)-\rd_r\varphi(T,r)).
\end{align*}
\end{proposition}
\begin{proof}
This follows from Lemma~\ref{lem:vec-field-comm}, after noting that $[Y_H, \rd_{Q_T}] = 0$.
\end{proof}
Next we set up pointwise estimates after integrating the commuted transport equations derived in Proposition~\ref{prop:comm-Vlasov}. Even though we consider only up to $N$ derivatives in Theorem~\ref{the:boot-f}, the following lemma is set up more generally up to $N+1$ derivatives in preparation for the next section (see Theorem~\ref{the:boot-top-plus-1}).
\begin{lemma}\label{lem:error-est}
Suppose $f$ solves \eqref{eq:vlasov-action angle}. 
\begin{enumerate}
\item For $i_1+i_2+i_3+i_4=I\leq N$, the following pointwise estimate holds for $t \in [0,T]$:
\begin{align*}
\|\derv{i_1}{i_2}{i_3}{i_4} f\|_{L^\i}(t)\lesssim \|\partial_{Q_T}^{i_2}\partial_H^{i_1+i_3}\partial_M^{i_4} f_{\ini} \|_{L^\i} +\mathbf{I}+\mathbf{II}+\mathbf{III}+\mathbf{IV},
\end{align*}
where 
$$\mathbf{I}:=\sum_{\substack{{i_1'\leq i_1}\\{1\leq i_2'\leq i_2+1}\\{i_3'\leq i_3}\\{i_4'\leq i_4-1}}}\int_0^t \epsilon \cdot \|\derv{i_1'}{i_2'}{i_3'}{i_4'}f \|_{L^\i}(\tau)\, \ud \tau,$$
$$\mathbf{II}:= \sum_{\substack{{i_1'\leq i_1}\\{1\leq i_2'\leq i_2+1}\\{i_3'\leq i_3-1}}}\int_0^t \|\derv{i_1'}{i_2'}{i_3'}{i_4}f\|_{L^\i}(\tau) \,\ud \tau,$$
$$\mathbf{III}:= \sum_{\substack{{i_1'+i_1''\leq i_1,\,i_2''\leq i_2}\\{i_3''\leq i_3,\,i_4''\leq i_4}\\{i_1''+i_2''+i_3''+i_4''\leq I-1}\\{I' \leq I -i_1''-i_2''-i_3''-i_4''}}}\int_0^t \langle \tau\rangle^{i_1'}\left\|\partial_r^{I'}\partial_s \varphi\right\|_{L^\i(r\in [\f{\mathfrak l_1} 2,\f{2}{\mathfrak h}])}(\tau)\cdot\|\derv{i_1''}{i_2''}{i_3''}{i_4''} \partial_H f\|_{L^\infty}(\tau) \ud\tau,$$
and 
\begin{align*}
\mathbf{IV}&:= \sum_{\substack{{\alpha\leq 1}\\{i_1'+i_1''\leq i_1,\,i_2''\leq i_2}\\{i_3''\leq i_3,\,i_4''\leq i_4}\\{i_1''+i_2''+i_3''+i_4''\leq I-1}\\{I' \leq I-i_1''-i_2''-i_3''-i_4''}\\{I''\leq I-I'-i_1''-i_2''-i_3''-i_4''}}}\int_0^t (1+\norm{\rd_r^{I''+1}\varphi}_{L^\i(r\in [\f{\mathfrak l_1} 2,\f{2}{\mathfrak h}])})\langle \tau\rangle^{i_1'}\left\|\partial_r^{I'+\alpha}(\varphi(\tau,r)-\varphi(T,r))\right\|_{L^\i(r\in [\f{\mathfrak l_1} 2,\f{2}{\mathfrak h}])}\\
&\hspace{25em} \times \|\derv{i_1''}{i_2''}{i_3''}{i_4''} \partial_{Q_T} f\|_{L^\i}(\tau)\, \ud \tau.
\end{align*}
\item For $i_2+i_3+i_4 = I = N+1$, the following pointwise estimate holds for $t \in [0,T]$:
\begin{align*}
\|\partial_{Q_T}^{i_2}\partial_H^{i_3}\partial_M^{i_4} f\|_{L^\i}(t)\lesssim \|\partial_{Q_T}^{i_2}\partial_H^{i_3}\partial_M^{i_4} f_{\ini} \|_{L^\i} +\mathbf{I}+\mathbf{I}'+\mathbf{II}+\mathbf{III}+\mathbf{IV},
\end{align*}
where $\mathbf{I}$, $\mathbf{II}$, $\mathbf{III}$ and $\mathbf{IV}$ are as above, and $\mathbf{I}'$ is only present when $i_4 \geq 1$ and is given by
\begin{equation}\label{eq:bfI'.def}
\mathbf{I}':= \int_0^t \|\rd_{Q_T} f \|_{L^\i}(\tau)\, \ud \tau.
\end{equation}
\end{enumerate}
\end{lemma}
\begin{proof}
We will integrate the equations \eqref{eq:comm-Vlasov} along characteristics. For this purpose, for $T_1$, $T_2$, $T_3$, $T_4$, $T_5$ and $T_6$ as in \eqref{eq:comm-Vlasov}, we first write 
$$T_i = T_i' + T_i'',$$
where $T_i'$ are the terms involving the top derivatives on $f$. For instance,
\begin{align*}
T_1':= &\: -\left((w\rd_{H_T} |_{(s,r,H_T,L)} Q_T)\left(\rd_r\varphi(T,r)-\rd_r\varphi(t,r)\right)\partial_{Q_T}\derv{i_1}{i_2}{i_3}{i_4} f\right), \\
T_1'':= &\: -\left[\derv{i_1}{i_2}{i_3}{i_4}, (w\rd_{H_T} |_{(s,r,H_T,L)} Q_T)\left(\rd_r\varphi(T,r)-\rd_r\varphi(t,r)\right)\partial_{Q_T}\right] f,
\end{align*} 
and similarly for the other terms. Integrating the equations \eqref{eq:comm-Vlasov} along characteristics, we obtain
\begin{align*}
\|\derv{i_1}{i_2}{i_3}{i_4} f\|_{L^\i}(t)= \|\partial_{Q_T}^{i_2}\partial_H^{i_1+i_3}\partial_M^{i_4} f_{\ini}\|_{L^\i}+ \int_0^t \sum_{i=1}^6 \|  T_i'' \|_{L^\i}(\tau) \ud \tau.
\end{align*}
Next we show that these terms can be controlled by the error terms in the statement of the lemma.

Using Leibnitz rule and Corollary~\ref{cor:der-H-Q}, we easily see that the term involving $\tau$ integral of $T''_1$ can be controlled by $\mathbf{IV}.$ Similarly, we see that term involving $T''_2$ can be easily controlled by $\mathbf{III}.$

Now we consider the terms arising from commutators, i.e., terms involving $\mathfrak{A}_1, \cdots, \mathfrak{A}_4$. We crucially note that since these terms arise from commutation, we only take at most $N$ further derivatives. This ensures that we do not have any issue with derivative loss.

We first consider $T''_3$ which involves $\mathfrak{A}_1$. The derivatives of the first half of $\mathfrak{A}_1$, i.e. $$\partial_M r(\rd_r\varphi(t,r)-\rd_r\varphi(T,r))\cdot\{(\partial_X\Omega)(H_T,L)\}$$ can be controlled by $\mathbf{IV}$ thanks to Leibnitz rule, \eqref{eq:precise-der-period.3} and Corollary~\ref{cor:der-inverse-map}. 

For the second half $\mathfrak{A}_1$, the precise expression is
$$- \sum_{i_4'+i_4''=i_4-1}Y_H^{i_1}\partial_{Q_T}^{i_2}\partial_H^{i_3}\partial_{M}^{i_4'}\left[(\partial_Z\Omega)(H_T,M)\partial_{Q_T}\partial_M^{i_4''}f\right].$$
Since this involves $\rd_Z\Omg$ (as opposed to $\rd_X\Omg$), we distinguish two cases. When $i_1+i_2+i_3+i_4\leq N$, there are at most $N-1$ further derivatives on $\rd_Z\Omg$, and thus we use \eqref{eq:precise-der-period.1} in Lemma~\ref{lem:der-Omega} to gain smallness and to bound the term by $\mathbf{I}$ and $\mathbf{IV}$. When $i_1+i_2+i_3+i_4= N+1$, there is a possibility of $N$ derivatives hitting on $\rd_Z\Omg$, in which case we need to apply \eqref{eq:precise-der-period.2} and be contend with no smallness. Nonetheless, in this case we must have no additional derivatives acting on $\rd_{Q_T} f$. This gives rise to the additional term $\mathbf{I}'$ in addition to $\mathbf{IV}$.

Time integral of $T''_4$ can be handled similarly but we use $\mathbf{II}$ and $\mathbf{IV}$ to bound it. Finally the terms $T''_5$ and $T''_6$ can be handled similarly. We emphasize that the commutator term involving $\mathfrak{U}_4$ has two derivatives of $\Omega$ but we only take at most $N-1$ futher derivatives (note that we choose $i_1=0$ in the case $I=N+1$ and thus the term $T''_6$ is absent). Thus we can bound this term using Proposition~\ref{prop:der-period} and Lemma~\ref{lem:der-Omega} since we take at most $N+1$ derivatives of $\Omega$.
\end{proof}

\subsection{Estimates for the initial data term}\label{sec:pointwise.data}

In this subsection and the next subsection, we control the terms on the right-hand side of the estimate in Proposition~\ref{prop:comm-Vlasov}. Here, we begin with the initial data term. Note that the data term is not automatically bounded because the derivatives are defined with respect to dynamical coordinates. Nonetheless, the bounds on the change of coordinate map derived in Section~\ref{sec:dyn.coord} are sufficiently strong for the following lemma:

\begin{lemma}\label{lem:ini-data}
For $f_{\ini}$ satisfying the assumptions in Theorem~\ref{thm:main}, the following estimate holds:
$$\sup_{(Q_T,H,L):(H,L)\in \mathcal S}|\derv{i_1}{i_2}{i_3}{i_4} f_{\ini}|\lesssim \delta\epsilon.$$
\end{lemma}
\begin{proof}
We write $Y_H$, $\partial_{Q_T}$, $\partial_H$ and $\partial_M$ in $(s,r,w,L)$ coordinates. Notice that $Y_H = \rd_H$ at $t = 0$. Finally, noting that the change of coordinate term $\partial_{Q_T}^{i_1}\partial_H^{i_2}\partial_M^{i_3}r$ is such that $i_1+i_2+i_3\leq N+1$, we use Corollary~\ref{cor:der-inverse-map} to get the required result. \qedhere
\end{proof}

\subsection{Estimates for the main error terms}\label{sec:pointwise.main.error}

In this subsection, we bound the errors $\mathbf{III}$ and $\mathbf{IV}$ defined as in Lemma~\ref{lem:error-est}.
\begin{lemma}\label{lem:III-est}
Under the assumptions of Theorem~\ref{the:boot-f}, the following estimate holds for all $i_1+i_2+i_3+i_4 = I \leq N$:
\begin{align*}
\mathbf{III}\lesssim \delta^{3/2} \epsilon^2 \jap{t}^{i_3} [1+ z(\jap{t},I-N+i_1)],
\end{align*}
where $z$ is as defined in \eqref{eq:z.def}.
\end{lemma}
\begin{proof}
It suffices to bound each term in the sum. Take $\{i_k'\}$ and $\{i_k''\}$ as required by the sums in $\textbf{III}$. It will be convenient to use the notation $I'' = i_1''+i_2''+i_3''+i_4''$.

Using the bootstrap assumptions \eqref{eq:varphi_t} and \eqref{eq:boot-f}, we have
\begin{equation}\label{eq:termIII.expanded}
\begin{split}
&\: \int_0^t \langle \tau\rangle^{i_1'}\left\|\partial_r^{I'}\partial_s \varphi\right\|_{L^\i(r\in [\f{\mathfrak l_1} 2,\f{2}{\mathfrak h}])}(\tau)\cdot\|\derv{i_1''}{i_2''}{i_3''}{i_4''} \partial_H f\|_{L^\infty}(\tau) \ud\tau
\\
\lesssim &\: \delta^{3/2}\epsilon^2\int_0^t \jap{\tau}^{i_1'}\jap{\tau}^{-\min\{N-I'+2,N\}} \jap{\tau}^{i_3''+1}\Big[1+\epsilon \Big( z(\jap{\tau},I''-N+i_1'')\Big) \Big]\ud\tau \\
\ls &\: \delta^{3/2}\epsilon^2 \Big( \int_0^t \jap{\tau}^{i_1'+i_3''+1} \jap{\tau}^{-(N-I'+2)} \ud\tau + \int_0^t \jap{\tau}^{i_1'+i_3''+1} \jap{\tau}^{-N} \ud\tau  \\
&\: \phantom{\delta^{3/2}\epsilon^2 \Big(}+ \ep \int_0^t \jap{\tau}^{i_1'+i_3''+1} \jap{\tau}^{-(N-I'+2)} \jap{\tau}^{I''-N+i_1''} \jap{\log \jap{\tau}}\ud\tau \\
&\: \phantom{\delta^{3/2}\epsilon^2 \Big(}+ \ep \int_0^t \jap{\tau}^{i_1'-N+i_3''+1} \jap{\tau}^{I''-N+i_1''}\jap{\log \jap{\tau}} \ud\tau \Big) \\
 =: &\:\delta^{3/2}\epsilon^2 [A_1+A_2+\ep(A_3+A_4)],
\end{split}
\end{equation}
where we have used the trivial bound
$$\ep z(\jap{\tau},I''-N+i_1'') \ls 1 + \ep \jap{\tau}^{I''-N+i_1''} \jap{\log \jap{\tau}}.$$

We now control each term on the right-hand side of \eqref{eq:termIII.expanded}. For $A_1$, note that $i_1'\leq i_1$, $I'\leq I$ and $i_3''\leq i_3$. Hence, 
\begin{equation}\label{eq:III.1}
A_1 = \int_0^t \jap{\tau}^{i_1'-N+I'-2+i_3''} \ud\tau \ls \jap{t}^{i_3} \int_0^t \jap{\tau}^{i_1-N+I-1} \ud\tau \ls \jap{t}^{i_3} z(\jap{t}, I-N+i_1).
\end{equation}

For $A_2$, we split into two cases: $i_1' \leq 2$ or $i_1' \geq 3$. If $i_1' \leq 2$, we use $i_3''\leq i_3$ and $N \geq 5$ to get
\begin{equation}\label{eq:III.2}
\begin{split}
A_2 = \int_0^t \jap{\tau}^{i_1'-N+i_3''+1} \ud\tau \ls \jap{t}^{i_3} \int_0^t \jap{\tau}^{-N+3} \ud\tau \ls \jap{t}^{i_3} \int_0^t \jap{\tau}^{-2} \ud\tau  \ls \jap{t}^{i_3}.
\end{split}
\end{equation}
If $i_1' \geq 3$, then we use $i_1' \leq i_1$, $i_1' \leq I$ to get $i_1' -i_1-I+2\leq 0$. Hence, using also $i_3''\leq i_3$, we have
\begin{equation}\label{eq:III.3}
\begin{split}
A_2 = \int_0^t \jap{\tau}^{i_1'-N+i_3''+1} \ud\tau \ls &\: \jap{t}^{i_3} \int_0^t \jap{\tau}^{I-N+i_1-1} \jap{\tau}^{i_1' -i_1-I+2} \ud\tau\\
\ls&\:  \jap{t}^{i_3} \int_0^t \jap{\tau}^{I-N+i_1-1}  \ud\tau \ls \jap{t}^{i_3} z(\jap{t}, I-N+i_1).
\end{split}
\end{equation}

For $A_3$, we use $i_1' + i_1'' \leq i_1$, $i_3'' \leq i_3$ and $I' \leq I$ to get
\begin{equation}\label{eq:III.4}
\begin{split}
A_3 = &\:\int_0^t \jap{\tau}^{i_1'+i_3''+1-N+I'-2} \jap{\tau}^{I''-N+i_1''} \jap{\log \jap{\tau}} \ud\tau \ls \jap{t}^{i_3} \int_0^t \jap{\tau}^{I-N+i_1-1} \jap{\tau}^{-N} \jap{\log \jap{\tau}} \ud\tau \\
\ls &\: \jap{t}^{i_3} \int_0^t \jap{\tau}^{I-N+i_1-1}  \ud\tau \ls \jap{t}^{i_3}z(\jap{t}, I-N+i_1).
\end{split}
\end{equation}
where in the penultimate inequality, we used $\jap{\tau}^{-N} \jap{\log \jap{\tau}}\ls 1$, which follows from $N\geq 5$.

Finally, $A_4$ can be treated in the same way as $A_3$ so that
\begin{equation}\label{eq:III.5}
\begin{split}
A_4 = &\: \int_0^t \jap{\tau}^{i_1'-N+i_3''+1} \jap{\tau}^{I''-N+i_1''} \jap{\log \jap{\tau}} \ud\tau \ls \jap{t}^{i_3} \int_0^t \jap{\tau}^{I-N+i_1-1} \jap{\tau}^{-N+2}\jap{\log \jap{\tau}} \ud\tau \\
\ls &\:\jap{t}^{i_3}z(\jap{t}, I-N+i_1).
\end{split}
\end{equation}

Plugging the estimates \eqref{eq:III.1}--\eqref{eq:III.5} into \eqref{eq:termIII.expanded}, we obtain the desired bound. \qedhere

\end{proof}

\begin{lemma}\label{lem:IV-est}
Under the assumptions of Theorem~\ref{the:boot-f}, the following estimate holds for all $i_1+i_2+i_3+i_4 = I \leq N$:
\begin{align*}
\mathbf{IV}\lesssim \delta^{3/2} \epsilon^2 z(\jap{t},I-N+i_1).
\end{align*}
\end{lemma}
\begin{proof}
First, observe that since $I\leq N$, we have $(1+\norm{\rd_r^{I''+1}\varphi}_{L^\i(r\in [\f{\mathfrak l_1} 2,\f{2}{\mathfrak h}])})\ls 1$, thanks to the bootstrap assumption \eqref{eq:varphi_r}. Then the proof is basically the same as Lemma~\ref{lem:III-est} except now 
$|\derv{i_1''}{i_2''}{i_3''}{i_4''} \partial_{Q_T} f|$ does not have the extra $\jap{t}$ growth unlike $|\derv{i_1''}{i_2''}{i_3''}{i_4''} \partial_{H} f|$ but on the other hand the bootstrap assumption \eqref{eq:varphi_diff} only allow $\jap{t}^{-1}$ decay for $\partial_r^{N+1}(\varphi(t,r)-\varphi(T,r))$. \qedhere
\end{proof}

\subsection{Proof of Theorem~\ref{the:boot-f}}\label{sec:pointwise.proof}

We finally prove Theorem~\ref{the:boot-f} now. We will carry out an induction in $i_3+i_4$, i.e., the number of derivatives in $\rd_{H}$ and $\rd_{M}$. We observe that in the term $\mathbf{II}$, we have $i_3' \leq i_3 -1$ (which gives a slower growth). In the term $\mathbf{I}$, even though $i_3' = i_3$ is allowed, it has an extra factor of $\epsilon$.
\begin{proof}[Proof of Theorem~\ref{the:boot-f}] 
Plugging the estimates from Lemma~\ref{lem:ini-data}, Lemma~\ref{lem:III-est} and Lemma~\ref{lem:IV-est} into Lemma~\ref{lem:error-est} (and recalling that we are in the case $i_1+i_2+i_3+i_4=I\leq N$), we obtain
\begin{align*}
\|\derv{i_1}{i_2}{i_3}{i_4} f\|_{L^\i}(t) \lesssim \delta\epsilon\jap{t}^{i_3}[1+\epsilon z(\jap{t},I-N+i_1)]+\mathbf{I}+\mathbf{II},
\end{align*}
where $z(\jap{t},I-N+i_1)$ is as in \eqref{eq:z.def} and $\mathbf{I}$, $\mathbf{II}$ are as in Lemma~\ref{lem:error-est}.

We now estimate the time integrals in $\mathbf{I}$, $\mathbf{II}$ by H\"older's inequality to obtain
\begin{equation}\label{eq:f.induction}
\begin{split}
&\: \|\derv{i_1}{i_2}{i_3}{i_4} f\|_{L^\i}(t) \\
\lesssim&\: \delta\epsilon\jap{t}^{i_3}[1+\epsilon z(\jap{t},I-N+i_1)]  \\
&\: +\sum_{\substack{{i_1'\leq i_1}\\{1\leq i_2'\leq i_2+1}\\{i_3'\leq i_3}\\{i_4'\leq i_4-1}}}t{\epsilon}\cdot \sup_{\tau \leq t} \|\derv{i_1'}{i_2'}{i_3'}{i_4'}f\|_{L^\i}(\tau) +\sum_{\substack{{i_1'\leq i_1}\\{1\leq i_2'\leq i_2+1}\\{i_3'\leq i_3-1}}}t \sup_{\tau \leq t} \|\derv{i_1'}{i_2'}{i_3'}{i_4}f \|_{L^\i}(\tau).
\end{split}
\end{equation}

We proceed by induction on $i_3+i_4$. When $i_3+i_4=0$, the last two terms on the right-hand side of \eqref{eq:f.induction} are absent and thus we get the desired result.  Assume that $\exists k \in \mathbb N\cup \{0\}$ such that the desired bound is true for $i_3+i_4\leq k$. Then since $i_4'\leq i_4-1$ for the second term and $i_3'\leq i_3-1$ for the third term in \eqref{eq:f.induction}, we get that
\begin{equation}\label{eq:f.final}
\begin{split}
\|\derv{i_1}{i_2}{i_3}{i_4} f\|_{L^\i}(t) &\lesssim \delta\epsilon\jap{t}^{i_3}[1+\epsilon z(\jap{t},I-N+i_1)]\\
&\quad+t{\epsilon}\cdot\delta\epsilon\jap{t}^{i_3}[1+\epsilon z(\jap{t},I-N+i_1)]\\
&\quad+t\cdot\delta\epsilon\jap{t}^{i_3-1}[1+\epsilon z(\jap{t},I-N+i_1)]\\
&\lesssim  \delta\epsilon\jap{t}^{i_3}[1+\epsilon z(\jap{t},I-N+i_1)],
\end{split}
\end{equation}
where we used the fact that $t\leq \epsilon^{-1}$ in the last step. 

We have thus shown that the estimate \eqref{eq:f.final} holds under the assumption \eqref{eq:boot-f}. By Lemma~\ref{lem:ini-data}, we also know that \eqref{eq:f.final} is verified initially. Hence, continuity implies that the bound \eqref{eq:f.final} holds in the $(t,Q_T,H,M)$ coordinates for all $t \in [0,T]$, as long as $T \in (0, T_{\text{B}})$.
\qedhere
\end{proof}

\section{Proof of Theorem~\ref{thm:main}: Top order bounds for $f$ and boundedness estimates for $\varphi$ and its derivatives}\label{sec:top.boundedness}

We continue to work under the assumptions of Theorem~\ref{thm:boot}.

This section will achieve two goals. First, we prove boundedness of $\varphi$ and its $\rd_r$ derivative in \textbf{Section~\ref{sec:phir.bounded.low.order}}, which in particular improve the bootstrap assumption \eqref{eq:varphi_r}. Second, we need to obtain a top-order estimate for $f$, which involves $N+1$ derivatives. In the process, we will also need to obtain some stronger estimates for the top derivatives of $\varphi$, using ideas introduced in Section~\ref{sec:phir.bounded.low.order}. This will be carried out in \textbf{Section~\ref{sec:f-top}}.

\subsection{Improving the bootstrap assumption \eqref{eq:varphi_r}}\label{sec:phir.bounded.low.order}
In this subsection, we will improve the bootstrap assumption \eqref{eq:varphi_r}. This relies on bounding $|\partial_r^\ell \varphi|$ pointwise by $f$ and its derivatives, and then using Theorem~\ref{the:boot-f}. 

For all the estimates related to $\varphi$ (including the estimates in this section and the decay estimates in the next section), we will repeatedly use the following simple algebraic fact. The point is that when $\rd_H$ hits $f$, it causes time growth. The following lemma says that $\rd_H$ can be exchanged to a linear combination of $\rd_{Q_T}$, $\rd_M$ and $\rd_L$. The upshot is that $\rd_{Q_T}$ and $\rd_M$ do not cause time growth while acting on $f$, and $\rd_L$ can be integrated by parts away to terms for which it is less harmful.
\begin{lemma}\label{lem:H-in-L-M}
 The following identities hold:
\begin{align}
\partial_H= &\: 2r^2(\partial_L-\partial_LQ_T\partial_{Q_T}-\partial_M) \label{eq:H-in-L-M}\\
\rd_r = &\: [\partial_rQ_T - 2r^2 (\partial_r H)(\partial_LQ_T)]\partial_{Q_T} +2r^2 (\partial_r H)\partial_L-2r^2 (\partial_r H)\partial_M. \label{eq:r-in-good-things}
\end{align}
 \end{lemma}
 \begin{proof}
We start with
$$\partial_L=\partial_L Q_T\partial_{Q_T}+\partial_L H \partial_H+\partial_M,$$
rearrange, and note that $\partial_L H=\frac{1}{2r^2}$ to obtain \eqref{eq:H-in-L-M}.

To prove \eqref{eq:r-in-good-things}, we first write 
$$\partial_r=\partial_rQ_T\partial_{Q_T}+\partial_r H\partial_H$$
and then use \eqref{eq:H-in-L-M} to deduce \eqref{eq:r-in-good-things}. \qedhere
\end{proof}

\begin{proposition}\label{prop:den-by-f}
For $0\leq I \leq N+3$. the following bound holds for all $t \in [0,T]$:
\begin{equation}
\sup_{ r\in [\f{\mathfrak l_1}{2}, \f{2}{\mathfrak h}]}|\partial_r^I\varphi|(t,r)\lesssim \sum_{i_1+i_2\leq \max\{0,I-2\}}\|\partial_{Q_T}^{i_1}\partial_M^{i_2} f\|_{L^\i}(t).
\end{equation}
\end{proposition}
\begin{proof}
From \eqref{eq:transport_spherical}, we know that 
$$\varphi(s,r)=- \pi\int_0^r\int_{r_2}^\infty \int_{-\infty}^\infty\int_{0}^\infty\frac{1}{r r_1}f(s,r_1,w,L)\d L\d w\d r_1\d r_2.$$
Note that the first two $\rd_r$ derivatives would either remove the $r_1$, $r_2$ integrals or hit on the $r$ (or $r_1$) factors. Either case, when restricted to $r\in [\f{\mathfrak l_1}{2}, \f{2}{\mathfrak h}]$, these would be harmless so that
\begin{equation}
\begin{split}
\sup_{ r\in [\f{\mathfrak l_1}{2}, \f{2}{\mathfrak h}]} | \rd^I \varphi(s,r)| \ls &\: \sum_{I' \leq \max\{0,I-2\}} \sup_{ r\in [\f{\mathfrak l_1}{2}, \f{2}{\mathfrak h}]} \Big| \rd_r^{I'} \int_{-\infty}^\infty\int_{0}^\infty f(s,r,w,L)\d L\d w \Big| \\
\ls &\: \sum_{I' \leq \max\{0,I-2\}} \sup_{ r\in [\f{\mathfrak l_1}{2}, \f{2}{\mathfrak h}]} \Big| \int_{-\infty}^\infty\int_{0}^\infty  \rd_r^{I'} f(s,r,w,L)\d L\d w \Big|.
\end{split}
\end{equation}
For each $\rd_r$, we write it as a linear combination of $\rd_{Q_T}$, $\rd_{L}$ and $\rd_{M}$ using \eqref{eq:r-in-good-things}.
Recalling our notation from Definition~\ref{def:coordinate.vector.fields}: note that these are coordinate derivatives from different coordinate systems! $(\rd_{Q_T}, \rd_M)$ are from the $(t,Q_T,H,M)$ coordinate system and $\rd_L$ is from the $(s,r,w,L)$ coordinate system. We will then integral by parts that $\rd_L$ away; the coefficients is sufficiently regular (by Lemma~\ref{lem:Jac}) to handle these derivatives. Notice that $\rd_L$ does not commute with $\rd_{Q_T}$ or $\rd_M$, but the commutator can be expressed as a linear combination of $\rd_{Q_T}$, $\rd_{M}$ and $\rd_{H}$ with sufficiently regular coefficients (since $(Q_T,H,M)$ are coordinates on a constant-$t$ hypersurface). When $\rd_H$ arises, we use \eqref{eq:H-in-L-M} again and continue the procedure.

At the end, we thus obtain that
\begin{equation}
\begin{split}
\sup_{ r\in [\f{\mathfrak l_1}{2}, \f{2}{\mathfrak h}]} | \rd^I \varphi(s,r)| \ls &\: \sum_{I' \leq \max\{0,I-2\}} \sup_{ r\in [\f{\mathfrak l_1}{2}, \f{2}{\mathfrak h}]} \Big| \rd_r^{I'} \int_{-\infty}^\infty\int_{0}^\infty f(s,r,w,L)\d L\d w \Big| \\
\ls &\: \sum_{i_1+i_2 \leq \max\{0,I-2\}} \sup_{ r\in [\f{\mathfrak l_1}{2}, \f{2}{\mathfrak h}]}  \int_{-\infty}^\infty\int_{0}^\infty  |\rd_{Q_T}^{i_1} \rd_M^{i_2} f(s,r,w,L)| \d L\d w.
\end{split}
\end{equation}
Finally, we obtain the desired estimates using the support properties of $f$. \qedhere
\end{proof}
Combining the above proposition with Proposition~\ref{the:boot-f}, we improve the bootstrap assumption \eqref{eq:varphi_r} as follows.
\begin{corollary}\label{cor:boot-r-der-phi}
For $0\leq I \leq N+2$, the following bound holds for every $t\in [0,T]$:
$$\sup_{t\in [0,T]} \sup_{ r\in [\f{\mathfrak l_1}{2}, \f{2}{\mathfrak h}]}|\partial_r^I \varphi|(t,r)\lesssim\delta\epsilon.$$
\end{corollary}
\subsection{Top order estimates for $f$} \label{sec:f-top}
In order to close the decay estimates $\rd_s \varphi$ (and $\varphi(t) - \varphi(T)$) in Section~\ref{sec:density}, we will need an estimate for $\partial_{Q_T}^{i_2}\partial_H^{i_3}\partial_M^{i_4} f$ for $i_2+i_3+i_4=N+1$. We derive such an estimate --- see \eqref{eq:boot-f-top-plus-1} --- in this subsection. (See also Steps~1 and 3 in the proof of Proposition~\ref{prop:T12} when this estimate is eventually used.) Note that this is one order higher than that provided by Theorem~\ref{the:boot-f}. On the other hand, this is simpler than the estimate in Theorem~\ref{the:boot-f} in that there are no $Y_H$ derivatives. 

We note that in order to close the top order estimates for $f$, we will need to control $\partial_r^{N+1}\partial_s\varphi$, which involve more derivatives than given in the bootstrap assumption \eqref{eq:varphi_t}. However, at this order, we longer need decay for these terms and thus they can be controlled along the lines of Section~\ref{sec:phir.bounded.low.order}.

The following is the main result of this subsection, which improves the bootstrap assumption \eqref{eq:boot-f-top-plus-1}.
\begin{theorem}\label{the:boot-top-plus-1}
Suppose that the assumptions of Theorem~\ref{thm:boot} hold. Then the bound 
$$\|\partial_{Q_T}^{i_2}\partial_H^{i_3}\partial_{M}^{i_4} f\|_{L^\i}(t)\leq C\delta \epsilon \jap{t}^{i_3+1},\quad i_2+i_3+i_4=N+1$$
holds for some $C>0$ depending only on $\mathfrak c_0$, $\mathfrak h_0$, $\mathfrak l_1$, $\mathfrak l_2$ and $N$, and independent of $\de$ and $\ep$.
\end{theorem}

In order to establish Theorem~\ref{the:boot-top-plus-1}, we will work under the following bootstrap assumptions for this subsection:
\begin{equation}\label{eq:boot-f-top-plus-1}
\|\partial_{Q_T}^{i_2}\partial_H^{i_3}\partial_{M}^{i_4} f\|_{L^\i}(t)\leq \delta^{3/4}\epsilon \jap{t}^{i_3+1},\quad i_2+i_3+i_4=N+1.
\end{equation}

As mentioned above, in order to prove Theorem~\ref{the:boot-top-plus-1}, we need estimates for $\partial_r^{N+1}\partial_s\varphi$



\begin{proposition}\label{prop:phi-t-by-f}
For $I\leq N+1$, the following estimate holds for every $t\in [0,T]$:
\begin{equation}
\sup_{ r\in [\f{\mathfrak l_1}{2}, \f{2}{\mathfrak h}]}|\partial_r^{I}\rd_s \varphi|(t,r)\lesssim \de \ep.
\end{equation}
\end{proposition}
\begin{proof}
Recall our notation that $\rd_s$ denotes the time derivative with fixed $(r,w,L)$, while $\rd_t$ denotes the time derivative with fixed $(Q_T,H,M)$. In particular,
\begin{equation}\label{eq:ds.to.dt}
\partial_s=\partial_t+\partial_s H\partial_H=\partial_t+\rd_s \varphi\partial_H.
\end{equation}
Starting with \eqref{eq:transport_spherical}, differentiate by $\rd_s$ and then using \eqref{eq:ds.to.dt}  and the equation \eqref{eq:vlasov-action angle}, we obtain
\begin{equation}
\begin{split}
\rd_s \varphi(s,r) = &\: -\pi \int_0^r \int_{r_2}^\infty \int_{-\infty}^\infty\int_{0}^\infty\frac{1}{rr_1}\rd_s f(s,r,w,L)\d L\d w\d r_1 \d r_2 \\
= &\: -\pi \int_0^r \int_{r_2}^\infty \int_{-\infty}^\infty\int_{0}^\infty\frac{1}{rr_1}\Omg(H_T,M) \partial_{Q_T} f \d L\d w\d r_1 \d r_2 \\
&\: -\pi \int_0^r \int_{r_2}^\infty \int_{-\infty}^\infty\int_{0}^\infty\frac{1}{rr_1}(w\rd_{H_T} |_{(s,r,H_T,L)} Q_T)\\
&\: \hspace{13em}\times\left(\rd_r\varphi(T,r_1)-\rd_r\varphi(s,r_1)\right)\partial_{Q_T}f \d L\d w\d r_1 \d r_2.
\end{split}
\end{equation}
(Notice that two terms involving $\rd_s \varphi \partial_{H}f$ cancel.)

We now apply the same trick as in Proposition~\ref{prop:den-by-f}, i.e., differentiate by $\rd_r^{N+1}$, use \eqref{eq:r-in-good-things} and then integrate the $\rd_H$ by parts away from $f$ by using \eqref{eq:H-in-L-M}. The same argument as in Proposition~\ref{prop:den-by-f} then gives 
\begin{equation*}
\begin{split}
\sum_{I=0}^{N+1} |\rd_r^{I} \rd_s \varphi(s,r)|\ls \sum_{i_1+i_2+I' \leq N-1}\Big| \int_{-\infty}^\infty (1+|\rd^{I'} (\rd_r\varphi(T,r)-\rd_r\varphi(s,r))| ) |\rd_{Q_T}^{i_1+1} \rd_M^{i_2} f|(s,r)\, \ud L \ud w  \Big|.
\end{split}
\end{equation*}

We thus obtain the desired estimate after using \eqref{eq:varphi_diff} and Theorem~\ref{the:boot-f}. \qedhere

%
\end{proof}

We are now ready to prove the estimates for top-order derivatives of $f$ as is required by Theorem~\ref{the:boot-top-plus-1}. To obtain the desired estimates, we will use Lemma~\ref{lem:error-est} with $i_1=0$. Before turning to the proof of Theorem~\ref{the:boot-top-plus-1}, we first estimate the terms $\mathbf{III}$ and $\mathbf{IV}$ from Lemma~\ref{lem:error-est} in the following lemma.
\begin{lemma}\label{lem:III-top-plus-1}
Under the assumptions of Theorem~\ref{the:boot-top-plus-1}, the following estimates hold for all $t \in [0,T]$ whenever $i_1=0$ and $i_2+i_3+i_4=I \leq N+1$:
\begin{align*}
\mathbf{III} + \mathbf{IV}\lesssim \delta^{7/4}\epsilon \jap{t}^{i_3+1}.
\end{align*}
\end{lemma}
\begin{proof}
For the term $\mathbf{III}$, we use Proposition~\ref{prop:phi-t-by-f} for derivatives of $\rd_s \varphi$ and Theorem~\ref{the:boot-f}, \eqref{eq:boot-f-top-plus-1} for derivatives of $f$ to obtain
\begin{equation*}
\begin{split}
\mathbf{III} \leq &\: \sum_{\substack{i_3''\leq i_3 \\ i_2'' + i_3'' + i_4'' \leq N \\ I' \leq N+1}} \int_0^t \|\partial_r^{I'} \rd_s \varphi \|_{L^\i(r\in [\f{\mathfrak l_1}2, \frac 2{\mathfrak h}])}\cdot\|\partial_{Q_T}^{i_2''}\partial_{H}^{i_3''}\partial_{M}^{i_4''} \partial_H f\|\tau \ud\tau \\
\ls &\:  \int_0^t (\de \ep) \cdot (\de^{3/4} \ep \brk{\tau}^{i_3+1}) \, \ud \tau \ls \de^{7/4} \ep\brk{t}^{i_3+1},
\end{split}
\end{equation*}
where we have used $\ep \brk{t} \ls 1$.

The term $\mathbf{IV}$ is very similar, except that we will use Corollary~\ref{cor:boot-r-der-phi} instead of Proposition~\ref{prop:phi-t-by-f}. The only difference is the extra term $$\int_{0}^t \norm{\rd_r^{N+3}\varphi}_{L^\i(r\in [\f{\mathfrak l_1} 2,\f{2}{\mathfrak h}])}\left\|\rd_r(\varphi(\tau,r)-\varphi(T,r))\right\|_{L^\i(r\in [\f{\mathfrak l_1} 2,\f{2}{\mathfrak h}])}\cdot \|\partial_{Q_T} f\|_{L^\i}(\tau)\, \ud \tau.$$
Using Proposition~\ref{prop:den-by-f} and bootstrap assumption \eqref{eq:boot-f-top-plus-1}, the bootstrap assumption \eqref{eq:varphi_diff} and 
Theorem~\ref{the:boot-f}, we get that 
\begin{align*}
\int_{0}^t (1+\norm{\rd_r^{N+3}\varphi}_{L^\i(r\in [\f{\mathfrak l_1} 2,\f{2}{\mathfrak h}])})&\left\|\rd_r(\varphi(\tau,r)-\varphi(T,r))\right\|_{L^\i(r\in [\f{\mathfrak l_1} 2,\f{2}{\mathfrak h}])}\cdot \|\partial_{Q_T} f\|_{L^\i}(\tau)\, \ud \tau\\
&\lesssim \int_0^t (\delta \epsilon)\cdot(\delta^{3/4}\epsilon\tau^{-2})\ud \tau\\
&\lesssim \de^{7/4} \ep \ls \de^{7/4} \ep\brk{t}^{i_3+1}.
\end{align*}
\qedhere

\end{proof}

Equipped with Lemma~\ref{lem:III-top-plus-1}, it is now easy to improve the bound \eqref{eq:boot-f-top-plus-1}.
\begin{proof}[Proof of Theorem~\ref{the:boot-top-plus-1}]
We proceed in the same way as in the proof of Theorem~\ref{the:boot-f}, except for the following:
\begin{enumerate}
\item We use case (2) instead of case (1) in Lemma~\ref{lem:error-est}.
\item Now we use Lemma~\ref{lem:III-top-plus-1} to control the terms $\mathbf{III}$ and $\mathbf{IV}$ from Lemma~\ref{lem:error-est}. 
\item At the top level of derivatives, there is the term $\mathbf{I}'$ in \eqref{eq:bfI'.def} when we apply case (2) of Lemma~\ref{lem:error-est} (which is absent in case (1)). Nevertheless, given the bound for $\rd_{Q_T} f$ derived in Theorem~\ref{the:boot-f}, we easily obtain $\mathbf{I}' \ls \de \ep \jap{t}$, which is sufficient to give the desired bound.
\end{enumerate} \qedhere
\end{proof}
\section{Proof of Theorem~\ref{thm:main}: Decay estimates for $\rd_s \varphi(t)$, $\varphi(t)-\varphi(T)$ and their derivatives}\label{sec:density}

We continue to work under the assumptions of Theorem~\ref{thm:boot}. 

The goal of this section will be to prove decay estimates. In particular, we will improve the bootstrap assumptions \eqref{eq:varphi_t}--\eqref{eq:varphi_diff} and complete the proof of Theorem~\ref{thm:boot}.

Recall from \eqref{eq:vlasov-action angle} that our equation takes the form
\begin{equation}
\partial_t f+\Omega(H_T,M)\partial_{Q_T}f=\bar{\mathfrak{G}},
\end{equation}
where $\Omega(X,Z)$ is defined as in Definition~\ref{def:omega} and $\bar{\mathfrak{G}}$ contains the nonlinear terms from \eqref{eq:vlasov-action angle}. Since $\rd_{Q_T} [\Omega(H_T,M)] \neq 0$ (recall that $\rd_{Q_T}$ is defined with respect to $(Q_T,H,M)$, not $(Q_T,H_T,M)$ coordinates\footnote{At this point, the reader may wonder why we do not use the $(t,Q_T,H_T,M)$ coordinate system instead. The reason is quite subtle. If we use the $(t,Q_T,H_T,M)$ coordinate system, then in \eqref{eq:vlasov-action angle}, we have a term $(\partial_r\varphi(T,\cdot) - \partial_r\varphi(t,\cdot) )\rd_{H_T} f(t,\cdot)$ which seems to have insufficient decay in the analogue of Proposition~\ref{prop:nonlin-diff}.}), it will be convenient to work with the following modified equation
\begin{equation}\label{eq:den-lin}
\partial_t f+\Omega(H,M)\partial_{Q_T}f=\mathfrak{G},
\end{equation}
where now $\mathfrak{G}=\bar{\mathfrak{G}}+(\Omega(H,M)-\Omega(H_T,M))\partial_{Q_T} f.$

Using Duhamel's principle for \eqref{eq:den-lin} we see that
$$f(t,Q_T,H,M)=S(t,Q_T,H,M) f_{\ini}+\int_0^t S(t-\tau,Q_T,H,M) \mathfrak{G}(\tau)\ud\tau$$
where $S$ is the solution semi-group for the equation
\begin{align*}
\partial_t h+\Omega(H,M)\partial_{Q_T}h&=0,\quad h|_{t=0}=h_0,
\end{align*}
i.e., $S$ is given explicitly by
\begin{equation}
S(t,Q_T,H,M) h=h_0(t,Q_T-t\Omega(H,M),H,M).
\end{equation}

Using \eqref{eq:Poisson_spherical} and writing $f(s,r,w,L)=\bar f(t,Q_T(r,w,L),H(s,r,w,L),M)$. We will be dropping this dependence of $Q_T$ and $H$ on $(s,r,w,L)$ variables and abuse notation to write $\bar f$ as $f$.
\begin{equation}\label{eq:phi-rep}
\begin{split}
\varphi(t,r)&=- \pi \int_0^r\int_{r_2}^\infty \int_{-\infty}^\infty\int_{0}^\infty\frac{1}{r r_1} f(t,Q_T,H,M)\d L\d w\d r_1\d r_2\\
&=-\pi \int_0^r\int_{r_2}^\infty \int_{-\infty}^\infty\int_{0}^\infty\frac{1}{r r_1}S(t,Q_T,H,M)  f_{\ini}(Q_T,H,M)\d L\d w\d r_1\d r_2\\
&\quad-\pi \int_0^r\int_{r_2}^\infty \int_{-\infty}^\infty\int_{0}^\infty\frac{1}{r r_1}\int_0^t S(t-\tau,Q_T,H,M) \mathfrak{G}(\tau)\ud\tau\d L\d w\d r_1\d r_2\\
&=: -\pi \mathcal{L}(t,r) - \pi \mathcal{N}(t,r).
\end{split}
\end{equation}
Notice that both $\mathcal{L}$ and $\mathcal{N}$ are functions of $(t,r)$ alone. We will use the convention (consistent with the above) that $\rd_s$ denotes the derivative in $t$ while holding $r$ fixed. Note that this is different from the $\rd_t$ derivative (for which $(Q_T,H,M)$ is held fixed).

In the remainder of this section, we will bound the $\calL$ and the $\calN$ terms in \eqref{eq:phi-rep}. It will be useful to expand the terms in Fourier series in $Q_T$. From now on, we use $\,{ }\widehat{ }\,$ to denote the Fourier coefficients of the Fourier series in $Q_T$, i.e., 
\begin{equation}\label{eq:Fourier.series}
h(t,Q_T,H,M) = \sum_{k \in \mathbb Z} e^{ikQ_T} \widehat{h}_k(t,H,M).
\end{equation}

The rest of the section will be organized as follows. In \textbf{Section~\ref{sec:phase.mixing.linear}}, we prove the needed estimates for $\calL$ and its derivatives. Section~\ref{sec:phase.mixing.nonlinear.1}--Section~\ref{sec:phase.mixing.nonlinear.3} will then be devoted to the estimates for $\calN$ and its derivatives. In \textbf{Section~\ref{sec:phase.mixing.nonlinear.1}}, we setup up the estimates by writing out all the terms. In \textbf{Section~\ref{sec:phase.mixing.nonlinear.2}}, we isolate and bound some easier terms. In \textbf{Section~\ref{sec:phase.mixing.nonlinear.3}}, we control the hardest nonlinear terms. In particular, these are terms for which we need to use the $Y_H$ vector field to understand resonance. In \textbf{Section~\ref{sec:density.everything}}, we put together all the estimates to prove Theorem~\ref{thm:boot}.

\subsection{Linear phase mixing in spherical symmetry for perturbations of Kepler potential}\label{sec:phase.mixing.linear} 

In this section we focus on the linear term $\calL$. The main results can be found in Proposition~\ref{prop:linear-est-den.1} and Proposition~\ref{prop:linear-est-den.2}.

We begin with a computation of the terms $\mathcal L_s$ and $\calL(t,r)-\calL(T,r)$.
\begin{proposition}
For $\mathcal{L}$ as in \eqref{eq:phi-rep}, the following holds:
\begin{equation}\label{eq:L-t}
\begin{split}
&\: \mathcal{L}_s(t,r)\\
=&\: -\sum_{\substack{{k\in\Z}\\{k\neq 0}}}\int_0^r\int_{r_2}^\infty \int_{-\infty}^\infty\int_{0}^\infty\frac{ik\Omega(H,M)}{r r_1}e^{ikQ_T}e^{-ikt\Omega(H,M)}  (\widehat{f}_{\ini})_k(H,M)\d L\d w\d r_1\d r_2 \\
&\: -\sum_{k\in\Z}\int_0^r\int_{r_2}^\infty \int_{-\infty}^\infty\int_{0}^\infty\Big[ikt (\rd_X\Omega)(H,M) (\widehat{f}_{\ini})_k(H,M) + (\rd_H \widehat{f}_{\ini})_k(H,M)\Big]\\
&\: \phantom{-\sum_{k\in\Z}\int_0^r\int_{r_2}^\infty \int_{-\infty}^\infty\int_{0}^\infty\frac{ik\Omega(H,M)e^{-ikt\Omega(H,M)}  }{r r_1}} \times \f{\rd_s \varphi e^{ikQ_T}e^{-ikt\Omega(H,M)}}{r r_1}  \d L\d w\d r_1\d r_2
\end{split}
\end{equation}
and
\begin{equation}\label{eq:L-diff}
\begin{split}
&\: \mathcal{L}(t,r)-\mathcal{L}(T,r)\\
&=\sum_{\substack{{k\in\Z}\\{k\neq 0}}}\int_0^r\int_{r_2}^\infty \int_{-\infty}^\infty\int_{0}^\infty\frac{1}{r r_1}e^{ikQ_T}e^{-ikt\Omega(H(t),M)}  (\widehat{f_{\ini}})_k(H(t),M)\d L\d w\d r_1\d r_2\\
&\quad-\sum_{\substack{{k\in\Z}\\{k\neq 0}}}\int_0^r\int_{r_2}^\infty \int_{-\infty}^\infty\int_{0}^\infty\frac{1}{r r_1}e^{ikQ_T}e^{-ikT\Omega(H(T),M)}  (\widehat{f}_{\ini})_k(H(T),M)\d L\d w\d r_1\d r_2\\
&\quad+\int_0^r\int_{r_2}^\infty \int_{-\infty}^\infty\int_{0}^\infty\frac{1}{r r_1}[(\widehat{f}_{\ini})_0(H(t),M)-(\widehat{f}_{\ini})_0(H(T),M)]\d L\d w\d r_1\d r_2.
\end{split}
\end{equation}
Here, $(H,M)$ are viewed as functions of $(T,r,w,L)$ or $(t,r,w,L)$; in \eqref{eq:L-t}, all $H = H(t,r,w,L)$, and in \eqref{eq:L-diff}, we used the shorthand $H(t) = H(t,r,w,L)$, $H(T) = H(T,r,w,L)$.
\end{proposition}
\begin{proof}
We use Fourier series in $Q_T$ (see \eqref{eq:Fourier.series}) to write
$$S(t,Q_T,H,M)  f_{\ini}(Q_T,H,M)=\sum_{k\in\Z}e^{ikQ_T}e^{-itk\Omega(H,M)}(\widehat{f}_{\ini})_k(H,M).$$
We plug this into the expression of $\mathcal L$. Using that $\rd_s t = 1$ and $\rd_s H = \rd_s \varphi$, we obtain the first identity. For the second identity, we again use the Fourier series expression above and subtract mode by mode. Note in particular the cancellation for the $k=0$ mode. \qedhere
\end{proof}

\begin{lemma}\label{lem:lin-k-est}
For $\alp\leq 1$ and every $k \in \mathbb Z\setminus\{0\}$, the following estimates hold for $I\leq N-1-\alp$:
\begin{equation}\label{eq:lim-k-est.main}
\begin{split}
\sup_{r\in [\f{\mathfrak l_1}{2}, \f{2}{\mathfrak h}]} &\Big| \partial_r^{I}  \int_{-\infty}^\infty\int_{0}^\infty [k \Omg(H,M)]^{\alp} e^{ikQ_T}e^{-ikt\Omega(H,M)}  (\widehat{f}_{\ini})_k(H,M)\ud L\ud w \Big|\\
&\lesssim \brk{k}^{-N+I}\brk{t}^{-N+I} \sum_{i_1+i_2+i_3\leq N}\int_{-\infty}^\infty\int_{0}^\infty\Big|\Big(\widehat{\partial_{Q_T}^{i_1}\partial_H^{i_2}\partial_M^{i_3}f_{\ini}} (H,M)\Big)_k\Big|\ud L\ud w,
\end{split}
\end{equation}
where $H$, $M$ on the left- and right-hand sides of \eqref{eq:lim-k-est.main} are understood as functions of $(t,r,w,L)$.
\end{lemma}
\begin{proof}
First consider $\alp = 0$ and take $I \leq N-1$. Using Leibnitz's rule, we obtain
\begin{align*}
\sup_{r\in [\f{\mathfrak l_1}{2}, \f{2}{\mathfrak h}]}& \Big| \partial_r^{I} \int_{-\infty}^\infty\int_{0}^\infty e^{ikQ_T}e^{-ikt\Omega(H,M)}  (\widehat{f}_{\ini})_k(H,M)\ud L\ud w \Big| \\
&\lesssim \sum_{i_1+i_2+i_3 =I} \Big| \int_{-\infty}^\infty\int_{0}^\infty \partial_r^{i_1}(e^{ikQ_T})\partial_r^{i_2}(e^{-ikt\Omega(H,M)})  \partial_r^{i_3}(\widehat{f}_{\ini})_k(H,M)\ud L\ud w \Big|.
\end{align*}
Note that if $t\leq 1$, then this is bounded above by $\sum_{i_1+i_3 \leq I}\int_{-\infty}^\infty\int_{0}^\infty \brk{k}^{i_1} |(\rd_r^{i_3} \widehat{f}_{\ini})_k|(H,M)\ud L\ud w$.. After expanding $\rd_r = \rd_r H\rd_H$ (since it acts on a $Q_T$-independent function) and bounded $\brk{k}$ in terms of $\rd_{Q_T}$, this is bounded by the right-hand side of \eqref{eq:lim-k-est.main}. Thus, \textbf{from now on, we assume $t \geq 1$}.

Now note that $$\partial_L \Omega(H,M)=\partial_L H (\partial_X\Omega)(H,M)+(\partial_Z \Omega)(H,M).$$
Thanks to Proposition~\ref{prop:precise-der-period}, the fact that $\partial_L H=\frac{1}{2r^2}$ and that $\f 1r$ is uniformly bounded on the support of $\widehat{f}_{\ini}$, there exists a $\kappa' >0$ so that 
$$|\partial_L \Omega(H,M)|>\kappa'\quad \hbox{on $\mathrm{supp}((\widehat{f}_{\ini})_k)$}.$$

We now write 
$$e^{-ikt\Omg(H,M)} = \f{ 1}{\Big(-ikt \rd_L \Omg(H,M)\Big)^{N-I+i_2}} \rd_L^{N-I+i_2} e^{-ikt\Omg(H,M)}$$
and integrate by parts in $L$ for $N-I+i_2$ times. We introduce a notation where $\mathfrak{B}(H,M)$ (resp.~$\mathfrak{B}(Q_T,H,M)$) denotes a smooth function of $(H,M)$ (resp.~$(Q_T,H,M)$) satisfying the estimate $|\mathfrak{B}|\ls 1$. This in particular incorporates all the derivatives of $\Omg(H,M)$ (up to $N+1$ order) as well as $\rd_r Q_T$ and other functions arising in the change of variables, which are controlled using Lemma~\ref{lem:Jac} and Proposition~\ref{prop:precise-der-period}. We thus obtain
\begin{align*}
 &\sum_{i_1+i_2+i_3 =I} \Big| \int_{-\infty}^\infty\int_{0}^\infty \partial_r^{i_1}(e^{ikQ_T})\partial_r^{i_2}(e^{-ikt\Omega(H,M)})  \partial_r^{i_3}(\widehat{f}_{\ini})_k(H,M)\ud L\ud w \Big| \\
 &\lesssim \sum_{\substack{i_1+i_2+i_3 =I \\ j_1+j_2 \leq N-I+i_2}} \Big| \int_{-\infty}^\infty\int_{0}^\infty\frac{\mathfrak{B}(H,M)}{(ikt)^{N-I+i_2}}\partial_L^{j_1}\partial_r^{i_1}(e^{ikQ_T})\partial_r^{i_2}(e^{-ikt\Omega(H,M)})  \partial_L^{j_2}\partial_r^{i_3}(\widehat{f}_{\ini})_k(H,M)\ud L\ud w \Big| \\
 &\lesssim \sum_{\substack{i_1+i_2+i_3 =I \\ j_1+j_2 \leq N-I+i_2}} \frac{|k|^{j_1+i_1}|kt|^{i_2}}{|kt|^{N-I+i_2}} \Big| \int_{-\infty}^\infty\int_{0}^\infty\mathfrak{B}(Q_T,H,M)  \partial_L^{j_2}\partial_r^{i_3}(\widehat{f}_{\ini})_k(H,M)\ud L\ud w \Big| \\
 &\ls \sum_{\substack{ i_1+i_2+i_3 =I \\ j_1+j_2 \leq N-I+i_2} } \frac{|k|^{j_1+i_1}}{|kt|^{N-I}}  \int_{-\infty}^\infty\int_{0}^\infty \Big| \partial_L^{j_2}\partial_r^{i_3}(\widehat{f}_{\ini})_k(H,M) \Big| \ud L\ud w.
\end{align*}

We now write re-express $\rd_r = \rd_r H \rd_H + \rd_r Q_T \rd_{Q_T}$ and $\rd_L = \rd_L H \rd_H + \rd_M$. We also use that $(H,M) \in \mathcal S$ on the support of $(\widehat{f}_{\ini})_k$ so that the domain of integration in $L$ and $w$ has finite volume (see Proposition~\ref{prop:support}). Hence, we obtain 
\begin{align*}
 &\sum_{i_1+i_2+i_3 =I} \Big| \int_{-\infty}^\infty\int_{0}^\infty \partial_r^{i_1}(e^{ikQ_T})\partial_r^{i_2}(e^{-ikt\Omega(H,M)})  \partial_r^{i_3}(\widehat{f}_{\ini})_k(H,M)\ud L\ud w \Big| \\
 &\ls \brk{k}^{-N+I}\brk{t}^{-N+I} \sum_{i_1+i_2+i_3\leq N}\int_{-\infty}^\infty\int_{0}^\infty\Big|\Big(\widehat{\partial_{Q_T}^{i_1}\partial_H^{i_2}\partial_M^{i_3}f_{\ini}} (H,M)\Big)_k\Big|\ud L\ud w \quad \hbox{for $I\leq N-1-\alp$}.
\end{align*}


Finally, notice that the $\alp = 1$ case in \eqref{eq:lim-k-est.main} is similar, except that the extra factor of $k$ restricts the number of derivatives to $I \leq N-2$. \qedhere

\end{proof}

Using Lemma~\ref{lem:lin-k-est}, we prove the following decay estimates for $\calL(t,r) - \calL(T,r)$.
\begin{proposition}\label{prop:linear-est-den.1}
The following estimates hold for all $t \in [0,T]$:
$$\sup_{r\in [\f{\mathfrak l_1}{2}, \f{2}{\mathfrak h}]}|\partial_r^I(\calL(t,r)-\calL(T,r))|\lesssim \delta \epsilon \jap{t}^{-\min\{N-I+2,N\}}, \quad \hbox{whenever $I\leq N+1$}.$$
\end{proposition}
\begin{proof}
Fix $I \leq N+1$. Using the representation \eqref{eq:L-diff}, we note that the first two $\rd_r$ derivatives would either remove the outermost $r_1$, $r_2$ integral or acts on the $\f{1}{rr_1}$ weights. In either case, they would not cause time growth. Thus, we have
\begin{equation*}
\begin{split}
&\: |\partial_r^I(\calL(t,r)-\calL(T,r))| \\
\ls &\: \sum_{I'\leq \max\{0,I-2\}} \sum_{k\in \mathbb Z\setminus \{0\}} \sup_{r\in [\f{\mathfrak l_1}{2}, \f{2}{\mathfrak h}]} \Bigg[ \Big| \partial_r^{I'}  \int_{-\infty}^\infty\int_{0}^\infty  e^{ikQ_T}e^{-ikt\Omega(H(t),M)}  (\widehat{f}_{\ini})_k(H(t),M)\ud L\ud w \Big| \\
&\: \qquad \qquad + \Big| \partial_r^{I'}  \int_{-\infty}^\infty\int_{0}^\infty e^{ikQ_T} e^{-ikT\Omega(H(T),M)}  (\widehat{f}_{\ini})_k(H(T),M)\ud L\ud w \Big| \Bigg] \\
&\quad+ \sum_{I'\leq \max\{0,I-2\}} \sup_{r\in [\f{\mathfrak l_1}{2}, \f{2}{\mathfrak h}]} \int_{-\infty}^\infty\int_{0}^\infty|[(\rd_{H}^{I'}\widehat{f}_{\ini})_0(H(t),M)-(\rd_{H}^{I'}\widehat{f}_{\ini})_0(H(T),M)]|\d L\d w.
\end{split}
\end{equation*}
For the first two terms we use the bound in Lemma~\ref{lem:lin-k-est} to obtain 
\begin{equation*}
\begin{split}
&\: \sum_{I'\leq \max\{0,I-2\}} \sum_{k\in \mathbb Z\setminus \{0\}} \sup_{r\in [\f{\mathfrak l_1}{2}, \f{2}{\mathfrak h}]} \Bigg[ \Big| \partial_r^{I'}  \int_{-\infty}^\infty\int_{0}^\infty  e^{ikQ_T}e^{-ikt\Omega(H(t),M)}  (\widehat{f}_{\ini})_k(H(t),M)\ud L\ud w \Big| \\
&\: \qquad \qquad + \Big| \partial_r^{I'}  \int_{-\infty}^\infty\int_{0}^\infty e^{ikQ_T} e^{-ikT\Omega(H(T),M)}  (\widehat{f}_{\ini})_k(H(T),M)\ud L\ud w \Big| \Bigg] \\
&\qquad\ls \: \sum_{k \in \mathbb Z\setminus \{0\}} \sum_{i_1+i_2+i_3\leq N+1}\int_{-\infty}^\infty\int_{0}^\infty\Big|\Big(\widehat{\partial_{Q_T}^{i_1}\partial_H^{i_2}\partial_M^{i_3}f_{\ini}} (H,M)\Big)_k\Big|\ud L\ud w\\
&\hspace{10em}\times |k|^{-1} \max\Big\{ \brk{t}^{-\min\{N-I+2,N\}}, \brk{T}^{-\min\{N-I+2,N\}} \Big\} \\
&\qquad\ls \: \brk{t}^{-\min\{N-I+2,N\}} \sum_{i_1+i_2+i_3\leq N+1} \int_{-\infty}^\infty\int_{0}^\infty\Bigg( \sum_{k  \in \mathbb Z\setminus \{0\}} \Big|\Big(\widehat{\partial_{Q_T}^{i_1}\partial_H^{i_2}\partial_M^{i_3}f_{\ini}} (H,M)\Big)_k\Big|^2 \Bigg)^{1/2} \ud L\ud w\\
 &\qquad\ls \:  \brk{t}^{-\min\{N-I+2,N\}} \sum_{i_1+i_2+i_3\leq N+1} \sup_{H,M} \Big\|  \partial_{Q_T}^{i_1}\partial_H^{i_2}\partial_M^{i_3}f_{\ini}  \Big\|_{L^2_{Q_T}} \\
 &\qquad \ls\: \brk{t}^{-\min\{N-I+2,N\}} \sum_{i_1+i_2+i_3\leq N+1} \| \partial_{r}^{i_1}\partial_w^{i_2}\partial_L^{i_3} f_{\ini}  \|_{L^\i} \ls \de \ep \brk{t}^{-\min\{N-I+2,N\}},
\end{split}
\end{equation*}
where in the last few steps, we used the Cauchy--Schwarz inequality, the Plancherel theorem, and bound $L^2_{Q_T}$ by $L^\i_{Q_T}$. We also used that $(H,M) \in \mathcal S$ on the support of $(\widehat{f}_{\ini})_k$ so that the domain of integration in $L$ and $w$ has finite volume (see Proposition~\ref{prop:support}), changed back to the original variables (using Lemma~\ref{lem:Jac}) and used the assumption \eqref{eq:main.assumption}. 

Thus, using the bound in Lemma~\ref{lem:lin-k-est}, we obtain 
\begin{equation*}
\begin{split}
&\: |\partial_r^I(\calL(t,r)-\calL(T,r))| \\
\ls &\: \sum_{I'\leq \max\{0,I-2\}} \sum_{k\in \mathbb Z\setminus \{0\}} \sup_{r\in [\f{\mathfrak l_1}{2}, \f{2}{\mathfrak h}]} \Bigg[ \Big| \partial_r^{I'}  \int_{-\infty}^\infty\int_{0}^\infty  e^{ikQ_T}e^{-ikt\Omega(H,M)}  (\widehat{f}_{\ini})_k(H,M)\ud L\ud w \Big| \\
&\: \qquad \qquad + \Big| \partial_r^{I'}  \int_{-\infty}^\infty\int_{0}^\infty e^{ikQ_T} e^{-ikT\Omega(H,M)}  (\widehat{f}_{\ini})_k(H,M)\ud L\ud w \Big| \Bigg] \\
\ls &\: \sum_{k \in \mathbb Z\setminus \{0\}} \sup_{H,M} \sum_{i_1+i_2+i_3\leq N+1}\Big|\Big(\widehat{\partial_{Q_T}^{i_1}\partial_H^{i_2}\partial_M^{i_3}f_{\ini}} (H,M)\Big)_k\Big||k|^{-1} \max\Big\{ \brk{t}^{-\min\{N-I+2,N\}}, \brk{T}^{-\min\{N-I+2,N\}} \Big\} \\
\ls &\: \brk{t}^{-\min\{N-I+2,N\}} \sup_{H,M} \sum_{i_1+i_2+i_3\leq N+1} \Bigg( \sum_{k  \in \mathbb Z\setminus \{0\}} \Big|\Big(\widehat{\partial_{Q_T}^{i_1}\partial_H^{i_2}\partial_M^{i_3}f_{\ini}} (H,M)\Big)_k\Big|^2 \Bigg)^{1/2} \\
 \ls &\:  \brk{t}^{-\min\{N-I+2,N\}} \sum_{i_1+i_2+i_3\leq N+1} \Big\| \| \partial_{Q_T}^{i_1}\partial_H^{i_2}\partial_M^{i_3}f_{\ini}  \|_{L^\i_{H,M}} \Big\|_{L^2_{Q_T}}\\
 \ls &\: \brk{t}^{-\min\{N-I+2,N\}} \sum_{i_1+i_2+i_3\leq N+1} \| \partial_{r}^{i_1}\partial_w^{i_2}\partial_L^{i_3} f_{\ini}  \|_{L^\i} \ls \de \ep \brk{t}^{-\min\{N-I+2,N\}},
\end{split}
\end{equation*}
where in the last few steps, we used the Cauchy--Schwarz inequality, the Plancherel theorem, bound $L^2_{Q_T}$ by $L^\i_{Q_T}$, change back to the original variables (using Lemma~\ref{lem:Jac}) and use the assumption \eqref{eq:main.assumption}. 

Finally, to bound the last term, we use mean value theorem and change coordinates back to $(r,w,L)$ to get the following
\begin{align*}
\sum_{I'\leq \max\{0,I-2\}}\int_{-\infty}^\infty\int_{0}^\infty&|[(\rd_{H}^{I'}\widehat{f}_{\ini})_0(H(t),M)-(\rd_{H}^{I'}\widehat{f}_{\ini})_0(H(T),M)]|\d L\d w\\
&\ls  \sum_{i_1+i_2+i_3\leq N+1} \| \partial_{r}^{i_1}\partial_w^{i_2}\partial_L^{i_3} f_{\ini}  \|_{L^\i}\int_{-\infty}^\infty\int_{0}^\infty|(\varphi(t,r)-\varphi(T,r))|\d L\d w.
\end{align*}

Now this poses no difficulty because we have sufficient decay from $(\varphi(t,r)-\varphi(T,r))$ and can be easily controlled by the required bound. \qedhere
\end{proof}

\begin{proposition}\label{prop:linear-est-den.2}
The following estimates hold for all $t \in [0,T]$:
$$\sup_{r\in [\f{\mathfrak l_1}{2}, \f{2}{\mathfrak h}]}|\partial_r^I \mathcal L_s|\lesssim \delta^{7/4} \epsilon \jap{t}^{-\min\{N-I+2,N\}},\quad I\leq N.$$
\end{proposition}
\begin{proof}
Fix $I \leq N$ and differentiate the expression from \eqref{eq:L-t}. The term 
\begin{equation}\label{eq:data.Ls.similar}
\sum_{\substack{{k\in\Z}\\{k\neq 0}}}\int_0^r\int_{r_2}^\infty \int_{-\infty}^\infty\int_{0}^\infty\frac{ik\Omega(H,M)}{r r_1}e^{ikQ_T}e^{-ikt\Omega(H,M)}  (\widehat{f}_{\ini})_k(H,M)\d L\d w\d r_1\d r_2
\end{equation}
can be treated exactly as in Proposition~\ref{prop:linear-est-den.1}, except for using the $\alp = 1$ (instead of $\alp = 0$) case of Lemma~\ref{lem:lin-k-est}.

We also need to consider the remaining terms, which are easier because $\rd_s \varphi$ already provides sufficient decay and we need not generate extra decay using stationary phase. Handling the higher derivatives of $f_{\ini}$ similarly as in the proof of Proposition~\ref{prop:linear-est-den.1}, we obtain
\begin{equation}\label{eq:data.extra.term.1}
\begin{split}
&\: \sum_{k\in\Z} \sum_{I' \leq \max\{0,I-2\}} \Big| \rd_r^{I'} \int_{-\infty}^\infty\int_{0}^\infty ikt (\rd_X\Omega)(H,M) (\widehat{f}_{\ini})_k(H,M) \rd_s \varphi e^{ikQ_T}e^{-ikt\Omega(H,M)}  \d L\d w \Big| \\
\ls &\: \sum_{k\in\Z} \sum_{i_1+i_2+i_3 \leq \max\{0, I-2\}} \Big|  \int_{-\infty}^\infty\int_{0}^\infty ikt \rd_r^{i_1}[(\rd_X\Omega)(H,M) (\widehat{f}_{\ini})_k(H,M) e^{ikQ_T}] \\
&\: \hspace{18em} \times (\rd_r^{i_2} \rd_s \varphi) (\rd_r^{i_3} e^{-ikt\Omega(H,M)})  \d L\d w \Big| \\
\ls &\: \sum_{k \in \Z} \brk{t} \Big( \sum_{I' \leq \max\{0,I-2\}} \sup_{r \in [\f{\mathfrak l_1}{2}, \f{2}{\mathfrak h}]} |\rd_r^{I'} \rd_s \varphi|(t,r) \Big) \Big( \sum_{i_1+i_2+i_3\leq N+1} \| \partial_{r}^{i_1}\partial_w^{i_2}\partial_L^{i_3} f_{\ini}  \|_{L^\i} \Big)\\
\ls &\: \brk{t} \Big(\de^{3/4} \ep \brk{t}^{-\min\{N-I+4,N\}}\Big) (\de \ep) \ls \de^{7/4} \ep \brk{t}^{-\min\{N-I+4,N\}},
\end{split}
\end{equation}
where we used the bootstrap assumption \eqref{eq:varphi_t} and the initial data assumption \eqref{eq:main.assumption} in the penultimate step, and we used $\ep \brk{t} \ls 1$ in the final step. 

The final term can be handled similarly so that we obtain 
\begin{equation}\label{eq:data.extra.term.2}
\begin{split}
&\: \sum_{k\in\Z} \sum_{I' \leq \max\{0,I-2\}} \Big| \rd_r^{I'} \int_{-\infty}^\infty\int_{0}^\infty  (\rd_H \widehat{f}_{\ini})_k(H,M) \rd_s \varphi e^{ikQ_T}e^{-ikt\Omega(H,M)}  \d L\d w \Big| \\
\ls &\: \de^{7/4} \ep \brk{t}^{-\min\{N-I+5,N+1\}}.
\end{split}
\end{equation}
Note that the $k=0$ is included here, but it does not pose any difficulties since there is sufficient decay from $\rd_s \varphi$. 

Finally, using \eqref{eq:data.extra.term.1}, \eqref{eq:data.extra.term.2}, and treating \eqref{eq:data.Ls.similar} as in Proposition~\ref{prop:linear-est-den.1} (as we mentioned above), we obtain the desired estimates. \qedhere
\end{proof}

\subsection{Setting up the nonlinear estimates} \label{sec:phase.mixing.nonlinear.1}
In this subsection we treat the source term $\mathcal{N}$ in \eqref{eq:phi-rep}. For the reader's convenience, we recall the terms that constitute $\mathcal{N}.$
\begin{equation}
\begin{split}
\mathcal{N}(t,r)&=\int_0^r\int_{r_2}^\infty \int_{-\infty}^\infty\int_{0}^\infty\frac{1}{r r_1}\int_0^t S(t-\tau,Q_T,H,M)[ \rd_s \varphi(\tau,r_1)\\
&\hspace{13em}\times \partial_H f(\tau,Q_T,H,M)]\ud\tau\d L\d w\d r_1\d r_2\\
&\quad+\int_0^r\int_{r_2}^\infty \int_{-\infty}^\infty\int_{0}^\infty\frac{1}{r r_1}\int_0^t S(t-\tau,Q_T,H,M)[(w\rd_{H_T} |_{(s,r,H_T,L)} Q_T)\\
&\hspace{7em}\times  \{\rd_r\varphi(\tau,r_1)-\rd_r\varphi(T,r_1)\}\partial_{Q_T} f(\tau,Q_T,H,M)]\ud\tau\d L\d w\d r_1\d r_2\\
&\quad+\int_0^r\int_{r_2}^\infty \int_{-\infty}^\infty\int_{0}^\infty\frac{1}{r r_1}\int_0^t S(t-\tau,Q_T,H,M)[\{\Omega(H,M)-\Omega(H_T,M)\}\\
&\hspace{17em}\times \partial_{Q_T} f(\tau,Q_T,H,M)]\ud\tau\d L\d w\d r_1\d r_2\\
&=:\mathcal{N}_1(t,r)+\mathcal{N}_2(t,r)+\mathcal{N}_3(t,r).
\end{split}
\end{equation}  
In each of the integral, $Q_T$, $H$ and $M$ inside the integral are understood as functions of $(r_1,w,L)$. Thus, the terms $\mathcal{N}_1$, $\mathcal{N}_2$ and $\mathcal{N}_3$ are indeed functions of $(t,r)$ alone.

It is convenient to write the term $\mathcal{N}_3$ to allow for easy comparison with $\mathcal{N}_2$. Using integral form of Taylor's theorem we get 
$$\Omega(H,M)-\Omega(H_T,M)=\int_0^1(\partial_X \Omega)(H+\sigma(H_T-H))\ud\sigma (H-H_T).$$
Thus using the fact that $H-H_T=\varphi(t,r)-\varphi(T,r)$, we see that
\begin{align*}
\mathcal{N}_3&=\int_0^r\int_{r_2}^\infty \int_{-\infty}^\infty\int_{0}^\infty\frac{1}{r r_1}\int_0^t S(t-\tau,Q_T,H,M)[\{\int_0^1(\partial_X \Omega)(H+\sigma(H_T-H))\ud\sigma\}\\
&\hspace{12em}\times (\varphi(\tau,r)-\varphi(T,r))\partial_{Q_T} f(\tau,Q_T,H,M)]\ud\tau\d L\d w\d r_1\d r_2.
\end{align*}
\begin{lemma}
The nonlinear term $\calN(t,r)$ can be decomposed into the following three terms:
\begin{equation}\label{eq:source-terms-phi}
\calN(t,r) = \sum_{i=1}^3 \calN_i(t,r),
\end{equation}
where the $\calN_i$'s are defined as
$$\calN_i(t,r):= \int_0^r\int_{r_2}^\infty \int_{-\infty}^\infty\int_{0}^\infty\int_0^t S(t-\tau,Q_T,H,M)[ \calR_i(\tau,T,Q_T,H,M)] \frac{\ud\tau\d L\d w \ud r_1 \ud r_2}{r r_1}$$
and the $\calR_i's$ are given by
\begin{align}
\calR_{1}(\tau,T,Q_T,H,M) = &\: \rd_s \varphi(\tau,r_1) \rd_H f(\tau, Q_T,H,M), \label{eq:R1.def} \\
\calR_{2}(\tau,T,Q_T,H,M) = &\: [\rd_r\varphi(\tau,r_1)- \rd_r\varphi(T,r_1)] (\mathfrak p \rd_{Q_T}f)(\tau,Q_T,H,M), \label{eq:R2.def} \\
\calR_{3}(\tau,T,Q_T,H,M) = &\: [\rd_r\varphi(\tau,r_1)- \rd_r\varphi(T,r_1)] (\mathfrak q \rd_{Q_T}f)(\tau,Q_T,H,M), \label{eq:R3.def}
\end{align}
with $r_1 = r_1(Q_T,H,M)$, and $\mathfrak p$ and $\mathfrak q$ are given by
\begin{equation}\label{eq:mfa.mfb.def}
\mathfrak p=w(\rd_{H_T} |_{(s,r,H_T,L)} Q_T),\quad \mathfrak q=\int_0^1(\partial_X \Omega)(H+\sigma(H_T-H))\ud\sigma.
\end{equation}
\end{lemma}

In the next two propositions, we will derive the expressions for $\mathcal{N}_s(t,r)$ and $\mathcal{N}(t,r)-\mathcal{N}(T,r)$, respectively.
\begin{proposition}\label{prop:nonlin-t}
For $\mathcal{N}$ as in \eqref{eq:source-terms-phi}, we have the following 
\begin{align}\label{eq:N-t}
\mathcal{N}_s(t,r)&=\sum_{i=1}^3 \sum_{j=1}^4 \mathcal{T}_{i,j}(t,r),
\end{align} 
where for each fixed $T$, every $\mathcal T_{i,j}$ term is a function of $(t,r)$ alone defined by
\begin{equation}\label{eq:Ti1}
\calT_{i,1} = \int_0^r\int_{r_2}^\infty \int_{-\infty}^\infty\int_{0}^\infty\frac{1}{r r_1} \calR_i(t,T,Q_T,H,M) \d L\d w\d r_1\d r_2,
\end{equation}
\begin{equation}\label{eq:Ti2}
\begin{split}
\mathcal{T}_{i,2}&=\sum_{k\in \Z\backslash\{0\}} \int_0^r\int_{r_2}^\infty \int_{-\infty}^\infty\int_{0}^\infty\int_0^t \frac{-ik\Omega}{r r_1} e^{ikQ_T} e^{-ik\Omega(t-\tau)}\widehat{(\calR_i)}_k(\tau,T,H,M) \ud\tau\d L\d w\d r_1\d r_2,
\end{split}
\end{equation}
\begin{equation}\label{eq:Ti3}
\begin{split}
\mathcal{T}_{i,3}&=\sum_{k\in \Z\backslash\{0\}} \int_0^r\int_{r_2}^\infty \int_{-\infty}^\infty\int_{0}^\infty\int_0^t \frac{-ik(t-\tau) \rd_s \varphi (\rd_X\Omg)(H,M)}{r r_1} \\
&\hspace{12em}\times e^{ikQ_T} e^{-ik\Omega(t-\tau)}\widehat{(\calR_i)}_k(\tau,T,H,M) \ud\tau\d L\d w\d r_1\d r_2,
\end{split}
\end{equation}
\begin{equation}\label{eq:Ti4}
\begin{split}
\mathcal{T}_{i,4}&=\sum_{k\in \Z} \int_0^r\int_{r_2}^\infty \int_{-\infty}^\infty\int_{0}^\infty\int_0^t \frac{\rd_s \varphi}{r r_1} e^{ikQ_T} e^{-ik\Omega(t-\tau)}\rd_H \widehat{(\calR_i)}_k(\tau,T,H,M) \ud\tau\d L\d w\d r_1\d r_2,
\end{split}
\end{equation}
with $\calR_i$ as in \eqref{eq:R1.def}--\eqref{eq:R3.def}.
Here, $\Omega = \Omega(H,M)$, all $Q_T$, $H$, $M$ are viewed as functions of $(r_1,w,L)$, and \hbox{ }${}\widehat{}$ denotes the coefficient of the Fourier series in $Q_T$.

\end{proposition}
\begin{proof}
We differentiate \eqref{eq:source-terms-phi} with respect to $\rd_s$ to get 
$\mathcal{N}_s= \sum_{i=1}^3\partial_s{\mathcal{N}_i}.$

For each term, we compute
\begin{equation*}
\begin{split}
\partial_s{\mathcal{N}_i} = &\: \int_0^r\int_{r_2}^\infty \int_{-\infty}^\infty\int_{0}^\infty \calR_i(\tau,T,Q_T,H,M) \frac{\ud\tau\d L\d w \ud r_1 \ud r_2}{r r_1} \\
&\: + \int_0^r\int_{r_2}^\infty \int_{-\infty}^\infty\int_{0}^\infty\int_0^t \rd_s[S(t-\tau,Q_T,H,M)[ \calR_i(\tau,T,Q_T,H,M)]] \frac{\ud\tau\d L\d w \ud r_1 \ud r_2}{r r_1}.
\end{split}
\end{equation*}
The first term gives rise to $\calT_{i,1}$. For the second term, we write 
\begin{equation}\label{eq:S.R.Fourier}
S(t-\tau,Q_T,H,M) \calR_i = \sum_{k \in \mathbb Z} e^{ikQ_T} e^{-ik\Omg(t-\tau)} \widehat{(\calR_i)}_k(\tau,H,M).
\end{equation} 
When computing the $\rd_s$ derivative of $e^{ikQ_T} e^{-ik\Omg(t-\tau)} \widehat{(\calR_i)}_k(\tau,H,M)$, we have a non-trivial contribution either from the $t$ factor in the exponent, or else from a $\rd_H$ derivative together with the fact that $\rd_s H = \rd_s \varphi$. In other words, we have
\begin{equation*}
\begin{split}
&\: \rd_s [S(t-\tau,Q_T,H,M) \widehat{(\calR_i)}_k(\tau,H,M)] \\
= &\: \sum_{k \in \mathbb Z} e^{ikQ_T} e^{-ik\Omg(t-\tau)} [-ik\Omg(H,M) \widehat{(\calR_i)}_k(\tau,H,M)_k(\tau,H,M) \\
&\: \hspace{9em}-ik(t-\tau) \rd_s \varphi (\rd_X\Omg)(H,M) \widehat{(\calR_i)}_k(\tau,H,M)(\tau,H,M) \\
&\: \hspace{22em}+ \rd_s \varphi \rd_H \widehat{(\calR_i)}_k(\tau,H,M)].
\end{split}
\end{equation*}
Plugging this back into the expression above, we obtain the terms $\calT_{i,2}$, $\calT_{i,3}$ and $\calT_{i,4}$. \qedhere
 \end{proof}
 
 Next we set up the term in $\mathcal{N}(t,r)-\mathcal{N}(T,r)$. 
 \begin{proposition}\label{prop:nonlin-diff}
For $\mathcal{N}$ as in \eqref{eq:source-terms-phi}, we have the following 
\begin{align}\label{eq:N-diff}
\mathcal{N}(t,r)-\mathcal{N}(T,r)&= \sum_{i=1}^3 \mathcal{D}_{i,1}(t,r) - (\sum_{i=1}^3 \mathcal{D}_{i,2}(\mathfrak{t},r)|_{\mathfrak{t}=t}^{\mathfrak{t} = T}),
\end{align} 
where $\calD_{i,j}$ is defined by
\begin{equation}
\calD_{i,1}(t,r) = \int_0^r\int_{r_2}^\infty \int_{-\infty}^\infty\int_{0}^\infty\int_t^T\frac{1}{r r_1}\widehat{(\calR_i)}_{0}(\tau,H,M)\ud\tau\ud L\ud w\ud r_1\ud r_2,
\end{equation}
\begin{equation}\label{eq:Di2}
\begin{split}
\calD_{i,2}(\mathfrak t,r) =& \sum_{k\in \mathbb Z\setminus\{0\}} \int_0^r\int_{r_2}^\infty \int_{-\infty}^\infty\int_{0}^\infty\int_0^{\mathfrak t} e^{ikQ_T} e^{-ik\Omega\cdot (\mathfrak t -\tau)} \widehat{(\calR_i)}_{k}(\tau,H,M)  \f{\ud\tau\d L\d w\d r_1\d r_2}{rr_1},
\end{split}
\end{equation}
with $\calR_i$ as in \eqref{eq:R1.def}--\eqref{eq:R3.def} and \eqref{eq:mfa.mfb.def}.
Here, $\Omega = \Omega(H,M)$, all $Q_T$, $H$, $M$ are viewed as functions of $(r_1,w,L)$, and \hbox{ }${}\widehat{ }\,$ denotes the coefficient of the Fourier series in $Q_T$.
 \end{proposition}
 \begin{proof}
We write $\mathcal{N}(t,r)-\mathcal{N}(T,r) = \sum_{i=1}^3( \mathcal{N}_i(t,r)-\mathcal{N}_i(T,r)).$ Now take the Fourier series in $Q_T$ as in \eqref{eq:S.R.Fourier} and subtract mode by mode. The $k=0$ mode gives $\calD_{i,1}(t,r)$ and the remaining modes give $\mathcal{D}_{i,2}(\mathfrak{t},r)|_{\mathfrak{t}=t}^{\mathfrak{t} = T}$.
\end{proof}
 
 \subsection{Estimating the easier nonlinear terms} \label{sec:phase.mixing.nonlinear.2}
 In this subsection, we treat the easier nonlinear terms $\mathcal{T}_{i,1}$ and $\mathcal{D}_{i,1}$ for $i\in\{1,2,3\}$. See Proposition~\ref{prop:T11} and Proposition~\ref{prop:D11} for the main estimates. As before, we will rely on an argument based on Lemma~\ref{lem:H-in-L-M} to move $\rd_H$ away from $f$ either by integration by parts or by exchanging $\rd_H$ for the better behaved $\rd_{Q_T}$ or $\rd_{M}$ derivatives.

We begin with a general lemma that will be used for the $\mathcal{T}_{i,1}$ terms. Noting as before that the outermost $r_1$, $r_2$ integrals in $\mathcal{T}_{i,1}$ pose no extra difficulty, we drop these integrals in this lemma.
\begin{lemma}\label{lem:T-der}
For $I\leq N-1$, $\mathfrak{C}(Q_T,H,M)$ such that $\sum_{i_1+i_2+i_3\leq N-1}|\partial_{Q_T}^{i_1}\partial_H^{i_2}\partial_M^{i_3} \mathfrak{C}|\lesssim 1$, $\psi(t,T,r)\in\{\partial_s \varphi, (\varphi(t,r)-\varphi(T,r))\}$ and $\partial\in\{\partial_{Q_T},\partial_H\}$, the following estimate holds for all $t \in [0,T]$:
\begin{align*}
&\left|\partial_r^I \int_{-\infty}^\infty\int_{0}^\infty \mathfrak{C}(Q_T,H,M) \psi(t,T,r)\partial f(t,Q_T,H,M) \d L\d w\right|\\
&\lesssim \sum_{i_1+i_2+i_3 \leq I} (\sup_{r\in [\f{\mathfrak l_1}{2}, \f{2}{\mathfrak h}]}|\partial_r^{i_1}\psi|(t,r))\cdot \|\partial_{Q_T}^{i_2}\partial_M^{i_3}(\partial f)\|_{L^\i}(t).
\end{align*}
\end{lemma}
\begin{proof}
We use the same strategy as in Proposition~\ref{prop:den-by-f}, which relies on Lemma~\ref{lem:H-in-L-M}. In other words, we first expand $\rd_r$ using \eqref{eq:r-in-good-things}, and then for every $\rd_L$ that hits on $\rd f$, we integrate by parts $\rd_L$ away. Altogether, we then obtain the desired estimate. \qedhere
\end{proof}

Lemma~\ref{lem:T-der} easily gives the following estimates for all $\mathcal T_{i,1}$
\begin{proposition}\label{prop:T11}
For $i \in \{1,2,3\}$, the following estimate holds for $t \in [0,T]$:
$$\sup_{r\in [\f{\mathfrak l_1}{2}, \f{2}{\mathfrak h}]}|\rd_r^I \mathcal T_{i,1}(t,r)| \ls \de^{7/4} \ep \jap{t}^{-\min\{N-I+2,N\}},\quad I \leq N.$$
\end{proposition}
\begin{proof}
Recall the definition of $\calT_{i,1}$ from \eqref{eq:Ti1} and \eqref{eq:R1.def}--\eqref{eq:R3.def}. Note that (similar to estimates in Section~\ref{sec:phir.bounded.low.order} and Section~\ref{sec:phase.mixing.linear}) the first two $\rd_r$ derivatives would either remove the $r_1$, $r_2$ integrals in $\calT_{i,1}$ and hit on the $r$ (or $r_1$) factors. Either case, when restricted to $r\in [\f{\mathfrak l_1}{2}, \f{2}{\mathfrak h}]$, they do not incur a loss in decay rate.

The proof is then a straightforward application of Lemma~\ref{lem:T-der} after using the bootstrap assumptions \eqref{eq:varphi_t}, \eqref{eq:varphi_diff} and the pointwise estimate for derivatives of $f$ in Theorem~\ref{the:boot-f}. \qedhere
\end{proof}

We now turn to the $\calD_{i,1}$ terms. Again, we begin with a couple of general lemmas.

\begin{lemma}\label{lem:D-der}
For $I\leq N-1$, $\mathfrak{C}(Q_T,H,M)$ such that $\sum_{i_1+i_2+i_3\leq N}|\partial_{Q_T}^{i_1}\partial_H^{i_2}\partial_M^{i_3} \mathfrak{C}|\lesssim 1$, $\psi(t,T,r)\in\{\partial_s \varphi, (\varphi(t,r)-\varphi(T,r))\}$ and $\partial\in\{\partial_{Q_T},\partial_H\}$ we have the following estimate
\begin{align*}
\sum_{\ell_1\in \Z}\sum_{\ell_2\in \Z} |\partial_r^I\int_{-\infty}^\infty\int_{0}^\infty\int_t^T  &\hat{\mathfrak{C}}_{-\ell_1}(H,M) \widehat{\psi(t,T,r)}_{\ell_2}\widehat{\partial f}_{\ell_1-\ell_2}(t,H,M)\ud\tau\d L\d w|\\
&\lesssim\sum_{i_1+i_2\leq I} \int_t^T(\sup_{r\in [\f{\mathfrak l_1}{2}, \f{2}{\mathfrak h}]}|\partial_r^{i_1} \psi|(\tau,r)) \|\partial_M^{i_2}(\partial f)\|_{L^\i}(\tau) \ud\tau.
\end{align*}
\end{lemma}
\begin{proof}
For every fixed $\tau \in [t,T]$,we estimate the integrand in a similar manner as in Lemma~\ref{lem:T-der}. Note that we need one more derivative on $\mathfrak{C}$ than in Lemma~\ref{lem:T-der} because we have passed to the Fourier series and needed to obtain summability.

Again arguing as in Proposition~\ref{prop:den-by-f} using Lemma~\ref{lem:H-in-L-M}, i.e., expanding $\rd_r$ using \eqref{eq:r-in-good-things}, and then integrating by parts $\rd_L$ away from $\rd f$, we obtain
\begin{equation}\label{eq:D.est.1}
\begin{split}
&\sum_{\ell_1,\ell_2\in \Z} \Big| \partial_r^I\int_{-\infty}^\infty\int_{0}^\infty\int_t^T \widehat{\mathfrak{C}}_{-\ell_1}(H,M) \widehat{\psi(t,T,r)}_{\ell_2}\widehat{\partial f}_{\ell_1-\ell_2}(t,H,M)\ud\tau\d L\d w \Big|\\
&\quad\lesssim\sum_{\ell_1,\ell_2\in \Z} \sum_{i_1+i_2+j_1+j_2+j_3 \leq I} \int_{-\infty}^\infty\int_{0}^\infty\int_t^T |\widehat{\partial_H^{i_1}\rd_M^{j_1}\mathfrak{C}}_{\ell_1}\widehat{\partial_H^{i_2} \rd_M^{j_2}\psi}_{\ell_2}\widehat{\partial_M^{j_3}\partial f}_{\ell_1-\ell_2}| \ud\tau\d L\d w.
\end{split}
\end{equation}
(We note that unlike in Lemma~\ref{lem:T-der}, there are no $\rd_{Q_T}$ derivatives since the Fourier coefficients are independent of $Q_T$.)

We now need to handle the $\ell_1$ and $\ell_2$ summation. Using Cauchy--Schwarz and Young's inequality, we have, for each fixed $(\tau,H,M)$,
\begin{equation}\label{eq:D.est.2}
\begin{split}
&\: \sum_{\ell_1,\ell_2\in \Z} |\widehat{\partial_H^{i_1}\rd_M^{j_1}\mathfrak{C}}_{\ell_1}\widehat{\partial_H^{i_2} \rd_M^{j_2}\psi}_{\ell_2}\widehat{\partial_M^{j_3}\partial f}_{\ell_1-\ell_2}| \\
\ls &\: \|\widehat{\partial_H^{i_1}\rd_M^{j_1}\mathfrak{C}}\|_{\ell^1} \|\widehat{\partial_H^{i_2} \rd_M^{j_2}\psi}\|_{\ell^2} \|\widehat{\partial_M^{j_3}\partial f}\|_{\ell^2} \\
\ls &\: (\sum_{\ell \in \mathbb Z} \brk{\ell}^{-2})^{\f 12} (\sum_{\alp\leq 1} \|\widehat{\partial_H^{i_1}\rd_M^{j_1} \rd_{Q_T}^{\alp} \mathfrak{C}}\|_{\ell^2}) \|\widehat{\partial_H^{i_2} \rd_M^{j_2}\psi}\|_{\ell^2} \|\widehat{\partial_M^{j_3}\partial f}\|_{\ell^2} \\
\ls &\: (\sum_{\alp\leq 1} \|\widehat{\partial_H^{i_1}\rd_M^{j_1} \rd_{Q_T}^{\alp} \mathfrak{C}}\|_{\ell^2}) \|\widehat{\partial_H^{i_2} \rd_M^{j_2}\psi}\|_{\ell^2} \|\widehat{\partial_M^{j_3}\partial f}\|_{\ell^2},
\end{split}
\end{equation}
where $\ell^p$ denotes the usual $\ell^p$ norm for the Fourier coefficient, and we have used on $\rd_{Q_T}$ to gain summability. We now plug \eqref{eq:D.est.2} back into \eqref{eq:D.est.1}, use Plancherel's theorem and that $Q_T \in \mathbb R/(2\pi \mathbb Z)$  (so that $L^\i_{Q_T} \hookrightarrow L^2_{Q_T}$)  to obtain
\begin{equation*}\label{eq:D.est.3}
\begin{split}
\hbox{\eqref{eq:D.est.1}} \ls \sum_{\substack{i_1+i_2+i_3+j_1+j_2 \leq I \\ \alp \leq 1}} \int_{-\infty}^\infty\int_{0}^\infty\int_t^T \Big(\sup_{Q_T} |\partial_H^{i_1}\rd_M^{j_1} \rd_{Q_T}^{\alp} \mathfrak{C} \partial_H^{i_2} \rd_M^{j_2}\psi \partial_M^{j_3}\partial f|\Big)(\tau,H,M) \ud\tau\d L\d w.
\end{split}
\end{equation*}
Finally, in order to conclude, we use the following facts: 
\begin{enumerate}
\item By assumption, $\sum_{i_1+j_1+\alp\leq N}|\partial_H^{i_1}\rd_M^{j_1} \rd_{Q_T}^{\alp} \mathfrak{C}|\lesssim 1$. 
\item The bounds on the change of coordinates map (Lemma~\ref{lem:Jac} and Corollary~\ref{cor:der-H-Q}) give $|\partial_H^{i_2} \rd_M^{j_2}\psi| \ls \sum_{i' \leq i_2+j_2} |\rd_r^{i'} \psi|$. 
\item $\supp(f) \subset \mathcal S$ and the estimates in Proposition~\ref{prop:support} imply that the region of the $\d L \d w$ integration is of finite volume and that on the support of $f$, $r \in [\f{\mathfrak l_1}2, \f{2}{\mathfrak h}]$. \qedhere
\end{enumerate}

\end{proof}

For small $I$, we need a more refined estimate for the term involving $\rd_s \varphi$ and $\rd_H f$ that we prove in the following lemma.
\begin{lemma}\label{lem:D-der-special}
For $I\leq 3$, $\mathfrak{C}(Q_T,H,M)$ such that $\sum_{i_1+i_2+i_3\leq N}|\partial_{Q_T}^{i_1}\partial_H^{i_2}\partial_M^{i_3} \mathfrak{C}|\lesssim 1$ we have the following estimate
\begin{align*}
\sum_{\ell_1\in \Z}\sum_{\ell_2\in \Z} |\partial_r^I\int_{-\infty}^\infty\int_{0}^\infty\int_t^T  &\hat{\mathfrak{C}}_{-\ell_1}(H,M) \widehat{\rd_s \varphi(t,T,r)}_{\ell_2}\widehat{\partial_H f}_{\ell_1-\ell_2}(t,H,M)\ud\tau\d L\d w|\\
&\lesssim\sum_{\alpha\leq 1}\sum_{i_1+i_2\leq I} \int_t^T(\sup_{r\in [\f{\mathfrak l_1}{2}, \f{2}{\mathfrak h}]}|\partial_r^{i_1+\alpha} \rd_s \varphi|(\tau,r)) \|\partial_M^{i_2+1-\alpha} f\|_{L^\i}(\tau) \ud\tau.
\end{align*}
\end{lemma}
\begin{proof}
The proof is very similar to that of Lemma~\ref{lem:D-der} except that we expand $\partial_H$ in terms of $\partial_L$ and $\partial_M$ using \eqref{eq:H-in-L-M} and integrate by parts the $\rd_L$ away. Also note that since $I\leq 3$ and $N\geq 8$, we have no more than $N$ derivatives falling on $\mathfrak{C}.$ We leave out the details.
\end{proof}

Using Lemma~\ref{lem:D-der} and Lemma~\ref{lem:D-der-special}, we obtain the desired estimates for the $\calD_{i,1}$ terms.
\begin{proposition}\label{prop:D11}
For $i \in \{1,2,3\}$, the following estimate holds:
$$\sup_{r\in [\f{\mathfrak l_1}{2}, \f{2}{\mathfrak h}]} |\rd_r^I \calD_{i,1}|(t,r) \ls \de^{7/4} \ep \brk{t}^{-\min\{N-I+2, N\}},\quad I \leq N+1.$$
\end{proposition}
\begin{proof}
After taking into account the first two derivatives that do not matter (similar to Proposition~\ref{prop:T11}), it suffices to bound the derivatives of
$$\int_{-\infty}^\infty\int_{0}^\infty\int_t^T \widehat{(\calR_i)}_{0}(\tau,H,M)\ud\tau\ud L\ud w.$$

Now each of the $(\widehat{\calR_i})_0$ can be written as $\sum_{\ell_1,\ell_2\in \mathbb Z} \widehat{\mathfrak{C}}_{-\ell_1}(H,M) \widehat{\psi(t,T,r)}_{\ell_2}\widehat{\partial f}_{\ell_1-\ell_2}(t,H,M)$, with $\mathfrak{C} = 1,\mathfrak{p},\mathfrak{q}$ (thus satisfying $\sum_{i_1+i_2+i_3\leq N}|\partial_{Q_T}^{i_1}\partial_H^{i_2}\partial_M^{i_3} \mathfrak{C}|\lesssim 1$), $\psi(\tau,r) \in \{\rd_s \varphi(\tau,r), \varphi(t,r)-\varphi(T,r)$ and $\rd \in \{\rd_H,\rd_Q\}$. We first focus on $(\widehat{\calR_1})_0.$ We split it into two cases as follows.

\emph{Case 1: $I\geq 4$.} In this case, we use the bound in Lemma~\ref{lem:D-der}. For the term $(\widehat{\calR_1})_0$, we use \eqref{eq:varphi_t} and Theorem~\ref{the:boot-f} to obtain
\begin{equation*}
\begin{split}
&\: \Big| \rd_r^{\max\{I-2,0\}} \int_{-\infty}^\infty\int_{0}^\infty\int_t^T \widehat{(\calR_1)}_{0}(\tau,H,M)\ud\tau\ud L\ud w \Big| \\
\ls &\: \sum_{i_1+i_2\leq \max\{I-2,0\}} \int_t^T(\sup_{r\in [\f{\mathfrak l_1}{2}, \f{2}{\mathfrak h}]}|\partial_r^{i_1} \rd_s \varphi|(\tau,r)) \|\partial_M^{i_2}(\partial_H f)\|_{L^\i}(\tau) \ud \tau \\
\ls &\: \sum_{i_1+i_2\leq \max\{I-2,0\}} \int_t^T (\de \ep \brk{\tau}^{-\min\{N-i_1+2,N\}})(\de^{3/4}\ep \brk{\tau}) \ud \tau \\
\ls &\: \max\{\de^{7/4} \ep^2 \brk{t}^{-N+I-2}, \de^{7/4} \ep^2 \brk{t}^{-N+2} \} \ls \de^{7/4} \ep \max\{\brk{t}^{-N+I-3},  \brk{t}^{-N+1} \}\\
\ls &\:\de^{7/4} \ep \brk{t}^{-N+I-2},
\end{split}
\end{equation*}
where in the second last inequality, we used $\brk{t} \ls \ep^{-1}$ and that $I\geq 4$, in the last inequality. Thus, recalling the argument in the beginning about first two $\rd_r$ derivatives, we obtain the desired bound.

\emph{Case 2: $I\leq 3$.} In this case, we use Lemma~\ref{lem:D-der-special}. 
\begin{equation*}
\begin{split}
&\: \Big| \rd_r^{\max\{I-2,0\}} \int_{-\infty}^\infty\int_{0}^\infty\int_t^T \widehat{(\calR_1)}_{0}(\tau,H,M)\ud\tau\ud L\ud w \Big| \\
\ls &\: \sum_{\alpha\leq 1}\sum_{i_1+i_2\leq \max\{I-2,0\}} \int_t^T(\sup_{r\in [\f{\mathfrak l_1}{2}, \f{2}{\mathfrak h}]}|\partial_r^{i_1+\alpha} \rd_s \varphi|(\tau,r)) \|\partial_M^{i_2+1-\alpha}f\|_{L^\i}(\tau) \ud \tau \\
\ls &\: \sum_{\alpha\leq 1}\sum_{i_1+i_2\leq \max\{I-2,0\}} \int_t^T (\de \ep \brk{\tau}^{-\min\{N-i_1-\alpha+2,N\}})(\de^{3/4}\ep) \ud \tau \\
\ls &\:\de^{7/4} \ep^2 \brk{t}^{-N+1} \ls \de^{7/4} \ep \brk{t}^{-N},
\end{split}
\end{equation*}
where in the last inequality we used the fact that $\brk{t} \ls \ep^{-1}$ and in the penultimate inequality, we used $i_1+\alpha-2\leq \max\{I-2,0\}+\alpha-2\leq 0,$ which holds since $I\leq 3$.

For $\calR_2$ and $\calR_3$, the analysis in Case 1 for the $\calR_1$ term is already sufficient. Using Corollary~\ref{cor:der-H-Q} and Proposition~\ref{prop:precise-der-period} to control $\mathfrak p$ and $\mathfrak q$, using \eqref{eq:varphi_diff} instead of \eqref{eq:varphi_t}, and using Theorem~\ref{the:boot-f} (noting that the estimates for $\partial_M^{i_2}(\partial_Q f)$ are better), we obtain for $i = 2,3$ that
\begin{equation*}
\begin{split}
&\: \Big| \rd_r^{\max\{I-2,0\}} \int_{-\infty}^\infty\int_{0}^\infty\int_t^T \widehat{(\calR_i)}_{0}(\tau,H,M)\ud\tau\ud L\ud w \Big| \\
\ls &\: \sum_{i_1+i_2\leq \max\{I-2,0\}} \int_t^T(\sup_{r\in [\f{\mathfrak l_1}{2}, \f{2}{\mathfrak h}]}|\partial_r^{i_1} (\rd_r \varphi (\tau,r) - \rd_r \varphi(T,r)) | \|\partial_M^{i_2}(\partial_{Q_T} f)\|_{L^\i}(\tau) \ud \tau \\
\ls &\: \sum_{i_1+i_2\leq \max\{I-2,0\}} \int_t^T (\de \ep \brk{\tau}^{-\min\{N-i_1+1,N\}})(\de^{3/4}\ep ) \ud \tau \\
\ls &\: \max\{\de^{7/4} \ep^2 \brk{t}^{-N+I-2}, \de^{7/4} \ep^2 \brk{t}^{-N+1} \} \ls \de^{7/4} \ep \max\{\brk{t}^{-N+I-3},  \brk{t}^{-N} \}\\
\ls &\:\de^{7/4} \ep \max\{\brk{t}^{-N+I-2},  \brk{t}^{-N} \}.
\end{split}
\end{equation*}
We then conclude the argument as in the $\calR_1$ case. \qedhere
\end{proof}

\subsection{Estimating the main nonlinear terms}\label{sec:phase.mixing.nonlinear.3}

We now turn to the main nonlinear terms $\calT_{i,2}$ and $\calD_{i,2}$.  For these terms, we continue to use Lemma~\ref{lem:H-in-L-M} to integrate by parts the $\rd_H$ away from $f$. In order to obtain sufficient decay, we need an additional idea to understand the resonance in order to obtain enough decay. As we mentioned in the introduction, we carry this out using the commuting vector field method as in \cite{VPL}.

In order to carry out such a scheme, we begin with a lemma concerning the relevant commuting vector field. 

\begin{lemma}\label{lem:pointwise-linear-den}
Consider the fixed mode linear transport equation
\begin{equation}\label{eq:lin-fou}
\begin{split}
\partial_t h+ik\Omega(H,M) h&=0\\
h|_{t=\tau}&=h_\tau,
\end{split}
\end{equation}
where $h_\tau$ is a function of $H,M$ alone.

For any $\eta \in \mathbb R$, $Y_{k,\eta\tau}:=i(kt+\eta\tau)(\partial_X\Omega)(H,M)+\partial_H$ commutes with \eqref{eq:lin-fou} and we have the following estimate for $i_1+i_2\leq N$:
$$|Y_{k,\eta\tau}^{i_1}\partial_M^{i_2} h|(t,H,M)\ls \sum_{i_2'\leq i_2}|Y_{k,\eta\tau}^{i_1}\partial_M^{i_2'} h_\tau|(H,M).$$
\end{lemma}
\begin{proof}
It is a straightforward computation to check that $Y_{k,\eta}$ commutes with \eqref{eq:lin-fou}. Further, $\partial_M$ generates a commutator term of the form $ik(\partial_Z\Omega)(H,M)h$. Since up to $N-1$ derivatives of $(\partial_Z\Omega)(H,M)$ are of size $\ls \epsilon$ (by \eqref{eq:precise-der-period.1}), these terms can be integrated up in the timescale with $T \leq \ep^{-1} (\log\f1{\ep})^{-1}$ in an induction argument (in the same way as the proof of Theorem~\ref{the:boot-f}). \qedhere
\end{proof}

Next we give a general lemma to treat the main terms $\mathcal{T}_{i,2}$ and $\mathcal{D}_{i,2}$.

\begin{lemma}\label{lem:main-terms-den}
For $\alpha\in\{0,1\}$ and $I\leq N-1-\alpha$, $\psi(t,r)$ a smooth function and $\partial\in\{\partial_{Q_T},\partial_H\}$, the following estimate holds for any $t_1,t_2\in \mathbb R$ with $t_1 < t_2 \leq T$ and for any $\ell,k \in \mathbb Z$:
\begin{align*}
&\Big| \partial_r^{I}\int_{-\infty}^\infty\int_{0}^\infty (k\Omega(H,M))^{\alpha}\int_{t_1}^{t_2} e^{ikQ_T}\widehat{\psi(\tau,r)}_{\ell}S_k(t-\tau)[\widehat{\partial f}_{k-\ell}(\tau,H,M)]\ud\tau\d L\d w \Big| \\
\ls &\: \sum_{\substack{i_1+i_2'+i_2''+i_3'\leq I}} \int_{t_1}^{t_2} \min \Big\{ \sum_{i_1'+i_1'' = i_1+\alp} \mathfrak F^{(k,\ell)}_{i_1',i_1'',i_2',i_2'',i_3'}(\tau) , \sum_{i_1'+i_1'' = i_1} |k|^{\alp} \mathfrak F^{(k,\ell)}_{i_1',i_1'',i_2',i_2'',i_3'}(\tau) \Big\}\, \ud \tau.
\end{align*}
where $\mathfrak F^{(k,\ell)}_{i_1',i_1'',i_2',i_2'',i_3'}$ will be defined in \eqref{eq:frkF}. Moreover, for any $\eta_1,\,\eta_2 \in \mathbb R$ and for any $N_1,N_2 \in (\mathbb N \cup \{0\})^2$, $\mathfrak F^{(k,\ell)}_{i_1',i_1'',i_2',i_2'',i_3'}$ satisfies the estimate
\begin{align*}
&\: \mathfrak F^{(k,\ell)}_{i_1',i_1'',i_2',i_2'',i_3'}(\tau)\\
&\quad\lesssim\sum_{\substack{j_2'' \leq i_2'' \\ N_1''\leq N_1,\, N_2'' \leq N_2 \\ K_1'+K_2'+K_3'+K_1''+K_2'' \leq N_1+N_2-N_1''-N_2''}}  {\mathcal{C}}_{N_1,N_2}  \sup_{(H,M)\in\mathcal S} |\widehat{(\rd_{Q_T}^{i_1'+K_1'}\partial_M^{i_2'+K_2'}\partial_H^{i_3'+K_3'}\psi)}_{\ell}|(\tau)\\
&\hspace{15em}\times|Y_{k,\eta_1\tau}^{N_1''} Y_{k,\eta_2\tau}^{N_2''}  \rd_{Q_T}^{i_1''+K_1''} \partial_M^{j_2''+K_2''}[\widehat{\partial f}_{k-\ell}(\tau,H,M)]| ,
\end{align*}
where ${\mathcal{C}}_{N_1,N_2}=\jap{kt+\eta_1\tau}^{-N_1}\jap{kt+\eta_2\tau}^{-N_2}.$

Here, the implicit constant is independent of $k$, $\ell$, $t_1$, $t_2$, $\eta_1$ and $\eta_2$.
\end{lemma}

Before turning to the proof, we first make a few remarks about the application of Lemma~\ref{lem:main-terms-den}.
\begin{remark}
The specific choice of $N_1$ and $N_2$, and the corresponding $\eta_1$ and $\eta_2$ will be made depending on $i_1'$, $i_2'$, $i_3'$, $i_1''$, $i_2''$ and $I$ when Lemma~\ref{lem:main-terms-den} is applied; see  the proof of Proposition~\ref{prop:T12}. It will in particular be important to choose $N_1$ and $N_2$ so as not to incur a loss of derivatives.
\end{remark}

\begin{remark}
The proof allows for having a larger number of $Y$'s with different $\eta$'s, i.e., with $Y_{k,\eta_1\tau}^{N_1}Y_{k,\eta_2\tau}^{N_2}Y_{k,\eta_3\tau}^{M_3}$ etc., but we stated the estimate as above because in later applications, we will use at most two different $\eta$'s.
\end{remark}

\begin{remark}\label{rmk:also.diff.psi}
In the application of Lemma~\ref{lem:main-terms-den}, it is important to note that when we extract $\jap{kt+\eta_1\tau}^{-N_1}$ or $\jap{kt+\eta_2\tau}^{-N_2}$ decay for $\max\{N_1,N_2\} \geq 1$, we need to differentiate both $\psi$ and $f$. This is reflected in the fact that there are $i_3'+N_1'+N_2'$ $\rd_H$ derivatives acting on $\psi$. This is different from the case of Vlasov--Poisson on the torus, but will become important when we estimate the error terms.
\end{remark}

\begin{proof}[Proof of Lemma~\ref{lem:main-terms-den}]
As before, we first introduce a schematic notation: we will use $\mathfrak{B}$ to catch all the functions that are generated due to the change of coordinates map and its derivatives. It will be understood that the function $\mathfrak{B}$ depends on $(\tau,Q_T,H,M)$ as well as the indices $i_1$, $i_2$, $i_3$, etc.~in the sums.  We may use $\mathfrak{B}$ to denote a different function from line to line as long as it satisfies the important property that
\begin{equation}
\sup_{\tau,Q_T,H,M}|\mathfrak{B}(\tau,Q_T,H,M)|\ls 1
\end{equation}
In particular, we absorb $\Omega(H,M)$ (and up to $I$ derivatives of it) into $\mathfrak{B}$ using Proposition~\ref{prop:precise-der-period}. 

Slightly abusing notation, we will denote its derivatives by $\mathfrak{B}$ itself. Here, it is important to count the total number of derivatives on the change of variables map. Since $I \leq  N-1$, there are at most $N-2$ derivatives on the change of variables map, which can indeed be controlled pointwise using Lemma~\ref{lem:Jac}.

We use Leibnitz rule to distribute the $\partial_r^I$ to get 
\begin{align*}
&\Big| \partial_r^{I}\int_{-\infty}^\infty\int_{0}^\infty(k\Omega(H,M))^{\alpha}\int_{t_1}^{t_2} e^{ikQ_T}\widehat{\psi(\tau,r)}_{\ell}S_k(t-\tau)[\widehat{\partial f}_{k-\ell}(\tau,H,M)]\ud\tau\d L\d w \Big|\\
&\qquad\lesssim \sum_{i_1+i_2+i_3\leq I} |k|^{\alpha}\cdot \Big|\int_{-\infty}^\infty\int_{0}^\infty \int_{t_1}^{t_2} \mathfrak{B} \partial_r^{i_1}(e^{ikQ_T})\partial_r^{i_2}\widehat{\psi(\tau,r)}_{\ell} \\
&\hspace{15em}\times \partial_r^{i_3}\left(S_k(t-\tau)[\widehat{\partial f}_{k-\ell}(\tau,H,M)]\right)\ud\tau\d L\d w \Big|\\
&\qquad\lesssim \sum_{i_1+i_2+i_3\leq I} |k|^{\alpha+i_1}\cdot \Big|\int_{-\infty}^\infty\int_{0}^\infty \int_{t_1}^{t_2} \mathfrak{B} e^{ikQ_T}\partial_r^{i_2}\widehat{\psi(\tau,r)}_{\ell} \\
&\hspace{15em}\times \partial_r^{i_3}\left(S_k(t-\tau)[\widehat{\partial f}_{k-\ell}(\tau,H,M)]\right)\ud\tau\d L\d w \Big|.
\end{align*}
Next, we rewrite $\rd_r^{i_2}$ and $\rd_r^{i_3}$ using $\partial_r=\partial_r H\partial_H+\partial_rQ_T \partial_{Q_T}$. Since $\widehat{\psi(\tau,r)}_{\ell}$ and $S_k(t-\tau)[\widehat{\partial f}_{k-\ell}(\tau,H,M)]$ are independent of $Q_T$. After absorbing $\rd_r H$ and its derivatives into $\mathfrak{B}$, it follows that
\begin{align*}
 &\sum_{i_1+i_2+i_3\leq I}|k|^{\alpha+i_1}\cdot \Big| \int_{-\infty}^\infty \int_{0}^\infty \int_{t_1}^{t_2} \mathfrak{B} e^{ikQ_T}\partial_r^{i_2}\widehat{\psi(\tau,r)}_{\ell} \\
&\hspace{13em}\times \partial_r^{i_3}\left(S_k(t-\tau)[\widehat{\partial f}_{k-\ell}(\tau,H,M)]\right)\ud\tau\d L\d w \Big|\\
&\lesssim  \sum_{i_1+i_2+i_3\leq I} |k|^{\alpha+i_1}\cdot \Big| \int_{-\infty}^\infty\int_{0}^\infty \int_{t_1}^{t_2} \mathfrak{B} e^{ikQ_T}\widehat{\partial_H^{i_2}\psi}_{\ell}(\tau,H,M) \\
&\hspace{13em}\times \partial_H^{i_3}\left(S_k(t-\tau)[\widehat{\partial f}_{k-\ell}(\tau,H,M)]\right)\ud\tau\d L\d w \Big|.
\end{align*}

But $\partial_H$ hitting $S_k(t-\tau)[\widehat{\partial f}_{k-\ell}(\tau,H,M)]$ can cause uncontrollable $t-\tau$ growth. To deal with that, we again use \eqref{eq:H-in-L-M}:
\begin{equation}
\partial_H=(\partial_L H)^{-1}(\partial_L-\partial_LQ_T\partial_{Q_T}-\partial_M),
\end{equation}   
integrate by parts in $L$ and transform $\partial_L$ back in $(t,Q_T,H,M)$ coordinates. Proceeding in this way we get,
\begin{align*}
& \sum_{i_1+i_2+i_3\leq I}|k|^{\alpha+i_1}\cdot \Big| \int_{-\infty}^\infty\int_{0}^\infty \int_{t_1}^{t_2}\mathfrak{B} e^{ikQ_T}\widehat{\partial_H^{i_2}\psi}_{\ell}(\tau,H,M) \\
&\hspace{13em}\times \partial_H^{i_3}\left(S_k(t-\tau)[\widehat{\partial f}_{k-\ell}(\tau,H,M)]\right)\ud\tau\d L\d w \Big|\\
&\lesssim  \sum_{\substack{i_1+i_2+i_3\leq I \\ j_1+j_2+j_3=i_3}} |k|^{\alpha+i_1}\cdot \Big| \int_{-\infty}^\infty\int_{0}^\infty\int_{t_1}^{t_2}\mathfrak{B}  \partial_L^{j_1}(e^{ikQ_T})\partial_L^{j_2}\widehat{\partial_H^{i_2}\psi}_{\ell}(\tau,H,M) \\
&\hspace{15em}\times \partial_M^{j_3}\left(S_k(t-\tau)[\widehat{\partial f}_{k-\ell}(\tau,H,M)]\right)\ud\tau\d L\d w \Big|\\
&\lesssim  \sum_{i_1+i_2+i_3+j_1+j_2'+j_2''+j_3\leq I} |k|^{\alpha+i_1+j_1}\cdot \Big| \int_{-\infty}^\infty\int_{0}^\infty\int_{t_1}^{t_2} \mathfrak{B}  e^{ikQ_T}\widehat{(\partial_H^{j_2'+i_2}\partial_M^{j_2''}\psi)}_{\ell}(\tau,H,M) \\
&\hspace{17em}\times \partial_M^{j_3}\left(S_k(t-\tau)[\widehat{\partial f}_{k-\ell}(\tau,H,M)]\right)\ud\tau\d L\d w \Big| \\
&\lesssim  \sum_{i_1+i_2'+i_2''+i_3\leq I} |k|^{\alpha+i_1}\cdot \Big| \int_{-\infty}^\infty\int_{0}^\infty\int_{t_1}^{t_2} \mathfrak{B}  e^{ikQ_T}\widehat{(\partial_M^{i_2'}\partial_H^{i_3}\psi)}_{\ell}(\tau,H,M) \\
&\hspace{17em}\times \partial_M^{i_2''}\left(S_k(t-\tau)[\widehat{\partial f}_{k-\ell}(\tau,H,M)]\right)\ud\tau\d L\d w \Big|,
\end{align*}
where in the last line we have simply relabelled the indices.

Finally, we use 
\begin{equation}\label{eq:k.split}
|k|^{\alp+i_1}\lesssim \min\{ \sum_{i_1'+i_1''=i_1+\alp}|\ell|^{i_1'}|k-\ell|^{i_1''}, |k|^{\alp} \sum_{i_1'+i_1''=i_1}|\ell|^{i_1'}|k-\ell|^{i_1''} \}
\end{equation}
to bound the last expression by 
$$\sum_{\substack{i_1+i_2'+i_2''+i_3'\leq I}} \int_{t_1}^{t_2} \min \Big\{ \sum_{i_1'+i_1'' = i_1+\alp} \mathfrak F^{(k,\ell)}_{i_1',i_1'',i_2',i_2'',i_3'}(\tau) , \sum_{i_1'+i_1'' = i_1} |k|^{\alp} \mathfrak F^{(k,\ell)}_{i_1',i_1'',i_2',i_2'',i_3'}(\tau) \Big\}\, \ud \tau,$$
where we have rewritten the Fourier multiplier as $\partial_{Q_T}$ derivatives and relabelled the indices so that $\mathfrak F^{(k,\ell)}_{i_1',i_1'',i_2',i_2'',i_3'}(\tau)$ takes the general form 
\begin{equation}\label{eq:frkF}
\begin{split}
&\: \mathfrak F^{(k,\ell)}_{i_1',i_1'',i_2',i_2'',i_3'}(\tau) \\
:= &\: \Big| \int_{-\infty}^\infty\int_{0}^\infty \mathfrak{B}  e^{ikQ_T}\widehat{(\rd_{Q_T}^{i_1'}\partial_M^{i_2'}\partial_H^{i_3'}\psi)}_{\ell}(\tau,H,M) \\
&\hspace{14em}\times \rd_{Q_T}^{i_1''} \partial_M^{i_2''}\left(S_k(t-\tau)[\widehat{\partial f}_{k-\ell}(\tau,H,M)]\right) \d L\d w \Big|.
\end{split}
\end{equation}
This finishes the first part of the lemma. 

We now begin the second part of the lemma and derive the estimates for $\mathfrak F^{(k,\ell)}_{i_1',i_1'',i_2',i_2'',i_3}$. Fix $N_1, N_2\in (\mathbb N \cup \{0\})^2$. We generate time decay using $Y_{k,\eta_1\tau}$ and $Y_{k,\eta_2\tau}$ as follows. For $Y_{k,\eta_1\tau}$, we write
\begin{align*}
1&=\Bigg[ \frac{i(kt+\eta_1\tau)(\partial_X\Omega)(H,M)}{i(kt+\eta_1\tau)(\partial_X\Omega)(H,M)} \Bigg]^{N_1} =  \frac{[Y_{k,\eta_1\tau}-\partial_H]^{N_1}}{i^{N_1}(kt+\eta_1\tau)^{N_1}[(\partial_X\Omega)(H,M)]^{N_1}}
\end{align*}
Since by Proposition~\ref{prop:precise-der-period}, $|(\partial_X \Omega)(H,M)|$ is strictly bounded from below, we can safely absorb $[(\partial_X\Omega)(H,M)]^{-N_1}$ in $\mathfrak{B}.$ After expanding out 
$$[Y_{k,\eta_1\tau}-\partial_H]^{N_1} = \sum_{\mathfrak N_1=0}^{N_1} {N_1 \choose \mathfrak N_1} Y_{k,\eta_1\tau}^{N_1-\mathfrak N_1} (-1)^{\mathfrak N_1} \partial_H^{\mathfrak N_1},$$ 
we rewrite $\partial_H^{\mathfrak N_1}$ using \eqref{eq:H-in-L-M} and integrate by parts away all factors of $\rd_L$. 

Similarly, we can use the same idea with $Y_{k,\eta_2\tau}$ applied $N_2$ times. Altogether, recalling also that $\mathfrak{B}$ denotes a bounded function, we thus obtain
\begin{align*}
&\: \mathfrak F^{(k,\ell)}_{i_1',i_1'',i_2',i_2'',i_3'}(\tau) \\
:= &\: \Big| \int_{-\infty}^\infty\int_{0}^\infty \mathfrak{B}  e^{ikQ_T}\widehat{(\rd_{Q_T}^{i_1'}\partial_M^{i_2'}\partial_H^{i_3'}\psi)}_{\ell}(\tau,H,M) \\
&\hspace{14em}\times \rd_{Q_T}^{i_1''} \partial_M^{i_2''}\left(S_k(t-\tau)[\widehat{\partial f}_{k-\ell}(\tau,H,M)]\right) \d L\d w \Big| \\
&\lesssim \sum_{\substack{N_1'\leq N_1,\, N_2' \leq N_2 \\ K_1'+K_2'+K_3'+K_1''+K_2'' \leq N_1+N_2-N_1'-N_2'}} \int_{-\infty}^\infty\int_{0}^\infty \jap{kt+\eta_1\tau}^{-N_1} \jap{kt+\eta_2\tau}^{-N_2} |\widehat{(\rd_{Q_T}^{i_1'+K_1'}\partial_M^{i_2'+K_2'}\partial_H^{i_3'+K_3'}\psi)}_{\ell}|\\
&\hspace{12em}\times|Y_{k,\eta_1\tau}^{N_1'} Y_{k,\eta_2\tau}^{N_2'}  \rd_{Q_T}^{i_1''+K_1''} \partial_M^{i_2''+K_2''}\left(S_k(t-\tau)[\widehat{\partial f}_{k-\ell}(\tau,H,M)]\right)| \d L\d w.
\end{align*}
Finally, we use Lemma~\ref{lem:pointwise-linear-den} to control the derivatives of $S_k(t-\tau)[\widehat{\partial f}_{k-\ell}(\tau,H,M)]$ and use the support properties of $f$ to bound the integral by the supremum. After relabelling $(N_1',N_2')$ to $(N_1'',N_2'')$, we obtain the desired bound.\qedhere

\end{proof}

\begin{lemma}\label{lem:Plancherel}
\begin{enumerate}
\item For any $n_1, n_2 \in \mathbb N \cup \{0\}$,
\begin{equation}\label{eq:Planchrel.1}
\begin{split}
&\: \sum_{k \in \mathbb Z\setminus\{0\}} \sum_{\ell \in \mathbb Z} \jap{\log\jap{\ell}}^{n_1} \jap{\log \jap{k-\ell}}^{n_2} \sup_{(H,M)\in \mathcal S} |\widehat{h}^{(1)}_{\ell}| |\widehat{h}^{(2)}_{k-\ell}|(H,M) \\
\ls &\: \sum_{\substack{a\leq 1 \\ b\leq 1}}\sup_{\substack{Q_T,H,M \\ (H,M)\in \mathcal S}} |\rd_{Q_T}^{a} h^{(1)}| |\rd_{Q_T}^b h^{(2)}| (Q_T,H,M).
\end{split}
\end{equation}
\item Suppose $\psi(\tau,r) \in \{ \rd_s \varphi(\tau,r), \varphi(T,r) - \varphi(\tau,r)\}$. Assume also that $i_1'+i_2'+i_3' \leq N$. Then for any $n_1, n_2 \in \mathbb N \cup \{0\}$,
\begin{equation}
\begin{split}
&\: \sum_{k \in \mathbb Z\setminus\{0\}} \sum_{\ell \in \mathbb Z} \jap{\log\jap{\ell}}^{n_1} \jap{\log \jap{k-\ell}}^{n_2} |(\widehat{\rd_{Q_T}^{i_1'} \rd_M^{i_2'} \rd_H^{i_3'} \psi})_\ell| |(\widehat{\rd_{Q_T}^{i_1''} \rd_M^{i_2''} \rd_H^{i_3''}} f)_{k-\ell} |(\tau,H,M) \\
\ls &\: \Big( \sum_{I' \leq i_1'+i_2'+i_3' +1} \sup_{r\in [\f{\mathfrak l_1}{2}, \f{2}{\mathfrak h}]} |\rd_r^{I'} \psi|(\tau,r) \Big)\Big( \sum_{j \leq i_1''+1} \|\rd_{Q_T}^{j} \rd_M^{i_2''} \rd_H^{i_3''} f \|_{L^\i}(\tau) \Big).
\end{split}
\end{equation}
\end{enumerate}
\end{lemma}
\begin{proof}
Part (1) follows from a combination of Young's inequality, the Cauchy--Schwarz inequaltiy and Plancherel's theorem. More precisely, we have 
\begin{equation}\label{eq:very.simple.Plancherel}
\begin{split}
&\: \hbox{LHS of \eqref{eq:Planchrel.1}} \\
\ls &\:\sup_{(H,M)\in \calS} \Big(\sum_{\ell\in \mathbb Z} \jap{\log\jap{\ell}}^{n_1} |\widehat{h}^{(1)}_{\ell}|(H,M)\Big) \Big(\sum_{k \in \mathbb Z} \jap{\log\jap{k}}^{n_2}|\widehat{h}^{(2)}_{k}|(H,M)\Big) \\
\ls &\: \sup_{(H,M)\in \calS} \Big(\sum_{\ell\in \mathbb Z} \brk{\ell}^2 |\widehat{h}^{(1)}_{\ell}|^2(H,M)\Big)^{1/2} \Big(\sum_{k \in \mathbb Z} \brk{k}^2 |\widehat{h}^{(2)}_{k}|^2(H,M)\Big)^{1/2} \\
&\: \qquad \times \Big(\sum_{\ell\in \mathbb Z} \brk{\ell}^{-2} \jap{\log\jap{\ell}}^{2n_1} \Big)^{1/2} \Big(\sum_{k \in \mathbb Z} \brk{k}^{-2} \jap{\log\jap{k}}^{2n_2} \Big)^{1/2}\\
\ls &\: \sup_{(H,M)\in \calS} \sum_{\substack{a\leq 1 \\ b\leq 1}} \|\rd_{Q_T}^{a}\widehat{h}^{(1)}_{\ell}\|_{L^2_{Q_T}}(H,M) \|\rd_{Q_T}^{b}\widehat{h}^{(2)}_{\ell}\|_{L^2_{Q_T}}(H,M) \\
\ls &\: \sum_{\substack{a\leq 1 \\ b\leq 1}}\sup_{\substack{Q_T,H,M \\ (H,M)\in \mathcal S}} |\rd_{Q_T}^{a} h^{(1)}| |\rd_{Q_T}^b h^{(2)}| (Q_T,H,M),
\end{split}
\end{equation} 
where we observed that $\sum_{\ell\in \mathbb Z} \brk{\ell}^{-2} \jap{\log\jap{\ell}}^{2n_1}\ls 1$ and $\sum_{k \in \mathbb Z} \brk{k}^{-2} \jap{\log\jap{k}}^{2n_2} \ls 1$.

Part (2) is then an application of part (1) after noting that
\begin{itemize}
\item $\widehat{f}_k$ is supported in $(H,M) \in \calS$ (Lemma~\ref{lem:support}),
\item $r(Q_T,H,M) \in [\f{\mathfrak l_1}{2}, \f{2}{\mathfrak h}]$ when $(H,M) \in \calS$ (Proposition~\ref{prop:support}),
\item $\psi$ is a function of $(\tau,r)$ alone so that so that we can transform $\rd_{Q_T}$, $\rd_M$, $\rd_H$ into $\rd_r$ (by Lemma~\ref{lem:Jac} since the total number of derivatives $\leq N+1$). \qedhere
\end{itemize}
\end{proof}

We need a slight variant of Lemma~\ref{lem:Plancherel}, where if there is an additional power of $|k|^{-2}$, then we do not need to lose a derivative on the right-hand side. 
\begin{lemma}\label{lem:Plancherel.2}
Suppose $\psi(\tau,r) \in \{ \rd_s \varphi(\tau,r), \varphi(T,r) - \varphi(\tau,r)\}$. Assume also that $i_1'+i_2'+i_3' \leq N$. Then
\begin{equation}
\begin{split}
&\: \sum_{k \in \mathbb Z\setminus\{0\}} \sum_{\ell \in \mathbb Z} |k|^{-2}|(\widehat{\rd_{Q_T}^{i_1'} \rd_M^{i_2'} \rd_H^{i_3'} \psi})_\ell| |(\widehat{\rd_{Q_T}^{i_1''} \rd_M^{i_2''} \rd_H^{i_3''}} f)_{k-\ell} |(\tau,H,M) \\
\ls &\: \Big( \sum_{I' \leq i_1'+i_2'+i_3' } \sup_{r\in [\f{\mathfrak l_1}{2}, \f{2}{\mathfrak h}]} |\rd_r^{I'} \psi|(\tau,r) \Big)\Big( \|\rd_{Q_T}^{i_1''} \rd_M^{i_2''} \rd_H^{i_3''} f \|_{L^\i}(\tau) \Big).
\end{split}
\end{equation}

\end{lemma}
\begin{proof}
Instead of \eqref{eq:very.simple.Plancherel}, we use $\sum_{k\in \mathbb Z\setminus\{0\}} |k|^{-2}\ls 1$ and then use Young's inequality and Plancherel's theorem to argue as follows:
\begin{equation}
\begin{split}
&\: \ \sum_{k \in \mathbb Z\setminus\{0\}} \sum_{\ell \in \mathbb Z} |k|^{-2}\sup_{(H,M)\in \mathcal S} |\widehat{h}^{(1)}_{\ell}| |\widehat{h}^{(2)}_{k-\ell}|(H,M)  
\ls \sup_{(H,M)\in \calS}  \sum_{\ell\in \mathbb Z}  |\widehat{h}^{(1)}_{\ell}| |\widehat{h}^{(2)}_{k-\ell}|(H,M)  \\
\ls &\: \sup_{(H,M)\in \calS} \Big(\sum_{\ell\in \mathbb Z} |\widehat{h}^{(1)}_{\ell}|^2(H,M)\Big)^{1/2} \Big(\sum_{k \in \mathbb Z} |\widehat{h}^{(2)}_{k}|^2(H,M)\Big)^{1/2} \\
\ls &\: \sup_{(H,M)\in \calS}  \| \widehat{h}^{(1)}_{\ell}\|_{L^2_{Q_T}}(H,M) \|\widehat{h}^{(2)}_{\ell}\|_{L^2_{Q_T}}(H,M) 
\ls \sup_{\substack{Q_T,H,M \\ (H,M)\in \mathcal S}} | h^{(1)}| |\ h^{(2)}| (Q_T,H,M).
\end{split}
\end{equation}
The remainder of the proof proceeds as in Lemma~\ref{lem:Plancherel}. \qedhere

\end{proof}

We now use Lemma~\ref{lem:main-terms-den} to get appropriate decay for $\mathcal{T}_{1,2}$. This should be viewed as a prototypical term; the estimates for $\mathcal{T}_{2,2}$ and $\mathcal{T}_{3,2}$ are similar, while the bounds for $\{\calT_{i,3}\}_{i=1,2,3}$, $\{\calT_{i,4}\}_{i=1,2,3}$, $\{\mathcal{D}_{i,2}\}_{i=1,2,3}$ are easier.

\begin{proposition}\label{prop:T12}
For $I\leq N-2$ and $\calR_1$ as in \eqref{eq:R1.def}, the following estimate holds for all $t \in [0,T]$:
\begin{equation}\label{eq:varphi-t-main}
\Big| \rd_r^I \sum_{k\in \Z\backslash\{0\}} \int_{-\infty}^\infty\int_{0}^\infty\int_0^t k\Omega e^{ikQ_T} e^{-ik\Omega(t-\tau)}\widehat{(\calR_1)}_k(\tau,T,H,M) \ud\tau\d L\d w \Big| \ls \de^{7/4} \ep \brk{t}^{-N+I}.
\end{equation}
Thus, for $\mathcal{T}_{1,2}$ as defined in \eqref{eq:Ti2}, the following estimate holds for all $t \in [0,T]$:
$$\sup_{r\in [\f{\mathfrak l_1}{2}, \f{2}{\mathfrak h}]}|\partial_r^{I}\mathcal{T}_{1,2}|(t,r)\lesssim \delta^{7/4}\epsilon \jap{t}^{-\min\{N-I+2,N\}},\quad I \leq N.$$
\end{proposition}
\begin{proof}
We focus on \eqref{eq:varphi-t-main} as the other inequality is an easy corollary. We apply Lemma~\ref{lem:main-terms-den} with $\psi=\rd_s \varphi$. We split the integral for $\mathcal T_{1,2}$ into $\tau \geq t/2$ and $\tau \leq t/2$ and thus according to Lemma~\ref{lem:main-terms-den}, it suffices to bound
\begin{equation}\label{eq:varphi-t-der.1}
\sum_{\substack{i_1+i_2'+i_2''+i_3'\leq I \\i_1'+i_1'' = i_1+1}} \sum_{k\in \Z\backslash\{0\}}\sum_{\ell\in\Z}  \int_{t/2}^{t} \mathfrak F^{(k,\ell)}_{i_1',i_1'',i_2',i_2'',i_3'}(\tau) \, \ud \tau
\end{equation}
and
\begin{equation}\label{eq:varphi-t-der.2}
 \sum_{\substack{i_1+i_2'+i_2''+i_3'\leq I \\i_1'+i_1'' = i_1}} \sum_{k\in \Z\backslash\{0\}}\sum_{\ell\in\Z} \int_{0}^{t/2} |k| \mathfrak F^{(k,\ell)}_{i_1',i_1'',i_2',i_2'',i_3'}(\tau) \, \ud \tau,
\end{equation}
where we used the first and second terms in the minimum in Lemma~\ref{lem:main-terms-den} in \eqref{eq:varphi-t-der.1} and \eqref{eq:varphi-t-der.2}, respectively.

We will bound the terms \eqref{eq:varphi-t-der.1} and \eqref{eq:varphi-t-der.2} for each \underline{fixed} admissible collection of $i_1'$, $i_1''$, etc.~in the four steps below. The terms will be estimated in different ways (with different choices of $N_1$, $N_2$, $\eta_1$, $\eta_2$ when applying Lemma~\ref{lem:main-terms-den}) according to the values of $i_1'$, $i_2'$, $i_3'$ and $I$.

We briefly comment on the strategy. For each of \eqref{eq:varphi-t-der.1} and \eqref{eq:varphi-t-der.2}, we will consider two cases: (1) $i_1'+i_2'+i_3'\leq \lceil \f N2 \rceil -1$ and $I\geq N-3$ (or $i_1'+i_2'+i_3'\leq \lceil \f N2 \rceil -1$ and $I\geq N-4$ in the case of \eqref{eq:varphi-t-der.1}) and (2) the complementary case. 

For the term \eqref{eq:varphi-t-der.1}, since $\tau \geq t/2$, the $\tau$-decay coming from $\rd_s \varphi$ can be converted into the desired $t$-decay. In case (1), because the total number of derivatives is large but relatively few of them fall on $\rd_s \varphi$, the estimates from \eqref{eq:varphi_t} and Theorem~\ref{the:boot-top-plus-1} are already sufficient to obtain the necessary bound. In case (2), it is possible (when $I=0$) that direct estimate would fall one power of $\jap{t}$ short. On the other hand, in this case, it must be that $f$ has fewer than the top number of derivatives. We thus use on $Y_{k,-\ell\tau}$ to generate a $\langle kt - \ell \tau\rangle^{-1}$ decay factor which is used for the $\tau$-integration and effectively saves one power of $t$.

For the term \eqref{eq:varphi-t-der.2}, since $\tau \leq t/2$, the $\tau$-decay from $\rd_s \varphi$ does not immediately translate to $t$-decay. We thus need to use $N-I-1$ copies of $Y_{k,-k\tau}$ and $Y_{k,-\ell\tau}$ to generate $t$-decay. (Note that the number of $Y_{k,-k\tau}$, $Y_{k,-\ell\tau}$ are quite tight in view of Remark~\ref{rmk:also.diff.psi}.) The vector field $Y_{k,-k\tau}$ effectively allows us to exchange $\tau$-decay for $t$-decay. On the other hand, the vector field $Y_{k,-\ell\tau}$ only generates $\jap{kt-\ell\tau}^{-1}$ decay but has the advantage that it does not lose $\tau$-decay. (Note, however, that this good property of not losing $\tau$-decay only applies when one $Y_{k,-\ell\tau}$ is used. This is because for every $Y_{k,-\ell\tau}$ we use, we may need to put a $Y_H$ on $\rd_s \varphi$, and if the number of $Y_H$ is large, it causes $\tau$-growth; see \eqref{eq:z.def}, \eqref{eq:boot-f}.)

When handling the term \eqref{eq:varphi-t-der.2}, in case (1), we have plenty of $\tau$-decay. Thus we use $N-I-1$ copies of $Y_{k,-k\tau}$ to convert $\tau$-decay to $t$-decay. In case (2), however, it is possible to not have enough $\tau$-decay. (We should note however that if we handle the summability in $k$ and $\ell$ carefully, the lack of $\tau$-decay only happens when $I = 0$.) We thus not only use $N-I-2$ copies of $Y_{k,-k\tau}$ to convert $\tau$-decay to $t$-decay, but also use a $Y_{k,-\ell\tau}$, which generates an additional $\jap{kt-\ell\tau}^{-1}$ decay without losing decay in $\tau$. To handle the error terms, we need to further divide the integral into the parts where $|\ell\tau| < |kt|/2$ and $|\ell\tau| \geq |kt|/2$. In the former case, $\jap{kt-\ell\tau}^{-1}$ can be converted to $\jap{kt}^{-1}$, which ultimately gives enough $t$-decay without losing too much $\tau$-decay. In the latter case, we generate the needed $kt$-decay by losing $\ell\tau$-decay, but then use $\langle kt - \ell \tau\rangle^{-1}$ for the $\tau$-integration to effectively save one power of $\tau$.

Let us also remark that since our main objective here is decay, if $t \leq 1$, then the desired bound follows relatively easily from \eqref{eq:varphi_t} and Theorem~\ref{the:boot-f}, without needing to use the $Y_{k,-k\tau}$ or $Y_{k,-\ell\tau}$ vector field. Hence, without loss of generality, \textbf{we assume from now on that $t\geq 1$}.

\pfstep{Step~1: Term \eqref{eq:varphi-t-der.1} when $i_1'+i_2'+i_3'\leq \lceil \f N2 \rceil -1$ and $I\geq N-4$} In this case, we pick $N_1 = N_2 = 0$. By Lemma~\ref{lem:main-terms-den}, we can bound the term in \eqref{eq:varphi-t-der.1} above by
\begin{equation}\label{eq:varphi-t-der.case.1}
\sum_{k\in \Z\backslash\{0\}}\sum_{\ell\in\Z} \sum_{j_2'' \leq i_2''} \int_{t/2}^t \sup_{(H,M)\in\mathcal S} |\widehat{(\partial_{Q_T}^{i_1'}\partial_M^{i_2'}\partial_H^{i_3'} \rd_s \varphi)}_{\ell}|(\tau)\cdot|\widehat{(\partial_{Q_T}^{i_1''}\partial_M^{j_2''}\partial_H f)}_{k-\ell}|(\tau) \ud \tau
\end{equation}

Since $i_1'+i_2'+i_3' \leq \lceil \f N2 \rceil - 1$ and $i_1''+i_2'' \leq i_1'+i_2'+i_3' + i_1''+i_2'' \leq I+1$, using Lemma~\ref{lem:Plancherel}, we have the following upper bound (after relabelling indices):
\begin{equation}\label{eq:varphi-t-der.case.1.1}
\begin{split}
\hbox{\eqref{eq:varphi-t-der.case.1}} 
\ls &\: \sum_{\substack{j' \leq \lceil \f N2 \rceil \\ j_1''+j_2''\leq I+2}} \int_{t/2}^t \sup_{r\in [\f{\mathfrak l_1}{2}, \f{2}{\mathfrak h}]} | \rd_r^{j'} \rd_s \varphi|(\tau)\cdot \|\partial_{Q_T}^{j_1''}\partial_M^{j_2''}\partial_H f \|_{L^\i}(\tau) \ud \tau.
\end{split}
\end{equation}

To bound the term in \eqref{eq:varphi-t-der.case.1.1}, we use the bootstrap assumption \eqref{eq:varphi_t} and Theorem~\ref{the:boot-top-plus-1} to obtain
\begin{equation*}
\begin{split}
\hbox{\eqref{eq:varphi-t-der.case.1.1}}
\ls &\: \int_{t/2}^t \delta^{3/4}\epsilon \jap{\tau}^{-\lfloor N/2 \rfloor -2} \cdot \de \ep \jap{\tau}^2 \ud \tau\\
\ls &\: \de^{7/4}\ep^2 \jap{t}^{- \lfloor N/2 \rfloor +1} \ls \de^{7/4}\ep^2 \jap{t}^{-N+I+1} \ls \de^{7/4} \ep \jap{t}^{-N+I},
\end{split}
\end{equation*}
where we used $I\geq N-4$ and $N\geq 8$ in the penultimate inequality and used $\ep \jap{t} \leq 1$ in the last inequality.

\pfstep{Step~2: Term \eqref{eq:varphi-t-der.1} when $i_1'+i_2'+i_3' \geq \lceil \f N2 \rceil$ or $I\leq N-5$} In either case, it must follow that 
\begin{equation}\label{eq:i'.large.means.i''.small}
i_1''+i_2''\leq N-4
\end{equation}
using $i_1'+i_2'+i_3'+i_1''+i_2''\leq I+1$ and $N \geq 8$. (Indeed, if $i_1'+i_2'+i_3' \geq \lceil \f N2 \rceil$, then $i_1''+i_2'' \leq I+1-(i_1'+i_2'+i_3') \leq I+1- \lceil \f N2 \rceil  \leq N-1-\lceil \f N2 \rceil\leq N-5$ since $N\geq 8$. On the other hand, if $I \leq N-5$, then $i_1''+i_2'' \leq I + 1 \leq N -4$.)

 In this case we choose $N_1=1$, $\eta_1=-\ell$ and $N_2 = 0$. 
According to Lemma~\ref{lem:main-terms-den}, we thus need to control the following term:
\begin{equation}\label{eq:varphi-t-der.case.2}
\begin{split}
&\: \sum_{k\in \Z\backslash\{0\}}\sum_{\ell\in\Z} \sum_{\substack{j_2'' \leq i_2'' \\K_1'+K_2'+K_3'+K_1''+K_2''+N_1''\leq 1}} \int_{t/2}^t \jap{kt-\ell\tau}^{-1} \\
&\: \hspace{3em} \times \sup_{(H,M)\in\mathcal S} |\widehat{(\partial_{Q_T}^{i_1'+K_1'}\partial_M^{i_2'+K_2'}\partial_H^{i_3'+K_3'} \rd_s \varphi)}_{\ell}|(\tau)\cdot|Y_{k,-\ell\tau}^{N_1''}[\widehat{(\partial_{Q_T}^{i_1''+K_1''}\partial_M^{i_2''+K_2''}\partial_H f)}_{k-\ell}(\tau)]| \ud \tau.
\end{split}
\end{equation}
We will need to use $\jap{kt-\ell\tau}^{-1}$ for integration in $\tau$. In particular, we will need to take $\sup_{\tau \in [t/2,t]}$ of $\rd_s \varphi$ and $f$ and will not be able to use Lemma~\ref{lem:Plancherel} (which is a fixed $\tau$) estimate. We thus need to be more careful to deal with the sums in $k$ and $\ell$.

To handle \eqref{eq:varphi-t-der.case.2}, we make two observations. First,
\begin{equation}\label{eq:varphi-t-der.case.2.obs.1}
\begin{split}
Y_{k,-\ell\tau}(\tau)[\widehat{(\partial_{Q_T}^{i_1''+K_1''}\partial_M^{j_2''+K_2''}\partial_H f)}_{k-\ell}(\tau)]&=((ikt-\ell\tau)\partial_X \Omega+\partial_H)_{|t=\tau}[\widehat{(\partial_{Q_T}^{i_1''+K_1''}\partial_M^{j_2''+K_2''}\partial_H f)}_{k-\ell}(\tau)]\\
&=\widehat{(Y_H\partial_{Q_T}^{i_1''+K_1''}\partial_M^{j_2''+K_2''}\partial_H f)}_{k-\ell}(\tau).
\end{split}
\end{equation}
Second, we carry out the $\tau$ integral as follows
\begin{equation}\label{eq:varphi-t-der.case.2.obs.2}
\begin{split}
\int_{t/2}^t \jap{kt-\ell\tau}^{-1} \ud \tau = &\: \begin{cases}
\ell^{-1} \int^{(k-\f{\ell}2)t}_{(k-\ell)t} \jap{s}^{-1}\ud s & \hbox{if $\ell >0$} \\
\jap{kt}^{-1} \cdot \f t2 & \hbox{if $\ell = 0$}\\
|\ell|^{-1} \int_{(k-\f{\ell}2)t}^{(k-\ell)t} \jap{s}^{-1}\ud s & \hbox{if $\ell < 0$}
 \end{cases} \\
\ls  &\: \jap{\ell}^{-1} \max\{1,\log{\jap{t}}, \log \jap{k}, \log \jap{k-\ell}\}.
 \end{split}
 \end{equation}
 
Consider now summand in the term \eqref{eq:varphi-t-der.case.2} for every fixed $\ell$. Using \eqref{eq:varphi-t-der.case.2.obs.1} and \eqref{eq:varphi-t-der.case.2.obs.2} above, we can bound the term \eqref{eq:varphi-t-der.case.2} as follows with either $\alp = 0$ or $\alp = 1$:
\begin{equation}\label{eq:varphi-t-der.case.2.1}
\begin{split}
&\: \sum_{k\in \Z\backslash\{0\}} \sum_{\substack{j_2'' \leq i_2'' \\K_1'+K_2'+K_3'+K_1''+K_2''+N_1''\leq 1}} \int_{t/2}^t \jap{kt-\ell\tau}^{-1} \\
&\: \hspace{3em} \times \sup_{(H,M)\in\mathcal S} |\widehat{(\partial_{Q_T}^{i_1'+K_1'}\partial_M^{i_2'+K_2'}\partial_H^{i_3'+K_3'} \rd_s \varphi)}_{\ell}|(\tau)\cdot|Y_{k,-\ell\tau}^{N_1''}[\widehat{(\partial_{Q_T}^{i_1''+K_1''}\partial_M^{i_2''+K_2''}\partial_H f)}_{k-\ell}(\tau)]| \ud \tau \\
\ls&\: \sum_{k\in \Z\backslash\{0\}} \sum_{\substack{j_2''\leq i_2'' \\ i_1'' + i_2'' \leq N-4 \\ K_1'+K_2'+K_3'+K_1''+K_2''+N_1''\leq 1}} \max\{1,\log{\jap{t}}, \log \jap{k}, \log \jap{k-\ell}\} \\
&\: \qquad \times \jap{\ell}^{-1-\alp} \jap{k-\ell}^{-2} \sup_{\substack{\tau \in [t/2,t] \\(H,M)\in\mathcal S}} |\widehat{(\partial_{Q_T}^{i_1'+K_1'+\alp}\partial_M^{i_2'+K_2'}\partial_H^{i_3'+K_3'} \rd_s \varphi)}_{\ell}|(\tau)\\
&\hspace{13em}\times |Y_{H}^{N_1''}[\widehat{(\partial_{Q_T}^{i_1''+K_1''+2}\partial_M^{j_2''+K_2''}\partial_H f)}_{k-\ell}(\tau)]|  \\
\ls &\:   \max\{1,\log{\jap{t}}, \log \jap{\ell}\} \jap{\ell}^{-1-\alpha} \\
&\: \qquad \times \sup_{\tau \in [t/2,t]}\sum_{\substack{j' \leq I+2+\alp \\ j_1''+j_2''+N_1'' \leq N -1 \\N_1'' \leq 1}} \Big( \sup_{r\in [\f{\mathfrak l_1}{2}, \f{2}{\mathfrak h}]} |\rd_r^{j'} \rd_s \varphi|(\tau)\Big)\cdot  \|Y_{H}^{N_1''}\partial_{Q_T}^{j_1''}\partial_M^{j_2''}\partial_H f\|_{L^\i}(\tau).
\end{split}
\end{equation}

To sum in $\ell$, we consider two cases: $|\ell| \leq \lfloor t^{\mathfrak{M}} \rfloor$ (for which we use the $\alp = 0$ estimate in \eqref{eq:varphi-t-der.case.2.1}) and $|\ell| \geq \lfloor t^{\mathfrak{M}} \rfloor +1$ (for which we use the $\alp = 1$ estimate in \eqref{eq:varphi-t-der.case.2.1}) for $\mathfrak{M}$ large to be chosen below. 
We thus have the following bound:
\begin{equation}\label{eq:log.split}
\begin{split}
 \hbox{\eqref{eq:varphi-t-der.case.2} } & \ls \sum_{\ell \in \mathbb Z} \min_{\alp \in \{0,1\}} (\hbox{RHS of \eqref{eq:varphi-t-der.case.2.1}}) \\
&\ls \sum_{|\ell|\leq \lfloor t^{\mathfrak{M}} \rfloor } \max\{1,\log{\jap{t}}, \log \jap{\ell}\} \jap{\ell}^{-1}\\
&\quad\qquad\times \sum_{\substack{j' \leq I+2\\ j_1''+j_2''+N_1'' \leq N-1 \\N_1'' \leq 1}} \sup_{\tau\in[t/2,t]}\Big( \sup_{r\in [\f{\mathfrak l_1}{2}, \f{2}{\mathfrak h}]} |\rd_r^{j'} \rd_s \varphi|(\tau)\Big)\cdot  \|Y_{H}^{N_1''}\partial_{Q_T}^{j_1''}\partial_M^{j_2''}\partial_H f\|_{L^\i}(\tau)\\
&\quad+\sum_{|\ell| \geq \lfloor t^{\mathfrak{M}} \rfloor +1} \max\{1,\log{\jap{t}}, \log \jap{\ell}\} \jap{\ell}^{-2}\\
&\quad\qquad\times \sum_{\substack{j' \leq N+1\\ j_1''+j_2''+N_1'' \leq N \\N_1'' \leq 1}} \sup_{\tau\in[t/2,t]} \Big( \sup_{r\in [\f{\mathfrak l_1}{2}, \f{2}{\mathfrak h}]} |\rd_r^{j'} \rd_s \varphi|(\tau)\Big)\cdot  \|Y_{H}^{N_1''}\partial_{Q_T}^{j_1''}\partial_M^{j_2''}\partial_H f\|_{L^\i}(\tau)\\
&\ls_{\mathfrak{M}} (\log \jap{t})^2 \sum_{\substack{j' \leq I+2\\ j_1''+j_2''+N_1'' \leq N-1 \\N_1'' \leq 1}} \sup_{\tau\in[t/2,t]} \Big( \sup_{r\in [\f{\mathfrak l_1}{2}, \f{2}{\mathfrak h}]} |\rd_r^{j'} \rd_s \varphi|(\tau)\Big)\cdot  \|Y_{H}^{N_1''}\partial_{Q_T}^{j_1''}\partial_M^{j_2''}\partial_H f\|_{L^\i}(\tau) \\
&\quad+ \jap{t}^{-\mathfrak{M}/2} \sum_{\substack{j' \leq N+1\\ j_1''+j_2''+N_1'' \leq N -1 \\N_1'' \leq 1}} \sup_{\tau\in[t/2,t]} \Big( \sup_{r\in [\f{\mathfrak l_1}{2}, \f{2}{\mathfrak h}]} |\rd_r^{j'} \rd_s \varphi|(\tau)\Big)\cdot  \|Y_{H}^{N_1''}\partial_{Q_T}^{j_1''}\partial_M^{j_2''}\partial_H f\|_{L^\i}(\tau).
\end{split}
\end{equation}

Finally, using the estimates in the bootstrap assumption \eqref{eq:varphi_t}, Theorem~\ref{prop:phi-t-by-f}, and Theorem~\ref{the:boot-f}, we obtain
\begin{equation*}
\begin{split}
\hbox{\eqref{eq:varphi-t-der.case.2} } 
\ls &\: \jap{\log{\jap{t}}}^2 \sup_{\tau\in[t/2,t]}\left( \delta^{3/4}\epsilon \jap{\tau}^{-\min\{N,N-I\}} + \delta^{3/4}\epsilon \brk{t}^{-\mathfrak{M}/2} \right) \cdot  \delta\epsilon\jap{t}[1+\epsilon \jap{t}]  \\
\lesssim &\: \delta^{7/4}\epsilon\jap{t}^{-\min\{N,N-I\}} = \delta^{7/4}\epsilon\jap{t}^{-N+I},
\end{split}
\end{equation*}
where we have fixed $\mathfrak{M}$ to be sufficiently large to deal with the second term, and used $\ep \jap{t} \jap{\log \jap{t}}^2 \ls 1$ (and also the weaker bound $\epsilon \jap{t} \ls 1$).

\pfstep{Step~3: Term \eqref{eq:varphi-t-der.2} when $i_1'+i_2'+i_3' \leq \lceil \f{N}2 \rceil -1$ and $I\geq N-3$} In this case, take $N_1 = N - I - 1$, $\eta_1 = -k$ and $N_2 = 0$. By Lemma~\ref{lem:main-terms-den}, we need to control the following term:
\begin{equation}\label{eq:varphi-t-der.case.3}
\begin{split}
&\: \sum_{k\in \Z\backslash\{0\}}\sum_{\ell\in\Z}\sum_{\substack{j_2''\leq i_2'' \\K_1'+K_2'+K_3'+K_1''+K_2''+N_1''\leq N-I-1 }} \int_0^{t/2} |k| \jap{k(t-\tau)}^{-(N-I-1)}\\
&\: \hspace{14em}\times \sup_{(H,M)\in\mathcal S} |\widehat{(\partial_{Q_T}^{i_1'+K_1'}\partial_M^{i_2'+K_2'}\partial_H^{i_3'+K_3'} \rd_s \varphi)}_{\ell}|(\tau) \\
&\: \hspace{16em}\times \sup_{(H,M)\in\mathcal S}|Y_{k,-k\tau}^{N_1''}[\widehat{(\partial_{Q_T}^{i_1''+K_1''}\partial_M^{j_2''+K_2''}\partial_H f)}_{k-\ell}(\tau)]| \ud \tau.
\end{split}
\end{equation}

Notice that
\begin{equation}
Y_{k,-k\tau}(\tau) = \rd_H.
\end{equation}
Using $\tau \in [0,t/2]$ and simply dropping the extra good factor of $|k|^{-(N-I-2)}$, we obtain
\begin{equation}\label{eq:varphi-t-der.case.3.jap}
\jap{k(t-\tau)}^{-(N-I-1)} \ls |k|^{-1} \jap{t}^{-(N-I-1)}.
\end{equation} 
We use the $|k|^{-1}$ factor in \eqref{eq:varphi-t-der.case.3.jap} to cancel with the $|k|$ factor in \eqref{eq:varphi-t-der.case.3}. Thus, using also Lemma~\ref{lem:Plancherel}, we obtain
\begin{equation}\label{eq:varphi-t-der.case.3.1}
\begin{split}
\hbox{\eqref{eq:varphi-t-der.case.3}} \ls &\: \jap{t}^{-(N-I-1)}  \sum_{\substack{j' \leq \min\{\lceil \f N2 \rceil,I+1\} \\ j_1''+j_2'' \leq I+1 \\ K'+K_1''+K_2'' + N_1'' \leq N-I-1}} \\
&\:\hspace{7em} \int_0^{t/2}\sup_{r\in [\f{\mathfrak l_1}{2}, \f{2}{\mathfrak h}]} |\rd_r^{j'+K'} \rd_s \varphi|(\tau)  \|\partial_{Q_T}^{j_1''+K_1''}\partial_M^{j_2''+K_2''}\partial_H^{N_1''+1} f\|_{L^\i}(\tau) \ud \tau.
\end{split}
\end{equation}

Note that 
\begin{itemize}
\item $j' + K' \leq \lceil N/2 \rceil + N -I-1 \leq \lceil N/2 \rceil + N - (N-3) -1 = \lceil N/2 \rceil + 2 \leq N$ so that bootstrap assumption \eqref{eq:varphi_t} can be applied for $|\rd_r^{j'+K'} \rd_s \varphi|(\tau)$,
\item $j_1''+K_1''+j_2''+K_2''+N_1''+1 \leq I+2 +N-I-1 = N+1$ so that Theorem~\ref{the:boot-top-plus-1} can be applied for $\|\partial_{Q_T}^{j_1''}\partial_M^{j_2''}\partial_H^{N_1''+1} f\|_{L^\i}(\tau)$.
\end{itemize}
Hence, we use bootstrap assumption \eqref{eq:varphi_t} and Theorem~\ref{the:boot-top-plus-1} to obtain that 
\begin{equation}
\begin{split}
&\: \sum_{\substack{j' \leq \min\{\lceil \f N2 \rceil,I+1\} \\ j_1''+j_2'' \leq I+1 \\ K'+K_1''+K_2'' + N_1'' \leq N-I-1}} \sup_{r\in [\f{\mathfrak l_1}{2}, \f{2}{\mathfrak h}]} |\rd_r^{j'+K'} \rd_s \varphi|(\tau)\cdot \|\partial_{Q_T}^{j_1''+K_1''}\partial_M^{j_2''+K_2''}\partial_H^{N_1''+1} f\|_{L^\i}(\tau)\\
\ls &\: \sum_{N_1'+N_1'' \leq N-I-1} (\de^{3/4} \ep \jap{\tau}^{-\min\{N-\lceil \f N2 \rceil - N_1'+2,N\}}) \cdot (\de \ep \jap{\tau}^{N_1''+2}) \\
\ls &\: \de^{7/4} \ep^2 \min\{ \jap{\tau}^{-\lfloor \f N2 \rfloor +N-I-1}, \jap{\tau}^{-N+(N-I-1)+2} \} = \de^{7/4} \ep^2 \min\{ \jap{\tau}^{\lceil \f N2 \rceil -I-1}, \jap{\tau}^{-I+1} \}.
\end{split}
\end{equation}
Since $I\geq N-3$ and $N\geq 8$,
$$I + 1- \lceil \f N2 \rceil \geq N-2-\lceil \f N2 \rceil = \lfloor \f N2 \rfloor - 2 \geq 2,\quad I - 1 \geq N-4 \geq 2.$$
Hence, the integrand in \eqref{eq:varphi-t-der.case.3.1} is integrable in $\tau$ so that 
\begin{equation}
\begin{split}
\hbox{\eqref{eq:varphi-t-der.case.3.1}} \ls \de^{7/4} \ep^2 \jap{t}^{-(N-I-1)} \int_0^{t/2} \jap{\tau}^{-2} \ud \tau \ls \de^{7/4} \ep^2 \jap{t}^{-(N-I-1)} \ls \de^{7/4} \ep \jap{t}^{-N+I},
\end{split}
\end{equation}
where at the end we used $\ep \jap{t} \leq 1$.

\pfstep{Step~4: Term \eqref{eq:varphi-t-der.2} when $i_1'+i_2'+i_3' \geq \lceil N/2 \rceil$ or $I\leq N-4$} We choose $\eta_1 = -k$, $N_1 = N-I-2$, $\eta_2 = -\ell$ and $N_2 = 1$. By Lemma~\ref{lem:main-terms-den}, we need to bound
 \begin{equation}\label{eq:varphi-t-der.case.4}
\begin{split}
&\: \sum_{k\in \Z\backslash\{0\}}\sum_{\ell\in\Z}\sum_{\substack{j_2''\leq i_2'' \\N_1''\leq N-I-2 \\ N_2'' \leq 1 \\ K_1'+K_2'+K_3'+K_1''+K_2'' +N_1''+N_2''\leq N-I-2}} \int_0^{t/2} |k| \jap{k(t-\tau)}^{-(N-I-2)} \jap{kt-\ell \tau}^{-1} \\
&\: \hspace{12em} \times \sup_{(H,M)\in\mathcal S} |\widehat{(\partial_{Q_T}^{i_1'+K_1'}\partial_M^{i_2'+K_2'}\partial_H^{i_3'+K_3'} \rd_s \varphi)}_{\ell}|(\tau)\\
&\: \hspace{12em}\times \sup_{(H,M)\in\mathcal S} |Y_{k,-k\tau}^{N_1''} Y_{k,\ell\tau}^{N_2''}[\widehat{(\partial_{Q_T}^{i_1''+K_1''}\partial_M^{j_2''+K_2''}\partial_H f)}_{k-\ell}(\tau)]| \ud \tau.
\end{split}
\end{equation}
 
 For each $k \in \mathbb Z \setminus \{0\}$, $\ell \in \mathbb Z$, we further divide the integral according to $|\ell \tau| < |kt|/2$ (Step~4(a)) and $|\ell \tau| \geq |kt|/2$ (Step~4(b)).
 
\pfstep{Step~4(a): $|\ell\tau|< |kt|/2$} When $|\ell\tau| < |kt|/2$, we have
$$ \jap{k(t-\tau)}^{-(N-I-2)} \jap{kt-\ell \tau}^{-1} \ls \jap{kt}^{-(N-I-1)}.$$
Thus, since $k\neq 0$ and we have assumed $t \geq 1$, it follows that
\begin{equation}\label{eq:annoying.cases}
 \jap{k(t-\tau)}^{-(N-I-2)} \jap{kt-\ell \tau}^{-1} \ls 
\begin{cases}
|k|^{-1} \jap{t}^{-(N-I-1)} & \hbox{if $I = N-3,N-2$} \\
|k|^{-3} \jap{t}^{-(N-I-1)} & \hbox{if $I \leq N-4$}
\end{cases}.
\end{equation}
In either case in \eqref{eq:annoying.cases}, we can use one power of $|k|^{-1}$ to cancel the power of $|k|$ in \eqref{eq:varphi-t-der.case.4}. We first handle the case where $I = N-3,N-2$. Using Lemma~\ref{lem:Plancherel}, we obtain
 \begin{equation}\label{eq:varphi-t-der.case.4.1}
 \begin{split}
&\:   \sum_{k\in \Z\backslash\{0\}}\sum_{\ell\in\Z} \sum_{\substack{j_2''\leq i_2'' \\N_1''\leq N-I-2 \\ N_2'' \leq 1 \\ K_1'+K_2'+K_3'+K_1''+K_2'' +N_1''+N_2''\leq N-I-1}}  \int_{\substack{\tau \in [0,t/2] \\ |\ell\tau| < |kt|/2}} [\cdots] \ud \tau\\
\ls &\: \sum_{\substack{N_1'' \leq N-I-2 \\ N_2'' \leq 1 \\ K' + K_1'' + K_2'' \leq N-I-1-N_1''-N_2''}} \sum_{\substack{ j' \leq i_1'+i_2'+i_3'+1\\ j_1'' + j_2'' \leq i_1'' + i_2'' +1}} \jap{t}^{-(N-I-1)}  \int_0^{t/2} \ \sup_{r\in [\f{\mathfrak l_1}{2}, \f{2}{\mathfrak h}]} |\rd_r^{j' + K'} \rd_s \varphi| (\tau,r) \\
&\:\hspace{16em} \times \| Y_{H}^{N_2''} \partial_{Q_T}^{j_1''+K_1''}\partial_M^{j_2''+K_2''}\partial_H^{N_1''+1} f\|_{L^\i}(\tau)\ud \tau.
 \end{split}
 \end{equation}
 
To bound the integrand in \eqref{eq:varphi-t-der.case.4.1}, note that 	
\begin{itemize}
\item $j' + K' \leq (i_1'+i_2'+i_3'+1)+ (N -I-2) +1 = (I+1)+(N-I-2) \leq N$ so that bootstrap assumption \eqref{eq:varphi_t} can be applied for $|\rd_r^{j'+K'} \rd_s \varphi|(\tau)$.
\item Since we are considering $I = N-3,N-2$, the assumption of this step implies $i_1'+i_2'+i_3' \geq \lceil N/2 \rceil \geq 4$ and thus $i_1''+i_2'' \leq N-6$. Thus, $j_1''+j_2''+K_1''+K_2''+N_1''+N_2''+1 \leq N-5 +(N-I-1)+1 = 2N -I -5 \leq N-2$, where we used again $I = N-3,N-2$. As a result, Theorem~\ref{the:boot-f} can be applied for $\|Y_{H}^{N_2''} \partial_{Q_T}^{j_1''+K_1''}\partial_M^{j_2''+K_2''}\partial_H^{N_1''+1} f\|_{L^\i}(\tau)$ (with the $z(\jap{t}, \cdots)$ factor $ \ls 1$).
\item The condition $i_1'+i_2'+i_3' \geq \lceil N/2 \rceil \geq 4$ also implies $\min\{N-i_1'-i_2'-i_3'-1-K'+2,N\} = N-i_1'-i_2'-i_3'-K'+1 \geq N-I-K'+1$.
\end{itemize}
 Therefore, 
\begin{equation}\label{eq:varphi-t-der.case.4.1.integrand}
\begin{split}
&\: \sum_{\substack{N_1'' \leq N-I-2 \\ N_2'' \leq 1 \\ K' + K_1'' + K_2'' \leq N-I-1-N_1''-N_2''}}  \sum_{\substack{ j' \leq i_1'+i_2'+i_3'+1\\ j_1'' + j_2'' \leq i_1'' + i_2'' +1}} \sup_{r\in [\f{\mathfrak l_1}{2}, \f{2}{\mathfrak h}]} |\rd_r^{j' + K'} \rd_s \varphi| (\tau,r) \| Y_{H}^{N_2''} \partial_{Q_T}^{j_1''+K_1''}\partial_M^{j_2''+K_2''}\partial_H^{N_1''+1} f \|_{L^\i}(\tau) \\
\ls &\: \sum_{\substack{N_1'' \leq N-I-2 \\ K' +N_1'' \leq N-I-1}} (\de^{3/4} \ep \jap{\tau}^{-\min\{N-i_1'-i_2'-i_3'-1-K'+2,N\}} )(\de \ep \jap{\tau}^{N_1''+1}) \\
\ls &\: \de^{7/4} \ep^2 \jap{\tau}^{-N+I-1+(N-I-1)+1} \ls \de^{7/4} \ep^2 \jap{\tau}^{-1}.
\end{split}
\end{equation}
Plugging \eqref{eq:varphi-t-der.case.4.1.integrand} into \eqref{eq:varphi-t-der.case.4.1}, we obtain
 \begin{equation}\label{eq:varphi-t-der.case.4.1.done}
 \begin{split}
\hbox{\eqref{eq:varphi-t-der.case.4.1}} \ls &\: \de^{7/4} \ep^2 \jap{t}^{-(N-I-1)} \int_0^{t/2} \jap{\tau}^{-1} \ud \tau \ls  \de^{7/4} \ep^2 \jap{t}^{-(N-I-1)} \jap{\log{\jap{t}}} \ls \de^{7/4} \ep \jap{t}^{-N+I},
 \end{split}
 \end{equation}
 where we have used $\ep \jap{t} \jap{\log \jap{t}} \ls 1$ at the end.

We turn to the other case in \eqref{eq:annoying.cases}, namely that $I \leq N- 4$. We instead apply Lemma~\ref{lem:Plancherel.2} to obtain
 \begin{equation}\label{eq:varphi-t-der.case.4.2}
 \begin{split}
&\:   \sum_{k\in \Z\backslash\{0\}}\sum_{\ell\in\Z}  \sum_{\substack{j_2''\leq i_2'' \\N_1''\leq N-I-2 \\ N_2'' \leq 1 \\ K_1'+K_2'+K_3'+K_1''+K_2'' +N_1''+N_2''\leq N-I-1}} \int_{\substack{\tau \in [0,t/2] \\ |\ell\tau| < |kt|/2}} [\cdots] \ud \tau\\
\ls &\: \sum_{k\in \Z\backslash\{0\}}\sum_{\ell\in\Z}  \sum_{\substack{j_2''\leq i_2'' \\N_1''\leq N-I-2 \\ N_2'' \leq 1 \\ K_1'+K_2'+K_3'+K_1''+K_2'' +N_1''+N_2''\leq N-I-1}} |k|^{-2} \jap{t}^{-(N-I-1)} \\
&\: \qquad \times \int_0^{t/2}\sup_{(H,M)\in\mathcal S} |\widehat{(\partial_{Q_T}^{i_1'+K_1'}\partial_M^{i_2'+K_2'}\partial_H^{i_3'+K_3'} \rd_s \varphi)}_{\ell}|\cdot|Y_{k,-k\tau}^{N_1''} Y_{k,\ell\tau}^{N_2''}[\widehat{(\partial_{Q_T}^{i_1''+K_1''}\partial_M^{j_2''+K_2''}\partial_H f)}_{k-\ell}]|(\tau)  \ud \tau \\
\ls &\: \sum_{\substack{j_2''\leq i_2'' \\N_1''\leq N-I-2 \\ N_2'' \leq 1 \\ K'+K_1''+K_2'' +N_1''+N_2''\leq N-I-1 \\j' \leq i_1' + i_2'}} \jap{t}^{-(N-I-1)}  \\
&\: \qquad \qquad \times  \int_0^{t/2} \sup_{r\in [\f{\mathfrak l_1}{2}, \f{2}{\mathfrak h}]} |\rd_r^{j' + K'} \rd_s \varphi| (\tau,r) \|Y_{H}^{N_2''} \partial_{Q_T}^{i_1''+K_1''}\partial_M^{j_2''+K_2''}\partial_H^{N_1''+1} f \|_{L^\i}(\tau)\ud \tau.
 \end{split}
 \end{equation}
 Note that 	
\begin{itemize}
\item $j' + K' \leq (i_1'+i_2')+ N -I-2 \leq N-2$ so that bootstrap assumption \eqref{eq:varphi_t} can be applied for $|\rd_r^{j'+K'} \rd_s \varphi|(\tau)$,
\item $i_1''+K_1''+j_2''+K_2''+N_1''+N_2''+1 \leq I +(N-I-2)+1+1 = N$ so that Theorem~\ref{the:boot-f} can be applied for $\|Y_{H}^{N_2''} \partial_{Q_T}^{j_1''+K_1''}\partial_M^{j_2''+K_2''}\partial_H^{N_1''+1} f\|_{L^\i}(\tau)$.
\end{itemize}
In anticipation of an estimate below (see \eqref{eq:varphi-t-der.case.4.4.2}), we provide a stronger bound than is needed for the integrand in \eqref{eq:varphi-t-der.case.4.2}, namely, we allow for $j' \leq i_1'+i_2'+1$:
 \begin{equation}\label{eq:varphi-t-der.case.4.2.integrand}
 \begin{split}
&\: \sum_{\substack{j_2''\leq i_2'' \\N_1''\leq N-I-2 \\ N_2'' \leq 1 \\ K'+K_1''+K_2'' +N_1''+N_2''\leq N-I-1 \\j' \leq i_1' + i_2'+1}}\sup_{r\in [\f{\mathfrak l_1}{2}, \f{2}{\mathfrak h}]} |\rd_r^{j' + K'} \rd_s \varphi| (\tau) \|Y_{H}^{N_2''} \partial_{Q_T}^{i_1''+K_1''}\partial_M^{j_2''+K_2''}\partial_H^{N_1''+1} f \|_{L^\i}(\tau) \\
\ls &\: \sum_{\substack{N_1''\leq N-I-2\\K'+N_1''\leq N-I-1}} (\de^{3/4} \ep \jap{\tau}^{-\min\{N-I-1-K'+2,N\}}) (\de \ep \jap{\tau}^{N_1''+1})\\
\ls &\:\max\{ \de^{7/4} \ep^2 \jap{\tau}^{-N+I-1+(N-I-1)+1}, \de^{7/4} \ep^2 \jap{\tau}^{-N+(N-I-2)+1}\} \ls \de^{7/4} \ep^2  \jap{\tau}^{-1} .
 \end{split}
 \end{equation}
 Therefore, substituting \eqref{eq:varphi-t-der.case.4.2.integrand} into \eqref{eq:varphi-t-der.case.4.2}, we argue as in \eqref{eq:varphi-t-der.case.4.1.done} to obtain
 \begin{equation}\label{eq:varphi-t-der.case.4.2.done}
 \begin{split}
\hbox{\eqref{eq:varphi-t-der.case.4.2}} \ls &\: \de^{7/4} \ep \jap{t}^{-N+I}.
 \end{split}
 \end{equation}
 
We remark that in the above estimate, it is important to have used Lemma~\ref{lem:Plancherel.2} instead of Lemma~\ref{lem:Plancherel} for otherwise there could be $N+1$ derivatives on $f$ and we will get a weaker decay rate.

 \pfstep{Step~4(b): $|\ell\tau|\geq |kt|/2$} Instead of \eqref{eq:annoying.cases}, when $|\ell\tau|\geq |kt|/2$, we have
 \begin{equation}\label{eq:before.integrate.with.the.logs}
 \jap{k(t-\tau)}^{-(N-I-2)} \jap{kt-\ell\tau}^{-1} 
 \ls \begin{cases}
|k|^{-1} \jap{t}^{-(N-I-1)} |\ell\tau| \jap{kt-\ell\tau}^{-1} & \hbox{if $I = N-3,N-2$} \\
|k|^{-3} \jap{t}^{-(N-I-1)} |\ell\tau| \jap{kt-\ell\tau}^{-1} & \hbox{if $I \leq N-4$}
\end{cases}.
 \end{equation}
As in Step~4(a), we use a power of $|k|^{-1}$ to cancel the $|k|$ factor in \eqref{eq:varphi-t-der.case.4}. Moreover, notice that since $\tau \leq t/2$,  if $|\ell\tau|\geq |kt|/2$, it must follow that $|\ell|\geq |k|/2$. We notice that
\begin{equation}\label{eq:integrate.with.the.logs}
\int_{0}^{t/2} \f{|\ell|}{\jap{kt-\ell\tau}} \ud \tau \ls \int_0^{\max\{(|k|+|\ell|/2)t,|k|t\}} \jap{s}^{-1} \ud s \ls 1 + \log\jap{\ell} + \log \jap{t},
\end{equation}
where we have used $|k|\leq 2|\ell|$ in the last estimate so that we do not have $\log\jap{k}$.

We have to proceed as in Step 2 because we cannot commute $\sup$ in $\tau$ and the sum over $k$ and $\ell.$ For $I = N-3, N-2$, we use the fact that the number of derivatives hitting $\rd_H f$ are at most $ \lceil N/2 \rceil$ and thus we can squeeze out $\jap{k-\ell}^{-2}$. In particular, for each fixed $\ell \in \mathbb Z$, we have the following estimates for $\alp \in \{0,1\}$:
\begin{equation}\label{eq:varphi-t-der.case.4.3}
 \begin{split}
&\:   \sum_{k\in \Z\backslash\{0\}} \sum_{\substack{j_2''\leq i_2'' \\N_1''\leq N-I-2 \\ N_2'' \leq 1 \\ K_1'+K_2'+K_3'+K_1''+K_2'' +N_1''+N_2''\leq N-I-1}}  \int_{\substack{\tau \in [0,t/2] \\ |\ell\tau| \geq |kt|/2}} [\cdots] \ud \tau \\
\ls &\:  \sum_{k\in \Z\backslash\{0\}} \sum_{\substack{N_1''\leq N-I-2 \\ N_2'' \leq 1 \\ K'+K_1''+K_2'' +N_1''+N_2''\leq N-I-1}} \sum_{\substack{ j' \leq i_1'+i_2'+1+\alp \\ j_1'' + j_2'' \leq i_1'' + i_2'' }}\jap{t}^{-(N-I-1)} \int_0^{t/2}\jap{kt-\ell\tau}^{-1}\jap{k-\ell}^{-2}\ud \tau\\
&\:  \qquad\times \jap{\ell}^{-\alp}\sup_{\tau \in [0,t/2]}  \Big[ \jap{\tau}  \sup_{r\in [\f{\mathfrak l_1}{2}, \f{2}{\mathfrak h}]} |\rd_r^{j' + K'} \rd_s \varphi| (\tau,r) \| Y_{H}^{N_2''} \partial_{Q_T}^{j_1''+K_1''+2}\partial_M^{j_2''+K_2''}\partial_H^{N_1''+1} f \|_{L^\i}(\tau) \Big] \\
\ls &\: \sum_{\substack{N_1''\leq N-I-2 \\ N_2'' \leq 1 \\ K'+K_1''+K_2'' +N_1''+N_2''\leq N-I-1}} \sum_{\substack{ j' \leq i_1'+i_2'+1+\alp \\ j_1'' + j_2'' \leq i_1'' + i_2'' }}\jap{t}^{-(N-I-1)} \max\{1,\log\jap{t},\log\jap{\ell}\} \jap{\ell}^{-1-\alp}\\
&\:  \qquad\times \sup_{\tau \in [0,t/2]}  \Big[ \jap{\tau}  \sup_{r\in [\f{\mathfrak l_1}{2}, \f{2}{\mathfrak h}]} |\rd_r^{j' + K'} \rd_s \varphi| (\tau,r) \| Y_{H}^{N_2''} \partial_{Q_T}^{j_1''+K_1''+2}\partial_M^{j_2''+K_2''}\partial_H^{N_1''+1} f \|_{L^\i}(\tau) \Big].
 \end{split}
 \end{equation}
 Similar consideration as for \eqref{eq:varphi-t-der.case.4.1} shows that $N_2''+j_1''+K_1''+j_2''+K_2''+N_1''+1\leq N-2$. In particular, we can use Theorem~\ref{the:boot-f} to control $\| Y_{H}^{N_2''} \partial_{Q_T}^{j_1''+K_1''+2}\partial_M^{j_2''+K_2''}\partial_H^{N_1''+1} f \|_{L^\i}(\tau)$. For $\rd_s\varphi$, we have $j' + K' \leq (I+1+\alp) + (N-I-1)-N_1'' = N-N_1''+\alp$ so that it can be controlled by \eqref{eq:varphi_t} and Theorem~\ref{prop:phi-t-by-f}.
 
We now follow the argument in \eqref{eq:log.split} and split the sum in $\ell$ into $|\ell| \leq \lfloor t^{\mathfrak{M}} \rfloor$ and $|\ell| \geq \lfloor t^{\mathfrak{M}} \rfloor +1$. Again, we will use $\alp = 0$ in the former case and $\alp = 1$ in the latter case. We then use bootstrap assumption \eqref{eq:varphi_t}, Theorem~\ref{prop:phi-t-by-f}, and Theorem~\ref{the:boot-f} to obtain
\begin{equation}\label{eq:varphi-t-der.case.4.3.done}
 \begin{split}
&\:   \sum_{k\in \Z\backslash\{0\}} \sum_{\ell \in \Z } \sum_{\substack{j_2''\leq i_2'' \\N_1''\leq N-I-2 \\ N_2'' \leq 1 \\ K_1'+K_2'+K_3'+K_1''+K_2'' +N_1''+N_2''\leq N-I-1}}  \int_{\substack{\tau \in [0,t/2] \\ |\ell\tau| \geq |kt|/2}} [\cdots] \ud \tau \\
\ls &\: \jap{t}^{-(N-I-1)}  \sum_{|\ell|\leq \lfloor t^{\mathfrak{M}} \rfloor } \max\{1,\log{\jap{t}}, \log \jap{\ell}\} \jap{\ell}^{-1} \\
&\: \qquad \qquad  \times \sum_{N_2''\leq 1}\sup_{\tau \in [0,t/2]} \jap{\tau} (\de^{3/4} \ep \jap{\tau}^{-2-N_1''}) (\de \ep \jap{\tau}^{1+N_1''}(1+ \ep \brk{t}))\\
&\quad+\jap{t}^{-(N-I-1)}  \sum_{|\ell| \geq \lfloor t^{\mathfrak{M}} \rfloor +1} \sum_{N_1'' \leq 1} \max\{1,\log{\jap{t}}, \log \jap{\ell}\} \jap{\ell}^{-2} \sup_{\tau \in [0,t/2]} (\de^{3/4} \ep) (\de \ep \brk{\tau}^2(1+\ep \jap{t})) \\
\ls &\: \de^{7/4} \ep^2 \jap{t}^{-(N-I-1)} (\log\jap{t})^2 + \de^{7/4} \ep^2 \jap{t}^{-\mathfrak{M}/2+2} \\
\ls &\: \de^{7/4} \ep \jap{t}^{-N+I},
 \end{split}
 \end{equation}
where we have chosen $\mathfrak M$ to be large, and used $\ep \jap{t}\jap{\log\jap{t}}^2\ls 1$.
 

We now turn to the $I \leq N- 4$ case. Notice that the extra $|k|^{-2}$ factor in \eqref{eq:before.integrate.with.the.logs} gives summability in $k$. We split the $\ell$ sum as above, i.e.,
\begin{equation}\label{eq:varphi-t-der.case.4.4}
 \begin{split}
 \sum_{\ell\in\Z} \sum_{k\in \Z\backslash\{0\}} \sum_{\substack{j_2''\leq i_2'' \\N_1''\leq N-I-2 \\ N_2'' \leq 1 \\ K_1'+K_2'+K_3'+K_1''+K_2'' +N_1''+N_2''\leq N-I-1}} \int_{\substack{\tau \in [0,t/2] \\ |\ell\tau| \geq |kt|/2}} [\cdots] \ud \tau
\ls &\:\sum_{|\ell|\leq \lfloor t^{\mathfrak{M}} \rfloor } (\cdots) + \sum_{|\ell| \geq \lfloor t^{\mathfrak{M}} \rfloor +1}(\cdots).
 \end{split}
 \end{equation}

 For the sum $\sum_{|\ell|\leq \lfloor t^{\mathfrak{M}} \rfloor }$, as before, we put additional derivatives on $\rd_s \varphi$ to gain powers of $\jap{\ell}^{-1}$. More precisely, using \eqref{eq:integrate.with.the.logs}, we control this sum by
\begin{equation}\label{eq:varphi-t-der.case.4.4.2}
 \begin{split}
 &\:\sum_{k\in \Z\backslash\{0\}}\sum_{|\ell|\leq \lfloor t^{\mathfrak{M}} \rfloor } \sum_{\substack{j_2''\leq i_2'' \\N_1'' \leq N-I-2 \\ N_2'' \leq 1 \\ K'+K_1''+K_2'' +N_1''+N_2''\leq N-I-1 \\ j' \leq i_1'+i_2'+1}} |k|^{-3} \jap{t}^{-(N-I-1)} \int_0^{t/2}\jap{kt-\ell\tau}^{-1} \ud \tau \\
&\: \qquad \times \sup_{\tau \in [0,t/2]} \Big[ \jap{\tau} \sup_{r\in [\f{\mathfrak l_1}{2}, \f{2}{\mathfrak h}]} |\rd_r^{j' + K'} \rd_s \varphi| (\tau) \| Y_{H}^{N_2''} \partial_{Q_T}^{i_1'' + K_1''}\partial_M^{j_2'' + K_2''}\partial_H^{N_1''+1} f\|_{L^\i}(\tau) \Big]\\
\ls &\: \sum_{\substack{j_2''\leq i_2'' \\N_1'' \leq N-I-2 \\ N_2'' \leq 1 \\ K'+K_1''+K_2'' +N_1''+N_2''\leq N-I-1 \\ j' \leq i_1'+i_2'+1}} \jap{t}^{-(N-I-1)} \jap{\log \jap{t}}^2\\
&\: \qquad \times \sup_{\tau \in [0,t/2]} \Big[ \jap{\tau} \sup_{r\in [\f{\mathfrak l_1}{2}, \f{2}{\mathfrak h}]} |\rd_r^{j' + K'} \rd_s \varphi| (\tau) \| Y_{H}^{N_2''} \partial_{Q_T}^{i_1'' + K_1''}\partial_M^{j_2'' + K_2''}\partial_H^{N_1''+1} f\|_{L^\i}(\tau) \Big]\\ 
\ls &\:  \jap{t}^{-(N-I-1)} \jap{\log\jap{t}}^{2} (\sup_{\tau \in [0,t/2]} \jap{\tau} \de^{7/4} \ep^2 \jap{\tau}^{-1}) \ls  \de^{7/4} \ep \jap{t}^{-N+I},
 \end{split}
 \end{equation}
where we have first used \eqref{eq:varphi-t-der.case.4.2.integrand}, and then used $\ep \jap{t}\jap{\log\jap{t}}^2\ls 1$ at the end.

The $\sum_{|\ell| \geq \lfloor t^{\mathfrak{M}} \rfloor +1}$ sum is easier since, as above, we can put one more derivative on $\rd_s\varphi$ and then gain large decaying powers in $\jap{t}$ after choosing $\mathfrak M$ to be large. More precisely, using \eqref{eq:integrate.with.the.logs}, bootstrap assumption \eqref{eq:varphi_t}, Theorem~\ref{prop:phi-t-by-f}, and Theorem~\ref{the:boot-f}, we control the sum by
\begin{equation}\label{eq:varphi-t-der.case.4.4.3}
 \begin{split}
 &\:\sum_{k\in \Z\backslash\{0\}} \sum_{|\ell| \geq \lfloor t^{\mathfrak{M}} \rfloor +1} \sum_{\substack{j_2''\leq i_2'' \\N_1'' \leq N-I-2 \\ N_2'' \leq 1 \\ K'+K_1''+K_2'' +N_1''+N_2''\leq N-I-1 \\ j' \leq i_1'+i_2'+2}} \jap{\ell}^{-1} |k|^{-3} \jap{t}^{-(N-I-1)} \int_0^{t/2}\jap{kt-\ell\tau}^{-1} \ud \tau \\
&\: \qquad \times \sup_{\tau \in [0,t/2]} \Big[ \jap{\tau} \sup_{r\in [\f{\mathfrak l_1}{2}, \f{2}{\mathfrak h}]} |\rd_r^{j' + K'} \rd_s \varphi| (\tau) \| Y_{H}^{N_2''} \partial_{Q_T}^{i_1'' + K_1''}\partial_M^{j_2'' + K_2''}\partial_H^{N_1''+1} f\|_{L^\i}(\tau) \Big]\\
\ls &\: \sum_{\substack{j'+K'\leq N+1\\ N_2''+i_1''+K_1''+j_2''+k_2''+N_1''+1 \leq N}} \jap{t}^{-(N-I-1)}  \max\{1,\log \jap{t},\log \jap{\ell}\} \jap{\ell}^{-2}\\
&\: \qquad \times \sup_{\tau \in [0,t/2]} \Big[ \jap{\tau} \sup_{r\in [\f{\mathfrak l_1}{2}, \f{2}{\mathfrak h}]} |\rd_r^{j' + K'} \rd_s \varphi| (\tau) \| Y_{H}^{N_2''} \partial_{Q_T}^{i_1'' + K_1''}\partial_M^{j_2'' + K_2''}\partial_H^{N_1''+1} f\|_{L^\i}(\tau) \Big]\\ 
\ls &\:  \jap{t}^{-(N-I-1)} \jap{t}^{-\mathfrak M/2} (\sup_{\tau \in [0,t/2]} \jap{\tau} \de^{7/4} \ep^2 \jap{\tau}^{N-I-2}) \ls  \de^{7/4} \ep^2 \jap{t}^{-N+I},
 \end{split}
 \end{equation}
 for $\mathfrak{M}$ sufficiently large.
Altogether, we thus have
 \begin{equation}\label{eq:varphi-t-der.case.4.4.done}
 \begin{split}
&\:   \sum_{k\in \Z\backslash\{0\}}\sum_{\ell\in\Z} \sum_{\substack{j_2''\leq i_2'' \\N_1''\leq N-I-2 \\ N_2'' \leq 1 \\ K_1'+K_2'+K_3'+K_1''+K_2'' +N_1''+N_2''\leq N-I-1}} \int_{\substack{\tau \in [0,t/2] \\ |\ell\tau| \geq |kt|/2}} [\cdots] \ud \tau
\ls  \de^{7/4} \ep \jap{t}^{-N+I}.
 \end{split}
 \end{equation}
 
 Combining \eqref{eq:varphi-t-der.case.4.1.done}, \eqref{eq:varphi-t-der.case.4.2.done}, \eqref{eq:varphi-t-der.case.4.3.done} and \eqref{eq:varphi-t-der.case.4.4.done} yields the desired bound in this case. \qedhere

\end{proof}

\begin{proposition}\label{prop:T22}
The terms $\calT_{2,2}$ and $\calT_{3,2}$ (defined in \eqref{eq:Ti2}) satisfy the following estimates for all $t \in [0,T]$:
$$ \sup_{r\in [\f{\mathfrak l_1}{2}, \f{2}{\mathfrak h}]}\Big(|\rd_r^I \calT_{2,2}|(t,r) + |\rd_r^I \calT_{3,2}|(t,r)\Big) \ls \de^{7/4} \ep \brk{t}^{-\min\{N-I+2,N\}},\quad I \leq N.$$
\end{proposition}
\begin{proof}
We first consider $\calT_{2,2}$ and $\calT_{3,2}$. This is essentially the same as $\calT_{1,2}$ in the sense that modulo the terms $\mathfrak p$ and $\mathfrak q$ (which are controlled by Corollary~\ref{cor:der-H-Q} and Proposition~\ref{prop:precise-der-period}), $\rd_H f$ is replaced by $\rd_{Q_T} f$ and $\rd_s \varphi(\tau,r)$ is replaced by $\rd_r\varphi(\tau,r) - \rd_r\varphi(T,r)$. Notice that by Theorem~\ref{the:boot-f} and Theorem~\ref{the:boot-top-plus-1}, the estimates for (the derivatives of) $\rd_{Q_T} f(\tau)$ is one power of $\tau$ better than that for (the derivatives of) $\rd_{H} f(\tau)$. On the other hand, comparing bootstrap assumptions \eqref{eq:varphi_diff} with \eqref{eq:varphi_t}, we see that the estimates for $\rd_r\varphi(\tau,r) - \rd_r\varphi(T,r)$ is one power of $\tau$ worse than that of $\rd_s \varphi(\tau,r)$. These two powers of $\tau$ cancel out and these two terms can otherwise be controlled in exactly the same manner as in the proof of Proposition~\ref{prop:T12}. \qedhere 
\end{proof}

\begin{proposition}\label{prop:T13}
For $i \in \{1,2,3\}$, the terms $\calT_{i,3}$ and $\calT_{i,4}$ (defined in \eqref{eq:Ti3} and \eqref{eq:Ti4}, respectively) satisfy the following estimates for all $t \in [0,T]$:
$$ \sup_{r\in [\f{\mathfrak l_1}{2}, \f{2}{\mathfrak h}]} \Big( |\rd_r^I \calT_{i,3}|(t,r) + |\rd_r^I \calT_{i,4}|(t,r)\Big) \ls \de^{7/4} \ep \brk{t}^{-\min\{N-I+2,N\}},\quad I \leq N.$$
\end{proposition}
\begin{proof}
We consider $\calT_{1,3}$ and $\calT_{1,4}$ as typical terms. The other terms are similar to them just as $\calT_{i,2}$ for $i \in \{2,3\}$ are similar to $\calT_{1,2}$ (see Proposition~\ref{prop:T22}). To handle the terms $\calT_{1,3}$ and $\calT_{1,4}$, we in turn only need to point out the difference with the term $\calT_{1,2}$ in the proof of Proposition~\ref{prop:T12}. 

\pfstep{Step~1: Estimates for $\calT_{1,3}$} For the term $\calT_{1,3}$, there is an extra factor of $(t-\tau) \rd_s \varphi$ and $\Omg$ is replaced by $\rd_X \Omg$. We can then repeat the argument in Proposition~\ref{prop:T12} since the $\ep$ in the estimates for $\rd_s \varphi$ and its derivatives (see \eqref{eq:varphi_t}) can cancel out with the growth of $t-\tau$ (and in fact $\rd_s \varphi$ provides additional $\tau$ decay which we need not use here). The derivative $\rd_X\Omg$ also does not pose any additional difficulties since in the argument we would at most differentiate this term $N$ times.

\pfstep{Step~2: Estimates for $\calT_{1,4}$} When compared with $\calT_{1,2}$, the term $\calT_{1,4}$ also has a favorable factor of $\rd_s \varphi$ (like $\calT_{1,3}$), and in addition it does not have the factor $k$. However, the main new difficulty is that we have $\rd_{H} (\widehat{R}_i)_k$ instead of $(\widehat{R}_i)_k$. 

As in Proposition~\ref{prop:T12}, we will control the term $\calT_{1,4}$ but without the outer $r_1$, $r_2$ integrals. The goal is to obtain an estimate $\de^{7/4} \ep \brk{t}^{-N+I}$. As in the proof of Proposition~\ref{prop:T12}, we divide the $t$-integral into the subintegrals over the regions $\tau \in [\f t2,t]$ and $\tau \in [0,\f t2]$. For $\tau \in [\f t2,t]$, the argument is exactly the same as Steps~1 and 2 of Proposition~\ref{prop:T12}. Indeed, in the case at hand, the integral over $\tau \in [\f t2,t]$ can be bounded above by 
\begin{equation}
\ep \sum_{i_1'+i_1''+i_2'+i_2''+i_3'\leq I+1} \sum_{k\in \Z\backslash\{0\}}\sum_{\ell\in\Z}  \int_{t/2}^{t} \brk{\tau}^{-2} \mathfrak F^{(k,\ell)}_{i_1',i_1'',i_2',i_2'',i_3'}(\tau) \, \ud \tau
\end{equation}
This is similar to \eqref{eq:varphi-t-der.1} (and even has additional smallness and $\tau$-decay from $\rd_s \varphi$). Just as for $\calT_{1,2}$, the total number of derivatives can be up to $I+1$; in $\calT_{1,2}$ this is due to the factor of $k$, while for $\calT_{1,4}$ here this is due to the $\rd_H$ derivative.

However, the integral over $\tau \in [0,\f t2]$ is slightly different (even though it is in fact much easier here and does not require separate cases comparing $|\ell\tau|$ and $|kt|$, etc.). Recall that in the proof of Proposition~\ref{prop:T12}, we first bound the integral by \eqref{eq:varphi-t-der.2}, and then in Steps~3 and 4, we used that we could cancel off the $|k|$ power using the $|k|^{-1}$ power from $\brk{k(t-\tau)}^{-1}$. In the case at hand, we can no longer carry this out because we have an $\rd_H$ derivative instead of $|k|$ power. We slightly modify the argument and use $\eta_1 = -k$, $N_1 = N - I -2$. This creates one fewer power of $\brk{t}$ decay, but we compensate it with the extra power of $\ep$ coming from $\rd_s \varphi$. More precisely, after using $\eta_1$ and $N_1$ as above and arguing as in\footnote{Note that we do not directly apply the statement of Lemma~\ref{lem:main-terms-den}, but the only small modifications to be made are to take into account the $\rd_H$ derivative and the extra factor of $\rd_s \varphi$, which decays $\ep \brk{\tau}^{-2}$ (by \eqref{eq:varphi_t}).} Lemma~\ref{lem:main-terms-den}, we need to bound
\begin{equation}\label{eq:T14.main.term}
\begin{split}
&\: \sum_{k\in \Z\backslash\{0\}}\sum_{\ell\in\Z}\sum_{\substack{j_2''\leq i_2'' \\i_1'+i_2'+i_3'+i_1''+i_2'' \leq I+1\\K_1'+K_2'+K_3'+K_1''+K_2''+N_1''\leq N-I-2 }} \int_0^{t/2} \ep \brk{\tau}^{-2} \jap{k(t-\tau)}^{-(N-I-2)}\\
&\: \hspace{14em}\times \sup_{(H,M)\in\mathcal S} |\widehat{(\partial_{Q_T}^{i_1'+K_1'}\partial_M^{i_2'+K_2'}\partial_H^{i_3'+K_3'} \rd_s \varphi)}_{\ell}|(\tau) \\
&\: \hspace{16em}\times \sup_{(H,M)\in\mathcal S}|Y_{k,-k\tau}^{N_1''}[\widehat{(\partial_{Q_T}^{i_1''+K_1''}\partial_M^{j_2''+K_2''}\partial_H f)}_{k-\ell}(\tau)]| \ud \tau.
\end{split}
\end{equation}
Note that in $i_1'+i_2'+i_3'+i_1''+i_2'' \leq I+1$, we have taken into account the extra $\rd_H$ derivative.

We use the rough bound $\jap{k(t-\tau)}^{-(N-I-2)} \ls \brk{t}^{-(N-I-2)}$ and Lemma~\ref{lem:Plancherel} to control \eqref{eq:T14.main.term}. Noting that Lemma~\ref{lem:Plancherel} costs one $\rd_{Q_T}$ derivative on each factor. Thus there are at most $I+2+ (N-I-2)+1 = N+1$ $\rd_r$ derivatives on $\rd_s \varphi$ and that there are at most $I+1+(N-I-2)+1+1 = N+1$ derivatives on $f$. Notice moreover that these terms cannot both have the top number of derivatives. Thus, after relabelling the indices, we  obtain
\begin{equation*}
\begin{split}
\hbox{\eqref{eq:T14.main.term}} \ls &\: \brk{t}^{-(N-I-2)}\int_0^{t/2} \ep \brk{\tau}^{-2} \\
&\: \qquad \qquad \times \sum_{\substack{ j' \leq i_1'+i_2'+i_3'+1 \\ j_1''+j_2'' \leq i_1''+i_2''+1 \\ i_1'+i_2'+i_3'+i_1''+i_2''\leq I+1 \\ K'+K_1''+K_2''+N_1'' \leq N-I-1}}  \sup_{r\in [\f{\mathfrak l_1}{2}, \f{2}{\mathfrak h}]} |\rd_r^{j'+K'} \rd_s \varphi|  \| \rd_{Q_T}^{j_1''+K_1''} \rd_M^{j_2''+K_2''}\rd_H^{N_1''+1} f \|_{L^\i}(\tau) \ud \tau. 
\end{split}
\end{equation*}
Notice that it is possible to have $\rd_r^{N+1} \rd_s \varphi$, which has one too many derivatives for the bootstrap assumption \eqref{eq:varphi_t}. To handle this, we split into two cases. First, we restrict to the sum where $j' + K' \leq N$. For this, we use \eqref{eq:varphi_t} for $\rd_r^{j'+K'}\rd_s\varphi$ and Theorem~\ref{the:boot-f}, Theorem~\ref{the:boot-top-plus-1} for $\rd_{Q_T}^{j_1''+K_1''} \rd_M^{j_2''+K_2''}\rd_H^{N_1''+1} f$ to bound the terms by
\begin{equation*}
\begin{split}
&\: \brk{t}^{-(N-I-2)} \sum_{\substack{ j' \leq i_1'+i_2'+i_3'+1 \\ j_1''+j_2'' \leq i_1''+i_2''+1 \\ i_1'+i_2'+i_3'+i_1''+i_2''\leq I+1\\ K'+K_1''+K_2''+N_1'' \leq N-I-1 \\ j'+K'\leq N}}\int_0^{t/2} \ep \brk{\tau}^{-2} (\de^{3/4} \ep \brk{\tau}^{-\min\{N-(j'+K')+2 , N\}}) (\de \ep \brk{\tau}^{N_1''+2}) \, \ud \tau \\
\ls &\: \brk{t}^{-(N-I-2)} \int_0^{t/2} \max\{ \de^{7/4} \ep^3 \brk{\tau}^{-N + (I+1) +(N-I-1)- 2}, \de^{7/4} \ep^3 \brk{\tau}^{-N+(N-I-1)} \} \, \ud \tau \\
\ls &\: \brk{t}^{-(N-I-2)} \de^{7/4}\ep^3 \int_0^{t/2} \brk{\tau}^{-1} \, \ud \tau \ls \de^{7/4}\ep^3 \brk{t}^{-(N-I-2)}\brk{\log \brk{t}} \ls \de^{7/4} \ep \brk{t}^{-N+I},
\end{split}
\end{equation*}
where in the final inequality, we used $\ep^2 \brk{t}^2 \brk{\log \brk{t}} \ls 1$.

When $j' + K' = N+1$, it must hold that $I=N-2$, $i_1''=i_2''=K_1''=K_2''=N_1''=0$. In this case, the term can be bounded using Proposition~\ref{prop:phi-t-by-f}  and Theorem~\ref{the:boot-f} by
\begin{equation*}
\begin{split}
&\: \brk{t}^{-(N-I-2)} \int_0^{t/2} \ep \brk{\tau}^{-2} \sum_{j_1''+j_2'' \leq 1} \sup_{r\in [\f{\mathfrak l_1}{2}, \f{2}{\mathfrak h}]} |\rd_r^{N+1} \rd_s \varphi|  \| \rd_{Q_T}^{j_1''} \rd_M^{j_2''}\rd_H f \|_{L^\i}(\tau)  \, \ud \tau \\
\ls &\: \brk{t}^{-(N-I-2)} \int_0^{t/2} \ep \brk{\tau}^{-2} (\de \ep)  (\de \ep \brk{\tau}) \, \ud \tau \\
\ls &\: \brk{t}^{-(N-I-2)} \de^2 \ep^3 \int_0^{t/2} \brk{\tau}^{-1} \, \ud \tau \ls \de^{7/4}\ep^3 \brk{t}^{-(N-I-2)}\brk{\log \brk{t}} \ls \de^{7/4} \ep \brk{t}^{-N+I},
\end{split}
\end{equation*}
where as before we used $\ep^2 \brk{t}^2 \brk{\log \brk{t}} \ls 1$. \qedhere
\end{proof}

\begin{proposition}\label{prop:D12}
For $I \leq N-1$ and $\calR_i$ as in \eqref{eq:R1.def}--\eqref{eq:R3.def}, the following estimate holds for all $\mathfrak t \in [0,T]$:
\begin{equation}\label{eq:D12}
\Big| \rd_r^{I}\sum_{k\in \mathbb Z\setminus\{0\}} \int_{-\infty}^\infty\int_{0}^\infty\int_0^{\mathfrak t} e^{ikQ_T} e^{-ik\Omega\cdot (\mathfrak t -\tau)} \widehat{(\calR_i)}_{k}(\tau,H,M)  \ud\tau\d L\d w \Big| \ls \de^{7/4} \ep \brk{t}^{-N+I}.
\end{equation}
Consequently, for $i \in \{1,2,3\}$, the term $\calD_{i,2}(\mathfrak t,r)$ (defined in \eqref{eq:Di2}) satisfy the following estimates for all $\mathfrak t \in [0,T]$:
$$ \sup_{r\in [\f{\mathfrak l_1}{2}, \f{2}{\mathfrak h}]}  |\rd_r^I \calD_{i,2}|(\mathfrak t,r) \ls \de^{7/4} \ep \brk{t}^{-\min\{N-I+2,N\}},\quad I \leq N+1.$$
\end{proposition}
\begin{proof}
We will only consider $\calD_{1,2}(\mathfrak t,r)$. The other terms $\calD_{1,2}(\mathfrak t,r)$ and $\calD_{1,3}(\mathfrak t,r)$ are similar to $\calD_{1,2}(\mathfrak t,r)$ after considerations similar to those in the proof of Proposition~\ref{prop:T22}. 

Again, as in Proposition~\ref{prop:T13}, we simply compare the term $\calD_{1,2}$ with the term $\calT_{1,2}$ in Proposition~\ref{prop:T12}. The difference is that the term $\calD_{1,2}$ has one fewer factor of $k\Omg$ so the estimates \eqref{eq:D12} for $I \leq N-2$ follow exactly as in Proposition~\ref{prop:T12} except for applying Lemma~\ref{lem:main-terms-den} with $\alp = 0$. However, we also need a new top-order estimates, i.e., for the term \eqref{eq:D12}, we need up to $I\leq N-1$.

Since in Steps~1 and 2 of Proposition~\ref{prop:T12} (when we controlled \eqref{eq:varphi-t-der.1}), we simply distributed the $k$ factor as derivatives, the analogous integral for $\tau \in [t/2,t]$ can be controlled in a similar manner. 

It remains to consider the $\tau \in [0,t/2]$ integral with the highest $I=N-1$ derivatives. In this case, we choose $N_1 = N_2 = 0$. By Lemma~\ref{lem:main-terms-den}, we need to bound
 \begin{equation}\label{eq:varphi-t-der.case.4.new}
\begin{split}
&\: \sum_{k\in \Z\backslash\{0\}}\sum_{\ell\in\Z}\sum_{\substack{j_2''\leq i_2''}} \int_0^{t/2}  \sup_{(H,M)\in\mathcal S} |\widehat{(\partial_{Q_T}^{i_1'}\partial_M^{i_2'}\partial_H^{i_3'} \rd_s \varphi)}_{\ell}|(\tau) |\widehat{(\partial_{Q_T}^{i_1''}\partial_M^{j_2''}\partial_H f)}_{k-\ell}|(\tau) \ud \tau.
\end{split}
\end{equation}
Using Lemma~\ref{lem:Plancherel}, we need to control up to $N$ derivatives of $\rd_s \varphi$ and $N$ derivatives of $\rd_H f$. Moreover, the derivatives must distribute so that both factors cannot simultaneously have the top number of derivatives. Hence, we can bound \eqref{eq:varphi-t-der.case.4.new} as follows, using \eqref{eq:varphi_t}, Theorem~\ref{the:boot-f} and Theorem~\ref{the:boot-top-plus-1}:
 \begin{equation*}
\begin{split}
\hbox{\eqref{eq:varphi-t-der.case.4.new}} \ls &\: \sum_{\substack{j' \leq N-1 \\ j_1'' + j_2'' \leq N}} \int_0^{t/2}  \Big( \sup_{r\in [\f{\mathfrak l_1}2, \f 2{\mathfrak h}]} | \rd_r^{j'} \rd_s \varphi |(\tau) \Big) \|\partial_{Q_T}^{j_1''}\partial_M^{j_2''}\partial_H f\|_{L^\i}(\tau) \ud \tau \\
&\: + \sum_{\substack{j' \leq N \\ j_1'' + j_2'' \leq N-1}} \int_0^{t/2}  \Big( \sup_{r\in [\f{\mathfrak l_1}2, \f 2{\mathfrak h}]} | \rd_r^{j'} \rd_s \varphi |(\tau) \Big) \|\partial_{Q_T}^{j_1''}\partial_M^{j_2''}\partial_H f\|_{L^\i}(\tau) \ud \tau \\
\ls &\: \int_0^{t/2} \Big[ (\de^{3/4} \ep \brk{\tau}^{-3}) (\de \ep \brk{\tau}^2) + (\de^{3/4} \ep \brk{\tau}^{-2}) (\de \ep \brk{\tau}) \Big]\, \ud \tau \ls \de^{7/4} \ep^2 \brk{\log\brk{t}} \ls \de^{7/4} \ep \brk{t}^{-1},
\end{split}
\end{equation*}
where at the end we used $\ep \brk{t} \brk{\log\brk{t}} \ls 1$. \qedhere

\end{proof}

\subsection{Putting everything together}\label{sec:density.everything}
 We now combine the above estimates to conclude the proof of Theorem~\ref{thm:boot}.
\begin{proof}[Proof of Theorem~\ref{thm:boot}] 
Recall that the needed improvement of \eqref{eq:varphi_r} has already been achieved in Corollary~\ref{cor:boot-r-der-phi}. It thus remains to improve the bootstrap assumptions \eqref{eq:varphi_t} and \eqref{eq:varphi_diff}.

For \eqref{eq:varphi_t}, we recall the expression for $\rd_s \varphi$ given by \eqref{eq:phi-rep}, \eqref{eq:L-t} and \eqref{eq:N-t}. The $\calL_s$ term is bounded by in Proposition~\ref{prop:linear-est-den.2}, while all the $\calT_{i,j}$ contribution to the nonlinear term $\calN_s$ can be bounded using the combination of Propositions~\ref{prop:T11}, \ref{prop:T12}, \ref{prop:T22} and \ref{prop:T13}. Altogether, we obtain 
$$\sup_{r\in [\f{\mathfrak l_1}{2}, \f{2}{\mathfrak h}]} |\rd_r^I \rd_s \varphi|(t,r) \ls \de \ep \brk{t}^{-\min\{N-I+2,N\}},\quad I \leq N,$$
as desired.

For \eqref{eq:varphi_diff}, we recall the expression for $\varphi(t,r) - \varphi(T,r)$ given in \eqref{eq:phi-rep}, \eqref{eq:L-diff} and \eqref{eq:N-diff}. The contribution from the linear term is bounded by Proposition~\ref{prop:linear-est-den.1}, and the nonlinear contribution is bounded by Propositions~\ref{prop:D11} and \ref{prop:D12}. Altogether, we obtain
$$\sup_{r\in [\f{\mathfrak l_1}{2}, \f{2}{\mathfrak h}]} |\rd_r^I (\varphi(t,r) - \varphi(T,r))| \ls \de \ep \brk{t}^{-\min\{N-I+2,N\}},\quad I \leq N+1,$$
as desired. \qedhere
\end{proof}
\section{Proof of Theorem~\ref{thm:main}: Putting everything together}\label{sec:continuity}

\begin{proof}[Proof of Theorem~\ref{thm:main}]
\pfstep{Step~1: Existence of solution} We first show that the solution exists up to time $T_{\mathrm{final}}= \ep^{-1}(\log \f 1{\ep})^{-2}$ and that \eqref{eq:varphi_t}, \eqref{eq:varphi_r} and \eqref{eq:varphi_diff} all hold. This is a standard continuity argument (in particular because the bootstrap assumptions \eqref{eq:varphi_t}, \eqref{eq:varphi_r} and \eqref{eq:varphi_diff} are given in the original coordinate system), but since \eqref{eq:varphi_diff} involves two times $s_1$ and $s_2$, we give a complete proof.

Define
\begin{equation*}
\begin{split}
B:= \{T_{\text{B}} \in ( 0,T_{\mathrm{final}}]: &\: \hbox{solution exists in $[0,T_{\text{B}})$ and} \\
&\: \hbox{\eqref{eq:varphi_t}, \eqref{eq:varphi_r} and \eqref{eq:varphi_diff} hold for $s \in [0,T_{\text{B}})$ and $0 \leq s_1 \leq s_2 < T_{\text{B}}$} \}.
\end{split}
\end{equation*} 
We show that $B$ is nonempty, open and closed (in the relative topology of $( 0,T_{\mathrm{final}}]$), and hence it must be $( 0,T_{\mathrm{final}}]$.
\begin{enumerate}
\item By local existence, $B$ must consist of some small time and thus $B$ is nonempty.\\
\item The closedness of $B$ follows just by continuity of the various quantities that show up in the bootstrap assumtions.\\
\item Finally, we turn to openness of $B$. Let $T_{\text{B}}\in B$. By Theorem~\ref{the:boot-f} and Theorem~\ref{the:boot-top-plus-1}, and the bounds on the change of coordinates in Lemma~\ref{lem:Jac}, we know that up to $N+1$ derivatives of $f$ are bounded. Moreover by Lemma~\ref{lem:support}, the support of $f$ remains compact in phase space and bounded away from $r=0$. By standard local existence, we can then extend the solution up to some time $T'_B>T_{\text{B}}$. 

By Theorem ~\ref{thm:boot}, we can improve the bounds for \eqref{eq:varphi_t}, \eqref{eq:varphi_r} and \eqref{eq:varphi_diff} from $\delta^{3/4}$ to $C\delta$ on $[0,T_{\text{B}})$. Taking $T'_B$ closer to $T_{\text{B}}$ if necessary, it follows from the improved bounds and continuity that \eqref{eq:varphi_t} and \eqref{eq:varphi_r} hold up to time $T_{\text{B}}'$. 

It remains to prove that \eqref{eq:varphi_diff} holds. Relabel $(s_1,s_2)$ as $(s,T)$ and assume that $0\leq s<T<T_{\text{B}}'.$ We consider different cases depending on $s$ and $T$.
\begin{itemize}
\item We first consider the case $T_{\text{B}}\leq s<T<T_{\text{B}}'.$ In this case, taking $T_{\text{B}}'$ closer to $T_{\text{B}}$ if necessary, the desired estimate follows from continuity.
\item Next we consider $s< T_{\text{B}}\leq T<T_{\text{B}}'$. Using the triangle inequality, we obtain
\begin{align*}
&\: \sup_{r\in [\f{\mathfrak l_1}{2}, \f{2}{\mathfrak h}]} |\partial_r^{k}(\varphi(s,r)-\varphi(T,r))| \\
\leq &\: \sup_{r\in [\f{\mathfrak l_1}{2}, \f{2}{\mathfrak h}]} |\partial_r^{k}(\varphi(s,r)-\varphi(T_{\text{B}},r))|+ \sup_{r\in [\f{\mathfrak l_1}{2}, \f{2}{\mathfrak h}]} |\partial_r^{k}(\varphi(T_{\text{B}},r)-\varphi(T,r))|.
\end{align*}
The first term is $\ls \delta\epsilon \jap{s}^{-\min\{N-k+2,N\}}$ by Theorem ~\ref{thm:boot}, and the second term can be made arbitrarily small by taking $T_{\text{B}}'$ closer to $T_{\text{B}}$. Hence the term in total is $\leq \delta^{3/4} \epsilon \jap{s}^{-\min\{N-k+2,N\}}$.
\item Finally, the case $s<T< T_{\text{B}}$ is trivial.
\end{itemize}
\end{enumerate}
Hence, $B=( 0,T_{\mathrm{final}}]$. In particular, the solution exists up to time $T_{\mathrm{final}}$ and the bounds \eqref{eq:varphi_t}, \eqref{eq:varphi_r} and \eqref{eq:varphi_diff} all hold.

\pfstep{Step~2: Proof of the estimates} Using the estimates for $\varphi$, $\rd_s \varphi$ and their derivatives obtained in Theorem~\ref{thm:boot}, we can induct in $|\alp|+|\bt|$ to obtain \eqref{eq:thm.very.weak}. (Note that the estimates for $\varphi$, $\rd_s \varphi$ are given in the original $(s,r)$ coordinates, independent of the dynamical action angle variables that we constructed.) 

Finally, the estimates \eqref{eq:thm.boundedness} and \eqref{eq:thm.decay} follow from Theorem~\ref{the:boot-f} and Theorem~\ref{thm:boot}, respectively. \qedhere
\end{proof}

\bibliographystyle{plain}
\bibliography{VPL}

\def\cprime{$'$} \def\cprime{$'$} \def\cprime{$'$}
\begin{thebibliography}{10}

\bibitem{lApBjJ2018}
L.~Andersson, P.~Blue, and J.~Joudioux.
\newblock Hidden symmetries and decay for the {V}lasov equation on the {K}err
  spacetime.
\newblock {\em Comm. Partial Differential Equations}, 43(1):47--65, 2018.

\bibitem{vaA1960}
V.~A. Antonov.
\newblock Remarks on the problems of stability in stellar dynamics.
\newblock {\em Soviet Astronom. AJ}, 4:859--867, 1960.

\bibitem{Arnold2013}
Vladimir~Igorevich Arnol'd.
\newblock {\em Mathematical methods of classical mechanics}, volume~60.
\newblock Springer Science \& Business Media, 2013.

\bibitem{BaDe85}
C.~Bardos and P.~Degond.
\newblock Global existence for the {V}lasov-{P}oisson equation in {$3$} space
  variables with small initial data.
\newblock {\em Ann. Inst. H. Poincar\'e Anal. Non Lin\'eaire}, 2(2):101--118,
  1985.

\bibitem{jB2017}
Jacob Bedrossian.
\newblock Suppression of plasma echoes and {L}andau damping in {S}obolev spaces
  by weak collisions in a {V}lasov-{F}okker-{P}lanck equation.
\newblock {\em Ann. PDE}, 3(2):Paper No. 19, 66, 2017.

\bibitem{jB2021}
Jacob Bedrossian.
\newblock Nonlinear echoes and {L}andau damping with insufficient regularity.
\newblock {\em Tunis. J. Math.}, 3(1):121--205, 2021.

\bibitem{jB2022}
Jacob Bedrossian.
\newblock A brief introduction to the mathematics of {L}andau damping.
\newblock {\em arXiv:2211.13707, preprint}, 2022.

\bibitem{jBnMcM2016}
Jacob Bedrossian, Nader Masmoudi, and Cl\'{e}ment Mouhot.
\newblock Landau damping: paraproducts and {G}evrey regularity.
\newblock {\em Ann. PDE}, 2(1):Art. 4, 71, 2016.

\bibitem{jBnMcM2018}
Jacob Bedrossian, Nader Masmoudi, and Cl\'{e}ment Mouhot.
\newblock Landau damping in finite regularity for unconfined systems with
  screened interactions.
\newblock {\em Comm. Pure Appl. Math.}, 71(3):537--576, 2018.

\bibitem{jBnMcM2020}
Jacob Bedrossian, Nader Masmoudi, and Cl\'{e}ment Mouhot.
\newblock Linearized wave-damping structure of {V}lasov--{P}oisson in {$\mathbb
  R^3$}.
\newblock {\em arXiv:2007.08580, preprint}, 2020.

\bibitem{lB2023}
L\'{e}o Bigorgne.
\newblock Decay estimates for the massless {V}lasov equation on {S}chwarzschild
  spacetimes.
\newblock {\em Ann. Henri Poincar\'{e}}, 24(11):3763--3836, 2023.

\bibitem{lBrVR2024}
L\'eo Bigorgne and Renato Velozo~Ruiz.
\newblock Late-time asymptotics of small data solutions for the
  {V}lasov--{P}oisson system.
\newblock {\em arXiv:2404.05812, preprint}, 2024.

\bibitem{Velozo.Ruiz.brothers+}
L\'eo Bigorne, Anibal Velozo~Ruiz, and Renato Velozo~Ruiz.
\newblock Modified scattering of small data solutions to the
  {V}lasov--{P}oisson system with a trapping potential.
\newblock {\em arXiv:2310.17424 preprint}, 2023.

\bibitem{BinTre2011}
James Binney and Scott Tremaine.
\newblock {\em Galactic dynamics}, volume~13.
\newblock Princeton university press, 2011.

\bibitem{eCcM1998}
E.~Caglioti and C.~Maffei.
\newblock Time asymptotics for solutions of {V}lasov-{P}oisson equation in a
  circle.
\newblock {\em J. Statist. Phys.}, 92(1-2):301--323, 1998.

\bibitem{sCjL2022}
Sanchit Chaturvedi and Jonathan Luk.
\newblock Phase mixing for solutions to 1{D} transport equation in a confining
  potential.
\newblock {\em Kinet. Relat. Models}, 15(3):403--416, 2022.

\bibitem{VPL}
Sanchit Chaturvedi, Jonathan Luk, and Toan Nguyen.
\newblock The {V}lasov-{P}oisson-{L}andau system in the weakly collisional
  regime.
\newblock {\em J. Amer. Math. Soc.}, 36(4):1103--1189, 2023.

\bibitem{jCxwZjbW2015}
Jing Chen, Xianwen Zhang, and Jinbo Wei.
\newblock Global weak solutions for the {V}lasov-{P}oisson system with a point
  charge.
\newblock {\em Math. Methods Appl. Sci.}, 38(17):3776--3791, 2015.

\bibitem{shCssK2016}
Sun-Ho Choi and Soonsik Kwon.
\newblock Modified scattering for the {V}lasov-{P}oisson system.
\newblock {\em Nonlinearity}, 29(9):2755--2774, 2016.

\bibitem{mD2005}
Mihalis Dafermos.
\newblock Spherically symmetric spacetimes with a trapped surface.
\newblock {\em Classical Quantum Gravity}, 22(11):2221--2232, 2005.

\bibitem{mDgH2006}
Mihalis Dafermos and Gustav Holzegel.
\newblock On the nonlinear stability of higher dimensional triaxial
  {B}ianchi-{IX} black holes.
\newblock {\em Adv. Theor. Math. Phys.}, 10(4):503--523, 2006.

\bibitem{mDaR2005}
Mihalis Dafermos and Alan~D. Rendall.
\newblock An extension principle for the {E}instein-{V}lasov system in
  spherical symmetry.
\newblock {\em Ann. Henri Poincar\'{e}}, 6(6):1137--1155, 2005.

\bibitem{mDaR2016}
Mihalis Dafermos and Alan~D. Rendall.
\newblock Strong cosmic censorship for surface-symmetric cosmological
  spacetimes with collisionless matter.
\newblock {\em Comm. Pure Appl. Math.}, 69(5):815--908, 2016.

\bibitem{jDoSjS2004}
Jean Dolbeault, \'{O}scar S\'{a}nchez, and Juan Soler.
\newblock Asymptotic behaviour for the {V}lasov-{P}oisson system in the
  stellar-dynamics case.
\newblock {\em Arch. Ration. Mech. Anal.}, 171(3):301--327, 2004.

\bibitem{xlDuan2022}
Xianglong Duan.
\newblock Sharp decay estimates for the {V}lasov-{P}oisson and
  {V}lasov-{Y}ukawa systems with small data.
\newblock {\em Kinet. Relat. Models}, 15(1):119--146, 2022.

\bibitem{pFzmOYbPkW2023}
Patrick Flynn, Zhimeng Ouyang, Benoit Pausader, and Klaus Widmayer.
\newblock Scattering map for the {V}lasov-{P}oisson system.
\newblock {\em Peking Math. J.}, 6(2):365--392, 2023.

\bibitem{rGjS1994}
Robert Glassey and Jack Schaeffer.
\newblock Time decay for solutions to the linearized {V}lasov equation.
\newblock {\em Transport Theory Statist. Phys.}, 23(4):411--453, 1994.

\bibitem{rGjS1995}
Robert Glassey and Jack Schaeffer.
\newblock On time decay rates in {L}andau damping.
\newblock {\em Comm. Partial Differential Equations}, 20(3-4):647--676, 1995.

\bibitem{eGtNiR2020a}
Emmanuel Grenier, Toan~T. Nguyen, and Igor Rodnianski.
\newblock Landau damping for analytic and {G}evrey data.
\newblock {\em Math. Res. Lett.}, 28(6):1679--1702, 2021.

\bibitem{eGtNiR2020b}
Emmanuel Grenier, Toan~T. Nguyen, and Igor Rodnianski.
\newblock Plasma echoes near stable {P}enrose data.
\newblock {\em SIAM J. Math. Anal.}, 54(1):940--953, 2022.

\bibitem{yG1999}
Yan Guo.
\newblock Variational method for stable polytropic galaxies.
\newblock {\em Arch. Ration. Mech. Anal.}, 150(3):209--224, 1999.

\bibitem{yG2000}
Yan Guo.
\newblock On the generalized {A}ntonov stability criterion.
\newblock In {\em Nonlinear wave equations ({P}rovidence, {RI}, 1998)}, volume
  263 of {\em Contemp. Math.}, pages 85--107. Amer. Math. Soc., Providence, RI,
  2000.

\bibitem{yGzwL2008}
Yan Guo and Zhiwu Lin.
\newblock Unstable and stable galaxy models.
\newblock {\em Comm. Math. Phys.}, 279(3):789--813, 2008.

\bibitem{yGgR1999}
Yan Guo and Gerhard Rein.
\newblock Stable steady states in stellar dynamics.
\newblock {\em Arch. Ration. Mech. Anal.}, 147(3):225--243, 1999.

\bibitem{yGgR2001}
Yan Guo and Gerhard Rein.
\newblock Isotropic steady states in galactic dynamics.
\newblock {\em Comm. Math. Phys.}, 219(3):607--629, 2001.

\bibitem{yGgR2007}
Yan Guo and Gerhard Rein.
\newblock A non-variational approach to nonlinear stability in stellar dynamics
  applied to the {K}ing model.
\newblock {\em Comm. Math. Phys.}, 271(2):489--509, 2007.

\bibitem{mHgRmScS2024}
Mahir {Had\v zi\'c}, Gerhard Rein, Matthew Schrecker, and Christopher Straub.
\newblock Quantitative phase mixing for {H}amiltonians with trapping.
\newblock {\em arXiv:2405.17153, preprint}, 2024.

\bibitem{mHgRmScS2023}
Mahir Had\v{z}i\'{c}, Gerhard Rein, Matthew Schrecker, and Christopher Straub.
\newblock Damping versus oscillations for a gravitational {V}lasov--{P}oisson
  system.
\newblock {\em arXiv:2301.07662, preprint}, 2023.

\bibitem{mHgRcS2022}
Mahir Had\v{z}i\'{c}, Gerhard Rein, and Christopher Straub.
\newblock On the existence of linearly oscillating galaxies.
\newblock {\em Arch. Ration. Mech. Anal.}, 243(2):611--696, 2022.

\bibitem{mH2007}
Mahir Hadzic.
\newblock A constraint variational problem arising in stellar dynamics.
\newblock {\em Quart. Appl. Math.}, 65(1):145--153, 2007.

\bibitem{dHKttNfR2021b}
Daniel Han-Kwan, Toan~T. Nguyen, and Fr\'{e}d\'{e}ric Rousset.
\newblock Asymptotic stability of equilibria for screened {V}lasov-{P}oisson
  systems via pointwise dispersive estimates.
\newblock {\em Ann. PDE}, 7(2):Paper No. 18, 37, 2021.

\bibitem{dHKttNfR2021a}
Daniel Han-Kwan, Toan~T. Nguyen, and Fr\'{e}d\'{e}ric Rousset.
\newblock On the linearized {V}lasov-{P}oisson system on the whole space around
  stable homogeneous equilibria.
\newblock {\em Comm. Math. Phys.}, 387(3):1405--1440, 2021.

\bibitem{dHKttNfR2024}
Daniel Han-Kwan, Toan~T. Nguyen, and Fr\'{e}d\'{e}ric Rousset.
\newblock Linear {L}andau damping for the {V}lasov--{M}axwell system in
  {$\mathbb R^3$}.
\newblock {\em arXiv:2402.11402, preprint}, 2024.

\bibitem{hjHaRjjlV2011}
Hyung~Ju Hwang, Alan Rendall, and Juan J.~L. Vel\'{a}zquez.
\newblock Optimal gradient estimates and asymptotic behaviour for the
  {V}lasov-{P}oisson system with small initial data.
\newblock {\em Arch. Ration. Mech. Anal.}, 200(1):313--360, 2011.

\bibitem{hjHjjlW2009}
Hyung~Ju Hwang and Juan J.~L. Vel\'{a}zquez.
\newblock On the existence of exponentially decreasing solutions of the
  nonlinear {L}andau damping problem.
\newblock {\em Indiana Univ. Math. J.}, 58(6):2623--2660, 2009.

\bibitem{adIbPxcWkW2023}
A.~D. Ionescu, B.~Pausader, X.~Wang, and K.~Widmayer.
\newblock On the stability of homogeneous equilibria in the {V}lasov-{P}oisson
  system on {$\Bbb R^3$}.
\newblock {\em Classical Quantum Gravity}, 40(18):Paper No. 185007, 32, 2023.

\bibitem{adIbPxcWkW2024b}
A.~D. Ionescu, B.~Pausader, X.~Wang, and K.~Widmayer.
\newblock Nonlinear {L}andau damping and wave operators in sharp {G}evrey
  spaces.
\newblock {\em arXiv:2405.04473, preprint}, 2024.

\bibitem{aIbPxcWkW2022}
Alexandru~D. Ionescu, Benoit Pausader, Xuecheng Wang, and Klaus Widmayer.
\newblock On the asymptotic behavior of solutions to the {V}lasov-{P}oisson
  system.
\newblock {\em Int. Math. Res. Not. IMRN}, (12):8865--8889, 2022.

\bibitem{adIbPxcWkW2024}
Alexandru~D. Ionescu, Benoit Pausader, Xuecheng Wang, and Klaus Widmayer.
\newblock Nonlinear {L}andau damping for the {V}lasov-{P}oisson system in
  {$\Bbb R^3$}: the {P}oisson equilibrium.
\newblock {\em Ann. PDE}, 10(1):Paper No. 2, 78, 2024.

\bibitem{feJ2021}
Fatima~Ezzahra Jabiri.
\newblock Static spherically symmetric {E}instein-{V}lasov bifurcations of the
  {S}chwarzschild spacetime.
\newblock {\em Ann. Henri Poincar\'{e}}, 22(7):2355--2406, 2021.

\bibitem{mK2021}
Markus Kunze.
\newblock {\em A {B}irman-{S}chwinger principle in galactic dynamics},
  volume~77 of {\em Progress in Mathematical Physics}.
\newblock Birkh\"{a}user/Springer, Cham, 2021.

\bibitem{Landau1946}
Lev~Davidovich Landau.
\newblock On the vibrations of the electronic plasma.
\newblock {\em Zh. Eksp. Teor. Fiz.}, 10:25, 1946.

\bibitem{mLfMpR2008}
Mohammed Lemou, Florian M\'{e}hats, and Pierre Raphael.
\newblock The orbital stability of the ground states and the singularity
  formation for the gravitational {V}lasov {P}oisson system.
\newblock {\em Arch. Ration. Mech. Anal.}, 189(3):425--468, 2008.

\bibitem{mLfMpR2011}
Mohammed Lemou, Florian M\'{e}hats, and Pierre Rapha\"{e}l.
\newblock A new variational approach to the stability of gravitational systems.
\newblock {\em Comm. Math. Phys.}, 302(1):161--224, 2011.

\bibitem{mLfMpR2012}
Mohammed Lemou, Florian M\'{e}hats, and Pierre Rapha\"{e}l.
\newblock Orbital stability of spherical galactic models.
\newblock {\em Invent. Math.}, 187(1):145--194, 2012.

\bibitem{dLB1962}
D.~Lynden-Bell.
\newblock {The Stability and Vibrations of a Gas of Stars}.
\newblock {\em Monthly Notices of the Royal Astronomical Society},
  124(4):279--296, 04 1962.

\bibitem{dLB1967}
D.~Lynden-Bell.
\newblock {Statistical Mechanics of Violent Relaxation in Stellar Systems}.
\newblock {\em Monthly Notices of the Royal Astronomical Society},
  136(1):101--121, 05 1967.

\bibitem{sdM1990}
Samir~D. Mathur.
\newblock Existence of oscillation modes in collisionless gravitating systems.
\newblock {\em Monthly Notices Roy. Astronom. Soc.}, 243(4):529--536, 1990.

\bibitem{mMpRhVDB2022}
M.~Moreno, P.~Rioseco, and H.~Van Den~Bosch.
\newblock Mixing in anharmonic potential well.
\newblock {\em J. Math. Phys.}, 63(7):Paper No. 071502, 9, 2022.

\bibitem{mMpRhVDB2023}
Matias Moreno, Paola Rioseco, and Hanne Van~Den Bosch.
\newblock Spectrum of the linearized {V}lasov--{P}oisson equation around steady
  states from galactic dynamics.
\newblock {\em arXiv:2305.05749, preprint}, 2023.

\bibitem{cM2013}
Cl\'{e}ment Mouhot.
\newblock Stabilit\'{e} orbitale pour le syst\`eme de {V}lasov-{P}oisson
  gravitationnel (d'apr\`es {L}emou-{M}\'{e}hats-{R}apha\"{e}l, {G}uo, {L}in,
  {R}ein et al.).
\newblock Number 352, pages Exp. No. 1044, vii, 35--82. 2013.
\newblock S\'{e}minaire Bourbaki. Vol. 2011/2012. Expos\'{e}s 1043--1058.

\bibitem{cMcV2011}
Cl\'{e}ment Mouhot and C\'{e}dric Villani.
\newblock On {L}andau damping.
\newblock {\em Acta Math.}, 207(1):29--201, 2011.

\bibitem{pS2022}
Stephen Pankavich.
\newblock Scattering and asymptotic behavior of solutions to the
  {V}lasov--{P}oisson system in high dimension.
\newblock {\em arXiv:2201.09464, preprint}, 2022.

\bibitem{bPkW2021}
Benoit Pausader and Klaus Widmayer.
\newblock Stability of a point charge for the {V}lasov--{P}oisson system: the
  radial case.
\newblock {\em Comm. Math. Phys.}, 385(3):1741--1769, 2021.

\bibitem{bPkWjqY2022}
Benoit Pausader, Klaus Widmayer, and Jiaqi Yang.
\newblock Stability of a point charge for the repulsive {V}lasov--{P}oisson
  system.
\newblock {\em arXiv:2207.05644 preprint}, 2022.

\bibitem{gR1994}
Gerhard Rein.
\newblock Static solutions of the spherically symmetric {V}lasov-{E}instein
  system.
\newblock {\em Math. Proc. Cambridge Philos. Soc.}, 115(3):559--570, 1994.

\bibitem{pRoS2020}
Paola Rioseco and Olivier Sarbach.
\newblock Phase space mixing in an external gravitational central potential.
\newblock {\em Classical Quantum Gravity}, 37(19):195027, 42, 2020.

\bibitem{oSjS2006}
\'{O}scar S\'{a}nchez and Juan Soler.
\newblock Orbital stability for polytropic galaxies.
\newblock {\em Ann. Inst. H. Poincar\'{e} C Anal. Non Lin\'{e}aire},
  23(6):781--802, 2006.

\bibitem{vSmT2024}
Volker Schlue and Martin Taylor.
\newblock Inverse modified scattering and polyhomogeneous expansions for the
  {V}lasov--{P}oisson system.
\newblock {\em arXiv:2404.15885, preprint}, 2024.

\bibitem{Sm16}
Jacques Smulevici.
\newblock Small data solutions of the {V}lasov-{P}oisson system and the vector
  field method.
\newblock {\em Ann. PDE}, 2(2):Art. 11, 55, 2016.

\bibitem{cS2024}
Christopher Straub.
\newblock Numerical experiments on stationary, oscillating, and damped
  spherical galaxy models.
\newblock {\em arXiv:2405.01235, preprint}, 2024.

\bibitem{Straub.thesis}
Christopher Straub.
\newblock {\em Pulsating Galaxies}.
\newblock PhD thesis, March 2024.

\bibitem{Velozo.Ruiz.brothers}
Anibal Velozo~Ruiz and Renato Velozo~Ruiz.
\newblock Small data solutions for the {V}lasov--{P}oisson system with a
  trapping potential.
\newblock {\em arXiv:2304.12017 preprint}, 2023.

\bibitem{velozo_ruiz_2022}
Renato~Adolfo Velozo~Ruiz.
\newblock {\em Linear and non-linear collisionless many-particle systems}.
\newblock PhD thesis, Apollo - University of Cambridge Repository, 2022.

\bibitem{Wa18.2}
Xuecheng Wang.
\newblock Decay estimates for the 3{D} relativistic and non-relativistic
  {V}lasov--{P}oisson systems.
\newblock {\em arXiv:1805.10837, preprint}, 2018.

\bibitem{gW1999}
G.~Wolansky.
\newblock On nonlinear stability of polytropic galaxies.
\newblock {\em Ann. Inst. H. Poincar\'{e} C Anal. Non Lin\'{e}aire},
  16(1):15--48, 1999.

\bibitem{bY2016}
Brent Young.
\newblock Landau damping in relativistic plasmas.
\newblock {\em J. Math. Phys.}, 57(2):021502, 68, 2016.

\bibitem{cZ2021}
Christian Zillinger.
\newblock On echo chains in {L}andau damping: traveling wave-like solutions and
  {G}evrey 3 as a linear stability threshold.
\newblock {\em Ann. PDE}, 7(1):Paper No. 1, 29, 2021.

\end{thebibliography}
\end{document}